\documentclass{book}
\usepackage{amsmath,amsfonts,amsthm,amssymb,mathtools}

\usepackage{thmtools}
\usepackage{framed}   
\usepackage[usenames,dvipsnames]{xcolor}
\usepackage{fancyhdr,fancyvrb}
\usepackage{mathptmx,graphicx,listings}
\usepackage{makeidx}
\usepackage[linesnumbered,ruled]{algorithm2e}

\newenvironment{solution}[1][Solution]{\proof[#1]}{\endproof}
\DefineVerbatimEnvironment{mi3listing}{Verbatim}{fontsize=\footnotesize}
\DeclareMathOperator{\lcm}{lcm}
\DeclareMathOperator{\ord}{ord}

\newcommand{\floor}[1]{\lfloor #1 \rfloor}

\newcommand{\mn}[1]{\min{\{ #1 \}}}

\newcommand{\leg}[2]{\left(\frac{#1}{#2}\right)}
\newcommand{\dv}[2]{#1 | #2}
\newcommand{\ndv}[2]{#1\nmid#2}
\newcommand{\ndd}[2]{#1\nmid#2}  %2026 move it down and off

\newcommand{\mb}[1]{\mathbb{#1}}
\newcommand{\mf}[1]{\mathrm{#1}}

\definecolor{mgray}{RGB}{ 85, 85, 85}
\definecolor{mbrown}{RGB}{128,40, 0}    % Nice brown
\definecolor{mblue4}{RGB}{  0,  0,127}  % Navy Blue
\definecolor{mcrimson}{RGB}{175,30,30}  % Nice crimson
\definecolor{mgreen1}{RGB}{ 0, 90, 50}  % OK

\colorlet{head0}{Black}
\colorlet{headb4}{mblue4}
\colorlet{head7}{mcrimson}
\colorlet{headg1}{mgreen1}
\colorlet{rule9}{Gray}
\theoremstyle{plain}% default

 {\endMakeFramed}

 {\endMakeFramed}

 {\endMakeFramed}

\renewenvironment{leftbar}{%
 \MakeFramed {\advance\hsize-\width \FrameRestore}}%
 {\endMakeFramed}
\renewenvironment{leftbar}{%
 \MakeFramed {\advance\hsize-\width \FrameRestore}}%
 {\endMakeFramed}

\declaretheoremstyle[headfont=\sffamily\bfseries,%
 headfont=\bfseries,%
 notebraces={}{},%
 bodyfont=\sffamily\itshape,%
 notefont=\sffamily\bfseries,%
 headpunct={},%
 headformat=\color{headg1}\NAME~\NUMBER\hfill\NOTE\smallskip\linebreak,%
 preheadhook=\begin{leftbar},%
 postfoothook=\end{leftbar},%
 ]{mythm}
 \declaretheorem[name=Theorem,style=mythm, numberwithin=chapter]
 {thm}
\declaretheoremstyle[headfont=\sffamily\bfseries,%
 notefont=\sffamily\bfseries,%
 notebraces={}{},%
 headpunct=,%
 bodyfont=\normalfont,
 headformat=\color{headb4}\NAME~\NUMBER\hfill\NOTE\smallskip\linebreak,%
 preheadhook=\begin{leftbar},%
 postfoothook=\end{leftbar},%
 ]{mylem}
 \declaretheorem[name=Lemma,style=mylem, numberwithin=chapter]
 {lem}
\declaretheoremstyle[headfont=\sffamily\bfseries,%
 notefont=\sffamily\bfseries,%
 notebraces={}{},%
 headpunct=,%
 bodyfont=\normalfont,
 headformat=\color{headb4}\NAME~\NUMBER\hfill\NOTE\smallskip\linebreak,%
 preheadhook=\begin{leftbar},%
 postfoothook=\end{leftbar},%
 ]{myprp}
 \declaretheorem[name=Proposition,style=myprp, numberwithin=chapter]
 {prp}

\declaretheoremstyle[headfont=\sffamily\bfseries,%
 notefont=\sffamily\bfseries,%
 notebraces={}{},%
 headpunct=,%
 bodyfont=\normalfont,
 headformat=\color{headb4}\NAME~\NUMBER\hfill\NOTE\smallskip\linebreak,%
 preheadhook=\begin{leftbar},%
 postfoothook=\end{leftbar},%
 ]{mycor}
 \declaretheorem[name=Corollary,style=mycor, numberwithin=chapter]
 {cor}

\declaretheoremstyle[headfont=\sffamily\bfseries,%
 notefont=\sffamily\bfseries,%
 notebraces={}{},%
 headpunct=,%
 bodyfont=\normalfont,
 headformat=\color{head7}\NAME~\NUMBER\hfill\NOTE\smallskip\linebreak,%
 preheadhook=\begin{leftbar},%
 postfoothook=\end{leftbar},%
 ]{mydfn}
 \declaretheorem[name=Definition,style=mydfn, numberwithin=chapter]
 {dfn}

\declaretheoremstyle[headfont=\sffamily\bfseries,%
 notefont=\sffamily\bfseries,%
 notebraces={}{},%
 headpunct=,%
 bodyfont=\normalfont,
 headformat=\color{head0}\NAME~\NUMBER\hfill\NOTE\smallskip\linebreak,%
 preheadhook=\begin{leftbar},%
 postfoothook=\end{leftbar},%
 ]{mynte}
 \declaretheorem[name=Note,style=mynte, numberwithin=chapter]
 {nte}
 \declaretheorem[name=Fact,style=mynte, numberwithin=chapter]
 {fct}
 \declaretheorem[name=Question,style=mynte, numberwithin=chapter]
 {que}

\declaretheoremstyle[headfont=\sffamily\bfseries,%
 notefont=\sffamily\bfseries,%
 notebraces={}{},%
 headpunct=,%
 bodyfont=\normalfont,
 headformat=\color{head0}\NAME~\NUMBER\hfill\NOTE\smallskip\linebreak,%
 preheadhook=\begin{leftbar},%
 postfoothook=\end{leftbar},%
 ]{myexa}
 \declaretheorem[name=Example,style=myexa, numberwithin=section]
 {exa}

\begin{document}

% 1. frontmatter
%%%%%%%%%%%%%%%%%%%%%%%%%%%%%%%%%%%%%%%%%%%%%%%%%%%%%%%%%%%%%%%%%%%%
\frontmatter

% 2. titlepage
%%%%%%%%%%%%%%%%%%%%%%%%%%%%%%%%%%%%%%%%%%%%%%%%%%%%%%%%%%%%%%%%%%%%
\title{      Lectures notes on number theory for computer science} 

\author{{\sc \Large Alexandros V. Gerbessiotis} \\ \\
 {\sc \Large CS Department} \\
 {\sc \Large NJIT} \\
 {\sc \Large Newark, NJ 07102.}\\  \\
 {\tt \Large Email:} {\tt \Large alexg@njit.edu}  \\ 
% {\bf \copyright 2019-2026. Alex. Gerbessiotis. All rights reserved.}\\
 }
\date{Printed on \today}

% 3. Preface-Abstract-Other
%%%%%%%%%%%%%%%%%%%%%%%%%%%%%%%%%%%%%%%%%%%%%%%%%%%%%%%%%%%%%%%%%%%%

\maketitle
\setcounter{page}{2}

\chapter*{Preface}

This brief, in the form of an e-book, 
is a collection of notes
that cover elementary and medium
level number theory with a target audience of
primarily computer science students. 
It can be used in the number
theory portion of a discrete mathematics course,
or a course on the mathematical foundations of 
computer science, or as background material
for a cryptography course.

Thematically it is split into five areas that
map to chapters.  
The first chapter is introductory and covers
topics including divisibility,
prime numbers, and modular arithmetic including modular
linear equations. The second chapter covers additional
topics such as Euler's totient function, 
units and inverses, the Chinese remainder theorem, 
and Fermat's and Euler's theorems. 
The following chapter covers primitive roots, quadratic
residues, the Jacobi and Legendre symbols, Gauss's lemma
and Eisenstein's theorem, and briefly discusses applications
of number theory to cryptography.
The fourth chapter is focused on traditional primality
testing methods  covering Miller's algorithms,
Rabin's conversion of a Miller algorithm 
into a probabilistic primality test algorithm, 
Solovay-Strassen's algorithm and several other 
peripheral results including Carmichael
numbers and  the equivalence of Miller's two algorithms.
Finally the last brief chapter can be 
viewed as an introduction to more advanced elements 
of number theory and its coverage includes multiplicative 
functions, the M\"obius function, Dirichlet products and
Dirichlet and M\"obius inversions.

Different parts of this e-book  are for
freshman to senior undergraduate students in computing
and in particular computer science. 
Graduate students with limited exposure to number
theory can use it to acquire a background suitable
for typical cryptography courses at the master's level.

% 4. Table of contents
%%%%%%%%%%%%%%%%%%%%%%%%%%%%%%%%%%%%%%%%%%%%%%%%%%%%%%%%%%%%%%%%%%%%

\tableofcontents

%%%%%%%%%%%%%%%%%%%%%%%%%%%%%%%%%%%%%%%%%%%%%%%%%%%%%%%%%%%%%%%%%%%%

\mainmatter

% 5. Parts-Chapters
%%%%%%%%%%%%%%%%%%%%%%%%%%%%%%%%%%%%%%%%%%%%%%%%%%%%%%%%%%%%%%%%%%%%

\chapter{Introductory Number Theory}

\section{Numbers}

\begin{dfn}[Integers]
The set of integers
is denoted as $\mb{Z}$.
\[
 \mb{Z} = \{  \ldots  , -3 , -2 , -1 , 0 ,
                    1 ,     2,     3, \ldots \} .
\]
\end{dfn}

\begin{dfn}[Natural integer numbers]
A {\bf natural integer number}, or natural number,
or ordinal number is a non-negative integer number.
The set of natural numbers is denoted as $\mb{N}$.
\[
 \mb{N} = \{ 0, 1 , 2, 3, \ldots \} .
\]
\end{dfn}

This definition varies in different textbooks.
Although we define a natural integer number as a
non-negative integer number,
several alternative sources describe a natural integer
number as a positive integer number thus excluding zero.

\begin{dfn}[Positive and negative integers]
The set of positive integers
is denoted as $\mb{Z}_+$,
the set of negative integers
is denoted as $\mb{Z}_-$,
and the set of non-zero integers
is denoted as $\mb{Z}^*$.
\[
 \mb{Z}_+ = \{ 1 , 2, 3, \ldots \} .
\]
\[
 \mb{Z}_- = \{ -1 , -2, -3, \ldots \} .
\]
\[
 \mb{Z}^* = \{ \pm 1 , \pm 2, \pm 3, \ldots \} .
\]
\end{dfn}

For the positive and negative integers one may also
use $ \mb{Z}_+^*$ and $ \mb{Z}_-^*$.
$\mb{Z}$ is the set of integers (positive, negative or zero).
For a positive integer $n>0$ the set
$\mb{Z}_n$ sometimes denoted as $\mb{Z}/n\mb{Z}$ or $\mb{Z}/n$
is the set of integers modulo $n$, thus representing the $n$
equivalence classes that the integers of $\mb{Z}$ can be
split into depending on the remainder of their division by $n$.

Integer one is an integer, a rational and algebraic and
thus a real number.
Number $4/5$  is rational and algebraic and thus a real number.
Number $\sqrt{2}$ is irrational, algebraic and thus a real number.
Number $\pi$ is irrational, transcendental and thus a real number.

\section{Divisibility and compositeness}

\begin{dfn}[Divisibility]
The symbol for divide is $|$.
Let $a \in \mb{Z} , b \in \mb{Z} , a \neq 0$.
We write $\dv{a}{b}$, $a \neq 0$, and read
$a$ divides $b$,
if there exists an integer $q \in \mb{Z}$ such that $b=aq$.
\end{dfn}

Integer $a$ is then a divisor or factor of $b$
and     $b$ is a multiple of $a$.
When $a$ divides $b$ we can also say $b$ is divided (evenly)
by $a$ or $b$ is divisible by $a$,
or equivalently the division of $b$ by $a$ leaves a remainder of zero.
For $\dv{a}{b}$, $a$ is never a zero; $b$  can be zero.

For a division of $b$ divided by $a$,
integer $b$ is the dividend in such a division.
Integer $a$ is the divisor of $b$ also known as a
factor of $b$.
The quotient $q$ is $q=b/a = \floor{b/a}$, and the integer remainder
$r$ is $r=b-a*q$. The remainder is unique if $0 \leq r < |b|$.
If $\dv{a}{b}$ then $q=b/a$ is an integer and $r=0$.

We say $a$ does not divide $b$ if there is no 
such $q \in \mb{Z}$ such that $b=aq$.
We then write $\ndv{a}{b}$.

\begin{dfn}
If $a$ does not divide $b$ we write $\ndv{a}{b}$ instead.
\end{dfn}

\begin{exa}
The following integers divide 50.
\[
 1, 2, 5, 10, 25, 50.
\]
We usually write down the positive divisors. The negative
divisors are implied. Therefore a complete list of divisors
would be as follows.
\[
\pm 1,\pm 2,\pm 5,\pm 10,\pm 25,\pm 50.
\]
\end{exa}

\begin{dfn}[Trivial divisor]
The trivial divisors of $n$ are $1$ and $n$.
If one includes negative numbers, they are $1$, $-1$, $n$ and $-n$.
\end{dfn}

\begin{dfn}[Odd, even]
An integer $n$ is even if it is a multiple of two.
Otherwise it is an odd integer.
\end{dfn}

\begin{exa}
Every integer $a$ is a divisor of $0$ that is,
$\forall a \in \mb{Z}$ we
have $\dv{a}{0}$.
$0$ is a divisor of itself and only itself.
\end{exa}
\begin{solution}
Since $0=a \cdot 0$ the first claim follows: $a$ is a divisor of $0$. 
Since $0=0 \cdot q$, integer 0 is a divisor of 0.
There is no way for any $d \neq 0$ to have $d= 0 \cdot q$. 
Thus 0 cannot be the divisor of any $d \neq 0$.
Thus 0 only divides 0; moreover, 0 is a multiple of every integer!
\end{solution}

\begin{exa}   
Both $5$ and $-5$ are divisors of $5$.
Both $1$ and $-1$ are divisors of $5$.
\end{exa}    
\begin{solution}
It is $5=5\cdot 1$ and $5=(-5)(-1)$. Replacing $5$ by $a$,
for every integer $a$, both $\pm a$ and $\pm 1$ are divisors of $a$.
Thus  integer $5$ has four divisors $+1,-1,+5,-5$.
And so so does $-5$. And in fact any $b\neq 0,1,-1$ has at least
four divisors which are $\pm b, \pm 1$. 
0 has an infinite number of divisors (in facts its
set of divisors is $\mb{Z}$).
Only $+1, -1$ have two divisors each (including the other one of the
pair). $+1$ is the positive unit and $-1$ the negative unit 
and collectively are known as the {    units} of $\mb{Z}$.
\end{solution}

\begin{exa}      
If $\dv{a}{b}$ then $\dv{-a}{b} , \quad \dv{a}{-b} , \quad \dv{a}{-b}$.
\end{exa}      
\begin{solution}
If $b=aq$ then $b = (-a)(-q)$ and $-b=a(-q)$ and $-b=a(-q)$.
\end{solution}

\begin{exa}      
If $\dv{a}{b}$ and $\dv{a}{c}$,
then $\dv{a}{b \pm c}$.
\end{exa}      
\begin{solution}
We have from $\dv{a}{b}$ that $b = a m$ for some $m \in \mb{Z}$.
Likewise $c= a n$ for some $n \in \mb{Z}$.
Adding or subtracting we have $b \pm c = a (m \pm n)$ and
the result follows.
\end{solution}

\begin{exa}      
If $\dv{a}{b}$ and $b \neq 0$ then $|a|\leq |b|$.
\end{exa}      
\begin{solution}
 Since  $\dv{a}{b}$  we have that $b = a m$ for some $m \in \mb{Z}$.
Since $b \neq 0$ we have $a \neq 0$ as well. Furthermore $m \neq 0$
as well. Thus $|m| \geq 1$. And thus $|b|=|am|=|a||m| \geq |a|$.
\end{solution}

\bigskip
   %THEOREM 11 %%
\begin{thm}[Properties of divisibility]
\label{podiv}
\label{p5div}
If $a,b,c,d,k,m \in \mb{Z}$, then
\begin{itemize}
\item[(a)] $\dv{1}{a}$. $\dv{a}{0}$ and $\dv{0}{0}$.  
Moreover,           $\dv{0}{a}$ implies $a=0$.
\item[(b1)] $\dv{a}{a}$ and of course $\dv{\pm a}{\pm a}$.
\item[(b2)] $\dv{a}{b} \wedge \dv{b}{c} \Rightarrow \dv{a}{c}$.
\item[(c1)] If $\dv{a}{b}$ then for every $k \in \mb{Z}$, $\dv{a}{kb}$.
\item[(c2)] If $\dv{a}{b}$ then for every $k \in \mb{Z}$, $\dv{ka}{kb}$.
\item[(d1)] if $\dv{a}{b}$ and $\dv{a}{c}$ then $\dv{a}{b+c}$.
\item[(d2) ] if $\dv{a}{b}$ and $\dv{a}{c}$ then $\dv{a}{b-c}$.
\item[(d3) ] if $\dv{a}{b}$ and $\dv{a}{c}$ then $\dv{a}{kb \pm mc}$.
\item[(f1) ] if $\dv{a}{b}$ and $b \neq 0$ then $|a|\leq |b|$.
\item[(f2) ] if $\dv{a}{b}$ and $a,b > 0$ then $a\leq b$.
\item[(f3) ] if $\dv{a}{b}$ and $\dv{b}{a}$ then $|a|=|b|$.
\item[(g) ] if $\dv{ka}{kb}$ and $k\neq 0$ then $\dv{a}{b}$.
\end{itemize}
\end{thm}
\begin{proof}
$ $\\\noindent
(a) $a = 1 \cdot a$ implies $\dv{1}{a}$. 
    $0 = a \cdot 0$ implies $\dv{a}{0}$.
    $0 = 0 \cdot 0$ implies $\dv{0}{0}$,
and $a = 0 \cdot q$ implies $a=0$ and this concludes the case.

%%%%%%%%%%%%
\noindent
(b1) $a  = a \cdot 1   $  implies $\dv{a}{a}$.
Moreover
     $-a = a \cdot (-1)$  concludes the case.

\noindent
(b2) If $\dv{a}{b$}, then $b=a q$, for some $q \in \mb{Z}$. 
     If $\dv{b}{c$}, then $c=b r$, for some $r \in \mb{Z}$. 
     This implies $c= b \cdot r = (a\cdot q) \cdot r = a \cdot (qr)$
i.e. $ \dv{a}{c}$.

%%%%%%%%%%%%
\noindent
(c1) If $\dv{a}{b$}, then $b=a q$, for some $q \in \mb{Z}$. Then
$kb=kaq= a (kq)$ i.e. $\dv{a}{kb}$.

\noindent
(c2) If $\dv{a}{b$}, then $b=a q$, for some $q \in \mb{Z}$. Then
$kb=kaq= (ka) q$ i.e. $\dv{ka}{kb}$.

%%%%%%%%%%%%
\noindent
(d1) If $\dv{a}{b$}, then $b=a q$, for some $q \in \mb{Z}$.
     If $\dv{a}{c$}, then $c=a r$, for some $r \in \mb{Z}$. 
Then $b+c = a (q+r)$ implies $\dv{a}{b+c}$.

\noindent
(d2) Moreover for the $a,b,c$ of (d1) we have
     $b-c = a (q-r)$ implies $\dv{a}{b-c}$.

\noindent
(d3) Furthermore for the $a,b,c$ of (d1) we have
     $kb \pm  mc = a (kq \pm mr)$ implies $\dv{a}{kb \pm mc}$.

%%%%%%%%%%%%
\noindent
(f1)   If $\dv{a}{b}$ then $b=aq$. Then $|b|=|aq|=|a| \cdot |q|$.
If $b \neq 0$ then $|b| > 0$. 
(Absolute values are positive or zero.) 
This implies that $|a| > 0$ and $|q|>0$.
The latter is equivalent to  $|q| \geq 1$. 
Then $|b|= |a| \cdot |q| \geq |a| \cdot 1  \geq |a|$.
Equivalently $|a| \leq |b|$.

\noindent
(f2)  If all of $a,b$ are positive we can drop the absolute values
from (f1) concluding $a \leq b$.

\noindent
(f3) From (f1) we have $|a| \leq |b|$. If $\dv{b}{a}$ we can
likewise conclude that $|b| \leq |a|$ thus deriving $ |a|=|b|$.

\noindent
(g) if $\dv{ka}{kb}$ and $k\neq 0$ then $kb=(ka)q$. For non-zero
$k$ dividing both sides we have $b=a\cdot q$ and thus $\dv{a}{b}$.
%(vi) Given that $\dv{a}{b}$ and $\dv{b}{a}$, if $b=0$ since $\dv{b}{a}$, 
%it means $a=0$ as well.  Likewise if $a=0$ then $b=0$ as well. 
%(So the case one of $a$,$b$ is zero implying that the other is zero
%and thus $|a|=|b|$ is resolved.)
%For the remainder, let us assume that
%$a \neq 0 , b \neq 0$. Then by way of (v) we have $|a| \leq |b|$ and
%swapping the roles of $a$ and $b$, also $|b| \leq |a|$. 
%Both imply that $|a|=|b|$ and thus $a=\pm b$ or
%$b = \pm a$.
\end{proof}

\begin{thm}[Uniqueness]
\label{uniqu}
For every $a \in \mb{Z}^* , b \in \mb{Z}$, if $\dv{a}{b}$, 
there is a unique integer $q \in \mb{Z}$ such that  $b=aq$.
\end{thm}
\begin{proof}
Suppose that $\dv{a}{b}$ i.e.  $b=aq$ with 
$q\in \mb{Z}$ and $a\in \mb{Z}^*$ i.e.  $a\neq 0$. If $q$ is not
unique, then there might exist a $q_1$ such that $b=aq_1$ and $q \neq q_1$.
Then $b=aq = a q_1$ implies $a(q-q_1 ) = 0$. Since $a\neq 0$, it must be
$q-q_1 =0$. Then $q=q_1$ but this contradicts to the existence of 
$q_1 \neq q$.
\end{proof}

\newpage

\section{Primes}

\begin{dfn}[Prime numbers]
A natural number $p>1$ is a prime number if it is not
the product of two smaller natural numbers.
\end{dfn}

For example for natural number $5$ the only product that gives
$5$ is $5 \cdot 1$. In that product $5$ is not a smaller
natural number than $5$ and thus $5$ is prime.
Natural number $6$ is not a prime number: $6=2 \cdot 3$
and both $2,3$ are smaller (natural numbers) than $6$.

%\begin{dfn}[Prime numbers are the units]
%An integer $p \in \mb{Z}^*$ is a prime (number) if and only if its only 
%divisors are the units $+1$ and $-1$ and $p$, and $-p$.
%\end{dfn}
%
%Based on the definition $1$ and $-1$ have two divisors each.
%Any other prime $p$, where $p \neq 1$ has four divisors each.
%This is summarized below. Under this definition the units $1$ 
%and $-1$ of $\mb{Z}$ are not considered prime numbers.

\begin{dfn}[Prime]
\label{prime}
An integer $p \in \mb{Z}_+^*$ is a prime (number) if and only 
if $p\neq 1$ and its only positive divisors are 1 and $p$.
\end{dfn}

%\begin{dfn}[Prime numbers]
An integer $p \in \mb{Z}^*$ is a prime (number) if and only 
if $p\neq \pm 1$ and its only divisors are the units $+1$
and $-1$ of $\mb{Z}$ and $p$, and $-p$ 
%(i.e. four integers for an integer $p \neq 0,-1,+1$)).
%\end{dfn}

A number that is not prime and not a unit ($+1$ or $-1$) is
a composite number.

\begin{dfn}[Composite]
\label{composite}
An integer $n  \in \mb{Z}^*$ such that $n \neq \pm 1$, it is either
a prime or  a composite  (integer) number.
\end{dfn}

\begin{dfn}[1 is a unit]
1 is neither a prime number nor a composite number.
It is a unit.
\end{dfn}

\begin{lem}[Composite]
An integer $n \in \mb{Z_+}^*$  is composite
if and only if it has a factor $a$ such that $1< a< n$.
Then there is another factor $q$ such that  $1< q< n$,
and $q=n/a$.
\end{lem}
\begin{proof}
If $n$ is composite, then $n$ is a multiple of integer
$a$ that is neither $1$ nor $n$. 
%(Between $|a|$ and $-|a|$ we pick the 
%positive choice and call it $a$.)
Then there exists $q$ such that $n=aq$ for some integer $q$.
If $a$ is neither $1$ nor $n$ then $1 < a < n$. (We also
used (f2) from Theorem~\ref{podiv}.)
Since $n=aq$, $q$ is also positive and $q>1$. 
Since $a>1$, we have  $n=aq>q$ implies
$q<n$. And since $a$ is not $n$ then $q$ cannot be $1$.
\end{proof}

\begin{thm}[Composite with a prime factor]
\label{compo}
\label{prfactor}
An integer $n \in \mb{Z}^*$ with $n>1$ has
a prime factor $p$ such that $\dv{p}{n}$.
\end{thm}
\begin{proof}
Let $Q$ be the set of natural integer numbers greater than one, that
have no prime factors. We shall show that $Q= \emptyset$.
$ $ \\ $ $
Say $Q$ is not empty. Then there is a minimum element say $m \in Q$.
Since $\dv{m}{m}$, and $m$ has no prime factors, $m$ cannot be a prime
number. Thus $m$ is composite and let $m=qr$, where $1 < q,r <m$.
Since either $q$ or $r$ is $<m$, and $m$ is the smallest element of
$Q$ it means $q \not\in Q$. By definition $q$ has a prime factor $p$
i.e. $\dv{p}{q}$ and $\dv{q}{m}$. By transitivity $\dv{p}{m}$.
The latter contradicts the fact that $m$ being in $Q$ it should
not have prime factors (such as $p$). Thus $Q$ has no minimum
element $m$ and thus it is empty!
%If $n$ is even a prime factor is 2. For the
%remainder we assume that $n$ is an odd composite number.
%Let $A$ be the set of odd composite integers that have 
%no prime factors (divisors). 
%Assume $A$ is not empty or we are done.
%$ $ \\ $ $
%By the Well-Ordered Set Principle (W.O.S.P) set $A$ 
%has a minimum and let it be $m$. 
%Since $m$ is composite there exists $a, b$ such
%that $m=a \cdot b$ and both $a,b< m$.
%$ $ \\  $ $
%{\bf Case 1.} Factors $a,b$ of $m$ are prime numbers:
%this can't be the case since $A$ does not contain
%integers and thus integer $m$ with prime factors.
%$ $ \\  $ $
%{\bf Case 2.} Factor $a,b$ of $m$ are composite
%numbers. Both $a,b$ are composite and $a<m, b<m$.
%Therefore they can't belong to $A$ that is
%$a \not\in A , b \not\in A$. Therefore both
%$a$ and $b$ have a prime factor; call $p$
%such a prime factor. Since $m=ab$ we have
%$m= p q$ for some $q\in \mb{Z}$. Then this
%means $m$ has a prime factor: then $m$ should
%not be in $A$. Contradiction.
%Thus $A$ must be empty and the theorem is proven.
\end{proof}

\begin{thm}[A factor less than $\sqrt{n}$ for $n$]
\label{prorprfactor}
For a composite integer $n>1$ one of its prime 
factors(divisors) is less than or equal to $\sqrt{n}$.
\end{thm}
\begin{proof}
%By Lemma~\ref{compo} we have $n = p q$ where $p$ is a
%prime factor. If $p \leq \sqrt{n}$ we are done.
%Otherwise $p > \sqrt{n}$. Then $q$ must be
%$q \leq \sqrt{n}$, since otherwise
%$p  > \sqrt{n}$ and $q>   \sqrt{n}$ would lead
%to $pq > n $ that is, $n>n$ an impossibily.
%Now if $q \leq \sqrt{n}$ indeed and $q$ is prime we
%are done. Otherwise $q$ is composite and by
%Lemma~\ref{compo} it has a prime factor $r$,
%where $r < q < \sqrt{n}$. If $r$ divides $q$ and since
%$q$ divides $n$ we also have that $\dv{r}{n}$.
%We have just found a prime factor of $n$,
%less than or equal to $\sqrt{n}$  
%and this is $r$!
Let $a,b$ be such that  $n=ab$.
We have $1 < a \leq b < n$. If $a > \sqrt{n}$ then $b \geq a >\sqrt{n}$
and therefore $n=ab > \sqrt{n} \cdot \sqrt{n} =n$, which is impossible.
It follows that $1 < a \leq \sqrt{n}$.
Thus $a$ has a prime factor $p$
by   Lemma~\ref{compo} and $p \leq a$.
Since $\dv{p}{a}$ and $\dv{a}{n}$ we conclude that $\dv{p}{n}$,
with $p \leq a \leq \sqrt{n}$.
\end{proof}

\begin{exa}
Show that if $\dv{2}{ab}$ for $a,b \in \mb{Z}$,  then
$\dv{2}{a}$ or $\dv{2}{b}$.
\end{exa}
\begin{solution}
If neither $\dv{2}{a}$ nor $\dv{2}{b}$, then
$a,b$ are odd numbers and thus $a=2n+1$ and $b=2m+1$.
Then since $ab$ is an even number by way
of $\dv{2}{ab}$ we obtain
\[
a b = (2n+1)(2m+1) = 4mn +2m +2n +1 = 2(2mn+m+n)+1,
\]
which is clearly an odd integer number and thus
not even contradicting the fact that $ab$ is even!
\end{solution}

\begin{exa}
Prime numbers  are irreducible.
Thus if $p$ is a prime number then $p$ cannot be written
in the form $p=qr$ where $q,r$ are both non units.
(In other words for an irreducible number either $q$
or $r$ is a unit.)
\end{exa}
\begin{solution}
Note 0 cannot be a prime number therefore $p \neq 0$.
Say prime number $p$ is not irreducible.
If a prime number $p$ is not irreducible, then $p=qr$ where
neither $q$ nor $r$ is a unit.
Then $\dv{p}{qr}$.
Moreover $\dv{q}{p}$ and $\dv{r}{p}$,
and therefore $|q|\leq|p|$ and $|r| \leq |p|$.
Then $|q| \leq |p|$ and $|r| \leq |p|$.
We claim $\ndv{p}{q}$ and $\ndv{p}{r}$.
If this was not the case (say the former)
then $\dv{p}{q}$ i.e. $|p|\leq |q|$ which with  $|q| \leq |p|$
implies $q= \pm p$ and then $r=\pm 1$. This contradicts
the assumption that $r$ is not a unit.
\end{solution}

\subsection{Inifinitely many primes}

\begin{lem}[A theorem by Euclid]
There are infinitely many prime numbers.
\end{lem}
\begin{proof}
If there are only finite prime numbers and let them ALL
be $p_1 , \ldots , p_m$, where $p_1= 2$.
Then form integer $n = p_1 \cdot p_2 \cdot \ldots \cdot p_m +1$.
There are two possibilities for $n$: (a) it is a prime number,
(b) it is not a prime number.

\noindent{\bf Case 1.} If $n$ is a prime number
Then $n = 2 \cdot \ldots \cdot p_i \cdot \ldots  \cdot p_m + 1 > 
  > 2p_i +1 > p_i $
given that all $p_i > 1$. We have just found one more prime number
beyond the $m$ ones that were declared ALL that there are: a
contradiction. 

\noindent{\bf Case 2.} If $n$ is a composite number, let $p$ be
a prime factor of $n$, i.e. $\dv{p}{n}$ with $p>1$. 
Such a $p$ exists by the previous Lemma (composite
with a prime factor). This $p$ cannot be one of the $m$
$p_1 , \ldots , p_m$. Why ? if $p$ was say $p=p_i$ then
$\dv{p}{p_i}$ implies by Theorem~\ref{podiv} (c1,d3) that
$\dv{p}{p_1 \cdot p_2 \cdot \ldots \cdot p_m}$.
Since $\dv{p}{n}$ and 
$n = p_1 \cdot p_2 \cdot \ldots \cdot p_m +1$
the $p$ divides their difference i.e. $\dv{p}{1}$ by
Theorem~\ref{podiv} (d2).
This means $p$ is one by Theorem~\ref{podiv} (f2).
But one is a unit not a prime number (plus also the fact
that $p>1$).
\end{proof}

\newpage

\section{Integer division}

   %THEOREM 13 %%
\begin{thm}[Division]
\label{divis} %pdiv
For $a \in \mb{Z}$ and $b \in \mb{Z}^*$ there exist unique integers
$q \in \mb{Z}$ and $r \in \mb{Z}$ such that \\
\[
a=b q+ r \quad , \quad 0 \leq r < |b| .
\]
\end{thm}
\begin{proof}
$ $ \\ $ $
{\bf Case 1 ($a\geq 0, b > 0$).}
Let $A$ be the set of all non-negative integers $a-bi$, where 
$i$ is such that $a-bi \geq 0$. 
For $i=0$, $a$ belongs to $A$, and thus $A$ is not empty. 
By the Well-ordered set principle $A$ has a 
minimum and let it be $r$. Since $r$ is in $A$, we 
have $r \geq 0$ and $r=a-bi$ for some integer $i$. 
All it remains to show is that 
$r < |b|=b$. Say that $r \geq b$ instead. Then $r-b\geq 0$.
Moreover
\[
r - b = (a-bi) - b = a -  b (i+1).
\]
Thus $r-b \geq 0$ and is of the form 
$a-b i^\prime$, with $i^\prime = i+1$.
Thus it belongs to $A$. Moreover $r-b$ is less than 
$r$, $r-b<r$ since $b$ is positive. We have found an 
element smaller than the minimum
element of $A$. This contradicts to the choice 
of $r$ as being the minimum;
we reached contradiction because we assume $r \geq b$.
Thus $r < b$. This thus establishes that  $0 \leq r < b$.
$ $ \\ $ $
We have yet to prove the pair $(q,r)$ is unique.
Let is not be and let another pair be 
$(q^\prime , r^\prime )$. Then
$a=bq+r = b q^\prime + r^\prime$,
where $0 \leq r , r^\prime < b$.
This gives $b(q - q^\prime ) = r^\prime -r $, and
           $|b(q - q^\prime )| = |r^\prime -r |$.
Adding $0 \leq  r$ and $r^\prime < b$ we get $r^\prime < r +b$ or
equivalently $r^\prime -r < b$. Thus $|b (q-q^\prime )| < b$.
As $b > 0$ we have $|q-q^\prime| < 1$. This can only be possible,
for $q=q^\prime$ since $q, q^\prime$ are integer. 
Then from $a=bq+r = b q^\prime + r^\prime$ if $q=q^\prime$,
we deduce also that $r=r^\prime$.
This answers positively the question about
the uniqueness of the pair $(q,r)$ for case 1.
$ $ \\ $ $
\smallskip\noindent
{\bf Case 2 ($a< 0 , b>0$).} Similarly as before $A$ 
is not empty because $a-bi \geq 0$ contains at least 
one element $a-ba \geq 0$, and the
rest of the discussion is similar to case 1.
$ $ \\ $ $ 
\smallskip\noindent
{\bf Case 3 ($a \in \mb{Z} , b<0$).} Then $|b| >0$, 
and of course $|b|=-b$.
By way of cases 1 and 2 for $a$ and $|b|$ we have that
$a=|b| q + r$, where $0 \leq r < |b|$. This is
equivalent to $a=(-b) q + r$, where $0 \leq r < |b|$. This is
equivalent to $a= b (-q)+ r$, where $0 \leq r < |b|$. The
claim is satisfied for case 3 as well.
$ $ \\ $ $
Conclusion.
For a pair $(a,b \neq 0)$  there is a unique pair 
$(q,r)$ with $a=bq+r$ and $0\leq r < |b|$.
\end{proof}

\begin{exa}
Find the quotient and the remainder of the division
of 37 by 5.
Then find the quotient and the remainder of the division
of -37 by 5.
\end{exa}
\begin{solution}
$ $ \\ $ $
(a) For 37 we have the following
\[
  37 = 5 \cdot 7 + 2.
\]
37 is the divident, 5 is the divisor and 7 the quotient.
The remainder is 2. Note that $0 \leq 2 < |5|$.
$ $ \\ $ $
(b) For $-37$ we have the following
\[
 -37 = 5 \cdot (-8) + 3.
\]
$-37$ is the divident, 5 is the divisor and $-8$ the quotient.
The remainder is 3. Note that $0 \leq 3 < |5|$.
One can also observe that
\[
 -37 = 5 \cdot (-7) - 2.
\]
However only the remainder $3$ is one that satisfies
the $0 \leq 3 < |5|$ requirement and is unique.
\end{solution}

\bigskip %THEOREM 14 %%
\begin{thm}[Division results]
\label{divir}
For $a,b$  as in division  ($b\neq 0$), where
$a=bq+r$ and $0 \leq r < |b|$, we have the following.
\begin{itemize}
\item[(i)]  if $\dv{d}{a}$ and $\dv{d}{b}$ then $\dv{d}{r}$.
\item[(ii)] if $\dv{d}{r}$ and $\dv{d}{b}$ then $\dv{d}{a}$.
\end{itemize}
\end{thm}

\begin{proof}
$ $ \\ $ $
(i) If $\dv{d}{a} and \dv{d}{b}$ then from the latter
$\dv{d}{bq}$ and thus from the former
$\dv{d}{a-bq}$ leading (given that $a=bq+r$) to $\dv{d}{r}$.
$ $ \\ $ $
(ii) If $\dv{d}{r},\dv{d}{b}$ the $\dv{d}{bq}$ and thus
$\dv{d}{bq+r}$ which leads to $\dv{d}{a}$.
\end{proof}

\bigskip %THEOREM 15 %%
\begin{thm}[Remainder of division by $b$]
Let $a,A \in \mb{Z}$ with $\dv{b}{a}$ and $\dv{b}{A}$. 
%For $a,A \in \mb{Z}$, with $\ndv{b}{a}$ and $\ndv{b}{A}$. 
The remainders of the divisions of $a, A$ with $b$
are the same if 
and only if $a-A$ is a multiple of $b$.
\end{thm}

\begin{proof}
\noindent{\bf $\Rightarrow$.} If $a=bq+r$ and $A=bQ + R$, and we
are given $r=R$, the $a-A=bq+r-qQ-r = b(q-Q)$. Since $q-Q\in \mb{Z}$,
we conclude that $\dv{b}{a-A}$.
$ $ \\ $ $
\noindent{\bf $\Leftarrow$.} 
%Say that $\dv{b}{a-A}$.
%Then let $a=bq+r$ and $A=bQ+R$, where $0 \leq r, R < |b|$.
%We have $a-A=b(q-Q) +(r-R)$. Since obviously $\dv{b}{b(q-Q)}$
%and also by assumption $\dv{b}{a-A}$ we conclude that 
%$\dv{b}{(a-A)-b(q-Q)}$ i.e. $\dv{b}{r-R}$ and also $\dv{|b|}{r-R}$.
%This would imply that $|b| \leq |r-R|$ or $|r-R|=0$.
%But $r,R < |b|$ which means that $r-R$ or $|r-R|$ are such 
%that $|r-R| < |b|$.
%Thus the former $|b| \leq |r-R|$ is false 
%because otherwise $|b| \leq |r-R|$ and $|r-R| < |b|$ lead by transitivity
%to the nonsense $|b| < |b|$.  We must have
%$|r-R|=0$ which implies $r=R$ since $R,r \geq 0$.
%
 Another way to prove $\Leftarrow$ it is directly as follows.
 Say that $\dv{b}{a-A}$, i.e. $a-A=b m$ for some $m \in \mb{Z}$.
 Then let $a=bq+r$ and $A=bQ+R$, where $0 \leq r, R < |b|$.
 From the former $ a-A = bm$ and $A=bQ+R$ we have
 $a-(bQ+R) = b m$ implying $a= b(m+Q) +R$. This latter equality
 implies $q = m+Q$ and $r=R$ as $R$ is such that $0 \leq R < |b|)$.
 Result is proven.
\end{proof}

\begin{exa}
Show the following: if $\dv{a^2}{b^2}$ then $\dv{a}{b}$,
for every $a,b \in \mb{Z}_+$.
\end{exa}
\begin{solution}
A simple solution that relies on rational numbers vs
irrational numbers is as follows.
If $\dv{a^2}{b^2}$ then there exists an integer $k$ such
that $b^2 = a^2 k$. Then, $a^2 / b^2 = k$ and therefore
$(b/a) = \sqrt{k}$. Therefore since $a,b$ are integer,
then $b/a$ is a rational number. $\sqrt{k}$ is algebraic
but also rational; it can't be irrational since that
would make $b/a$ irrational contradicting to the rationality
of $b/a$. Since $\sqrt{k}$ is rational, and $k$ integer,
there can only be that $\sqrt{k}$ is also integer.
Then $b=a  \sqrt{k}$ and thus $\dv{a}{b}$.
\end{solution}

\begin{exa}
Show the following: if $\dv{a^2}{b^2}$ then $\dv{a}{b}$,
for every $a,b \in \mb{Z}_+$.
\end{exa}
\begin{solution}
Let for the sake of contradiction that $\ndv{a}{b}$.
$  $  \\  $ $
\begin{equation}
\label{a2b2a}
\dv{a^2}{b^2} \Rightarrow \exists a_2 : b^2 = a^2 a_2 .
\end{equation}
Consider,
\[
 B(a) = \{  bt : t \in \mb{Z}_+ , \dv{a^2}{b^2} \wedge \dv{a}{bt} \} .
\]
Let $k$ be the smallest value such that $bk \in B(a)$.
$B(a)$ is not empty.
Since  $\dv{a}{b^2}$ for $t=b$ we have $bt = b^2 \in B(a)$.
The former is true because
because $\dv{a^2}{b^2}$ implies
$b^2 = a^2 r$ and thus $\dv{a}{a^2 r}$ implies
$\dv{a}{b^2}$ i.e. $\dv{a}{bt}$ for $t=b$.
Since $k$ is the smallest $t$ that generates the min
element of $B(a)$ by the W.O.S.P.
we have by way of $\dv{a}{bt}$, $t=k$.
\begin{equation}
\label{a2b2b}
bk = a a_1
\end{equation}
Consider
\begin{equation}
\label{a2b2c}
 a b a_1 = ab \frac{bk}{a} = b^2 k = a^2 a_2 k.
\end{equation}
We further conclude by dividing by $a \neq 0$,
the following
\begin{equation}
\label{a2b2d}
 a b a_1 =  a^2 a_2 k \Rightarrow ba_1 = a a_2 k.
\end{equation}
By the division theorem for $a_1$ and $k$ we have the following.
\begin{equation}
\label{a2b2e}
a_1 = k q +r , 0\leq r < k \Rightarrow
 r= a_1 -k q \Rightarrow
 br= ba_1 -b k q \Rightarrow
 br= aa_2 k -a a_1 q \Rightarrow
 br= a (a_2 k - a_1 q) \Rightarrow \dv{a}{br}
\end{equation}
In the previous derivation we uses in sequence
Eq.(\ref{a2b2d}),
Eq.(\ref{a2b2b}).
Therefore $br$ is in $A$ since $\dv{a^2}{b^2}$
and $\dv{a}{br}$. Since $r < k$ we have that
$br$ is smaller than the smallest element of $B(a)$
which is $bk$. This leads to contradiction.
\end{solution}

\begin{lem}
Integer $d$ such that $\dv{d}{n}$ if and only if the
remainder of the division of $n$ by $d$ is $0$.
\end{lem}
\begin{proof}
$ $\\
\noindent
$\Rightarrow$ (only-if).
$ $ \\ $ $
If $d$ is such that $\dv{d}{n}$ this means
$n=qd$ that is $n=qd+0$. Given the 
uniqueness of $q,r$ from Theorem~\ref{divis} we 
conclude $q=q$ and $r=0$.
$ $ \\ $ $
$\Leftarrow$ (if).
$ $ \\ $ $
If $d$ is such that $n=d \cdot q + r$ with $r=0$,
then $n =d \cdot q$ which implies $\dv{d}{n}$.
\end{proof}

\section{Greatest common divisor}

\begin{dfn}[Greatest common divisor]
For integers $a,b \in \mb{Z}$, $\gcd(a,b)$ is the {\bf greatest
common divisor} of $a$ and $b$ if and only if $\dv{\gcd(a,b)}{a}$
and $\dv{\gcd(a,b)}{b}$ and every other divisor $c$ of $a,b$ is such
that $c \leq \gcd(a,b)$.
Therefore,
\begin{itemize}
\item[(i)]  $\dv{\gcd(a,b)}{a}$ and $\dv{\gcd(a,b)}{b}$,
\item[(ii)] $\dv{c}{a}$ and $\dv{c}{b}$ $\Rightarrow$ $c \leq \gcd(a,b)$.
\end{itemize}
Moreover, $\gcd(a,b) \leq |a|$ and $\gcd(a,b) \leq |b|$.
\end{dfn}

Note that if $\dv{c}{a}$ then by Theorem~\ref{podiv}(v) we have $|c| \leq |a|$.
Likewise if $\dv{c}{b}$ we have $|c|\leq |b|$. Combining the two we have
$|c| \leq \max(|a|,|b|)$. Thus the set of common divisors of $a,b$ is finite.
A finite set of integers always has a maximum, and thus there is a
uniquest largest integer $d >0$ such that $\dv{d}{a}$ and $\dv{d}{b}$. We
call $d$ the {\bf greatest common divisor} of $a,b$ and denote it by
$\gcd(a,b)$ thus $d=\gcd(a,b)$.

\begin{exa}
(a) The greatest common divisor of $1$ and $n$ is $1$.\\
(b) If $\dv{a}{b}$ then $\gcd(a,b)=b$.\\
(c) $\gcd(5,15)=5$. $\gcd(30, 105) = 15$. \\
(d) The common divisors of 30 and 105 are 
$\{ 1, 3, 5, 15, \}$. If we include negative numbers then it is
$\{ \pm 1, \pm 3, \pm 5, \pm 15, \}$, twice as many.
\end{exa}

\begin{nte}
The $\gcd(0,0)$ is not defined as the set of common divisors is an infinite 
set and it does not have a maximum.
In the remainder when $\gcd(a,b)$ is considered, we would assume
that $a \neq 0$ or $b \neq 0$ (or both). Thus it cannot be that
both $a$ and $b$ are zero.
One way to express this is by having
$|a|+|b| \neq 0$.
\end{nte}

\begin{dfn}
Let $S(a)$ be the set of divisors of $a$.
\end{dfn}

\begin{fct}[Simple GCD facts]
\label{sgcdf}
Let $a,b \in \mb{Z}$ such that $|a|+|b| \neq 0$. Then
the following apply.
\begin{itemize}
\item[(i)   ]  $\gcd(a,b) > 0$ and also $\gcd(a,b) \geq 1$.
\item[(ii)  ]  $\gcd(a,b)= \gcd(|a|,|b|)$.
\item[(iii) ]  $\gcd(a,b)= \gcd(b,a)$.
\item[(iv)  ]  $\gcd(a,1)=1$.
\item[(v)   ]  $\gcd(a,0)= |a|$ for all $a\neq 0$.
\item[(vi)  ]  $\gcd(a,b)=|a|$ if and only if $\dv{a}{b}$.
\end{itemize}
\end{fct}

\begin{proof} 
$ $\\

\noindent
(i) Obviously. $gcd(a,b)$ is the maximum of the
common divisors of $a$ and $b$ i.e. the maximum elements of 
$S(a) \cap S(b)$. One positive element of this set of common
divisors is 1 and thus $gcd(a,b) \geq 1$ in addition to $\gcd(a,b) >0$.

\noindent
(ii) Since $S(a) = S(|a|)$ and $S(b)=S(|b|)$ we have that
$S(a) \cap S(b) = S(|a|) \cap S(|b|)$. Thus
$\gcd(a,b)$ is equal to $\gcd(|a|,|b|)$ since the set
of common divisors are equal to each other.

\noindent
(iii) It is a consequence of the fact that $S(a) \cap S(b) = S(b) \cap S(a)$.

\noindent
(iv) Since $\dv{1}{a}$ trivially, and the largest divisor of 1 is 1
itself the result follows. (Note that $S(1) = \{ -1,+1 \}$.)

\noindent
(v) $S(a) \cap S(0)  = S(a) \cap \mb{Z} = S(a)$. The result follows as
the largest element of $S(a)$ is $|a|$.

\noindent
(vi) 
$ $ \\ $ $
$\Rightarrow$. If $\gcd(a,b)=|a|$ then $\dv{|a|}{a}$ and $\dv{|a|}{b}$.
For the latter there exist $q$ such that $b=q|a|$. If $a$ is non-negative,
the $b=qa$ as well. If $a$ is negative $b=q(-a) = (-q) a$.
The former concludes $\dv{a}{b}$ and so does the latter.
$ $ \\ $ $
$\Leftarrow$. If $\dv{a}{b}$ then $\dv{-a}{b}$ and thus $\dv{|a|}{b}$. 
Since trivially $\dv{a}{a}$ we conclude that 
$\dv{-a}{a}$ and thus $\dv{|a|}{a}$.  
Then $\dv{|a|}{\gcd(a,b)}$. 
This by Theorem~\ref{podiv}(v) implies $|a| \leq \gcd(a,b)$ 
(the gcd is always positive thus no absolute value sign around it is needed).
By definition $\dv{\gcd(a,b)}{a}$ and thus $\dv{\gcd(a,b)}{-a}$.
Thus $|\gcd(a,b)| \leq |a|$.
Combining $|\gcd(a,b)| \leq |a|$ and  $|a| \leq \gcd(a,b)$ the result follows.
\end{proof}

\bigskip %THEOREM 19 %%
\begin{thm}[GCD divided]
\label{gcddi}
If $d=\gcd(a,b)$ then $\gcd(\frac{a}{d}, \frac{b}{d} ) =1 $.
\end{thm}

\begin{proof}
If $d=\gcd(a,b)$ then by Theorem~\ref{gcd} we have that $d=r_k$.
By the last equation of $\gcd$'s we have $\gcd(r_{k-1},r_k )  = r_k = d$.
Given that $d$ divides itself ($r_k =d$) we also have that $\dv{d}{r_{k-1}}$.
From the prior expression, now that we have $\dv{d}{r_k}$ and $\dv{d}{r_{k-1}}$
using $\gcd(r_{k-2},r_{k-1}) =\gcd(r_{k-1},r_k ) = r_k =d$ we have that
$\dv{d}{r_{k-2}}$ as well. Continuing likewise we have $\dv{d}{r_2}$ and 
$\dv{d}{r_1}$. 
Likewise from the third equation of Theorem~\ref{gcd} we have $\dv{d}{r}$
and from the second $\dv{d}{b}$. 
Concluding from the first (division) equation we have
$\dv{d}{a}$. We can thus divide by $d$ all equations. This shows that
$\gcd(\frac{a}{d}, \frac{b}{d} ) = \frac{d}{d} =1$ as all 
remainders will be divided by $d=r_k$ as well.
\end{proof}

         %THEOREM 17 %%
\begin{thm}
\label{t17}
\label{gcddiv}
If $x,y \in \mb{Z}$, then
\[
  \gcd(y,x) =  \gcd(x, y) = \gcd( x, xq+y)
\]
for all integers $q \in \mb{Z}$.
\end{thm}

\begin{proof}
Let $d=\gcd(x,y)$. 
Let $d_1 = \gcd( x, xq+y)$.
If $\dv{d}{x}$ and $\dv{d}{y}$ by divisilibity
we have $x=dx_1$ and thus $xq=d x_1 q$ and then
we have $\dv{d}{xq}$ and $\dv{d}{y}$. Thus $\dv{d}{xq+y}$ 
in addition to $\dv{d}{x}$. 
Thus $d$ is a common divisor of $xq+y$ and $x$. 
Thus $d \leq d_1 = \gcd(x,xq+y)$.
For $d_1 = \gcd(x,xq+y)$ we have $\dv{d_1}{x}$ and $\dv{d_1}{xq+y}$. 
From the former, we conclude that  $\dv{d_1}{xq}$; combining 
it with the latter we have $\dv{d_1}{xq+y-xq}$ i.e.
$\dv{d_1}{y}$. Thus $\dv{d_1}{x}$ and $\dv{d_1}{y}$ and 
therefore $d_1$ is a common
divisor of $x,y$. Thus $d_1 \leq d=\gcd(x,y)$.
By way of $d_1 \leq d$ and $d \leq d_1$ we conclude $d=d_1$.
\end{proof}

From Theorem~\ref{divis} in order to compute the 
$\gcd(a,b)$ we formulate the division operation.
\[
a = b q + r
\]
where $0 \leq r < |a|$.
Then by way of Theorem~\ref{t17} 
\[
\gcd(a,b)=\gcd(b,a) = 
\gcd(b, bq+r) \overset{Th.~\ref{t17}}{=} \gcd(b, r).
\]

For calculating the $\gcd(a,b)$ at a minimum $|a|+|b| \neq 0$ that is,
we can't determine the gcd of two numbers that are both 0. 
The set of common divisors is 
then $\mb{Z}$ and has no maximum.
If $b=0$ then $\gcd(a,b) = a$.
If $a=b$ then $\gcd(a,b) = a$.
Thus in general we need to
compute $\gcd(a,b)$ for $a>b$ or $b>a$.
If $b<a$ then we rename (swap) $a$ and $b$ so
that they become $a>b$.
We may assume without loss of generality that $a>b$.

\subsection{The gcd algorithm by Euclid}

\bigskip %THEOREM 18 %%
\begin{thm}[GCD: Euclid's algorithm]
\label{gcd}
For $a > b $, $b \neq 0$ 
the  $\gcd(a,b)$ can be calculated  using
iterated division, as follows.
\begin{align*}
  a    &= b    q   +r     & 0\leq r   < |b|&& \gcd(a,b)&= \gcd(b, r)   \\
  b    &= r    q_1 + r_1  & 0\leq r_1 <  r && \gcd(b,r)&= \gcd(r, r_1 ) \\
  r    &= r_1  q_2 + r_2  & 0\leq r_2 < r_1&& \gcd(r,r_1)&=\gcd(r_1, r_2 ) \\
       &\ldots              &           &&             &&\\
r_{k-2}&= r_{k-1}q_k +r_k & 0\leq r_k < r_{k-1}&&\gcd(r_{k-2},r_{k-1}) &= 
       \gcd(r_{k-1},r_k ) \\
r_{k-1}&= r_k  q_{k+1} + 0 &&&\gcd(r_{k-1},r_k ) &= \gcd(r_k , 0 ) = r_k 
\end{align*}
The {\bf last non-zero remainder} is the gcd:
\[
\gcd(a,b) = r_k.
\]
Moreover, for  $a>b$ the sequence of
steps is finite as the sequence of remainders is a
decreasing sequence of positive numbers eventually
reaching zero since
$ |b| > r > r_1 > r_2 > \ldots > r_{k-1} > r_k >0$.
\end{thm}

\begin{proof}
$ $ \\ $  $
{Finite number of steps.} 
$ $ \\ $ $
The sequence of remainders
is a decreasing sequence starting below $|b|$. Thus after
no more that $|b|$ steps it will reach a remainder of $0$.
$ $ \\ $ $
{Special and trivial cases.}
$ $ \\ $ $
If the numbers are negative we consider their absolute values
and continue.
If $a=b=0$, there is no $\gcd(a,b)$ as $S(0) = \mb{Z}$.
If $a>b=0$, then $\gcd(a,b) = a$.
If $b>a=0$, then $\gcd(a,b) = b$.
We deal with some more trivial cases next.
Otherwise, if $a=b$ then $\gcd(a,b) = a =b$.
$ $ \\ $ $
{Bottom-to-top proof.}
$ $ \\ $ $
Starting from the bottom $\dv{r_k}{r_{k-1}}$.
Then from the penultimate equation we have $r_k$ dividing
both itself and $r_{k-1}$ and thus $\dv{r_k}{r_{k-1}q_k + r_k}$.
Thus $\dv{r_k}{r_{k-2}}$. Working likewise we
show that $r_k$ divides $r$ and $b$ of the first equation
and thus also divides $a$.  Thus $r_k$ divides both $a,b$.
Therefore $r_k \leq \gcd(a,b)$.
Pick an arbitrary integer $d$ dividing $a$ and $b$.
Working downwards  we show that $d$ divides $r_k$ as well.
Thus $d \leq r_k$. Set $d=\gcd(a,b)$. This translates
into $\gcd(a,b) \leq r_k$; combined with $r_k \leq \gcd(a,b)$
leads to $r_k = \gcd(a,b)$, as needed.
$ $ \\ $ $ 
{ Top-to-bottom proof: start with $a>b>0$.}
$ $ \\ $ $
If $b>a>0$, we swap $a$ and $b$.
Thus for all remaining cases we assume that  $a>b>0$ in the remainder.
$ $ \\ $ $
Therefore, the algorithm computes $\gcd(a,b)$
for $a > b$ and $b \neq 0$. Furthermore, $|a|+|b| \neq 0$.
It is obvious from the statement that
by using Theorem~\ref{gcddiv}  the following apply.
\[
\gcd(a,b) = \gcd(b,r) = \gcd(r, r_1 )   = \gcd(r_1 , r_2 ) =
            \ldots = \gcd(r_{i-1},r_i ) = \ldots
          = \gcd(r_{k-1},r_k ) = \gcd(r_k , 0 ) =  r_k .
\]
That is, the last, non-zero remainder, is the gcd of $a,b$.
Moreover the sequence of remainders purely decreases i.e.
$|b| > r > r_1 > r_2 > \ldots > r_{k-1} > r_k$.
Thus after a finite number of no more than  $b$ steps
$r_k$ will be determined.
$ $ \\ $ $
We use Theorem~\ref{divis} repeteadly starting with
$x=b$ and $y=a$.
\[
\gcd(x,y) = \gcd(y,x   ) \Rightarrow
\gcd(a,b) = \gcd(b,a) \Rightarrow
\gcd(a,b) = \gcd(b,bq+r)  ,
\]
and continuing with $x=b$ and $y=r$ and proceeding
similarly.
\[
\gcd(x,y) = \gcd(x,xq+y ) \Rightarrow
\gcd(b,r) = \gcd(b,bq+r ) \Rightarrow
\]
It is obvious from the statement of the theorem and the
previous derivations that
\[
\gcd(a,b) = \gcd(b,r) = \gcd(r, r_1 )   = \gcd(r_1 , r_2 ) = 
            \ldots = \gcd(r_{i-1},r_i ) = \ldots
          = \gcd(r_{k-1},r_k ) = \gcd(r_k , 0 ) =  r_k
\]
That is, the last, non-zero remainder, is the gcd.
Moreover the sequence of remainders purely decreases 
$|b| > r > r_1 > r_2 > \ldots > r_{k-1} > r_k$. 
Thus after a finite number of no more than  $|b|$ steps 
$r_k$ will be determined.
\end{proof}

For the golden ration $\phi$ we have that 
$\phi \approx  1.6183$, and $\lg{\phi} \approx 0.6942$ and
$1/\lg{\phi} \approx 1.44 $.

\begin{lem}
(a) In Euclid's algorithm the maximum number of division
steps $k+2$ is for $a,b$ that are Fibonacci numbers.
$ $ \\ $ $
(b) For $a,b \in \mb{Z}$, with  $a>b >0$
 the number of divisions in Euclid's Algorithm is
at most   $\lg{b} / \lg{\phi} + 2 \approx 1.44 \lg{b} +2$,
where $\phi = (1+ \sqrt{5})/2$ is the golden ratio.
\end{lem}
\begin{proof}
$ $ \\ $ $
The worst-case number of division steps is for two
consecutive Fibonacci numbers,  say $a=F_n$ and $b=F_{n-1}$,
where $F_0 = 0$ and $F_1 =1$.
Then the first division is,
$F_n = F_{n-1}+F_{n-2}$, $n>1$,
and the last division
$F_2 = F_1 + F_0$, where $F_0 =0$ and $F_1 =1 $.
All quotients generated are one, the smallest possible
quotient.
The total number of divisions is $n-1$.
A bound would be provided for $n-1$ with respect to $b$.
$ $ \\ $ $
Note that $F_n \geq \phi^{n-2}$ for all $n \geq 2$.
This follows from $\phi^2 = \phi +1 $ and induction on $n$.
Clearly $F_2 = 1 \geq \phi^0$. Moreover
\[
F_n = F_{n-1}+F_{n-2} \geq \phi^{n-3} + \phi^{n-4}
                      \geq (1+ \phi )   \phi^{n-4}
                      \geq     \phi^2   \phi^{n-4}
                      \geq              \phi^{n-2}.
\]
If the first division involves $a>b$, this means
\[
b \geq F_{n-1} \geq \phi^{n-3}
\Leftrightarrow
  n -1 \leq \lg{b} / \lg{\phi} +2 .
\]
Then the number of division steps $n-1$ is upper bounded
by $1.44 \lg{b} +2$.
Of course the last division step is trivial, so the
bound is in fact at most
$1.44 \lg{b} +1$.
\end{proof}

The worst-case (i.e. longest) division is for 
two consecutive Fibonacci numbers
$F_n = F_{n-1}+F_{n-2}$, $n>1$, with  $F_0 =0$ 
and $F_1 =1 $ generate a quotient 
of 1 every iteration! For $F_n$ the number 
of divisions is $n-2$ as the last
division is $F_2 = F_1 + F_0 = F_1 + 0$. The 
$n$-th Fibonacci number exceeds
$\phi^{n-2}$.

\begin{cor}
For $a,b < N$ the number of divisions in Euclid's Algorithm 
will
be less than $\lg{N} / \lg{\phi} \approx 1.48 \lg{N}$.
\end{cor}
\begin{proof}
We can do a bit better by bounding the number of divisions in terms of
$b$ only.
In the worst case, for an $n$ step division $a \geq F_{n+2}$ and
$b \geq F_{n+1}$. But $F_{n+1} \geq \phi^{n-1}$. Thus
$n-1 \leq \lg{b} / \lg{\phi}$. $\lg{\phi} \approx 0.687$ and thus
$n-1 \leq 1.471\lg{b}$ i.e. $n \leq 1.471 \lg{N} +1$.
\end{proof}

\section{Extended GCD}

\bigskip %THEOREM 20 %%
\begin{thm}[Extended-GCD]
\label{egcd}
{\bf B\'ezout identity.}
Let $a,b \in \mb{Z}$, where $|a|+|b| \neq 0$.
Let $d=\gcd(a,b)$ then there exist integers $x,y$ such that
\begin{equation}
\label{egcd0}
   d= ax+by
\end{equation}
Moreover, $d$ is the smallest positive integer that can 
be written as a linear combination of $a,b$.
\end{thm}

\begin{proof}
Consider 
\[
A= \{ au+bv | u,v \in \mb{Z}  \ \wedge \  au+bv > 0 \}.
\]
Note that for $u=a$, $v=b$, $ a\cdot a + b \cdot b > 0$ 
since $|a|+|b| \neq 0$.  Set $A$ has thus at least one 
element and it is not empty. Thus from the well-ordered
set principle there must exist a minimum element for $A$ and 
let it be $d^\prime$.  Since $d^\prime \in A$ there exist 
$x , y$ such that 
\[
a x + b y = d^\prime .
\]
$ $ \\  $  $
\noindent
{A. Show $d^\prime$ is a divisor of $a$.}
We first show that $\dv{d^\prime}{a}$. Let
us form the division operation for the two integers i.e.
\[
 a = d^\prime q^\prime + r^\prime
\]
where $0 \leq r^\prime < d^\prime$. Then we have that
\[
 r^\prime = a - d^\prime q^\prime = a -(ax+by) q^\prime = 
a \cdot (1- x q^\prime ) + b \cdot (-  y q^\prime ).
\]
Note that because of the division of $a$ by $d^\prime$ we 
know that $0 \leq r^prime$.
We examine two cases.
$ $ \\ $ $ 
\noindent
{Case 1:  $0 < r^\prime$.} 
$ $ \\ $ $ 
If $r^\prime > 0$ then it is a 
member of $A$ since 
$r^\prime =  a \cdot (1- x q^\prime ) +  b \cdot (-  y q^\prime )$.
But this cannot happen since $r^\prime < d^\prime$ and 
$d^\prime$ is the minimum element of $A$.
It would imply the existence of an element of $A$ (i.e. $r^\prime$) 
smaller than the minimum element of $A$ (i.e. $d^\prime$)!
$ $ \\ $ $
\noindent
{Case 2: $0=r^\prime$.} The only other possibility is that 
$r^\prime = 0$.  But then, $a = d^\prime q^\prime + r^\prime$ 
implies $a = d^\prime q^\prime $ i.e.
$d^\prime$ is a divisor of $a$.
$ $ \\ $ $
{B.  Show $d^\prime$ is a divisor of $b$.} Similar to 
the case involving $a$.
$ $ \\ $ $
{ Conclusion. }
$ $ \\ $ $
Thus $d^\prime $ divides $a$ and $b$ and is a common
divisor of $a,b$.  It should be that $d^\prime \leq d$
since $d$ is the gcd of $a$ and $b$
and is the greatest common divisor of $a,b$.
Furthermore, $\dv{d}{a}$ and $\dv{d}{b}$ and therefore
$\dv{d}{ax+by}$ and thus $\dv{d}{d^\prime}$. This
implies $d \leq d^\prime$
%If $g$ is another common divisor of $a$ and $b$
%we conclude  from $a x + b y = d^\prime$ that $\dv{g}{ax+by}$ i.e.
%$\dv{g}{d^\prime}$.
%Thus $g \leq d^\prime$.
%Thus any common divisor of $a,b$ such
%as $g$ is no more than $d^\prime$. Such a common divisor
%also includes the gcd of $a$ and $b$ and thus $d \leq d^\prime$.
This $d \leq d^\prime$ along with the $d^\prime \leq d$ shows
that $d=d^\prime$.
%(We also proved that every common divisor of $a,b$ also divides the
%gcd of $a,b$.)
In conclusion
\[
  d = \gcd(a,b) =  d^\prime = a x + b y.
\]
\end{proof}

\subsection{Some corollaries}

\begin{cor}
\label{egcd1}
Let $a,b \in \mb{Z}$, where $|a|+|b| \neq 0$.
Let $d=\gcd(a,b)$.
Show that set $A$ defined as follows
\[
A= \{ au+bv | u,v \in \mb{Z} \ \wedge \        au+bv > 0 \},
\]
contains multiples of $d$.
\end{cor}
\begin{proof}
By Theorem~\ref{egcd},
the smallest positive element of $A$
was proven to be $d$.
Moreover every element $z$ of $A$ is of the form $z=au+bv$.
Since $\dv{d}{a}$ and $\dv{d}{b}$ we have that $\dv{d}{au+bv}$
that is $\dv{d}{z}$. Thus every element of $A$ is a multiple of $d$.
\end{proof}

\begin{cor}
\label{egcd2}
Let $a,b \in \mb{Z}$, where $|a|+|b| \neq 0$.
Let $a,b$ be such that $\gcd(a,b)=1$. Then there exist
$x, y \in \mb{Z}$ such that
\[
 ax+by = 1.
\]
\end{cor}
\begin{proof}
From Theorem~\ref{egcd} and Corollary~\ref{egcd1}
we know that there exist $x,y$
such that $ax+by = \gcd(a,b)$. Since $a,b$ are such that
$\gcd(a,b)=1$,  we obtain $ax+by =1$, as needed.
\end{proof}

\begin{cor}
\label{egcd3}
Let $\dv{a}{bc}$ and $\gcd(a,b)=1$. Then
$\dv{a}{c}$.
\end{cor}
\begin{proof}
From the previous problem we have that there exist
$x,y $ such that $ax+by =1$. Then,
\[
ac x + bc y = c.
\]
Since we are given $\dv{a}{bc}$ we conclude $\dv{a}{bcy}$.
Moreover, $\dv{a}{acx}$. Then
$\dv{a}{acx+bcy}= \dv{a}{c}$, as needed.
\end{proof}

\begin{cor}
If $d=\gcd(a,b)$ then $S(a) \cap S(b) = S(d)$.
(In other words, $\dv{m}{a} \mf{\; and \;} \dv{m}{b} \iff \dv{m}{d}$.)
\end{cor}

\begin{proof}
Every common divisor of $a$ and $b$ is a divisor of $d=\gcd(a,b)$ since
$d=d^\prime = ax+by$ derived and used in Theorem~\ref{egcd}.
That is $S(a) \cap S(b) \subseteq S(d)$.

Moreover for $t \in S(d)$ then $\dv{t}{d}$ and since $\dv{d}{a}$ and
$\dv{d}{b}$ by transitivity we have $\dv{t}{a}$ and $\dv{t}{b}$.
The former show that $t\in S(a)$ and the latter that $t\in S(b)$.
Both of them show that $t \in S(a) \cap S(b)$.
Thus $S(d) \subseteq S(a) \cap S(b)$.
\end{proof}

\begin{cor}
Let $a,b \in \mb{Z}$, $|a|+|b| \neq 0$.
Let $m \in \mb{Z}_+^*$.
Then 
\[
 \gcd( ma,mb ) = m \gcd( a,b). 
\]
\end{cor}

\begin{proof}
Let $d=\gcd(a,b)$. Since $\dv{d}{a}$ we have $a=dq$ and thus
$ma=(md)q$. Thus $\dv{md}{ma}$. Likewise $\dv{md}{mb}$.
Thus $md \leq \gcd(mb,ma)$.

Since $\dv{m}{ma}$ and $\dv{m}{mb}$ we have $\dv{m}{\gcd(ma,mb)}$.
Thus $\gcd(ma,mb) = m q$ for some $q$.
Moreover $\dv{\gcd{ma,mb}}{ma}$ implies $\dv{mq}{ma}$ i.e $\dv{q}{a}$.
Likewise $\dv{q}{b}$. Thus $\dv{q}{d}$ i.e. $q\leq d$.
We conclude that $\gcd(ma,mb) =mq \leq md$.
From $md \leq \gcd(mb,ma)$ previously and
     $\gcd(ma,mb) =mq \leq md $ the corollary follows.
\end{proof}

\begin{exa}
Calculate the $\gcd(30,105)$.
\end{exa}
\begin{solution}
We note that $30 <105$ so swapping takes place
\[
\begin{array}{lll}
               & \gcd(30,105)=        &:\gcd(105,30 )=   \\
               & 105=30 \cdot 3 +15   &:\gcd( 30,15) =   \\
               & 30 = 15\cdot 2 + 0   &:\gcd(15,0)   =   \\
               &                      &: 15              
\end{array}
\]
Therefore $\gcd(30,105)=15$.

\end{solution}

\begin{exa}
For $a=105$ and $b=30$ determine  $x,y$ such that
\[
  ax + b y = \gcd(a,b).
\]
\end{exa}
\begin{solution}
We observe that $\gcd(a,b)=\gcd(280,151)=1$.
We use Euclid's algorithm for the gcd as follows.
\[
\begin{array}{lll}
   105         &=&    30 \cdot 3  +  15 \\
    30         &=&    15 \cdot 2  +   0 \\
\end{array}
\]
The first non-zero remainder is $15$.
Thus $\gcd(105,30 )=15$.
We then use the extended GCD to find $x,y$ as follows.
\[
\begin{array}{lll}
   15          &=&   105 -  30 \cdot 3  \\
\end{array}
\]
We work from the last equation upwards, one step and one level
at a time.
\begin{eqnarray*}
 15 &=&  105- 30\cdot 3 \\
    &=&  105 \cdot 1 + 30 \cdot (-3).
\end{eqnarray*}
Therefore $x=1$ and $y=-3$.
\end{solution}

\subsection{More gcd results}

\bigskip %COROLLARY  %%
\begin{cor}
For three $a,b,c \in \mb{Z}$, we have $\gcd(a,b,c) = \gcd( (a,b),c)
= \gcd(d,c)$, where $d=\gcd(a,b)$.
\end{cor}

\begin{proof}
$\gcd(a,b,c)$ belongs to $S(a) \cap S(b) \cap S(c) = 
\left( S(a) \cap S(b) \right) \cap S(c)$. The result
follows.
\end{proof}

\bigskip %COROLLARY  %%
\begin{cor}
For $p$ a prime, and $a \in \mb{Z}^*$, $\ndv{p}{a}$
if and only if $\gcd(a,p)=1$.
\end{cor}

\begin{proof}
If $p$ is a prime $S(p) = \{ -1, +1, -p , +p \}$,
then $p$ does not divide $a$ means that neither $p$ nor $-p$ are
divisors of $a$. The only possible divisors are +1 and -1.
Thus $S(a) \cap S(p) = \{ +1 , -1 \}$. Thus the gcd is 1, and
thus $\gcd(a,p)=1$ as needed.

If $\gcd(a,p)=1$, then $p$ cannot divide $a$. This is because if
$\dv{p}{a}$ since obviously $\dv{p}{p}$ we would have from a prior
property that $\dv{p}{\gcd(a,p)}$ and (in fact $\gcd(a,p)=|p|$).
For this to happen $p$ should be $-1$ or $+1$. But $p$ is prime
and this can't happen.
\end{proof}

\bigskip %COROLLARY  %%
\begin{cor}
For $a,b,m \in \mb{Z}^*$ and $\gcd(a,b)=1$ and $\dv{a}{bm}$ then
$\dv{a}{m}$.
\end{cor}

\begin{proof}
Since $\gcd(a,b)=1$ we have $1=ax+by$ for some $x,y$.
Multiplying by $m$ we get $m=axm+bmy$. Obviously
$\dv{a}{axm}$. Since $\dv{a}{m}$ we have $\dv{a}{bmy}$ as well.
Thus $\dv{a}{axm+bmy}$ i.e. $\dv{a}{m}$.
\end{proof}

\bigskip %COROLLARY  %%
\begin{cor}
For $a,b \in \mb{Z}^*$ and prime $p$ is such that $\dv{p}{ab}$,
then $p$ divides either $a$ or $b$.
\end{cor}

\begin{proof}
Say $p$ does not divide $a$. By a prior theorem since $p$ is
prime this means $\gcd(p,a)= \gcd(a,p)=1$. Since $\dv{p}{ab}$
by the previous theorem we have $\dv{p}{b}$.
\end{proof}

\bigskip %COROLLARY  %%
\begin{cor}
\label{thm25}
If $\gcd(a,b)=1$ and $\dv{a}{c}$ and $\dv{b}{c}$ then $\dv{ab}{c}$.
\end{cor}

\begin{proof}
If $\gcd(a,b)=1$ we have $1=ax+by$. Then $c= acx+bcy$.
We have $\dv{a}{c}$ i.e. $c=aq$. Then $cb=abq$ and thus $bcy=ab(qy)$.
The latter implies $\dv{ab}{bcy}$.
We also have $\dv{b}{c}$ i.e. $c=br$. Then $ac=abr$. Thus $acx=ab(rx)$.
The latter implies $\dv{ab}{acx}$.
Thus $\dv{ab}{bcy}$ and $\dv{ab}{acx}$ imply $\dv{ab}{acx+bcy}$.
Therefore $\dv{ab}{c}$.
\end{proof}

\subsection{Relatively prime integers}

\begin{dfn}
For two integers $a,b \in \mb{Z}$ if $\gcd(a,b) =1$ the 
two integers $a$ and $b$ are called {\bf relatively prime}.
\end{dfn}

\subsection{Extended GCD calculation}

\begin{lem}
Let $a \in \mb{N}^{*}$, $b \in \mb{Z}$ such that
$a > b$, $b \neq 0$.
$ $ \\ $ $
Then algorithm ExtendedGCD(a,b)  returns
(x,y) such that
\[
  ax + b y = \gcd(a,b) = d .
\]
Furthermore show that
if $d=1$ and $b>0$ then
$ y \bmod a$ and
$ x \bmod b$ are the inverses
of
$b \bmod a$ and
$a \bmod b$ respectively.
\end{lem}

\begin{proof}
$ $ \\ $ $ 
\begin{algorithm}[H]
\KwIn{$a,b$, where $a \in \mb{N}^{*}$, $b \in \mb{Z}$, $a>b$, $b\neq 0$}
\KwOut{$(x,y,d)$}
 $(x,y,d,k,l,m) = (1,0,a,0,1,b)$ ;

 \While{$m >  0$}{
     $t = d / m $;

     $(x,y,d,k,l,m) =
      (k,l,m,x-t \cdot k, y-t \cdot l, d- t \cdot m) $\;
 }
 \Return{$(x,y,d)$};
\caption{Extendedgcd(a,b)}
\label{egcdAlg}
\end{algorithm}
Details Omitted.
\end{proof}

\begin{exa}
Trace Extendedgcd(a,b) for $a=280$ and $b=105$.
\end{exa}
\begin{solution}
\[
\begin{array}{lllll}
   \gcd(280,105)  &  &                         \\
(x,y,d,k,l,m)     &=& (1,0,280,0,1,105)  &,&
                   t = 280/105 = 2. \\
(x,y,d,k,l,m)     &=& (0,1,105, 1-2*0, 0-2*1 , 280-2*105) &,& \\
                  &=& (0,1,105,  1   ,  -2   , 70       ) &.& \\
(x,y,d,k,l,m)     &=& (0,1,105,  1   ,  -2   , 70       ) &,&
                   t = 105/70 =1. \\
(x,y,d,k,l,m)     &=& (1,-2,70, 0-1*1, 1-1*(-2) , 105-1*70 ) &,& \\
                  &=& (1,-2,70, -1   ,  3    , 35       ) &.&  \\
(x,y,d,k,l,m)     &=&  (1,-2,70, -1   ,  3    , 35       ) &,&
                   t = 70 / 35=2. \\
(x,y,d,k,l,m)     &=& (-1,3,35, 1-2*(-1),-2-2*( 3) , 70 -2*35 ) &,& \\
                  &=& (-1,3,35,  3   ,  -8   ,  0       ) &.& \\
(x,y,d)           &=& (-1,3,35)                           &.& \\
\end{array}
\]
\end{solution}

\newpage

\section{Least common multiples}

\noindent
\begin{dfn}[Least common multiple :lcm]
For $a,b \in \mb{Z}$ we denote by $\lcm(a,b)$ the
{\bf least common multiple } of $a$ and $b$,
the least positive integer divisible by both $a,b$.
\end{dfn}

For $a,b \in \mb{Z}$ only
the positive multiples are  then considered.

\begin{prp}
Let $a,b \in \mb{Z}$ are such that $|a|+|b| \neq 0$,
and let $n= \lcm(a,b)$.
If $m$ is a common multiple of $a,b$ show
that $\dv{n}{m}$.
\end{prp}

\noindent
This says that any common multiple of $a,b$
is a multiple of the least common multiple (of $a,b$).

\noindent
If $T(a)$ is the set of (positive) multiples of $a$,
then $T(a) = T(-a) = T(|a|)$.
Moreover $T(a) \cap T(b) = T(|a|) \cap T(|b|)$.
Therefore $\lcm(a,b)=\lcm(|a|,|b|)$.

\begin{proof}
If one of $a,b$ is equal to zero, then all multiples
of zero is zero, common muliples can only be zero
and
the statement is true trivially.
$ $ \\ $ $
Otherwise divide $m$ by $n$.
\[
    m = n q + r,  0\leq r < n.
\]
We are given that $m$ is a common multiple of $a,b$
that is $\dv{a}{m}$ and $\dv{b}{m}$.
Moreover $n$ is the lcm of $a,b$
that is $\dv{a}{n}$ and $\dv{b}{n}$.
Therefore $\dv{a}{r}$ and $\dv{b}{r}$.
Thus $r$ is a common multiple of $a,b$.
Moreover $r< n$ and the common multiple $r$ is positive
and LESS than the least common multiple $a,b$,
a contradiction unless $r=0$.
Then $m=nq$ and thus $\dv{n}{m}$.
\end{proof}

Moreover $\lcm(ma,mb) = m \lcm(a,b)$.

\begin{prp}
For $a, b \in \mb{Z}$ we have
$\lcm(am,bm) = m \lcm(a,b)$.
\end{prp}
\begin{proof}
$ $ \\ $ $
(a)
Let $X=\lcm(a,b)$ and $Y=\lcm(am,bm)$. We are to show
that $Y=Xm$. It suffices to show $Y \geq Xm$ and $Y \leq Xm$.
$ $ \\ $ $
Since $X=\lcm(a,b)$ then $a/X$ and $b/X$.
Therefore $X=a d_1 $ and $X=b d_2$ for some $d_1 , d_2$.
Consequently $Xm =mad_1$ and $Xm=mbd_2$ respectively.
Furthermore $\dv{ma}{Xm}$ and $\dv{mb}{Xm}$. Thefore,
since $Y=\lcm(am,bm)$, we have $Xm \geq Y$.
$ $ \\ $ $
Consider $Y=\lcm(am,bm)$. Them $\dv{m}{Y}$, and
therefore there exists a  $y$ such that
$Y=m  y$.
\[
Y=\lcm(am,bm) \Rightarrow \dv{m}{Y}
              \Rightarrow \exists y : Y=my.
\]
Combining the latter with $Y=my$ we obtain
\[
Y=my , Y=amd_3 ,  Y=bmd_4
\Rightarrow  my=amd_3 , \quad my = bmd_4
\Rightarrow   y=a d_3 , \quad  y = b d_4
\Rightarrow   \dv{a}{y}, \quad \dv{b}{y},
\]
\[
\Rightarrow   \dv{a}{y}, \quad \dv{b}{y},
\Rightarrow   y \geq X
\Rightarrow   my \geq mX
\Rightarrow   Y  \geq Xm.
\]
From $Xm \geq Y$ and $Y\geq mX$ we conclude
$Y=Xm$ as needed.
\end{proof}

\bigskip %THEOREM 26 %%
\begin{thm}
\label{lcm1}
Let $a,b \in \mb{N}$. If $\gcd(a,b)=1$ then $\lcm(a,b)=ab$.
\end{thm}

\begin{proof}
Let $m=\lcm(a,b)$. Since $\gcd(a,b)=1$ we have $1=ax+by$.
Then $m=axm+bym$. Since $\dv{a}{m}$ and $\dv{b}{m}$.
Then $m=ap$ and $m=bq$. Moreover substituting to the previous
equation we have $m=ax(bq)+by(ap)=ab(xq)+ab(yp)= ab (xq+yp)$.
Thus $\dv{ab}{m}$. This implie $ab \leq m$.
Moreover $m$ is the least common multiple of $a,b$.
One such multiple is $ab$. Thus $ab \geq m$.
The $ab\leq m$ and the just shown $ab \geq m$ implies
$m=ab$.
\end{proof}

\begin{exa}
Proof Theorem~\ref{lcm1} using other arguments.
\end{exa}
\begin{solution}
 Let $d=\gcd(a,b)$. Let $a_1 = a/d$ and $b_1 = b/d$.
Then $1 = d/d= \gcd(a/d,b/d) = \gcd(a_1, b_1 )$.
From the previous problem, part (b), we have that
for $\gcd(a/d,b/d)=1$,
\[
\frac{1}{d} \lcm(a,b) = \lcm(a_1 , b_1 )= \lcm( a/d , b/d)=
  ab /d^2 = \frac{a \cdot b}{\gcd(a,b)}
\Leftrightarrow   \lcm(a,b) \cdot \gcd(a,b) = a \cdot b.
\]
\end{solution}

\bigskip %THEOREM 27 %%
\begin{thm}
Let $a,b \in \mb{N}$. Then $\gcd(a,b) \cdot \lcm(a,b)=ab$.
\end{thm}

\begin{proof}
Let $d=\gcd(a,b)$. Let $a_1 = a/d$ and $b_1 = b/d$.
Then $\gcd(a/d,b/d) = \gcd (a_1, b_1 ) =1$.
From the previous theorem we have $\lcm( a_1 , b_1 )=
\lcm( a/d , b/d)=  ab /d^2$.
\end{proof}

%38
If $T(a)$ is the set of (positive) multiples of $a$, then
$T(a) = T(-a) = T(|a|)$. Moreover $T(a) \cap T(b) = T(|a|) \cap T(|b|)$.
Therefore $\lcm(a,b)=\lcm(|a|,|b|)$.

\begin{cor}
For $a,b \in \mb{Z}^*$ if $\dv{a}{b}$, then $\lcm(a,b)=b$.
\end{cor}

\begin{proof}
If $\dv{a}{b}$, then any element $t \in T(b)$ is $t\geq b$.
Moreover $b$ or $|b|$ is in $T(a)$. The result follows.
\end{proof}

\begin{thm}[lcm of $a_1 , \ldots , a_{n}$]
Let $a_1 , a_2 , \ldots , a_n \in \mb{Z}$.
Let
\begin{eqnarray}
\label{lcmp1}
\lcm( a_1     , a_2 )     &=& r_1     \nonumber \\
\lcm( r_1     , a_3 )     &=& r_2     \nonumber \\
 \ldots                   &=& \ldots            \\
\lcm( r_{n-3} , a_{n-1} ) &=& r_{n-2} \nonumber \\
\lcm( r_{n-2} , a_n )     &=& r_{n-1} \nonumber
\end{eqnarray}
Then show that
\begin{equation}
\label{lcmp2}
\lcm( a_1 , a_2 , \ldots , a_n) = r_{n-1}.
\end{equation}
\end{thm}

\begin{proof}
$ $ \\ $ $
We first work top to bottom.
It is the case that
$\dv{r_1}{r_2}$,
$\dv{r_2}{r_3}$,
$\ldots $,
$\dv{r_{n-1}}{r_n}$.
It is clear that $\dv{a_1}{r_1}$ and since
$\dv{r_1}{r_2}$ we have $\dv{a_1}{r_2}$.
Proceeding similarly we have
$\dv{a_1}{r_i}$ for all $i=1 , \ldots , n-1$.
Likewise we can show $\dv{a_i}{r_{n-1}}$,
and thus $r_{n-1}$ is a common multiple of
$a_1 , \ldots , a_n$ and thus by the previous
problem
$\lcm( a_1 , a_2 , \ldots , a_n) $ divides
$r_{n-1}$ or equivalently,
\[
\lcm( a_1 , a_2 , \ldots , a_n) \leq r_{n-1} .
\]
$ $ \\ $ $
Let $c$ be a common multiple of $a_1 , \ldots a_n$.
Then $\dv{a_i}{c}$ for all $i$.
In particular $\dv{a_1}{c}$ and $\dv{a_2}{c}$.
Thus $c$ is a common multiple of $a_1 , a_2$.
From the previous problem this implies
that $\lcm(a_1 , a_2)$ which is $r_1$ divides
$c$. Thus $\dv{r_1}{c}$.
From the following equation we have
$\dv{r_1}{c}$ just established and $\dv{a_3}{c}$
since $c$ is a common multiple of all $a_i$.
With a similar argument as before we conclude
$\dv{r_2}{c}$.
By induction we have $\dv{r_i}{c}$ including
$\dv{r_{n-1}}{c}$, which generates a
\[
 r_{n-1} \leq c.
\]
Combining the two inequalities derived we have
\[
\lcm( a_1 , a_2 , \ldots , a_n) \leq r_{n-1} \leq c,
\]
for every common multiple $c$ of $a_1 , \ldots , a_n$,
and $r_{n-1}$ is one such common multiple.
The only possibility is that
$r_{n-1} =\lcm( a_1 , a_2 , \ldots , a_n)$,
as needed,
that is there is no common multiple smaller than
$r_{n-1}$ because $r_{n-1}$ is the least common
multiples of $a_1 , \ldots , a_n$.
\end{proof}

\newpage

\section{Diophantine equations}

\bigskip %THEOREM 28 %%
\begin{thm}
\label{diot}
Let $a,b \in \mb{Z}^*$.
The linear Diophnatine equation
\begin{equation}
\label{dio}
ax+by = c
\end{equation}
has an integer solution $x,y$ if and only if 
$\dv{d}{c}$, where $d=\gcd(a,b)$.
If a particular solution $(x_0 ,y_0 )$ exists, then there 
are infinitely
many solution of the form
\[
  x = x_0 + \frac{mb}{d}  , \quad\quad y = y_0 - \frac{ma}{d},
\]
where $m \in mb{Z}$.
\end{thm}

\begin{proof}
$ $ \\ $ $
$\Rightarrow$. Let $x,y$ be a solution for $ax+by=c$.
Then since $d=\gcd(a,b)$ we have $\dv{d}{a}$ and $\dv{d}{b}$
and thus $\dv{d}{ax+by}$ implying $\dv{d}{c}$. If $\ndv{d}{c}$
obviously there are no $x,y$.
$ $ \\ $ $
$\Leftarrow$.
%By way of Theorem~\ref{egcd} there exist $x,y$ such that
Let $\dv{d}{c}$. By way of Theorem~\ref{egcd} 
of the extended GCD algorithm or the definition of $\gcd(a,b)$
there exist $x,y$ such that
\begin{equation}
\label{dio1}
ax+by = d
\end{equation}
where $d=\gcd(a,b)$. Since $\dv{d}{c}$ there exists a $q$ such that
$c=dq$. Multiplying Equation~\ref{dio1} by $q$ we get
\begin{equation}
\label{dio2}
a(xq)+b(yq) = dq=c
\end{equation}
Thus a solution has been found through Theorem~\ref{egcd}. 
The solution is $x_0 = xq$ and $y_0 = yq$. 
Let $(x_0 , y_0)$ be a solution pair. Suppose there is another
solution pair $(x,y)$. Then
\begin{equation}
\label{dio3}
a x_0 +b y_0  = c  \quad \quad , \quad \quad a x   + b y   = c 
\end{equation}
Subtracting one from the other we get
\begin{equation}
\label{dio4}
 a(x_0 -x) + b (y_0 -y) = 0 \Rightarrow  a(x_0 - x) = b (y-y_0 ).
\end{equation}
For $d=\gcd(a,b)$, $a_1 = a/d$ and $b_1 = b/d$ are both integer.
Moreover $\gcd( a_1  , b_1 )=1$ from a prior result.
Dividing by $d$ equation~\ref{dio4} we have.
\begin{equation}
\label{dio5}
 \frac{a}{d}(x_0 - x) = \frac{b}{d} (y-y_0 ) \Rightarrow
      a_1   (x_0 - x) =   b_1       (y-y_0 ) 
\end{equation}
Since $\gcd(a_1  , b_1 )=1$, since $\dv{a_1}{a_1 (x_0 -x)}$ it 
means $\dv{a_1}{b_1 (y-y_0)}$. Being relatively prime $a_1, b_1$ this
is equivalent to $\dv{a_1}{y-y_0}$. Likewise $\dv{b_1}{x-x_0}$.
From the latter we obtain that there exists, $m$  such that
\begin{eqnarray*}
\dv{b_1}{x-x_0} &\Rightarrow&  x-x_0 = m b_1 \\
                &\Rightarrow&  x = x_0 + m \frac{b}{d}
\end{eqnarray*}
From Eq.(\ref{dio4}) and $x-x_0 = m b_1$ we obtain the following
\begin{eqnarray*}
x-x_0 = m b_1  , \quad
a_1   (x_0 - x) =   b_1       (y-y_0 )
 &\Rightarrow& - m a_1 b_1 = b_1 (y - y_0 )  \\
 &\Rightarrow&  y = y_0 - m \frac{a}{d}.
\end{eqnarray*}
The result thus follows.
\end{proof}

\begin{cor}
If $\gcd(a,b,c)=1$ and $\gcd(a,b) =d > 1$ then Equation~\ref{dio} has
no solution.
\end{cor}

\begin{proof}
Say that Equation~\ref{dio} has an integer solution $(x,y)$ i.e.
$ax+by =c$. Since $d=\gcd(a,b)$ it means $\dv{d}{a}$ and $\dv{d}{b}$.
Then $\dv{d}{ax+by}$ i.e. $\dv{d}{c}$. If $d$ is a common divisor of
$a,b,c$ it means $\dv{d}{1}$. Then it can only be $d=1$. This contradicts
the fact $d>1$. Thus Equation~\ref{dio} cannot have any solution.
\end{proof}

\begin{cor}
\label{edio}
If $\gcd(a,b)=1$ then Equation~\ref{dio} has one solution.
\end{cor}

\begin{cor}
Let $a,b \in \mb{Z}^*$.
The linear Diophantine equation
\begin{equation}
\label{dio10}
ax+by = c
\end{equation}
has an integer solution $x,y$ if and only 
if $\dv{d}{c}$, where $d=\gcd(a,b)$.
If a particular solution $(x_1 ,y_1 )$ exists for $ax+by=d$,
many solutions of the form
\[
  x = \frac{c}{d} x_1 + m \cdot \frac{b}{d}  , \quad\quad
  y = \frac{c}{d} y_1 - m \cdot \frac{a}{d},
\]
exist, where $m \in \mb{Z}$.
\end{cor}
\begin{proof}
Using the extended GCD algorithm,
by way of Eq.(\ref{dio2}), $x_1 , y_1$ are solutions
of $ax+by = \gcd(a,b)$. If $\dv{d}{c}$ then there exists
a $q \in \mb{Z}$ such that $c=dq$.
Multiplying Eq.(\ref{dio2}) by $q$ we derive
Eq.(\ref{dio3}) and thus $x_0 = x_1 q$ and $y_0 = y_1 q$ are
solutions of Eq.(\ref{dio3}).
Therefore a solution $(x_1 , y_1)$ of Eq.(\ref{dio2})
gives rise to a solution
\[
x_0 = x_1 q = \frac{c}{d} x_1  , \quad
y_0 = y_1 q = \frac{c}{d} y_1
\]
of Eq.(\ref{dio3}), and using the previous
problem many solutions of the form
\[
  x =
            x_0 + m \cdot \frac{b}{d}  =
\frac{c}{d} x_1 + m \cdot \frac{b}{d}  , \quad\quad
  y =
            y_0 - m \cdot \frac{a}{d} =
\frac{c}{d} y_1 - m \cdot \frac{a}{d},
\]
exist, where $m \in \mb{Z}$.
\end{proof}

\begin{cor}
If $\gcd(a,b,c)=1$ and $\gcd(a,b) =d > 1$ 
then Equation~\ref{dio} has no solution.
\end{cor}
\begin{proof}
Say that Equation~\ref{dio} has an integer solution $(x,y)$ i.e.
$ax+by =c$. Since $d=\gcd(a,b)$ it means $\dv{d}{a}$ and $\dv{d}{b}$.
Then $\dv{d}{ax+by}$ i.e. $\dv{d}{c}$. If $d$ is a common divisor of
$a,b,c$ it means $\dv{d}{1}$. Then it can only be $d=1$. This contradicts
the fact $d>1$. Thus Equation~\ref{dio} cannot have any solution.
If $\gcd(a,b)=1$ then Equation~\ref{dio} has one solution.
\end{proof}

\newpage

\begin{exa}      
Show $\gcd( 1024, 640 ) =128$ we have
\end{exa}      

\begin{proof}
\begin{eqnarray*}
  1024 &=& 640 \cdot 1 + 384 \\
   640 &=& 384 \cdot 1 + 256 \\
   384 &=& 256 \cdot 1 + 128 \\
   256 &=& 128 \cdot 2 + 0  
\end{eqnarray*}
Obviously $\gcd(1024,640)= 128$, the last non zero remainder.
Reversing the order of the equation we have.
\begin{eqnarray*}
   128 &=&  384 + 256 \cdot (-1) \\
       &=&  384 + (640 + 384 \cdot (-1) ) \cdot (-1) \\
       &=&  640 \cdot (-1) + 384 \cdot (2) \\
       &=&  640 \cdot (-1) + (1024 + 640 \cdot (-1) ) \cdot 2 \\
       &=&  640 \cdot (-3) + 1024 \cdot 2 
\end{eqnarray*}
Therefore $\gcd(1024,640) = 128 = 1024 \cdot 2 + 640 \cdot (-3)$.
\end{proof}

\begin{exa}      
Find the solutions, if any, of  $1024x+640y= 256$.
\end{exa}      

\begin{proof}
By the previous example $\gcd(1024,640)=d=128$. It is $\dv{128}{256}$.
Thus one solution of the Diophantine is 
$x_0 = 2 \cdot (256/128) = 4$ and $y_0 = (-3) \cdot (256/128) = -6$.

Other solutions are
\[
x = 4 + m (640/128) = 4+5m  \quad , \quad y = -6 -8m
\]
A simple calculation confirms the latter solutions
\[
1024(4+5m) + 640(-6-8m) = 256.
\]
\end{proof}

The previous method outlined in Corollary~\ref{edio} is tedious.
A better approach for solving $ ax+by = d$, where $d=\gcd(a,b)$ starts with.

\[
\left[ \begin{array}{c|cc}
     a & 1 & 0 \\
     b & 0 & 1 \\  
       \end{array} \right]
\]

The first column are the the Dividend ($a$) and Divisor ($b$) of the
division operation. Eventually through repeated division with remainder
(the previous vertical operations become horizontal) will generate a
matrix such as the one below. Its first row entries contains
the $\gcd(a,b)$ and $x,y$.

\[
\left[ \begin{array}{c|cc}
   \gcd(a,b)   & x & y \\
     0         & ? & ? \\  
       \end{array} \right]
\]

\newpage
\begin{exa}      
Show this for $a=1024 , b= 640 $ and division
\begin{eqnarray*}
  1024 &=& 640 \cdot 1 + 384 \\
   640 &=& 384 \cdot 1 + 256 \\
   384 &=& 256 \cdot 1 + 128 \\
   256 &=& 128 \cdot 2 + 0  
\end{eqnarray*}
\end{exa}      

\begin{proof}
\begin{eqnarray*}
\left[ \begin{array}{r|rr}
  1024 & 1 & 0 \\
   640 & 0 & 1 \\  
       \end{array} \right] 
&\rightarrow&
\left[ \begin{array}{r|rr}
   384 & 1 & -1 \\
   640 & 0 &  1 \\  
       \end{array} \right]  \\
&\rightarrow&
\left[ \begin{array}{r|rr}
   384 & 1 & -1 \\
   256 &-1 &  2 \\  
       \end{array} \right]  \\
&\rightarrow&
\left[ \begin{array}{r|rr}
   128 & 2 & -3 \\
   256 &-1 &  2 \\  
       \end{array} \right]  \\
&\rightarrow&
\left[ \begin{array}{r|rr}
   128 & 2 & -3 \\
    0  &-5 & -8 \\  
       \end{array} \right]  \\
\end{eqnarray*}
In the first transition, we subtract the second row from the first per
the first division.
In the second transition, we subtract the first row from the second per
the second division.
In the third  transition, we subtract the second row from the first  per
the third  division.
In the fourth transition, we subtract twice the first row from the second  per
the fourth  division. (It is twice because the quotient in this case is a 2.)

\end{proof}

There is yet another matrix form representation of the Extended-Euclid's
algorithms (i.e. Extended GCD). For this in Theorem~\ref{gcd} we rewrite
the first line as in $a=bq+r = b q_0 + r_0 $,
and the second line $b=rq_1 + r_1 = r_0 q_1 +r_1 $. 
The remaining lines remain the same.
The product of $k+2$ matrices can be computed as a 
$2 \times 2$ matrix with entries $x_1 , \ldots x_4$. 
\begin{eqnarray*}
\left[  \begin{array}{r}
   a \\
   b \\
       \end{array} \right] &=&
\left[ \begin{array}{rr}
    q_0 & 1 \\
    1   & 0  \\
       \end{array} \right]   \times
\left[  \begin{array}{r}
   b   \\
   r_0 \\
       \end{array} \right]  =
\left[ \begin{array}{rr}
    q_0 & 1 \\
    1   & 0  \\
       \end{array} \right]   \times
\left[ \begin{array}{rr}
    q_1 & 1 \\
    1   & 0  \\
       \end{array} \right]   \times
\left[  \begin{array}{r}
   r_0 \\
   r_1 \\
       \end{array} \right]  = \ldots  \\
&=& 
\left[ \begin{array}{rr}
    q_0 & 1 \\
    1   & 0  \\
       \end{array} \right]   \times
\left[ \begin{array}{rr}
    q_1 & 1 \\
    1   & 0  \\
       \end{array} \right]   \times \ldots \times
\left[ \begin{array}{rr}
    q_{k+1} & 1 \\
    1   & 0  \\
       \end{array} \right]   \times
\left[  \begin{array}{r}
   r_k \\
    0  \\
       \end{array} \right]  
\\ &=&
   \prod_{i=0}^{i=k+1}
\left[ \begin{array}{rr}
    q_{i} & 1 \\
    1     & 0  \\
       \end{array} \right]   \times
\left[  \begin{array}{r}
   r_k \\
    0  \\
       \end{array} \right]   =
\left[ \begin{array}{rr}
    x_1   & x_2 \\
    x_3   & x_4  \\
       \end{array} \right]   \times
\left[  \begin{array}{r}
   r_k \\
    0  \\
       \end{array} \right]   \\
 &\Rightarrow& \\
\left[  \begin{array}{r}
   r_k \\
    0  \\
       \end{array} \right]   &=&
\left[ \begin{array}{rr}
    x_1   & x_2 \\
    x_3   & x_4  \\
       \end{array} \right]^{-1}   \times
\left[  \begin{array}{r}
    a  \\
    b  \\
       \end{array} \right]   \\
&\Rightarrow& \\
\left[  \begin{array}{r}
   \gcd(a,b) \\
    0  \\
       \end{array} \right]   &=&
\left[ \begin{array}{rr}
     x    &  y  \\
     *    &  *   \\
       \end{array} \right]        \times
\left[  \begin{array}{r}
    a  \\
    b  \\
       \end{array} \right]    \Rightarrow \gcd(a,b) = ax + by.
\end{eqnarray*}

\newpage
\begin{exa}      
Show this for $a=1024 , b= 640 $ and division
\begin{eqnarray*}
  1024 &=& 640 \cdot 1 + 384 \\
   640 &=& 384 \cdot 1 + 256 \\
   384 &=& 256 \cdot 1 + 128 \\
   256 &=& 128 \cdot 2 + 0  
\end{eqnarray*}
\end{exa}      
\begin{proof}
\begin{eqnarray*}
\left[  \begin{array}{r}
   r_k \\
    0  \\
       \end{array} \right]   &=&
\left[ \begin{array}{rr}
    x_1   & x_2 \\
    x_3   & x_4  \\
       \end{array} \right]^{-1}   \times
\left[  \begin{array}{r}
    a  \\
    b  \\
       \end{array} \right]   , \quad\quad
\left[ \begin{array}{rr}
    x_1   & x_2 \\
    x_3   & x_4  \\
       \end{array} \right]   =
\left[ \begin{array}{rr}
    1     & 1   \\
    1     & 0   \\
       \end{array} \right]        \times
\left[ \begin{array}{rr}
    1     & 1   \\
    1     & 0   \\
       \end{array} \right]        \times
\left[ \begin{array}{rr}
    1     & 1   \\
    1     & 0   \\
       \end{array} \right]        \times
\left[ \begin{array}{rr}
    2     & 1   \\
    1     & 0   \\
       \end{array} \right]  =        
\left[ \begin{array}{rr}
    8     & 3   \\
    5     & 2   \\
       \end{array} \right] 
\end{eqnarray*}

Then,
\begin{eqnarray*}
\left[  \begin{array}{r}
   r_k \\
    0  \\
       \end{array} \right]   &=&
\left[ \begin{array}{rr}
     8    &  3  \\
     5    &  2   \\
       \end{array} \right]^{-1}   \times
\left[  \begin{array}{r}
    1024  \\
    640  \\
       \end{array} \right]   =
\left[ \begin{array}{rr}
     2    & -3  \\
    -5    &  8   \\
       \end{array} \right]        \times
\left[  \begin{array}{r}
    1024  \\
    640  \\
       \end{array} \right]   
\end{eqnarray*}

We just need to find the top element of the vector i.e. its
$r_k$ values. Obviously $r_k = 128= 1024 \cdot 2 + 640 \cdot (-3) $.
At the same time $(x,y)=(2,-3)$ as well.
\end{proof}

\bigskip %THEOREM 29 %%
\begin{thm}
For $a,p \in \mb{Z}^*$ such that $\gcd(a,p)=1$, there is a unique
$x$ such that $a a^\prime \equiv 1 \pmod p$.
\end{thm}

\begin{proof}
Since $\gcd(a,p)=1$ we have by Theorem~\ref{gcd} that
there exists $x,y$ such that $ax+py =1 =\gcd(a,p)$.
Furthermore $-py = ax-1$. Since $\dv{p}{py}$ we have $\dv{p}{ax-1}$.
Thus $ax \equiv 1 \pmod p$. Consider $a^\prime = x \pmod p$.
It is still $a a^\prime \equiv \pmod p$. Furthermore any divisor
of $a^\prime$ and $m$ must also divide $1$, i.e. $\gcd(a^\prime , p)=1$.
\end{proof}

\newpage

\section{Fundamental theorem of arithmetic}

\subsection{Unique factorization}

\begin{thm}
Let $n \in \mb{N}$ with $n>1$.
Then $n$ can be expressed as the product of prime
numbers (natural numbers with two divisors one, and itself).
The representation is unique up to a permutation of the
ordering of the primes.
\end{thm}

We usually represent the prime numbers from the
smallest (left-most side) to the largest (right-most side).
If a prime number $p$ appears $k$ times we also write $p^k$.
If $n$ itself is a prime number the expression
"product" is used degeneratively. The product of one prime is 
the prime itself.

\begin{proof}
$ $ \\ $ $
We prove the result by induction: $n>1$ becomes $n\geq 2$.
$ $ \\ $ $
{\bf Base case $n=2$.} Integer $n=2$ is the degenerate case
of $n=2$ where there is only one prime factor of $n=2$ and this
is itself.
$ $ \\ $ $
{\bf Induction hypothesis: $H(m)$, $2 \leq m < n$.}
For $n \geq 2$ we have that
for every integer $m$ such that $2 \leq m < n$, then $m$
can be expressed as the products of primes as noted in the
statement of the problem.
$ $ \\ $ $
{\bf Inductive step: $n$} Consider integer $n$.
We distinguish two cases: (a) $n$ is a prime number
and we are done, (b) $n$ is not a prime number.
In the latter case there exists by Proposition~\ref{compo} a
prime factor of $1<p < n$. Call then $P=n/p$. Then
$1< P < n$. By the induction hypothesis $P$ can be expressed
as a product of primes, and so does $n$, since $n= P \cdot p$,
where $P$ is a product of primes then, and $p$ is also a prime.
$ $ \\ $ $
{\bf Uniqueness of the representation.}
Let $p_1 p_2 \ldots p_r$ is a representation and
$q_1 q_2 \ldots q_s$ is another representation of $n$.
\[
n = p_1 p_2 \ldots p_{r-1} p_r = q_1 q_2 \ldots q_{s-1} q_s ,
\]
and let
$1< p_1 \leq p_2 \leq \ldots \leq p_r$,
and
$1< q_1 \leq q_2 \leq \ldots \leq q_r$,
and furthermore $1 \leq r \leq s$.
By using induction we will show that
$r=s$ and $p_i = q_i$ for all $i=1, 2, \ldots , r(=s)$.
Induction is on $r$.
$ $ \\ $ $
{\bf Base case $r=1$.} We have
\[
n = p_1                = q_1 q_2 \ldots q_s .
\]
On the left-hand side we have $p_1$ a prime number.
If $s > 1$ on the right-hand side we have the product of
at least two primes. This is a contradiction. Therefore $s=1$.
Then $p_1 =q_1$ and we are done.
$ $ \\ $ $
{\bf Induction hypothesis.}
We assume that the result is true for $r-1$, where $r\geq 2$.
Consider the following.
Consider the following.
\[
n = p_1 p_2 \ldots p_{r-1} p_r = q_1 q_2 \ldots q_{s-1} q_s ,
\]
$p_r$ is a factor of $n$. It also thus divides the $q$ product.
Since $p_r$ is a prime number and all $q_i$ are prime numbers
it can't by that $\dv{p_r}{q_i}$ and $p_r < q_i$ because then
$q_i$ would have a prime factor  and could not be a prime number
itself. Thus the only way $\dv{p_r}{q_i}$ is for $p_r = q_i$.
In other words $\dv{p_r}{q_i}$ implies $p_r \leq q_i$.
We might also work out similarly to show that $q_s \leq p_j$.
Then
\[
q_s \leq p_j \leq p_r \leq q_i \leq q_s
\]
implies $p_r =q_s$. Since both $p_r > 1 , q_s >1$ we
divide the expression for $n$ by $p_r =q_s$. We then
obtain
\[
 p_1 p_2 \ldots p_{r-1}  = q_1 q_2 \ldots q_{s-1} ,
\]
and by the induction hypothesis $r-1=s-1$
and $p_1 =q_1 $, $\ldots $, $  p_{r-1}= q_{s-1}$.
\end{proof}

\subsection{Fundamental theorem}

\medskip
\begin{lem}
\label{pfactor}
For a prime number $p$ if $\dv{p}{p_1 p_2}$ then
$\dv{p}{p_1}$ or $\dv{p}{p_2}$. This can be generalized
for $\dv{p}{p_1 p_2 p_3 \ldots}$.
\end{lem}
\begin{proof}
If $\dv{p}{p_1 p_2}$ given that $p$ is prime its only
positive divisors are $1,p$. If $\ndv{p}{p1}$, then
$\gcd(p,p_1 )=1$. Thus from a prior result $\dv{p}{p_2}$.
By induction we can prove its generalization.
\end{proof}

\bigskip %THEOREM 32 %%
\begin{thm}[Fundamental theorem of arithmetic]
\label{fund}
If $n>1$ there there exists unique prime numbers
$p_1 < \ldots < p_k$ and natural integers $a_1 , \ldots , a_k >0$
such that
\[
  n = p_1^{a_1} p_2^{a_2} \ldots p_k^{a_k}.
\]
\end{thm}

\begin{proof}
%Let $Q$ be the set of integers not conforming to the precondition of
%Theorem~\ref{fund} i.e. integers with no unique factorization.
%
%If $Q$ is not empty by the well-ordered set principle $Q$ has a
%minimum element and let it be $m$. If $m$ is prime then
%it is its own factorization and we have a violation
%by the fact  $m$ being inside $Q$. Thus a prime $m$ cannot be in $Q$. 
%In other words
%$m$ is a composite, and only composites are in $Q$. 
%Then $m=qr$ where $1< q \leq r <m$.
%By the minimality of $m$, both $q,r$ cannot be in $Q$ that is they 
%have unique factorizations as products of prime. 
%This means that $m =qr$ is the product of primes.
%
%The only way $m$ is in violation is that it has two or more
%different prime factorizations.
%
%Let $m=p_1^{a_1}p_2^{a_2}\ldots p_k^{a_k}=q_1^{b_1}q_2^{b_2}\ldots q_k^{b_t}$.
% where $p_1 < \ldots   < p_k$ and $q_1 < \ldots < q_t$ are all primes.
%Since $p_1$ is prime and $\dv{p_1}{m}$ it follow from Lemma~\ref{pfactor}
%that $\dv{p_1}{q_i}$. All $p_* , q_*$ are prime and this translates into
% $p_1=q_i$. Consider $M=m/p_1$. We have
%$M =p_1^{a_1-1}p_2^{a_2}\ldots p_k^{a_k} =
%    q_1^{b_1}  q_2^{b_2}\ldots q_i^{b_i -1}\ldots q_k^{b_t}$.
% $M<m$. Thus $M$ cannot be a member of $Q$ and thus the two
% factorization of $M$ in terms of the $p_*$'s or in terms of the
% $q_*$ must be identical. This implies that $k=t$ and
% $p_i = q_i$ for all indexes $i$. Moreover $a_i = b_i$ as well.
% But them $m$'s factorization would also be unqiue.
% That is $m$ should not be in $Q$. Thus $Q$ should be empty!
If $n$ is prime this is true easily.
Let $n$ be a composite natural number. Then by Theorem~\ref{prfactor}
it has a prime factor and let its smallest one be $q_1$ i.e.
\[
n= q_1 \cdot a_1 , \quad  \mbox{\ with\ } a_1 < n.
\]
If $a_1$ is prime we have decomposed $n$ into two prime numbers.
If $a_1$ is composite and  by Theorem~\ref{prfactor} it has
a prime factor and let its smallest one be $q_2$ i.e.
\[
a_1 = q_2 \cdot a_2 ,\quad  \mbox{\  with\ } a_2 < a_1.
\]
\[
n = q_1 a_1 = q_1 ( q_2 a_2 ) = q_1 q_2 a_2.
\]
If $a_2$ is composite we repeat this until he hit a $a_{k-1} = q_k$.
That way
\[
n =  q_1 q_2  \ldots q_{k-1} q_k.
\]
Let us assume that there is another factorization of $n$
\[
n =  r_1 r_2  \ldots r_{t-1} r_t.
\]
The we have
\[
  q_1 q_2  \ldots q_{k-1} q_k  =  r_1 r_2  \ldots r_{t-1} r_t.
\]
$r_1$ divides the right hand side. It also divides the left-hand side.
By Lemma~\ref{pfactor} one of $q_i$ divided by $r_1$, i.e. $\dv{r_1}{q_i}$.
Because all of $r_1 , q_i$ are prime numbers this can only mean 
$r_1 =q_i$ for some $i$.  The smallest $r_i$ is $r_1$. 
The smallest $q_i$ is $q_1$. Thus $r_1 =q_1$.
Because primes are $\neq 0$, we can factor out $r_1$.
\[
  q_2  \ldots q_{k-1} q_k  =   r_2  \ldots r_{t-1} r_t.
\]
Continuing likewise if without loss of generality $k < t$ we 
will eventually have
\[
  1  =   r_{k+1}  \ldots r_{t-1} r_t.
\]
Prime numbers are $>1$. Their product cannot be equal to $1$. 
Thus this can only mean that $k=t$ as well.
\end{proof}

\subsection{Finding the gcd and lcm}

\bigskip %THEOREM 33 %%
\begin{thm}[GCD UF]
Let $a_i  , b_i  \in \mb{N}$ fo all $i$. 
Moreover $a,b \geq 1$.
If
\[
  a = p_1^{a_1} p_2^{a_2} \ldots p_k^{a_k},  \quad
  b = p_1^{b_1} p_2^{b_2} \ldots p_k^{b_k},  \quad a_k , b_k \geq 0.
\]
then in order to find $d=\gcd(a,b)$ we have
\[
d=\gcd(a,b) =  p_1^{c_1} p_2^{c_2}\ldots p_k^{c_k}  =
               p_1^{\min(a_1 ,b_1 )}p_2^{\min(a_2 ,b_2 )} \ldots 
                                    p_k^{\min(a_k ,b_k )},
\]
where $c_i = \min (a_i , b_i ) \geq 0$,
and then in
order to find $m=\lcm(a,b)$ we have
\[
m=\lcm(a,b) =  p_1^{d_1} p_2^{d_2}\ldots p_k^{d_k}  =
               p_1^{\max(a_1 ,b_1 )}p_2^{\max(a_2 ,b_2 )} \ldots
                                    p_k^{\max(a_k ,b_k )},
\]
where $d_i = \max (a_i , b_i ) \geq 0$,
\end{thm}

\begin{proof}
We shall show that $D =  p_1^{c_1} p_2^{c_2}\ldots p_k^{c_k}$ satisfy 
the GCD properties.  Because $c_i \leq a_i$ for all $i$, 
we conclude $\dv{D}{a}$.  Because $c_i \leq b_i$ for all $i$, 
we conclude $\dv{D}{b}$.
Let $g$ by any common divisor of $a$ and $b$.  If $g$ divides $a$,then
Then $g =  p_1^{g_1} p_2^{g_2}\ldots p_k^{g_k}$ with $g_i \leq a_i$.
If $g$ divides $b$,then also $g_i \leq b_i$, i.e. $g_i \leq\min(a_i , b_i)$.
This means that $\dv{g}{D}$. This is equivalent to also having $g\leq D$. 
Thus any common divisor $g$ of $a,b$ is $g\leq D$. This means
$D$ is  $d=\gcd(a,b)$.
$ $ \\ $ $
We work similarly for the lcm.
\end{proof}

\begin{cor}
Let $a_i \geq 0 , b_i \geq 0$. Moreover $a,b \geq 1$.
Let
\[
  a = p_1^{a_1} p_2^{a_2} \ldots p_k^{a_k},  \quad
  b = p_1^{b_1} p_2^{b_2} \ldots p_k^{b_k},  \quad a_i , b_i \geq 0,
 i= 1 , \ldots , k.
\]
Then 
\[
 \lcm(a,b) \cdot \gcd(a,b) = a \cdot b.
\]
\end{cor}

\begin{proof}
By the previous problem
we have
$c_i = \min (a_i , b_i ) \geq 0$, and
$d_i = \max (a_i , b_i ) \geq 0$, and
\[
 p_i^{c_i} p_i^{d_i} =
 p_i^{\min (a_i , b_i )}
 p_i^{\max (a_i , b_i )}
 = p_i^{a_i} p_i^{b_i}.
\]
Therefore
\[
\lcm(a,b) \gcd(a,b) = \prod_i p_i^{c_i} p_i^{d_i}
                    = \prod_i p_i^{a_i} p_i^{b_i}
                    = \prod_i p_i^{a_i} \prod_i  p_i^{b_i}
                    = a \cdot b.
\]
\end{proof}

\begin{exa}
Find the gcd and lcm of 256, and 8192
\end{exa}
\begin{solution}
We have $a= 256=2^8 =p_1^8$ and $b=8192=2^{13} =p_1^{13}$,
where $p_1 =2$ and $a_1 =8$ and $b_1 = 13$.
Then
$c_1 = \min (a_i , b_i ) = 8   $, and
$d_1 = \max (a_i , b_i ) = 13  $, and
$\gcd(a,b) =2^{c_1} =  2^8 = 256$
and
$\lcm(a,b) =2^{d_1} =  2^{13} = 8192$.
\end{solution}

\begin{exa}
Find the gcd and lcm of 60, and 630.
\end{exa}
\begin{solution}
We have
\[
a= 60 = 2^2 \cdot 3 \cdot 5 ,
\quad
b=630 = 2 \cdot 3^2 \cdot 5 \cdot 7
\]
Then
\[
\gcd(a,b) = \gcd(60,630) =
  2^{\min(1,2)} 3^{\min(1,2)} 5^{\min(1,1)} 7^{\min(0,1)} = 2^1 3^1 5^1 = 30.
\]
\[
\lcm(a,b) = \lcm(60,630) =
  2^{\max(1,2)} 3^{\max(1,2)} 5^{\max(1,1)} 7^{\max(1,1)} = 2^2 3^2 5^1 7^1 
= 1260.
\]
\end{solution}

\subsection{Prime number theorem}

Euclid's theorem reproved.

\bigskip %THEOREM 34 %%
\begin{thm}[Infinitely many primes]
\label{imp}
There are infinitely many prime numbers distinct from each other.
\end{thm}

\begin{proof}
Suppose that there are finitely many prime numbers i.e.
\[
p_1 < p_2 < \ldots < p_n
\]
that is $n$ distinct prime numbers exist. Then form
the product $N= p_1 p2 \ldots p_n +1$. Since $N>1$
by Theorem~\ref{prfactor} there is at least one 
prime $p$ dividing $N$. This $p$ cannot be any of the
$p_1 \ldots p_n$. Why ? Say $p=p_i$ for some $i$.
Then $\dv{p}{N}$ and $\dv{p}{p_1 \ldots p_n}$ which
would imply $\dv{p}{N-p1\ldots p_n}$. The latter
implies $\dv{p}{1}$ i.e. $p\leq 1$ but given $p$
is a prime number we must have $p>1$. A contradiction.
Thus $p$ is a prime number other than the ones
of  the finite  group $p_1 , \ldots , p_n$.
\end{proof}

%\bigskip
%\begin{dfn}
%Let $\pi (n)$ be the number of prime numbers less than or equal
%to $n$. Then $\pi(n) \rightarrow \infty $ as $n\rightarrow \infty$.
%\end{dfn}

\bigskip %THEOREM 35 %%
\begin{thm}[Prime number theorem]
\label{pnt}
Let $\pi (n)$ be the number of prime numbers less than or equal
to $n$. Then $\pi(n) \Rightarrow \infty $ as $n\Rightarrow \infty$.
It is $\pi(n) \approx n / \ln{n}$, or equivalently the following
holds.
\begin{equation}
\label{pn}
        \lim_{n \rightarrow \infty} \frac{\pi (n)}{ n / \ln{n}} =1.
\end{equation}
\end{thm}

\begin{proof}
Omitted.
\end{proof}

\begin{thm}
\label{pnt2}
Let $\pi (n)$ be the number of prime numbers less than or equal
to $n$.
Let $p_n$ be the $n$-th prime number.
The following applies.
\begin{equation}
\label{pofn}
        \lim_{n \rightarrow \infty} \frac{\pi (n)}{ n / \ln{n}} =1
        \Leftrightarrow
        \lim_{n \rightarrow \infty} \frac{p_n}{ n  \ln{n}} =1
\end{equation}
\end{thm}

\begin{proof}
$ $ \\ $ $
$\Rightarrow$.
We start with Eq.(\ref{pn}).
\begin{eqnarray}
\label{pn1}
 \lim_{n \rightarrow \infty} \frac{\pi (n)}{ n / \ln{n}} =
 \lim_{n \rightarrow \infty} \frac{\pi (n) \ln{n}}{ n } &=& 1 \nonumber \\
 \lim_{n \rightarrow \infty} \ln{\left( \frac{\pi (n) \ln{n}}{ n }\right)}
&=& \ln{1} \nonumber \\
 \lim_{n \rightarrow \infty}
   \left(  \ln{\left( \pi (n) \right)} +
           \ln{\left( \ln{n}  \right)} -
                      \ln{n} \right)
&=& 0 \nonumber \\
 \lim_{n \rightarrow \infty} \ln{n} \cdot
   \left(  \frac{\ln{\pi (n)}}{\ln{n}} +
           \frac{\ln{\left( \ln{n}  \right)}}{\ln{n}} -
                      1 \right)
&=& 0 \nonumber \\
 \lim_{n \rightarrow \infty} \ln{n} \cdot
   \left(  \frac{\ln{\pi (n)}}{\ln{n}} - 1 \right)
&=& 0 \nonumber \\
 \lim_{n \rightarrow \infty}
    \ln{\pi (n)}
&=&
 \lim_{n \rightarrow \infty}   \ln{n}
 \nonumber \\
 \lim_{n \rightarrow \infty}
    \frac{\ln{\pi (n)}}{\ln{n}}
&=& 1
\end{eqnarray}
We multiply Eq.(\ref{pn}) with Eq.(\ref{pn1}) to derive the following.
\begin{eqnarray}
\label{pn2}
 \lim_{n \rightarrow \infty}
    \frac{\ln{\pi (n)}}{\ln{n}} =1 &\wedge &
        \lim_{n \rightarrow \infty} \frac{\pi (n)}{ n / \ln{n}} =1
\Leftrightarrow \nonumber \\
 \lim_{n \rightarrow \infty}
 \frac{\pi (n) \cdot {\ln{\pi (n)}}}{n}
 &=& 1
\end{eqnarray}
In the latter Eq.(\ref{pn2}) we substitute
$n= p_n$ and $\pi (n) = \pi (p_n ) = n$.
We then obtain the following.
\begin{eqnarray}
\label{pn3}
 \lim_{n \rightarrow \infty}
 \frac{\pi (n) \cdot {\ln{\pi (n)}}}{n}
 &=& 1 \Leftrightarrow \nonumber \\
 \lim_{n \rightarrow \infty}
 \frac{n \cdot \ln{n}}{p_n}
 &=& 1 .
\end{eqnarray}
This concludes this case, after an inversion.
$ $ \\ $ $
$\Leftarrow$.
Shown similarly.
\end{proof}

\subsection{Mersenne numbers}

\begin{dfn}[Mersenne numbers]
The Mersenne number of order $k$ is denoted as $M_k = 2^k -1$,
where $k$ is a prime number.
\end{dfn}

\begin{thm}[Mersenne prime numbers]
If  $M_k = 2^k -1$ is a prime number, then
$M_k$ is a Mersenne prime (number).
\end{thm}

\begin{exa}
Find the first eight Mersenne numbers.
\end{exa}
\begin{solution}
The first eight Mersenne numbers are as follows.
\[
2^2 -1 =3, \quad
2^3 -1 =7, \quad
2^5 -1 =31, \quad
2^7 -1 =127, \quad
\]
\[
2^{11} -1 =2047, \quad
2^{13} -1 =8191, \quad
2^{17} -1 =131071, \quad
2^{19} -1 =524287, \quad
\]
\end{solution}

\begin{exa}
Find the first five Mersenne prime numbers.
\end{exa}
\begin{solution}
The first five  Mersenne primes are as follows.
\[
2^2 -1 =3, \quad
2^3 -1 =7, \quad
2^5 -1 =31, \quad
2^7 -1 =127, \quad
2^{13} -1 =8191, \quad
\]
Note that 2047 is missing from the second list.
This is because $2047 = 23 \cdot 89$.
One can verify 131071 and 524287 are both
Mersenne primes as well.
\end{solution}

\subsection{Fermat numbers}

\begin{dfn}[Fermat numbers]
The Fermat   number of order $k$ is denoted as $F_k = 2^{2^k}+1$,
$k \geq 0$.
\end{dfn}

\begin{thm}[Fermat prime numbers]
If  $F_k $ is a prime number, then
$F_k$ is a prime Fermat  number,
otherwise it is a composite Fermat number.
\end{thm}

\begin{exa}
Find the first five Fermat numbers.
\end{exa}
\begin{solution}
\[
F_0 = 3 , \quad F_1 = 5, \quad F_2 = 17 , \quad F_3 = 257,
\]
\[
F_4 = 65537 , \quad F_5 = 4294967297.
\]
All first give  are prime Fermat numbers (to mychecking).
But for the last one
\[
4294967297 = 641 * 6700417,
\]
we conclude $F_5$ is a composite Fermat number.
\end{solution}

\begin{exa}
Find the first five   prime Fermat numbers.
\end{exa}
\begin{solution}
Consider,
\[
(F_n -1)^2 +1 = (2^{2^n} +1 -1)^2 +1  = 2^{2 \cdot 2^n} +1 = F_{n+1} .
\]
Moreover
\[
F_{n+1} -2 = 2^{2^{n+1}} -1 = ( 2^{2^n} +1 ) ( 2^{2^n} -1 ) = F_n (F_n -2).
\]
By induction
\[
F_{n+1} -2 =F_n \cdot (F_n -2) = F_n \cdot \left( F_{n-1} (F_{n-1} -2) \right)
           =F_n \cdot  F_{n-1}  (F_{n-1} -2)
           = F_n F_{n-1} \ldots F_0 .
\]
Therefore
\[
\dv{F_k}{F_{n+1}-2},  \quad 0\leq k \leq n.
\]
\end{solution}

\newpage

\section{A short review of algebra definitions}

We recite some definitions.
This section can be skipped as needed.

\begin{dfn}[Equivalence relation]
A binary relation $R$ on $X$ is an equivalence relation on $X$
if for every $x,y,z \in X$ the following properties are
satisfied.
\begin{itemize}
\item (a) Reflexivity $xRx$ for every $x \in X$,
\item (b) symmetry $xRy$ then $yRx$, and
\item (c) transitivity $xRy$ and $yRz$ then $xRz$.
\end{itemize}
\end{dfn}

\begin{dfn}[Equivalence classes]
In an equivalence relation we can decompose $X$ into
disjoint subsets known as equivalence classes.
Then for each $x \in X$ the set of elements equivalent to $x$
define
\[
      [ x / R ]       = \{ y \in X : xRy \}.
\]
Then
\[
p : X \rightarrow X/R \ \mathrm{with} \  p(x) = [ x ].
\]
\end{dfn}

\noindent
For an equivalence relation $R$ if $aRb$ the $[a]=[b]$.
(It is the same as $[a/R]=[b/R]$.)

\begin{dfn}[Residue classes modulo $n$]
Let $n \in \mb{N}$.
The set of all residue classes modulo $n$ or congruences
mod $n$ (or modulo $n$) is
denoted in simple form by $\mb{Z}/n\mb{Z}$.
\[
 \mb{Z}/n\mb{Z} = \{ 0, 1 , 2, \ldots , n-1 \} .
\]
This is supposed to mean the following.
\[
 \mb{Z}/n\mb{Z} = \{ (0)_n , (1)_n , (2)_n , \ldots , (n-1)_n \} .
\]
\end{dfn}

We read this set `` $\mb{Z}$ modulo $n$ '' or `` $\mb{Z}$ mod $n$ ''.
Sometimes we write $\mb{Z}_n$ instead of $\mb{Z}/n\mb{Z}$.
The latter notation explains set $\mb{Z}$ is split by $n$ into
$n$ equivalence classes.  Two integer $x,y$ belong to the same
equivalence class if
%$ x \equiv y \pmod n$ or
$(x-y)$ or $(y-x)$ is a
multiple of $n$. The equivalence class of
$x$ is $x \bmod n$.

\subsection{Binary operators and operations}

\begin{dfn}[Binary operator]
A binary operator defines a binary operation.
\end{dfn}

A binary operation is the operation implied by a binary operator.

\begin{dfn}[Binary operation]
For a set $S$ and a binary operator $\oplus$, a
binary operation is a rule that maps    to each
pair $(x,y) \in S \times S$ a unique $z \in S$.
\end{dfn}

Operations $+$, $*$ are binary operations for a
set $S = \mb{N}$ or $S = \mb{R}$.
In all definitions below we assume $S$ is non-empty.

\subsection{Groupoids and semigroups}

\begin{dfn}[Magma/Groupoid]
Let $S$ be a set, and $\oplus$ be a
closed operator $S \times S \mapsto S$ defining
a total function  as follows.
\[
\forall (x,y) \in S \times S:  (x,y) \mapsto x \oplus y \in S .
\]
The structure $S$ equipped with $\oplus$ as structured
is a magma. We then say $(S, \oplus )$ defines  a magma,
also known as a groupoid.
\end{dfn}

Operator $\oplus$, and when introduced operator $\otimes$,
are closed operators $S \times S \mapsto S$ defining
a total function  as follows.
\[
\forall (x,y) \in S \times S:  (x,y) \mapsto x \oplus y \in S .
\]
This is also known as a closure property.

\begin{dfn}[Semigroup]
Let $S$ be a non-empty set, and $\oplus$ be an operator
$\oplus : S \times S \mapsto S$.
If operator $\oplus$ has the closure property and is associative,
then the structure $S$ equipped with $\oplus$ as structured
is a semigroup.
\begin{enumerate}
\item Closure property: $ \forall x \in S, \forall y \in S:
x \oplus y \in S$.
\item Associativity : $ (x \oplus y) \oplus z = x \oplus (y \oplus z)
\quad \forall x \in S, \forall y \in S, \forall z \in S$,
\end{enumerate}
We then say equivalently that $(S, \oplus )$ defines  a semigroup.
\end{dfn}

A monoid is a semigroup with an identity element.

\begin{dfn}[Monoid]
An $(S, \oplus )$ that is a semigroup it becomes a monoid if
there exists an element $e$ in $S$ such that
\[
\forall x \in S : e \oplus x = x \oplus e = x.
\]
Then  $(S, \oplus )$ is called a monoid.
\end{dfn}

A monoid equipped with an inverse (every element has
an inverse element) is a group.

\subsection{Groups}

\begin{dfn}[Group]
An $(S, \oplus )$ that is a monoid it becomes a group
if every element of $S$ has an inverse element.
\[
\forall x \in S , \exists y \in S : x \oplus y = y \oplus x = e,
\]
where $e$ is the identity element of the underlying monoid.
Then  $(S, \oplus )$ is called a group.
\end{dfn}

Below we give a standalone definition of a group.

\begin{dfn}[Group]
An $(S, \oplus )$ with $S$ non-empty is a group if  it has the
following properties.
\begin{enumerate}
\item Closure property: $ \forall x \in S, \forall y \in S:
x \oplus y \in S$.
\item Associativity : $ (x \oplus y) \oplus z = x \oplus (y \oplus z)
\quad \forall x \in S, \forall y \in S, \forall z \in S$,
\item Identity: there exists an $e \in S$ such that
$ \forall x \in S : e \oplus x = x \oplus e = x$.
\item Inverse:
$\forall x \in S , \exists y \in S : x \oplus y = y \oplus x = e$.
\end{enumerate}
Then  $(S, \oplus )$ is called a group.
\end{dfn}

\begin{exa}
For $S=\mb{Z}$ and $\oplus = +$ we define $e=0$.
The inverse $y$ of $x$ is then denoted as $-x$.
$(\mb{Z}, +)$ is a group. It is also an Abelian group,
defined next.
\end{exa}

\begin{dfn}[Commutative or Abelian group]
An $(S, \oplus )$ that is a group becomes a commutative
group if
\[
\forall x \in S,  \forall y \in S : x \oplus y = y \oplus x.
\]
Then  $(S, \oplus )$ is called a commutative group,
also known as an Abelian group.
\end{dfn}

$(\mb{R}, +)$ is also a commutative (abelian) group.

\begin{exa}
For $S=\mb{R}$ and $\oplus = +$ we define $e=0$
The inverse $y$ of $x$ is then denoted as $-x$.
$(\mb{Z}, +)$ is a group.
\end{exa}

\begin{exa}
For an $\oplus = +$ operator we use  0 as its identity
element that is, $e=0$. We denote the inverse of $a$
as $-a$.
For an $\oplus = *$ operator we use  1 as its identity
element that is, $e=1$.
We denote the inverse of an $a \neq 0$ as
$a^{-1}$ or $1/a$.
\end{exa}
A group with operator $\oplus = +$ is known as
an additive group.
A group with operator $\oplus = *$ is known as
a multiplicative group.

\begin{dfn}[Order of a group]
The order of a group $(S, \oplus)$ is the number of its elements,
and it is denoted $|S|$ (cardinality symbol).
\end{dfn}

\begin{dfn}[Subgroup]
For a group $(S, \oplus)$ and a non-empty subset $T$ of
$S$ that is $T \subseteq S$, we say $T$ a sugroup of $S$
if $(T, \oplus)$ is a group.
\end{dfn}

The order of a group is the cardinality of its underlying
set $S$.

$(\mb{R},  +    )$ is also an abelian group.
$(\mb{Z}_+,  +    )$ is not an abelian group. There is
no identity element then.
$(\mb{Z}_+,  *    )$ is not an abelian group. There is
an identity element then which is $1$, but there is no inverse
for most elements (other than one)!
$(\mb{Q}_+,  *    )$ is  an abelian group.
$(\mb{R}_+,  *    )$ is  an abelian group.
$(\mb{Q}^* , *    )$ is  an abelian group.
$(\mb{R}^* , *    )$ is  an abelian group.

\begin{exa}
For $S=\mb{R} - \{ 0 \} = \mb{R}^{x}$
and $\oplus = \cdot$ we define $e=1$
The inverse $y$ of $x$ is then denoted as $1/x$.
$(\mb{R}^{x}, \cdot )$ is a group.
\end{exa}

$(\mb{R}^{x}, \cdot )$ is also an abelian group.

\begin{dfn}[Cyclic group]
For a group $(S, \otimes)$  and $a \in S$,
the elements $a^i$, where $i$ is an integer, form a subgroup
of $S$, called the subgroup generated by $a$.
A group is cyclic if there is an element $a \in S$ that
the subgroup generated by $a$ is $S$ itself.
We can the write $S$ as follows.
\[
 S = \{ e , a , a^2 , \ldots , a^{k-1} \},
\]
where $e$ is the identity of the group,
and $k$ is the smallest (positive) integer such that
$a^k =e $.
\end{dfn}

If $S$ is infinite,
\[
 S = \{ \ldots a^{-2}, a^{-1} , e , a , a^2 , \ldots  \},
\]

\subsection{Rings and Integral domains}

\begin{dfn}[Ring]
A ring is an algebraic structure consisting of a
set $S$ in which two binary closed operations
are defined say $\oplus , \otimes$.
Then, $(S, \oplus , \otimes)$ is a ring
if (a) $(S, \oplus )$ is an abelian group,
(b) $(S , \otimes )$ is a semigroup,
and
(c) $ \otimes$ is distributive over $\oplus$,
with $e$ being the identity element of $(S, \oplus )$.
A ring denoted as  $(S, \oplus, \otimes )$
has the following properties.
\begin{enumerate}
\item Closure property for $\oplus$:
      $ \forall x \in S, \forall y \in S: x \oplus y \in S$.
\item Associativity for $\oplus$:
      $\forall x \in S, \forall y \in S, \forall z \in S,
       (x \oplus y) \oplus z = x \oplus (y \oplus z) $
\item  Commutativity for $\oplus$:
      $ \forall x \in S , \forall y \in S : x \oplus y = y \oplus x$.
\item Identity for $\oplus$: there exists an $e=0 \in S$ such that
      $ \forall x \in S : 0 \oplus x = x \oplus 0 = x$.
\item Additive inverse for $\oplus$:
      $\forall x \in S ,\exists y=-x \in S:x \oplus y=y \oplus x=e$.
\item Closure property for $\otimes$:
      $ \forall x \in S, \forall y \in S: x \otimes y \in S$.
\item Associativity for $\otimes$:
      $ \forall x \in S,  \forall y \in S, \forall z \in S,
      (x \otimes y) \otimes z = x \otimes (y \otimes z) $.
\item Distribution of $\otimes$ with respect to $\oplus$: \\
      $ \forall x \in S , \forall y \in S , \forall z \in S :
      x \otimes (y \oplus z ) = (x \otimes y) \oplus (x \otimes z)
      \wedge
      (x \oplus y ) \otimes z = (x \otimes z) \oplus (y \otimes z).$
%
%\item Identity for $\otimes$: there exists an $e=1 \in S$ such that
%      $ \forall x \in S : 1 \oplus x = x \oplus 1 = x$.
%
%
\end{enumerate}
Then  $(S, \oplus , \otimes )$ is called a ring.
\end{dfn}

All of $\mb{Z}, \mb{Q}, \mb{R}, \mb{C}$ are rings;
in fact they are commutative rings for $\otimes$ being commutative
as well.
Non-commutative rings are $n \times n$ matrices
whose elements are from the rings above.

\begin{dfn}[Commutative ring]
$(S, \oplus , \otimes )$ is called a commutative ring if
$(S, \oplus , \otimes )$ is a ring and
$\otimes$ is also commutative.
\[
\forall x \in S, \forall y \in S : x \otimes y = y \otimes x.
\]
\end{dfn}

\begin{dfn}[Ring with identity]
$(S, \oplus , \otimes )$ is called a ring with identity if
$(S, \oplus , \otimes )$ is a ring and
$\otimes$ supports an identity element let us call it $1$.
\[
\text{Identity for}\  \otimes: \exists e=1 \in S:
       \forall x \in S : 1 \otimes x = x \otimes 1 = x.
\]
\end{dfn}

\begin{dfn}[Integral domain]
Let $(S, \oplus , \times  )$ be a non-zero
commutative ring with $0= e \neq 1$, where
$e$ or $0$ is the identity of $\oplus$ and
$1$ is the identity of $\otimes$,
that has no zero divisors  which is equivalent
to the following.
\[
a \neq 0 , b \neq 0  \Rightarrow c=ab \neq 0
\ \text{or equivalently} \
 c=ab =0 \Rightarrow  a=0 \vee b=0.
\]
Then the structure
$(S, \oplus , \otimes )$ is called an
 integral domain.
\end{dfn}

\begin{dfn}[Division ring]
A division ring $(S, \oplus , \otimes )$  is a
ring with identity $1 \neq 0$
with an inverse element over $\otimes$, that is
\[
\forall x \in S, x \neq 0, \exists a \in S : ax=xa=1.
\]
\end{dfn}

\begin{dfn}[Ideal]
For a ring $(S, \oplus , \otimes)$,
a left ideal $I$ of $S$ is a subset of $S$ that
is a subgroup of the structure $(S, \oplus )$ of the ring
that is closed under left multiplication over
elements of $S$. Thefore $e \in I$, $e=0$ being
the identity of $(S, \oplus)$, and for every
$a,b \in I$ and $x \in S$ the following are true.
\[
a + b \in I, \quad
-a    \in I, \quad
x a   \in I .
\]
A right ideal is defined similarly.
A two-sided ideal is a left and right ideal;
oftentimes it is called an ideal.
If the ring is commutative, then the left, right, and
two-sided ideals coincide and the term ideal identifies
all three of them.  Then the ideal $I$ is an abelian subgroup.
\end{dfn}

\subsection{Fields}

A field can be defined as an extension of a ring as follows.

\begin{dfn}[Field from ring]
A field
$(S, \oplus , \otimes)$
is a division ring that supports a commutative $\otimes$.
\[
\forall x \in S, \forall y \in S : x \otimes y = y \otimes x.
\]
\end{dfn}

A field can also be defined as an extension of two groups  as follows.

\begin{dfn}[Field from groups]
For a field $(S, \oplus , \otimes )$,
$(S, \oplus )$ is an abelian group,
$(S, \otimes )$ is an abelian group but with an inverse
defined only for $S- \{ 0 \}$.
Moreover  $\otimes$ is distributive with respect to $\oplus$.
\end{dfn}

An equivalent standalone definition of a field follows.

\begin{dfn}[Field]
$(S, \oplus , \otimes)$ is a field with two binary
operators $\oplus , \otimes$ if $S$ has
at least two elements, there exist $1 \neq 0$
where $0$ is the identity of $\oplus$,
and   $1$ is the identity of $\otimes$
and it supports the following properties.
\begin{enumerate}
\item Closure property for $\oplus$:
      $ \forall x \in S, \forall y \in S: x \oplus y \in S$.
\item Associativity for $\oplus$:
      $\forall x \in S, \forall y \in S, \forall z \in S,
       (x \oplus y) \oplus z = x \oplus (y \oplus z) $
\item  Commutativity for $\oplus$:
      $ \forall x \in S , \forall y \in S : x \oplus y = y \oplus x$.
\item (Additive) Identity for $\oplus$: there exists an $e=0 \in S$ such that
      $ \forall x \in S : 0 \oplus x = x \oplus 0 = x$.
\item Additive inverse for $\oplus$:
      $\forall x \in S ,\exists y=-x \in S:x \oplus y=y \oplus x=e$.
\item Closure property for $\otimes$:
      $ \forall x \in S, \forall y \in S: x \otimes y \in S$.
\item Associativity for $\otimes$:
      $ \forall x \in S,  \forall y \in S, \forall z \in S,
      (x \otimes y) \otimes z = x \otimes (y \otimes z) $.
\item  Commutativity for $\otimes$:
      $ \forall x \in S , \forall y \in S : x \otimes y = y \otimes x$.
\item (Multiplicative) Identity for $\otimes$: there exists
      an $e=1 \in S$, $1 \neq 0$ such that
      $ \forall x \in S :  x \otimes 1 = 1 \otimes x = x$.
\item (Multiplicative) Inverse for $\otimes$:
      $\forall x \in S , \exists y=-x \in S:x \otimes y=y \otimes x=e$.
\item Distribution of $\otimes$ with respect to $\oplus$:
      $ \forall x \in S , \forall y \in S , \forall z in S $:
\[
      x \otimes (y \oplus z ) = (x \otimes y) \oplus (x \otimes z),
\]
\[
      (x \oplus y ) \otimes z = (x \otimes z) \oplus (y \otimes z).
\]
%
%\item Identity for $\otimes$: there exists an $e=1 \in S$ such that
%      $ \forall x \in S : 1 \oplus x = x \oplus 1 = x$.
%
%
\end{enumerate}
Then  $(S, \oplus , \otimes )$ is called a field.
\end{dfn}

Roughly speaking for a field
$(S, \oplus , \otimes )$,
$(S, \oplus )$ is an abelian group,
$(S, \otimes )$ is short of an abelian group, with
``short'' defined that only elements $x \neq e$ i.e.
$x \neq 0$ have a multiplicative inverse $y= 1/x = x^{-1}$,
that is,  $(S - \{ e \} , \otimes )$ is a group.

\begin{dfn}[Finite field]
A finite field is a field that has a finite
number of elements or equivalently $|S| < \infty$.
The order of a field is the number of its elements.
\end{dfn}

\begin{thm}[Galois]
Finite fields exists for an order $q$ that is a prime
power that is $q=p^k$, where $k\geq 1$.
\end{thm}

A finite field is also called a Galois field and denoted
$GF(q)$.

\newpage

\section{Modular arithmetic}

\begin{nte}
The terms congruence and residue, 
congruence class and residue class will be 
used interchangeably in the remainder.
\end{nte}

\begin{dfn}[Congruence modulo $n$]
Let $n \in \mb{Z}_+^*$. 
For $a,b \in \mb{Z}$ we say
$a \equiv b \pmod n$ if $n / (a-b)$,
or equivalently, if the remainder of the
division of $a$ by $n$ is equal 
to the remainder of the division of $b$ by $n$.
\[
\forall a,b \in \mb{Z} \quad   a \equiv b \pmod n
\quad \Leftrightarrow
\quad
   \dv{n}{a-b}.
\]
\end{dfn}

Congruence modulo $n$ and congruence mod $n$ mean
the same thing.
We can read $a \equiv b \pmod n $  by saying that
''a is congruent to b modulo n''.
Then the ''difference of a and b is a multiple of n''
or ''n divides the difference of a and b''.
Note that several times the $\equiv$ is replaced by $=$
and the parentheses around the mod are dropped.

\bigskip %THEOREM 37 %%
\begin{prp}[Modular division]
\label{moddiv}
Let $n \in \mb{N}$ and $n>1$ and let $a\in \mb{Z}$.
Then there exists a unique integer $0\leq r < n$ such that 
\[
a \equiv r \pmod n.
\]
\end{prp}

\begin{proof}
If $n > 1   $ and $a \in \mb{Z}$ division (Theorem~\ref{divis}) implies
that there are unique $q,r$ such that 
$a = n q+ r$, with $0 \leq r < |n|$. If in addition $n$ is positive
(e.g. $n > 1$), then $0 \leq r < n$.
Then $a-r = n q$ i.e. $a \equiv r \pmod n$.
\end{proof}

\begin{dfn}[Least residue]
If $a \equiv b \pmod{n}$ we refer to $b$ as the
residue of $a$ modulo $n$. For a $b$ such that
$0 \leq b < n$, we refer to $b$ as the least non-negative
residue of $a$ modulo $n$.
\end{dfn}

\begin{nte}
Oftentimes the least non-negative residue of $a$ modulo
$n$ is denoted as $a \bmod n$.
Therefore
\[
 a \equiv b \pmod n  \Leftrightarrow  a \bmod n = b \bmod n.
\]
\end{nte}

\begin{lem}
\label{eqrlem}
Let $n \in \mb{Z}_+$, and $a,b \in \mb{Z}$.
The following conditions are equivalent.
\begin{enumerate}
\item $a \equiv b \pmod n$,
\item $a=b+nk$, for some $k \in \mb{Z}$, and
\item $\dv{n}{(a-b)}$.
\end{enumerate}
\end{lem}
\begin{proof}
$ $ \\ $ $
1. Since $a \equiv b \pmod n$, there exist $a_1 , b_1 ,m \in \mb{Z}$
such as: $a=a_1 n + m$, $b= b_1 n + m$, where $0 \leq m < n$.
Then, $a-b= (a_1 -a_2 )n$ i.e. $a=b + (a_1 - a_2 )n $ and
therefore $a=b+nk$, where $k=(a_1 - a_2 )$, thus proving 2.
Furthermore, $a-b= (a_1 -a_2 )n$ is by definition equivalent to
$\dv{n}{(a-b)}$ thus proving 3.
$ $ \\ $ $
2. Let $a=b+nk$ for some $k$. Let $b= n q +r$ for some $0\leq r < n$.
The $a= nk + nq + r = n(k+q) + r$. Therefore $a \equiv b \pmod n$
thus proving 1. Moreover $a=b+nk$ implies $a-b=nk$ which is
equivalent to $\dv{n}{(a-b)}$ thus proving 3.
$ $ \\ $ $
3. Let $\dv{n}{(a-b)}$. This mean $a-b = n k$ for some integer $k$.
This is 2, and from 2 we might derive 1 as well.
\end{proof}

%Note that several times the $\equiv$ is replaced by $=$
%and the parentheses around the mod are dropped.
%This defines the modulo operation, the integer remainder
%of an integer division.

Note that the $|$ (divisible/divisibility relation) is reflexive
and transitive but not symmetric, and thus it is not an equivalence
relation.
In fact if $\dv{a}{b}$ and $\dv{b}{a}$ we have $|a|=|b|$ or for
positive integers $a=b$.

\bigskip %THEOREM 36 %%
\begin{cor}[Properties of modular arithmetic]
\label{pmod}
Let $n \in \mb{N}$ with  $n>1$.
The $\bmod n$ operation has the following properties.
\begin{itemize}
\item[1.] {\bf Reflexive}  $a\equiv a \pmod n$.
\item[2.] {\bf Symmetric}  $a\equiv b \pmod n \iff b \equiv a \pmod n$.
\item[3.] {\bf Transitive} If $a\equiv b \pmod n$ and $b\equiv c \pmod n$,
then  $a\equiv c \pmod n$. 
\item[4.] {\bf Translation} 
              If  $a   \equiv b   \pmod n$ then for any integer $c$,
then
                           $a+c \equiv b+c \pmod n$.
\item[5.] {\bf Scaling} 
              If  $a   \equiv b   \pmod n$ then for any integer $c$,
then
                           $a\cdot c \equiv b\cdot c \pmod n$.
\item[6.] {\bf Additivity}
               If  $a_1 \equiv b_1 \pmod n$ and $a_2 \equiv b_2 \pmod n$,
then
                   $a_1 + a_2 \equiv b_1 + b_2  \pmod n$.
\item[7.] {\bf Subtractivity}
               If  $a_1 \equiv b_1 \pmod n$ and $a_2 \equiv b_2 \pmod n$,
then
                   $a_1 - a_2 \equiv b_1 - b_2  \pmod n$.
\item[8.] {\bf Multiplicativity}
               If  $a_1 \equiv b_1 \pmod n$ and $a_2 \equiv b_2 \pmod n$,
then
                   $a_1 \cdot a_2 \equiv b_1 \cdot b_2  \pmod n$.
\item[9.] {\bf Exponentiativity}
              If  $a   \equiv b   \pmod n$ then for any integer $c>0$,
then
                           $a^c \equiv b^c \pmod n$.
\end{itemize}
\end{cor}

\begin{proof}
$ $ \\ $ $
Properties (1)-(3) follow from the proof of
Lemma~\ref{eqrlem}.
$ $ \\ $ $
(1)  For $\dv{n}{a-a}$ i.e. $\dv{n}{0}$ obviously.
$ $ \\ $ $
(2)  If $\dv{n}{a-b}$ then $\dv{n}{b-a}$ obviously.
$ $ \\ $ $
(3)
If $\dv{n}{a-b}$ and $\dv{n}{b-c}$ we have respectively
$a-b = k n$ and $b-c = m n$ for some $k,m \in \mb{Z}$.
Adding the two together
we get $a-c= (k+m)n$ i.e. $\dv{n}{a-c}$
i.e. $a \equiv c \pmod n$.
$ $ \\ $ $
(4) The translation property follows similarly.
If  $a   \equiv b   \pmod n$ then $a=kn+b$.
Then $(a+c) = kn + (b+c)$ or
     $(a-c) = kn + (b-c)$  and the result follows.
$ $ \\ $ $
(5) The scaling property follows similarly.
If  $a   \equiv b   \pmod n$ then $a=kn+b$.
Then $(a \cdot c) = (kn + b) \cdot c$ or
     $ac = (kc)n + b c $ and the result follows.
$ $ \\ $ $
(6)-(7) They follow similarly to (4)-(5).
If  $a_1   \equiv b_1   \pmod n$ then $a_1 =kn+b_1$.
If  $a_2   \equiv b_2   \pmod n$ then $a_2 =ln+b_2$.
Then, $a_1 + a_2 \equiv b_1 + b_2 \pmod n $,
or
$a_1 - a_2 \equiv b_1 - b_2 \pmod n $,
$ $ \\ $ $
(8) Moreover,
$a_1 a_2 = (k b_2 + l b_1 + kln)n + b_1 b_2$ and
thus  $a_1 \cdot a_2 \equiv b_1 \cdot b_2 \pmod n $.
$ $ \\ $ $
(9) For $a   \equiv b   \pmod n$ implying $a=kn+b$,
we use the binomial theorem to show that
\[
a^c = (kn+b)^c  \equiv b^c \pmod n ,
\]
as needed.
Or it can follow by induction and the scaling property (4).
The result then follows.
\end{proof}

A note on division.  Operation $\bmod n$ was used in the context
of modular addition, subtraction and multiplication. We have been
silent about division. This is because of the following.
$4 \not\equiv 2 \pmod 6$ yet
$3\cdot 4 \equiv 3 \cdot 2 \pmod 6 $.
Also $2 \cdot 3 \equiv 0         \pmod 6$
yet $2     \not\equiv       0 \pmod 6$
and $3     \not\equiv       0 \pmod 6$.

\begin{exa}      
Find $3^{49} \pmod {19}$.
\end{exa}      

\begin{proof}
Repeated squares can help avoiding doing 49 multiplications.
The binary representation of 49 is \\
$49=(110001)_2$. Or in other words $49 = 2^{5} + 2^{4} + 2^{0}$.
Then $3^{49} = 3^{32} \times 3^{16} \times 3^{1} =
               3^{2^5} \times 3^{2^4} \times 3^{2^0}$.
\begin{eqnarray*}
    3^1    &\equiv& 3          \pmod {19} \\
    3^2    &\equiv& 9          \pmod {19} \\
    3^4    &\equiv&81 \equiv 5 \pmod {19} \\
    3^8    &\equiv&25 \equiv 6 \pmod {19} \\
    3^{16} &\equiv&36 \equiv 17\pmod {19} \\
    3^{32} &\equiv&289\equiv 4 \pmod {19} \\
\end{eqnarray*}
Note that $3^{16} \equiv 36 \equiv 17 \equiv -2 \pmod {19}$.
Then $3^{32} \equiv 4 \pmod {19}$ does not need to deal with 
a $17^2 = 289$!

We then combine the powers of 2 in the exponent of three
as dictacted by the binary representation of 49.
That is  $3^{49} \equiv  3^1   \pmod {19} \cdot 3^{16} \pmod {19} \cdot
                         3^{32}\pmod {19} 
                 \equiv  3 \cdot 17 \cdot 4 
                 \equiv          13 \cdot 4 
                 \equiv 14 \pmod {19}$

However, a  nice trick might have worked better if a 1 or -1 was
encountered earlier.
\begin{eqnarray*}
    3^1    &\equiv& 3          \pmod {19} \\
    3^2    &\equiv& 9          \pmod {19} \\
    3^3    &\equiv&27 \equiv 8 \pmod {19} \\
    3^4    &\equiv&24 \equiv 5 \pmod {19} \\
    3^5    &\equiv&15          \pmod {19} \\
    3^6    &\equiv&45 \equiv 7 \pmod {19} \\
    3^7    &\equiv&21 \equiv 2 \pmod {19} \\
    3^8    &\equiv& 6 \equiv 6 \pmod {19} \\
    3^9    &\equiv&18 \equiv -1\pmod {19} \\
\end{eqnarray*}
The $3^{49} = 3^{45} \cdot 3^{4} = ( 3^{9} )^5 \cdot 3^{4} \equiv
   (-1)(-1)(-1)(-1)(-1) 5 \equiv -5 \equiv 14 \pmod{19}$
\end{proof}

\bigskip %THEOREM 38 %%
\begin{thm}[Cancellation law]
If $ac \equiv bc \pmod n$,
and $\gcd(n,c)=1$ then $a\equiv b \pmod n$.
\end{thm}

\begin{proof}
It is $ac+qn = bc +r n$ i.e. $ac-bc = sn$ for $s=r-q$ for some integer
$r,q$. Thus $\dv{n}{ac-bc}$ or $\dv{n}{(a-b)c}$. Since $\gcd(n,c)=1$,
we have that $\dv{n}{a-b}$. This implies $ a \equiv b \pmod n$.
%and also $a \equiv b \pmod n$ as  just used/introduced.
\end{proof}

\subsection{Modular linear equations}

Consider the ring $( \mb{Z}/n\mb{Z} , + , \cdot )$.

\bigskip %THEOREM 39 %%
\begin{thm}[Linear congruences]
\label{tlincon}
The modular equation (linear congruence)
\begin{equation}
 a x \equiv b \pmod n.
\end{equation}
has a solution if and only if $\dv{\gcd(a,n)}{b}$.
\end{thm}

If $\ndv{\gcd(a,n)}{b}$ the linear
congruence has NO solutions.
The linear congruence if it has one solution, then
it has an infinite number of solutions in $\mb{Z}$.

\begin{proof}
We can rewrite the congruence as a Diophantine equation
as follows.
\begin{equation}
\label{tlincone}
\exists k : \  ax -b = k n \Rightarrow   a x + n (-k)  = b.
\end{equation}
Then, by way of Theorem~\ref{diot},
%, Eq.~(\ref{tlincon}) or its equivalent 
Eq.(\ref{tlincone})
has a solution if and only if $\dv{d}{b}$, where $d=\gcd(a,n)$.
Moreover if one solution is $(x_0 , k_0 )$ there are
more solutions of the form
\[
  x = x_0 + \frac{mn}{d}  , \quad\quad k = k_0 - \frac{ma}{d},
\]
for $m \in \mb{Z}$.
%From $ ax \equiv b \pmod n$ we have that $ax-b = kn$ i.e.
%$ax+kn = b$. 
%The solution of this Diophantine equation
%exists for $\dv{\gcd(a,n)}{b}$.
\end{proof}

Since $d=\gcd(a,n)$ then $n/d$ and $a/d$ are integers above.
If $d=1=\gcd(a,n)$ then there is only one solution mod $n$;
all other solutions are equivalent to this mod $n$.

%HERE2
A self-contained proof can be stated as follows.
Consider the ring $( \mb{Z}/n\mb{Z} , + , \cdot )$.
\begin{thm}[Linear congruence redo]
\label{modeq1}
The modular equation
\[
 ax \equiv b \pmod n,
\]
has a solution if and only if $d=\gcd(a,n)$ is such that
$\dv{d}{b}$.
\end{thm}
\begin{proof}
$ $ \\ $ $
$\Rightarrow$.
$ $ \\ $ $
Consider $d=\gcd(a,n)$ thus implying $\dv{d}{a}$
and $\dv{d}{n}$. If $ ax \equiv b \pmod n$ has a solution,
then there exists a $k \in \mb{Z}$ such that
$ax-b = k n$. By way of $\dv{d}{a}$ , $\dv{d}{n}$,
we have $\dv{d}{ax-kn}$ and thus $\dv{d}{b}$.
$ $ \\ $ $
$\Leftarrow$.
$ $ \\ $ $
For $d=\gcd(a,n)$ it implies $\gcd(a/d,n/d)=1$.
Then $a/d$ is a unit mod $n/d$. Therefore there exists a $y$
such that
\[
  (a/d) y \equiv 1 \pmod{(n/d)}
\Rightarrow
   (a/d)y -1 = k (n/d),
\]
for some integer $k \in \mb{Z}$.
If we multiply this equation by $b$ we derive the following
one.
\[
   (a/d)y -1 = k (n/d)
\Rightarrow
   a (b/d)y - b = k (b/d) n
\Rightarrow
   a (b/d)y  = b+   k (b/d) n .
\]
Since $\dv{d}{b}$ we have that $b/d$ is an integer and thus we
further derive the following.
\[
   a (b/d)y  = b+   k (b/d) n
\Rightarrow
   a x       = b+   K  n ,
\]
where $x= by/d$ and $K=k(b/d)$. The solution to the
modular equation is $x=by/d$, where $y$ is
the solution of the modular equation
$ (a/d) y \equiv 1 \pmod{(n/d)}$.
\end{proof}

\bigskip %THEOREM 39 %%
\begin{thm}[Linear congruences with $d=1$]
\label{tlincon1}
Let $1=d=\gcd(a,n)$. If $\dv{d}{b}$ the linear congruence
below has one solution mod $n$.
\begin{equation}
 a x \equiv b \pmod n.
\end{equation}
\end{thm}
\begin{proof}
$ $ \\ $ $
We can rewrite the congruence as a Diophantine equation
as follows.
\[
\exists k : \  ax -b = k n \Rightarrow   a x + n (-k)  = b.
\]
If $d=1=\gcd(a,n)$
by way of Theorem~\ref{diot}, we have that
\[
 a x + n k = 1
\]
has a solution $( x^\prime , k^\prime )$ since $\gcd(a,n)=1$ divides
the 1 of the right hand side.
\[
 a x^\prime + n k^\prime = 1.
\]
Then multiplying by $b$ both sides of it we obtain the
following.
\[
 a (bx^\prime ) + n (bk^\prime ) = b
\]
Therefore the linear congruence
\[
ax  \equiv b \pmod n \Leftrightarrow
 a x + n (-k)  = b  ,
\]
has a solution $x_0 = bx^\prime$, $k_0 = -bk^\prime$.
\[
 a x_0 + n (-k_0 )  = b  .
\]
Other solutions of the linear congruence are
\[
  x = x_0 + \frac{mn}{d}  , \quad\quad k = k_0 - \frac{ma}{d},
\]
for $m \in \mb{Z}$. Since $d=1$ all these are as follows.
\[
  x = x_0 + mn , \quad\quad k = k_0 - ma,
\]
and thus all other $x$ solutions of the
congruence are $x \equiv x_0 \mod n$.
Congruence-wise there is only one solution $(x_0 )_n$ i.e.
$x_0$.
\end{proof}

\begin{thm}[Linear congruences with $d>1$]
Let $d=\gcd(a,n) >1$. If $\dv{d}{b}$ the linear congruence
below has exactly $d$ solutions mod $n$.
\label{tlincond}
\begin{equation}
 a x \equiv b \pmod n.
\end{equation}
\end{thm}
\begin{proof}
$ $ \\ $ $
We can rewrite the congruence as a Diophantine equation
as follows.
\[
\exists k : \  ax -b = k n \Rightarrow   a x + n (-k)  = b.
\]
If $d=\gcd(a,n) >1$
by way of Theorem~\ref{diot}, we have that
\[
 a x + n k = b
\]
has a solution $( x^\prime , k^\prime )$ since $\dv{d}{b}$.
Therefore we have the following.
\[
 a x^\prime + n  k^\prime  = b.
\]
Therefore the linear congruence
\[
ax \equiv b \pmod n \Leftrightarrow  ax + n (-k) = b
\]
has a solution $x_0 = x^\prime $ and $k_0 = - k^\prime$.
\[
 ax_0 + n (- k_0 ) = b.
\]
Other solutions of the linear congruence are, again from
Theorem~\ref{diot},
\[
  x = x_0      + \frac{mn}{d}  , \quad\quad k = k_0      - \frac{ma}{d},
\]
for $m =0 , \ldots , d-1$.
This is because
\[
a (x_0      + \frac{mn}{d} )
=
a x_0     + \frac{amn}{d}
=
a x_0     + nm\frac{a}{d}
=
a x_0      + K n
\equiv  b \pmod n,
\]
since $\frac{a}{d}$ is an integer by way of $d=\gcd(a,n)$ which
implies $\dv{d}{a}$ and thus $ K= m a /d$ is also an integer.
Note that
\[
 a x_0      + n (- k^\prime ) = b \Rightarrow
 a x_0      \equiv b \pmod n.
\]
$ $ \\  $ $
Consider the solutions   $x_i = x_0      + (in)/d$ for
$i=0, \ldots d-1$.
Any two solutions $x_i , x_j$ with $0\leq i < j < d$ are
such that $x_i \not\equiv x_j \pmod n$.
To prove this, say
$x_i \equiv x_j \pmod n$.
Then
\[
x_i \equiv x_j \pmod n \Rightarrow
 (i-j)n /d \equiv 0 \pmod n \Rightarrow
 (i-j)n/d = k n \Rightarrow
 (i-j) = k d,
\]
for some integer $k$. Since $\dv{d}{kd}$
we have $\dv{d}{i-j}$ and thus $d \leq |i-j|$
which is impossible
since $|i-j| < d$.
Therefore it must be that
$x_i \not\equiv x_j \pmod n$ for $i\neq j < d$.
$ $ \\ $ $
Consider now a solution for general $t \geq d$.
\[
 x_t =  x_0      + (tn)/d .
\]
Let $t= Qd +R $, where $0 \leq R < d$.
We obtain the following
\begin{eqnarray*}
x_t &=&  x_0      + (tn)/d  \\
    &=&  x_0      + (Qd+R)n/d \\
    &=&  x_0      + Qn+ (Rn)/d \\
    &=&  x_0      + Qn+ (Rn)/d \\
    &=&  x_0      + (Rn)/d \pmod n \\
\end{eqnarray*}
implying that any  other solution $x_t$ for general
$t$ maps modulo $n$ to a solution
$x^\prime + (in)/d$ where $i=R$ is in the range
$0 \leq i=R < d$.
This concludes the proof that the number of
solutions modulo $n$ is indeed $d = \gcd(a,n)$.
\end{proof}

%HERE

The proof technique in the previous theorem involving
the relationship between $x_i$ and $x_j$ is to be
used in the proof of Fermat's little theorem.
We will show that for prime $p$, and an integer
$a$ such that $\gcd(a,p)=1$ we have $a^{p-1} \equiv 1 \pmod p$.
In the proof we form $ia$ and $ja$ to 
show that $ia \not\equiv ja \pmod p$,
for $i \neq j$. Moreover $ia$ for all $i=1, \ldots , p-1$
form a permutation of $1,2, \ldots , p-1$ and
thus 
$\prod_{i=1}^{p-1} ia = a^{p-1} (p-1)! \equiv (p-1)! \pmod p$.
Given that $p$ is prime and $\gcd(i,p)=1$, for all $i=1, \ldots , p-1$,
we conclude $a^{p-1} \equiv 1 \pmod p$.
Moreover $a^p \equiv a \pmod p$.

%We can then generalize that for all $a$ we have $a^p \equiv a \pmod p$.
%For $a$ such that $\gcd(a,p)=1$ the result involving Fermat's
%little theorem is used.
%If $\gcd(a,p) \neq 1$, given that $p$ is prime implies
%$\dv{p}{a}$. Then $a^p -a$ is still a multiple of $p$
%and therefore $a^p \equiv a \pmod p$.

\bigskip %THEOREM 40 %%
\begin{thm}
\label{modn}
The modular equation
\[
ax \equiv 1 \pmod n
\]
has a solution if and only if $\gcd(a,n)=1$.
\end{thm}

\begin{proof}
The $\dv{\gcd(a,n)}{b}$ of the previous theorem for $b=1$
becomes $\dv{\gcd(a,n)}{1}$. Thus $\gcd(a,n) \leq 1$.
The gcd is always a positive integer i.e. $\gcd(a,n) \geq 1$.
Thus $\gcd(a,n)=1$ as needed. 
\end{proof}

\bigskip %THEOREM 41 %%
\begin{thm}
The modular equation, for prime $p$,
\[
ax \equiv 1 \pmod p
\]
has a solution for $x$ if $\ndv{p}{a}$.
\end{thm}

\begin{proof}
By the previous theorem for a solution to exist
$\gcd(a,p)=1$. Since $p$ is a prime its only 
positive divisors are $1$ and $p$. Given that
$p$ cannot divide $a$, we have that the
$\gcd(a,p)=1$. The result follows.
\end{proof}

Equivalently, it can be stated as follows.

\bigskip %THEOREM 42 %%
\begin{thm}
If $p$ is a prime number,
the modular equation, for prime $p$ and $\ndv{p}{a}$,
then there exists an $x$ such that $1\leq x \leq p-1$ 
such that the modular equation,
\[
ax \equiv 1 \pmod p
\]
has a solution for $x$.
The $x$ is sometimes denoted as the inverse of $a$ modulo
$p$ i.e.  $a^{-1}$. 
\end{thm}

The $a$ is called a {    unit} modulo $p$,
as it has an inverse. An $a$ that is not a unit is called
a {    zero divisor} modulo $p$.

\medskip
\begin{dfn}
The $a$ such that $a \cdot a^{-1} \equiv 1 \pmod p$  
is called a {\bf unit} modulo $n$, as it has an inverse. 
\end{dfn}

\smallskip
\begin{dfn}
The $a$ such that  there does not exist an $a^{-1}$
such that $a \cdot a^{-1} \equiv 1 \pmod p$  
is called a {\bf zero divisor} modulo $p$.
\end{dfn}

\bigskip %THEOREM 43 %%
\begin{thm}
Let $n>1$ be an integer and $\ndv{n}{a}$. The following are
equivalent.
\begin{itemize}
\item[(a)] $a$ is a zero divisor $\pmod n$,
\item[(b)] $a$ has no inverse $\pmod n$,
\item[(c)] there exists a $x\in Z$ such that $\ndv{n}{x}$ and
$ax \equiv  0 \pmod n$.
\end{itemize}
\end{thm}

\begin{proof}
Statements (a) and (b) are true and equivalent by the prior
definition and introduction of unit and zero divisor.

Suppose that $(b)$ is true and $a$ has no inverse.
By Theorem~\ref{modn} $\gcd(a,n) >1$. Let $\gcd(a,n)=d>1$.
The $a=dr$ and $n=ds$ for some integer $r,s$.
For $1< < z < n$ we have $z \not\equiv \pmod n$. Furthermore,
$a s =  (dr)s = (ds)r = n r \equiv 0 \pmod n$. Statement (c)
follows from Statement (b).

Suppose that statement (c) is true. There there exists
an $s\in Z$ such that $\ndv{n}{a}$ and $as \equiv 0 \pmod n$.
We are going to prove $a$ has no inverse. Let us assume that
$a$ has an inverse, then $ a a^{-1} \equiv 1 \pmod n$, and then
\[
 0 \equiv as \equiv as a^{-1} \equiv (aa^{-1})s \equiv s \pmod n.
\]
The latter implies that $\dv{n}{s}$ that contradicts the
assumption that $\ndv{n}{s}$! Thus $a$ has no inverse
and statement (b) is true coming from (c). Thus statements
(b) and (c) are equivalent.
\end{proof}

\bigskip %THEOREM 44 %%
\begin{thm}
For prime $p$, $ab \equiv 0 \pmod p$ implies  either
$a\equiv 0 \pmod p$
or
$b\equiv 0 \pmod p$.
Thus for a prime $p$ there no zero divisors other than 0
$\pmod p$.
\end{thm}

From Theorem~\ref{modeq1} we conclude the following.

\begin{cor}
The modular equation
\begin{equation}
\label{modn1}
ax \equiv 1 \pmod n
\end{equation}
has a solution if and only if $\gcd(a,n)=1$.
\end{cor}

\begin{cor}
The modular equation
\begin{equation}
\label{modn2}
ax \equiv 1 \pmod p,
\end{equation}
for prime $p$
has a solution for $x$ if and only if $\ndv{p}{a}$.
\end{cor}

\section{Residue classes or congruence classes}

Note that the $|$ (divisible/divisibility relation) is reflexive
and transitive but not symmetric, and thus it is not an equivalence
relation.
In fact if $\dv{a}{b}$ and $\dv{b}{a}$ we have $|a|=|b|$ or for
positive integers $a=b$.

\begin{lem}[Congruence is an equivalence relation]
The congruence $\bmod n$ is an equivalence relation:  it is
reflexive, symmetric and transitive on $\mb{Z}$.
\end{lem}
\begin{proof}
$ $ \\ $ $
1. It is straightforward to show $\bmod{n}$ is reflexive.
For every $a$ we have $a \equiv a \bmod n$ since $a-a=0 \cdot n$.
$ $ \\ $ $
2. It is also symmetric since if $a \equiv b \pmod n$ this
implies $a-b = n k$ for some integer $k$. Then $b-a= n (-k)$
and therefore $b \equiv a \pmod n$.
$ $ \\ $ $
3. Furthermore, $a \equiv b \pmod n$ and $b \equiv c \pmod n$
imply $ a-b = k n $ and $b -c = l n$ for some integers $k,l$.
Therefore adding up the equalities we obtain
$a-c = (k+l) n$ anf therefore $a \equiv c \pmod n$ thus
proving transitivity.
\end{proof}

\begin{dfn}[Congruence or residue class]
The residue class or congruence class of $a$ mod $n$
or $a$ modulo $n$ is denoted as $(a)_n$.
%If $a \equiv b \pmod{n}$,  and $b$ is  the (least)
%residue of $a$ modulo $n$.
\end{dfn}

\begin{dfn}[Congruence or residue classes]
Let $n \in \mb{Z}_+$.
We define classes $(a)_n$ (or $[a]_n$), for all
$a=0, 1, \ldots , n-1$, $a \in \mb{Z}$.
\[
 (a)_n = \{ b \in \mb{Z} | b \equiv a \pmod n \}
       = \{ b= a + k n | k \in \mb{Z} \}.
\]
$(0)_n , (1)_n , \ldots , (n-1)_n$ are the $n$ equivalence
class modulo $n$. Every integer $b$ belongs to one of those
classes depending on its remainder after a division with $n$.
Thus $b \in (a)_n$ if and only if $b \equiv a \pmod n$.
\end{dfn}

One might use $(a)_n$ or $[a]_n$, or $a+n\mb{Z}$,
or $a \bmod n$ to denote the equivalence class
$a$ modulo (or mod) $n$.
%The $a \equiv b \pmod n$ is equivalent to $\dv{n}{a-c}$ or
%the equivalence class $a \bmod n$ is equal to
%the equivalence class $c \bmod n$.
Arithmetic on equivalence classes is known as modular arithmetic.

\begin{exa}
Integers $\bmod $ 3 can fall into three equivalence
classes. Describe them.
\end{exa}
\begin{solution}
\[
(0)_3 = \{  \ldots,  -3, 0 , +3 , +6 , \ldots  \},
\]
then
\[
(1)_3 = \{  \ldots,  -2, 1 , +4 , +7 , \ldots  \},
\]
and finally,
\[
(2)_3 = \{  \ldots,  -1, 2 , +5 , +8 , \ldots  \}.
\]
Integers $5$ and $8$ belong to $(2)_3$ because the remainder of
the integer division of 5 by 3 is a 2. And so is the
remainder of the division of 8 by 3.
\end{solution}

\begin{dfn}[Equivalence classes $\bmod n$]
Let $\mb{Z}_n$ or $\mb{Z}/n\mb{Z}$
be the set of equivalence classes of integers modulo $n$.
That is,
\[
 \mb{Z}/n\mb{Z} =
 \mb{Z}_n = \{ (0)_n , (1)_n , \ldots , (n-1)_n \}.
\]
If $a\in (k)_n$ then $a= nq +k$ i.e. $a \equiv k \pmod n$.
If $b\in (m)_n$ then $b= nr +m$ i.e. $b \equiv m \pmod n$.
Then $(a+b) \in (k+m)_n$. Naturally $(k+m)_n$ is $((k+m)\pmod n)_n$.
Moreover $(ab) \in (km)_n$.
We can then define operations on the elements of $Z_n$ as follows.
\begin{itemize}
\item[1.] {\bf (Addition)} $(k)_n +(m)_n = (k+m)_n$, or
$(k+n \mb{Z})+(m+n\mb{Z}) = (k+m) + n \mb{Z}$.
\item[2.] {\bf (Multiplication)} $(k)_n (m)_n = (km)_n$, or
$(k+n \mb{Z})\cdot (m+n\mb{Z}) = (k \cdot m) + n \mb{Z}$.
\item[3.] {\bf (Additivity)} If $(a_1)_n = (b_1)_n$ and
                               $(a_2)_n = (b_2)_n$ then
                 $ (a_1 + a_2 )_n = (b_1 + b_2 )_n$.
\item[4.] {\bf (Multiplicativity)} If $(a_1)_n = (b_1)_n$ and
                               $(a_2)_n = (b_2)_n$ then
                 $ (a_1 \cdot a_2 )_n = (b_1 \cdot b_2 )_n$.
\item[5.] {\bf (Exponentiativity)} If $(a_1)_n = (b_1)_n$ and
                               $c \in \mb{Z}_+$,
                 $ (a_1^c )_n = (b_1^c )_n$.
\end{itemize}
\end{dfn}

\begin{nte}
Several times, instead of using the correct notation
\[
\mb{Z}/n\mb{Z}= \mb{Z}_n = \{ (0)_n , (1)_n , \ldots , (n-1)_n \},
\]
one can use a simplified notation, as follows.
\[
\mb{Z}/n\mb{Z} = \mb{Z}_n = \{ 0 , 1 , \ldots , n-1 \}.
\]
\end{nte}

\noindent
Thus  we won't use the correct notation $(5)_7 +(3)_7 = (1)_7$.
Instead we will write $ 5 + 3 \equiv 1 \pmod 7$.

\begin{thm}
For a positive integer $n$,
\begin{enumerate}
\item  $(a)_n = (b)_n$ if and only if $a \equiv b \pmod n$,
\item For $0\leq a,b<n$ either
$(a)_n = (b)_n$ or $(a)_n \neq (b)_n$,
\item There are exaclty $n$ residue class mod $n$ and no more
and they
cover all of $\mb{Z}$.
\end{enumerate}
\end{thm}
\begin{proof}
$ $ \\ $ $
(1) Consider $(a)_n = (b)_n$. Thus for an element $x$
$x \in (a)_n$ if and only if $x \in (b)_n$.
If $ x \in (a)_n$ then $x =nk +a$. If $x \in (b)_n$ then
$x=nl+b$ for some integer $k,l$. Therefore $nk+a = nl+b$
which leads to $a-b = n (l-k)$ and therefore
$a \equiv b \pmod n$.
Consider $x \in (a)_n$. Then $x \equiv a \pmod n$.
Since $a \equiv b \pmod n$ by transitivity $x \equiv b \pmod n$,
and thus $x \in (b)_n$. Therefore $(a)_n \subseteq (b)_n$.
Similarly we show $(b)_n \subseteq (a)_n$, and the result
follows.
$ $ \\ $ $
(2) Consider $a, b$. If $a \equiv b \pmod n$ then $(a)_n = (b)_n$
from part (1). Otherwise $a \not\equiv \pmod n$. Then
$(a)_n \neq (b)_n$. If this was not the case, then there would
exist $x \in (a)_n \cap (b)_n$. Arguments similar to the proof
in (1) then lead to $a \equiv b \pmod n$, a contradiction.
$ $ \\ $ $
(3) Consider integer $a \in \mn{Z}$ and take the integer
division of $a$ by $n$: $a=nq+r$ where $q$ is the quotient
and $r$ the remainder, where $0 \leq r < n$.
Then $a \equiv r \pmod n$, and thus $a \in (r)_n$.
For $0 \leq i,j < n$, $i\neq j$ we have $i \not\equiv j \pmod n$.
Thus those $n$ congruence classes ($(0)_n , \ldots (n-1)_n$
divide $\mb{Z}$.
\end{proof}

\begin{exa}
Ring $\mb{Z}/6\mb{Z}$  is not an integral domain.
Consider $2 \not\equiv 0 \pmod 6$.
Consider $3 \not\equiv 0 \pmod 6$.
Then
 $2 \cdot 3 \equiv 0 \pmod 6$.
\end{exa}

\begin{exa}
$6\mb{Z}$ is an ideal $I$ (the set of multiples of $6$).
Then $\mb{Z}/6\mb{Z}$  is the set of cosets
\[
\{ a + I : a \in \mb{Z} \} ,
\]
where operations $ + , \cdot $ are defined as follows.
\[
(a+I) + (b+I) = (a+b) + I ,
\]
\[
(a+I) \cdot (b+I) = (a\cdot b) + I .
\]
Two elements $a,b$ are equivalent $a \equiv b$
or equivalently $a+I = b+ I$ if $a-b \in I$.
If $\mb{Z}$ is a ring, $\mb{Z}/6\mb{Z}$ is a quotient
ring also known as a factor ring.
\end{exa}

\subsection{Complete system of congruences}

\begin{dfn}[Complete system of congruences]
Let $n \in \mb{Z}_+$, and $A$ be a set of integers.
$A$ is a complete set of congruences mod $n$ if and only if
\begin{itemize}
\item $|A|=n$, and
\item $A$ contains exactly one element from each one of the
$n$ congruence classes mod $n$.
\end{itemize}
\end{dfn}

\begin{exa}
Let $n=3$. Then $A=\{ 9, 4, 5 \}$ is a complete system
of congruences (residues).
Note
$9 \equiv 0 \pmod 3$,
$4 \equiv 1 \pmod 3$, and
$5 \equiv 2 \pmod 3$.
\end{exa}

\begin{thm}
Let $n \in \mb{Z}_+$, and $A$ be a set of integers.
$A$ is a complete set of congruences mod $n$ if and only if
\begin{itemize}
\item $|A|=n$, and
\item no two elements of $A$ are congruent mod $n$.
\end{itemize}
\end{thm}
\begin{proof}
$ $ \\ $ $
$\Rightarrow$. If $A$ is a complete set of congruences mod $n$
the two conditions above are obviously true.
$ $ \\ $ $
$\Leftarrow$. If no two elements of $A$ are congruent
mod $n$, then they belong to different congruence classes
mod $n$. Given that $A$ has $n$ elements all the congruence
classes are represented in $A$ and thus $A$ is a complete
system of congruences.
\end{proof}

\subsection{Reduced system of congruences}

A congruence class $(b)_n$ is relatively prime to $n$ if and only
if $\gcd(b,n)=1$.

\begin{dfn}
Let $n \in \mb{Z}_+$.Then
$\phi (n)$ is the number of congruence
classes modulo $n$ which are relatively prime to $n$.
Then the set of integers
\[
\left\{ a_1 , a_2 , \ldots , a_{\phi(n)} \right\}
\]
defines a reduced system of congruences if the set contains
exactly one element from each congruence class mod $n$ that
is relatively prime to $n$.
\end{dfn}

\begin{dfn}[Reduced system of congruences]
Let $n \in \mb{Z}_+$, then $\phi (n)$ is the number of congruence
classes modulo $n$ which are relatively prime to $n$.
A congruence class $(b)_n$ is relatively prime to $n$ if and only
if $\gcd(b,n)=1$.
Then the set of integers
\[
\left\{ a_1 , a_2 , \ldots , a_{\phi(n)} \right\}
\]
defines a reduced system of congruences if the set contains
exactly one element from each congruence class mod $n$ that
is relatively prime to $n$.
\end{dfn}

\begin{exa}
Let $n=10$. Start with
\[
\{ 0,1,2,3,4,5,6,7,8,9 \}
\]
Delete $0,2,4,5,6,8$ to obtain
\[
\{ 1,3,7,9 \}
\]
since $0,2,4,5,6,8$ are not relatively prime to $10$.
This provides a reduced system of congruences
Another one is
\[
\{ 11,23,37,-40 \} ,
\]
with $11 \in (1)_{10}$ and so on.
\end{exa}

\begin{thm}
Let $n \in \mb{Z}_+$. Let $A$ be a set of integers.
Then $A$ is a reduced system of congruences mod $n$ if and
only if
\begin{enumerate}
\item $A$ has exactly $\phi (n)$ elements,
\item no two elements $a,b$ are congruent mod $n$
(i.e. $a \not\equiv b \pmod n \quad , \forall a,b \in A$)
\item each element of $A$ is relatively prime to $n$
(i.e. $\forall a \in A, \quad \gcd(a,n)=1$).
\end{enumerate}
\end{thm}
\begin{proof}
A reduced set of congruences possesses those properties by
default. Let us prove the converse.
Let $A$ be a set of integers having those three properties.
Because no two elements of $A$ are congruent mod $n$, it means
they belong to difference congruence classes. All classes
of $A$'s elements are relatively prime to $n$. There are
$\phi (n)$ elements or classes represented in $A$.
Thus the $\phi (n)$ elements of $A$ represent all $\phi (n)$
classes whose elements are relatively prime to $n$.
Thus $A$ is a reduced set of congruences.
\end{proof}

\begin{prp}
If $A$ is a reduced set of congruences, then $kA$ is also
one for $\gcd(k,n)=1$.
\end{prp}

\begin{exa}[Euler's theorem]
Let $n \in \mb{Z}_+$ and $a$ such that $\gcd(a,n)=1$.
The following is true.
\[
 a^{\phi (n)} \equiv 1 \pmod n .
\]
\end{exa}
\begin{solution}
If $x_1 , \ldots , x_{\phi (n)}$ is a reduced system
of congruences mod $n$, then
If $a x_1 , \ldots , a x_{\phi (n)}$ is also a reduced system
of congruences mod $n$. Then
we have the following
\begin{eqnarray*}
 x_1 \ldots x_{\phi (n)} &\equiv& a x_1 , \ldots , a x_{\phi (n)}
\pmod n \\
 x_1 \ldots x_{\phi (n)} &\equiv& a^{\phi(n)} x_1 , \ldots ,  x_{\phi (n)} 
\pmod n \\
a^{\phi(n)} &\equiv& 1 \pmod n.
\end{eqnarray*}
\end{solution}

\subsection{Modular equation redo}

The following variant of Theorem~\ref{modeq2} is stated.
\begin{thm}
\label{modeq2}
The modular equation
\[
ax \equiv b \pmod n
\]
has a solution if and only if $\dv{\gcd(a,n)}{b}$.
Let $\gcd(a,n)=d$, and
$a_1 =a /d , b_1 =b/d, n_1 = n/d$.
If $x$ is a solution of the
modular equation then, we have the following
\[
  x \equiv  b_1 a_1^{\phi (n_1) -1 } \pmod{n_1}.
\]
\end{thm}

\begin{proof}
$ $ \\ $ $
From $ ax \equiv b \pmod n$ we have that $ax-b = kn$ i.e.
$ax+kn = b$, for some $k \in \mb{Z}$. The
solution of this Diophantine equation
exists for $\dv{\gcd(a,n)}{b}$.
Let  $\gcd(a,n)=d$, and
$a_1 =a /d , b_1 =b/d, n_1 = n/d$.
$ $ \\ $ $
If $\dv{d}{b}$ then we have the following for some $k \in \mb{Z}$.
\begin{eqnarray*}
  ax                  &\equiv& b \pmod{n}         \Leftrightarrow \\
           (a_1 d ) x &\equiv& b_1 d \pmod{n_1 d} \Leftrightarrow \\
 (a_1 d ) x - b_1 d   &=&                n_1 d k  \Leftrightarrow \\
 a_1  x - b_1         &=& n_1  k                  \Leftrightarrow \\
 a_1  x               &\equiv&   b_1 \pmod{n_1} .
\end{eqnarray*}
By Euler's theorem it is
$a_1^{\phi (n_1 )} \equiv 1 \pmod{n_1}$.
We multiply the modular equation above by $a_1^{\phi (n_1 ) -1} $.
\begin{eqnarray*}
 a_1 x               &\equiv&   b_1 \pmod{n_1}  \\
 a_1 x a_1^{\phi (n_1 )-1} &\equiv& b_1 a_1^{\phi (n_1 )-1}\pmod{n_1}\\
  x a_1^{\phi (n_1 )} &\equiv& b_1 a_1^{\phi (n_1 )-1}\pmod{n_1}\\
  x                   &\equiv& b_1 a_1^{\phi (n_1 )-1}\pmod{n_1}.
\end{eqnarray*}
The result follows.
\end{proof}

\subsection{Congruence summary}

\begin{prp}
\label{abelplus}
The set
$\mb{Z}/n\mb{Z}$
of integers mod $n$
satisfies
the following properties over addition.
\begin{enumerate}
\item Closure property:
$(i)_n + (j)_n \in \mb{Z}/n\mb{Z} \quad \forall (i)_n , (j)_n
\in \mb{Z}/n\mb{Z}$.
\item Associativity property:
$( (i)_n + (j)_n )+ (k)_n = (i)_n + ( (j)_n + (k)_n )  \quad
\forall (i)_n , (j)_n , (k)_n \in \mb{Z}/n\mb{Z}$.
\item Commutativity property:
$(i)_n + (j)_n = (j)_n + (i)_n  \quad
\forall (i)_n , (j)_n \in \mb{Z}/n\mb{Z}$.
\item Identity that is $(0)_n$,
$\forall (i)_n \in \mb{Z}/n\mb{Z} :
(i)_n + (0)_n = (0)_n +(i)_n = (i)_n$.
\item (Additive) inverse property:
$\forall (i)_n \in \mb{Z}/n\mb{Z} \quad \exists (j)_n \in \mb{Z}/n\mb{Z} :
(i)_n + (j)_n = (j)_n +(i)_n = (0)_n$. Moreover $(j)_n = (-i)_n = -(i)_n$.
\end{enumerate}
\end{prp}

\begin{prp}
\label{abeltimes}
The set
$\mb{Z}/n\mb{Z}$
of integers mod $n$
satisfies
the following properties over multiplication.
\begin{enumerate}
\item Closure property:
$(i)_n \cdot (j)_n \in \mb{Z}/n\mb{Z} \quad \forall (i)_n , (j)_n
\in \mb{Z}/n\mb{Z}$.
\item Associativity property:
$( (i)_n \cdot (j)_n ) \cdot  (k)_n =
(i)_n \cdot ( (j)_n \cdot (k)_n )  \quad
\forall (i)_n , (j)_n , (k)_n \in \mb{Z}/n\mb{Z}$.
\item Commutativity property:
$(i)_n \cdot (j)_n = (j)_n \cdot (i)_n  \quad
\forall (i)_n , (j)_n \in \mb{Z}/n\mb{Z}$.
\item Identity that is $(1)_n$,
$\forall (i)_n \in \mb{Z}/n\mb{Z} :
(i)_n \cdot  (1)_n = (1)_n \cdot (i)_n = (i)_n$.
\item Distribution property of multiplication over addition:
$
(i)_n \cdot ( (j)_n  +    (k)_n )   =
((i)_n \cdot  (j)_n)  + ((i)_n \cdot   (k)_n )$.
\end{enumerate}
\end{prp}

Therefore $(\mb{Z}/n\mb{Z} , +, \cdot)$ is a ring,
since $(\mb{Z}/n\mb{Z} , +)$ is an abelian group,
$(\mb{Z}/n\mb{Z} , \cdot )$ is a semigroup, and $\cdot$
is distributive over $+$. Furthermore, $\cdot$ is commutative thus
making  $(\mb{Z}/n\mb{Z} , +, \cdot)$ a commutative ring.

\begin{prp}[$\mb{Z}_n$ or $\mb{Z}/n\mb{Z}$ is a commutative ring.]
$\mb{Z}_n$ or $\mb{Z}/n\mb{Z}$  is such that
$( \mb{Z}/n\mb{Z} , + , \cdot )$ is a commutative
ring.
\end{prp}

\begin{prp}
The ring $(\mb{Z}/n\mb{Z} , +, \cdot)$ is a field for
a prime integer $n$.
\end{prp}

\newpage

\chapter{Intermediate Number Theory}

\section{Euler's totient function}

\begin{dfn}[Euler's $\phi$ function]
For any $n \in \mb{N}$,  let 
\[
n= p_1^{a_1} p_2^{a_2} \ldots p_k^{a_k} ,
\]
where $p_1, p_2, \ldots , p_k$ are $k$ distinct
prime numbers, $k \geq 1$, and 
$a_1 , a_2 , \ldots a_k \in \mb{N}$
with $a_i \geq 1$.
We define $\phi (n)$ as follows.
\[
 \phi (n) = 
p_1^{a_1 -1} ( p_1 -1 )
p_2^{a_2 -1} ( p_2 -1 )
\ldots
p_k^{a_k -1} ( p_k -1 ) .
\]
\end{dfn}

Moreover if we define $\phi (1) = 1$ then we
can write $ \phi (n)$ as follows as well.
\[
 \phi (n) = n
   \left( 1 - \frac{1}{p_1} \right)
   \left( 1 - \frac{1}{p_2} \right)
    \ldots
   \left( 1 - \frac{1}{p_k} \right).
\]

\begin{lem}
\label{phil1}
If $p,q$ are prime numbers and then $\gcd(p,q)=1$,
then
\[
 \phi (pq) = \phi (p) \phi (q).
\]
\end{lem}
\begin{proof}
The number of integers relatively prime to $pq$ are
the multiples of $p$, the multiple of $q$, and if
we include in this count $pq$ make sure it is not included
twice. Thus out of the $pq$ integers $1, 2, \ldots , pq$.
There are $pq/p= q$ multiples of $p$ (including $pq$)
and
there are $pq/q= p$ multiples of $q$ (including $pq$),
and thus
the number of integers relatively prive to $pq$ are
\[
pq - p - q +1,
\]
where the $+1$ is because in $-p$ and $-q$ the contribution
of $pq$ was subracted twice and we adjust with a $+1$
accordingly.
Then
\[
pq - p - q +1 = (p-1) (q-1) = \phi (p) \phi (q).
\]
\end{proof}

\begin{lem}
\label{phil2}
Show that for any $N,M \in \mb{N}$
with $N= p^n$ and $M=p^m$ with $p$ a prime number,
and $n,m \in \mb{N}$ with  $n \geq 0 , m \geq 0$,
we have that
\[
 \phi (N \cdot M ) =
\phi (N) \cdot \phi (M) \cdot
                     \frac{d}{\phi (d)},
\]
where $d= \gcd(N,M)$.
\end{lem}

\noindent
{\bf Note.} $\phi (1)=1$.

\begin{proof}
$ $ \\ $ $
If $n=0$ then $N=p^0 = 1$, $\gcd(N,M)=1$,
and therefore
\[
  \phi (N \cdot M ) = \phi ( M )    \wedge
\phi (N) \cdot \phi (M) \cdot
                     \frac{d}{\phi (d)} =
   \phi(1) \phi(M) \cdot \frac{1}{\phi(1)} = \phi (M),
\]
and the result is true.
$ $ \\ $ $
If $m=0$ a similar result is obtained.
$ $ \\ $ $
If $n > 0, m > 0$ then
\[
\phi (N \cdot M ) = \phi (p^{n+m} ) =
                  = p^{n+m-1} (p-1)
                  = \left( p^{k+l-1} (p-1) \right),
\]
where $k=\max{(n,m)}$ and $l= \min{(n,m)}$.
On the other hand
\[
\phi (N) \cdot \phi (M) \cdot
                     \frac{d}{\phi (d)} =
 p^n (p-1) p^m (p-1) \frac{p^l}{p^{l-1} (p-1)} =
 p^{n+m-1} (p-1).
\]
The result is then proven.
\end{proof}

\begin{lem}
\label{phil3}
Show that for any $N,M \in \mb{N}$ we have that
\[
 \phi (N \cdot M ) = \phi (N) \cdot \phi (M) \cdot
                     \frac{d}{\phi (d)},
\]
where $d= \gcd(N,M)$.
\end{lem}
\begin{proof}
Let
$N=
  p_1^{n_1}
  p_2^{n_2}
  \ldots
  p_r^{n_r}
$
and
$M=
  p_1^{m_1}
  p_2^{m_2}
  \ldots
  p_r^{m_r}
$
where $n_i \geq 0$ and $m_i \geq 0$.
Let $l_i = \min{(n_i ,m_i )}$, $i=1, \ldots , r$.
Let $k_i = \max{(n_i ,m_i )}$, $i=1, \ldots , r$.
Then $k_i + l_i = n_i + m_i$.
Then we have the following.
\begin{eqnarray*}
 \phi (N) \cdot \phi (M) \cdot \frac{d}{\phi (d)}
&=&
 \phi \left( p_1^{n_1 } p_2^{n_2 } \ldots p_r^{n_r } \right)
\cdot
 \phi \left( p_1^{m_1 } p_2^{m_2 } \ldots p_r^{m_r } \right)
\cdot
 \frac{       p_1^{l_1} p_2^{l_2} \ldots p_r^{l_r} }{
       \phi ( p_1^{l_1} p_2^{l_2} \ldots p_r^{l_r} ) } \\
&=&
 p_1^{n_1 -1 } p_2^{n_2 -1 } \ldots p_r^{n_r -1 } \cdot
 ( p_1  -1 )(p_2 -1 ) \ldots ( p_r -1)      \\
&\cdot&
 p_1^{m_1 -1 } p_2^{m_2 -1 } \ldots p_r^{m_r -1 } \cdot
 ( p_1  -1 )(p_2 -1 ) \ldots ( p_r -1 )      \\
 &\cdot&
 \frac{       p_1^{l_1} p_2^{l_2} \ldots p_r^{l_r} }{
        p_1^{l_1 -1} (p_1 -1)  p_2^{l_2 -1 } (p_2 -1)
        \ldots p_r^{l_r-1} (p_r -1) } \\
&=&
 p_1^{k_1 +l_1 -2 } p_2^{k_2 +l_2 -2} \ldots p_r^{k_r +l_r -2 } \cdot
 ( p_1  -1 )^2 (p_2 -1 )^2  \ldots ( p_r -1 )^2       \cdot
 \frac{       p_1 p_2 \ldots p_r }{
         (p_1 -1)  (p_2 -1) \ldots (p_r -1) }  \\
&=&
 p_1^{k_1 +l_1 -1 } p_2^{k_2 +l_2 -1} \ldots p_r^{k_r +l_r -1 } \cdot
 ( p_1  -1 )   (p_2 -1 )    \ldots ( p_r -1 )         \\
&=&
 p_1^{n_1 +m_1 -1 } p_2^{n_2 +m_2 -1} \ldots p_r^{n_r +m_r -1 } \cdot
 ( p_1  -1 )   (p_2 -1 )    \ldots ( p_r -1 )        \\
&=&
 \phi (N \cdot M ).
\end{eqnarray*}
The result is proven.
\end{proof}

\begin{thm}
\label{phit1}
Let $N,m \in \mb{N}$ and $N,m > 0$.
Then there are $m \cdot \phi (N)$ integers relatively 
prime to $N$ among $1,2, \ldots , N \cdot m -1, N \cdot m$.
\end{thm}
\begin{proof}
\[
N=
  p_1^{n_1}
  p_2^{n_2}
  \ldots
  p_{r-1}^{n_{r-1}}  ,
  p_r^{n_r}  ,
\]
where $n_i  >   0$, $r>0$ and $1 < p_1 < p_2 < \ldots < p_r $.
The result uses induction on $r$.
$ $ \\ $ $
{\bf Base case $r=0$.} Then $N=1$, and $\phi (N) =1$.
True by inspection.
The number of integers in $1,2, \ldots , m-1,m$ that are
relatively prime to $N=1$ is $m$, obviously.
$ $ \\ $ $
{\bf Induction hypothesis: $H(k)$, $k<r$}.
Let us assume that the theorem is true for an
\[
n =
  p_1^{n_1}
  p_2^{n_2}
  \ldots
  p_{r-1}^{n_{r-1}}  ,
\]
Then $Nm = \left( p_r^{n_r} m \right) n $ and by the induction
hypothesis among the integers
$1,2, \ldots , N \cdot m -1, N \cdot m$ there
are $ \left( p_r^{n_r} m \right) \phi (n)$ integers
relatively prime to $n$ such as $x$, where $\gcd(x,n)=1$.
We now want to find the number of integers $y$ relatively
prime to $N=n p_r^{n_r}$, where $\gcd(y,N)=1$. If $y$
is relatively prime to $N$ is also so  to $n$. But if it
is to $n$ it might not be so for $N$: it might be a multiple
of $p_r$. Therefore we need to discard those $x$ that are not
relatively prime to $p_r^{n_r}$, i.e. all multiples of $p_r$.
Since $p_r$ is a prime number,
the number that are NOT relatively prime to
$p_r^{n_r} $ are the multiple of $p_r$.
Among the $ \left( p_r^{n_r} m \right) \phi (n)$ integers
$ \left( p_r^{n_r} m \right) \phi (n) / p_r$ of them are
NOT relative prime to $p_r^{n_r} $ as they are multiple of $p_r$
and thus the number of integers relatively prime to $N$
in the range $1,2, \ldots , N \cdot m -1, N \cdot m$
is
\[
\left( p_r^{n_r} m \right) \phi (n) -
\left( p_r^{n_r} m \right) \phi (n) / p_r
= m  p_r^{n_r -1} \left( p_r -1 \right) \phi (n)
= m  \phi (N ).
\]
The result has been proven.
\end{proof}

\begin{cor}
From Theorem~\ref{phit1} conclude that
the number of integers relatively
prime to $N$ is $\phi (N)$.
\end{cor}
\begin{proof}
For the Corollary set $m=1$ in Theorem~\ref{phit1}.
\end{proof}

\bigskip %THEOREM 51 %%
\begin{thm}
\label{thm51}
If $\gcd(m,n)=1$ then $\phi(mn)=\phi(m) \phi(n)$.
\end{thm}

\begin{proof}
If $a$ is a unit $\pmod {mn}$ it means $\gcd(a,mn)=1$.
Using a Chinese remainder theorem-based method
consider function $g$ defined on 
$\mb{Z}_{mn} \rightarrow \mb{Z}_m \times \mb{Z}_n$.
\[
g(A) = ( A \pmod m , A \pmod n ).
\]
we will show that $g$ is  a one-to-one and onto bijection.
Note that if $A$ is a unit $\pmod {mn}$ then by Proposition~\ref{uniteq}
it is a unit $\pmod m$ and a unit $\pmod n$. Consider $G(A_1 )=G(A_2 )$.
Then $ A_1 \equiv A_2 \pmod m$
and  $ A_1 \equiv A_2 \pmod n$.
Thus $\dv{m}{A_1 - A_2}$ and
     $\dv{n}{A_1 - A_2}$. If $\gcd(m,n)=1$ as it is,
then $\dv{mn}{A_1 - A_2}$ as well implying $A_1 \equiv A_2 \pmod {mn}$. 
Say $A_1$ is a unit $\pmod m$ and $A_2$ a unit $\pmod n$. Then
from $\gcd(m,n)=1$ and say Corollary~\ref{cor7} there is a unique
$A \pmod {mn}$ such that $A \equiv A_1 \pmod m$ and
                        $A \equiv A_2 \pmod n$.
Thys $g(A) = (A_1 , A_2 ) = (A \pmod m , A \pmod n )$.
Thus the two sets $\mb{Z}_{mn}$ and $\mb{Z}_m \times \mb{Z}_n$ have the same
number of elements thus $\phi (mn) = \phi(m) \phi(n)$.
\end{proof}

\begin{cor}
\label{cor11}
For $p^k$, $k>1$, the number of units of $\mb{U}_{p^k}$ is $ p^k$
 minus the multiples of $p$ which is $p^{k-1}$. 
Thus $|\mb{U}_{p^k}| = p^k - p^{k-1} = p^{k-1} (p-1)$.
Therefore $\phi(p^k) = p^k -p^{k-1}$.
\end{cor}

\begin{cor}
If $n = p_1^{a_1} p_2^{a_2} \ldots p_k^{a_k}$,
then 
\[
 \phi (n) = n
   \left( 1 - \frac{1}{p_1} \right)
   \left( 1 - \frac{1}{p_2} \right)
    \ldots
   \left( 1 - \frac{1}{p_k} \right).
\]
\end{cor}

\begin{proof}
By way of Theorem~\ref{thm51},
$\phi(n) = \phi (p_1^{a_1}) \ldots  \phi ( p_k^{a_k} )$.
Furthermore, from Corollary~\ref{cor11} we have
\[
\phi(n) = \phi (p_1^{a_1}) \ldots  \phi ( p_k^{a_k} ) =
          p_1^{a_1} - p_1^{a_1 -1} \ldots p_k^{a_k} - p_k^{a_k -1}  =
     n    (1- 1/ p_1 ) \ldots (1- 1/ p_k ).
\]
\end{proof}

\subsection{More on the totient function}

\begin{prp}
For all integer $n >0$ the following applies.
\[
  n = \sum_{\dv{d}{n}} \phi (d).
\]
\end{prp}
\begin{proof}
For $d=1,2, \ldots , n$ define
\begin{equation}
\label{phin1}
 S_d = \{ 1\leq a \leq n : \gcd(a,n)= d \}
\end{equation}
If $d\neq 1$ and $\ndv{d}{n}$ then $S_d = \emptyset$.
It is straightforward to conclude
\begin{equation}
\label{phin2}
 \sum_{d=1}^{n} |S_d | = n
\end{equation}
From Eq.(\ref{phin1}) we have the following, given that $\dv{d}{a}$
and thus $a/d$ is an integer and so is $n/d$.
\begin{equation}
\label{phin3}
 S_d = \{ 1\leq a \leq n : \gcd(a,n)= d \}
     = \{ 1\leq \frac{a}{d} \leq \frac{n}{d} :
          \gcd(\frac{a}{d},\frac{n}{d})= 1 \} .
\end{equation}
From the latter we conclude $|S_d| = \phi(n/d)$.
Eq.(\ref{phin3}) by way of Eq.(\ref{phin2}) gives
the following.
\begin{equation}
\label{phin4}
n =
 \sum_{d=1}^{n} |S_d | =
 \sum_{d=1}^{n} \phi(n/d) =
 \sum_{\dv{d}{n}, \dv{\frac{n}{d}}{n}} \phi(n/d) =
 \sum_{           \dv{\frac{n}{d}}{n}} \phi(n/d) =
 \sum_{\dv{D}{n}} \phi(D) =
 \sum_{\dv{d}{n}} \phi(d) .
\end{equation}
The last few equations involved letter relabelings.
\end{proof}

\begin{exa}
A slightly shorter version of the proof follows.
\end{exa}
\begin{solution}
Consider an integer $k$ among $1,2, \ldots , n$.
Let $\gcd(k,n) =d$.
Note that  $\gcd(k,n) =d$ implies
$\gcd(k/d, n/d)=\gcd( m,n/d) = 1$.
Let $k/d= m$, or in other words $k=dm$.
Then we insert $k$
into set $S_d = \{ m d : 1 \leq m \leq n/d , \gcd(m,n/d)=1 \}$.
The cardinality of $S_d$
is $|S_d | = \phi (n/d)$.
Then
\[
|\cup_{d : \dv{d}{n}} S_d | = |\{ 1,2, \ldots , n \}| =n
  \Leftrightarrow
\sum_{\dv{d}{n}} |S_d| = n
  \Leftrightarrow
\sum_{\dv{d}{n}} \phi (n/d) = n
  \Leftrightarrow
\sum_{\dv{d}{n}} \phi (d) = n ,
\]
sice $n/d$ is a divisor of $n$ if and only if
$d$ is a divisor of $n$.
\end{solution}

%%%%%% Modular equations

\section{Units}

Let $n$ be a positive integer.
The set of all residue classes mod $n$ is
denoted as $\mb{Z}/n\mb{Z}$, and rarely
$\mb{Z}_n$ or $\mb{Z}/n$.
It is the set of integers modulo $n$, thus representing
the $n$ equivalence classes that the integers of $\mb{Z}$ can be
split into depending on the remainder of their division by $n$.
$\mb{Z}_n$ is a cyclic group under addition, and a commutative
ring under multiplication and addition.
The ring is a field for a prime $n$.

\subsection{Units in $\mb{Z}$ }

\begin{dfn}[\bf{Units in $\mb{Z}$}]
A divisor of one is called a unit.
The only units of $\mb{Z}$ are $+1$ and $-1$.
\end{dfn}

\begin{dfn}[\bf{Alternative definition of a unit in $\mb{Z}$}]
An element $a$ in $\mb{Z}$ that has an inverse in $\mb{Z}$
is also known as a unit.
Only $+1$ and $-1$ have inverses in $\mb{Z}$.
\end{dfn}

In $\mb{Z}$, $+1$ or $-1$ are indeed the
only invertible elements of $\mb{Z}$.

\begin{prp}
There are no other units in $\mb{Z}$, other than $+1$ and
$-1$.
\end{prp}
\begin{proof}
Let $u$
be a third unit in $\mb{Z}$ with  $u \neq \pm 1$.
Then $ 1= u q$ for some integer $q \in \mb{Z}$.
Since $u \neq \pm 1$ and integer $|u| \geq 2$,
noting $1=uq$ implies $u \neq 0$ as well. But
$|u|\geq 2 $ and $uq=1$ implies $| q| \leq 1/2$
and thus $q=0$. Then $uq=0$, a contradiction
to $uq=1$.
\end{proof}

\subsection{Inverses}

\begin{dfn}
Let $n \in \mb{Z}_+$ and consider $\mb{Z}/n\mb{Z} $
(also known as $\mb{Z}_n$). Let $a \in \mb{Z}_n$.
The inverse $b$ of $a$ mod $n$ is defined as
\[
  b \cdot a \equiv 1 \pmod n.
\]
We can then say that $a$ and $b$ are  multiplicative
inverses of one another.
\end{dfn}

We write $b$ as $1/a$ or $a^{-1}$.
Note that implicit in this definition is that
$a,b \in \mb{Z}/n\mb{Z} $ or equivalently
$ 1 \leq a, b <n$.
We might use in the remainder for simplicity $\mb{Z}_n$
for $\mb{Z}/n\mb{Z} $.

\begin{prp}
The inverse of $a$ mod $n$ exists if and only if
$\gcd(a,n)=1$.
\end{prp}

This is a byproduct of the extended GCD or results derived
for a diophantine equation discussion.
We can prove then a more general result.

\begin{cor}
Let $n>1$ be an integer and $\ndv{n}{a}$. The following are
\begin{itemize}
\item[(a)] $a$ is a zero divisor $\pmod n$,
\item[(b)] $a$ has no inverse $\pmod n$,
\item[(c)] there exists a $s\in Z$ such that $\ndv{n}{s}$ and
$as \equiv  0 \pmod n$.
\end{itemize}
\end{cor}
\begin{proof}
Statements (a) and (b) are true and equivalent by the prior
definition and introduction of unit and zero divisor.
$ $ \\ $ $
Suppose that $(b)$ is true and $a$ has no inverse.
By Theorem~\ref{modn} $\gcd(a,n) >1$. Let $\gcd(a,n)=d>1$.
The $a=dr$ and $n=ds$ for some integer $r,s$.
For $1< < z < n$ we have $z \not\equiv \pmod n$. Furthermore,
$a s =  (dr)s = (ds)r = n r \equiv 0 \pmod n$. Statement (c)
follows from Statement (b).
$ $ \\ $ $
Suppose that statement (c) is true. There there exists
an $s\in Z$ such that $\ndv{n}{a}$ and $as \equiv 0 \pmod n$.
We are going to prove $a$ has no inverse. Let us assume that
$a$ has an inverse, then $ a a^{-1} \equiv 1 \pmod n$, and then
\[
 0 \equiv as \equiv as a^{-1} \equiv (aa^{-1})s \equiv s \pmod n.
\]
The latter implies that $\dv{n}{s}$ that contradicts the
assumption that $\ndv{n}{s}$! Thus $a$ has no inverse
and statement (b) is true coming from (c). Thus statements
(b) and (c) are equivalent.
For prime $p$, $ab \equiv 0 \pmod p$ implies  either
$a\equiv 0 \pmod p$
or
$b\equiv 0 \pmod p$.
Thus for a prime $p$ there no zero divisors other than 0
$\pmod p$.
\end{proof}

\begin{prp}
There are $\phi (n)$ numbers $a$ in $\mb{Z}_n$  for
which the inverse $a^{-1} = 1/a$ exists.
\end{prp}
This is a consequence of the fact that there are
$\phi(n)$ natural numbers less that $n$ that are relatively
prime to $n$.

\begin{exa}
For $n=10$ there are only $\phi (10) = 4$, $a$'s with
inverses. These are $a= 1,3,7,9$ with $a^{-1} = 1, 7,3, 9$
respectively.
\end{exa}

\begin{thm}
The ring $(\mb{Z}/n\mb{Z} , +, \cdot)$ is a field
if and only if $n$ is a prime number.
\end{thm}
\begin{proof}
From Proposition~(\ref{abelplus}) we have that
$\mb{Z}/n\mb{Z}$
is an abelian group over $+$.
From Proposition~(\ref{abelplus}) we have that
$\mb{Z}/n\mb{Z}$ minus the $0$ (additive identity)
is an abelian group over $\cdot$,
if a (multiplicative) inverse property is true i.e.
$\forall (i)_n \in \mb{Z}/n\mb{Z} \quad \exists (j)_n \in \mb{Z}/n\mb{Z} :
(i)_n \cdot (j)_n = (j)_n \cdot (i)_n = (1)_n$.
Moreover $(j)_n = 1/ (i)_n = (i)_n^{-1}$.
For the latter to be the case $n$ must be a prime number.
(This is derived from prior results.)
\end{proof}

\subsection{Units in $\mb{Z}/n\mb{Z}$}

Consider $\mb{Z}/n\mb{Z}$ also known as $\mb{Z}_n$.
Any $g \in \mb{Z}/n\mb{Z}$ that is invertible and thus
$gx \equiv 1 \pmod n$ has a solution for $x$, the inverse
of $g$ moduloe $n$, is known as a unit of
$\mb{Z}_n$ and thus of $ \mb{Z}/n\mb{Z}$.
We further conclude that $g$ is a unit if and only
if $\gcd(g,n)=1$. Then $g x \equiv 1 \pmod n$ has a
solution for $x = g^{-1}$.

\begin{dfn}[\bf{Units in $\mb{Z}/n\mb{Z}$}]
Any $g \in \mb{Z}/n\mb{Z}$ that has an inverse
modulo $n$, that is, there exists an $x$
that satisfies the following modular equation
\[
gx \equiv 1 \pmod n ,
\]
is called a unit modulo $n$.
The $x$ is also known as the inverse of $g$ modulo $n$.
Moreover $\gcd(g,n)=1$.
\end{dfn}

\begin{dfn}[\bf{Zero divisors in $\mb{Z}_n$}]
An $g$ that is not a unit  is called a zero divisor
modulo $n$.
\end{dfn}

\begin{dfn}[\bf{Set of units $\mb{U}_n$}]
The set of units of $\mb{Z}_n$ is denoted as
$\mb{U}_n$.
It is also denoted as     $\mb{Z}_n^x$.
It is also denoted as     $(\mb{Z}/n\mb{Z})^x$.
It is also denoted as     $\mb{Z}_n^*$.
\end{dfn}

\begin{prp}
\label{uniteq}
Let $m, n >1$ be integer. The following two statements
are equivalent.
\begin{itemize}
\item[(i)]  $a$ is a unit $\pmod {mn}$.
\item[(ii)] $a$ is a unit $\pmod m$ and $a$ is a unit $\pmod n$.
\end{itemize}
\end{prp}

\begin{proof}
$ $ \\  $ $
$(i) \Rightarrow (ii)$.
If $a$ is a unit $\pmod {mn}$ then it means 
$\gcd(a,mn)=1$.
We claim that this implies that  
$\gcd(a,m)=\gcd(a,n)=1$. 
If this was not so, and say $\gcd(a,m)=d >1$,
then  $d$ becomes a common divisor of $a$
and $m$ (and consequently of of $mn$ as well). 
This would imply
$\gcd(a,mn) >1$, a contradiction to $\gcd(a,mn)=1$.
$ $ \\  $ $
Then $\gcd(a,m)=\gcd(a,n)=1$.
Therefore $ax \equiv 1 \pmod m$ has a solution for $x$
and thus $a$ is unit mod $m$.
Likewise, $ay \equiv 1 \pmod n$ has a solution for $y$
and thus $a$ is unit mod $n$. Case completed.
$ $ \\ $ $
$(ii) \Rightarrow (i)$.
If $a$ is a unit mod $m$ then $\gcd(a,m)=1$.
If $a$ is a unit mod $n$ then $\gcd(a,n)=1$.
The former imply $\gcd(a,mn)=1$.
$ $ \\ $ $
If the latter was not so, then $\gcd(a,mn)=d >1$.
There is a
prime factor $p$ of $d$ i.e. $\dv{p}{d}$.
Then $\dv{p}{a}$ and thus $\dv{p}{mn}$. The latter
implies $\dv{p}{m}$ or $\dv{p}{n}$. One
or the other combined with $\dv{p}{a}$ implies
that $\dv{p}{\gcd(a,m)}$ or $\dv{p}{\gcd(a,n)}$
contradicting that $\gcd(a,m)=\gcd(a,n)=1$.
The result then follows.
\end{proof}

\subsection{The totient function as cardinality of a set}

\begin{dfn}[\bf{Cardinality} of $\mb{U}_n$]
The set of units $\mb{U}_n$ has cardinality
$\phi (n)$, where $\phi (n)$ is Euler's
totient function.
\end{dfn}

Let $p$ be a prime. Every non-zero element of
$\mb{Z}_p$ i.e. $1, \ldots , p-1$ is relatively
prime to $p$, and thus $\phi(p)=p-1$.
Moreover $|\mb{U}_p|=\phi(p)=p-1$.

\subsection{Units of rings}

\begin{dfn}[Set of units of ring $S$ is $S^{x}$]
Let $(S, + , \cdot )$ be a ring, and let $x \in S$.
We say that $x$ is invertible or equivalently that
$x$ is a unit, if there exists a $y \in S$ such that
\[
   x \cdot  y = y \cdot  x = 1,
\]
where the multiplication $\cdot$ implies the multiplicative
operation of $S$, and $1$ is the identity element of $S$
over $\cdot$.
The set of units of $S$ is denoted as $S^{x}$.
\end{dfn}

Note that for $x$, the $y$ such that
$ x \cdot  y = y \cdot  x = 1$ is uniquely defined.
For there was a $z$ such that
$ x \cdot  z = z \cdot  x = 1$,
then
$y = y (x z) = (y x) z = z$. The unique $y$ that is
the inverse of $x$ is sometimes denoted as $x^{-1}$.

\begin{prp}
For a ring $(S, + , \cdot )$,
we have that $(S^{x}, \cdot )$ is a group.
\end{prp}

\begin{dfn}[Division ring]
Let $(S, + , \cdot )$ be a ring.
If all of its non-zero elements are invertible,
then $S$ is a division ring. Moreover, if $S$ is
commutative then $S$ is a field.
\end{dfn}

\begin{lem}
If $a   \in \mb{Z}/n\mb{Z} $ is a unit modulo $n$ then
$\gcd(a,n)=1$.
\end{lem}

\newpage

\section{Chinese remainder theorem}

% 11 and 20 
% 101 = 11 * 9 + 2
% 101 = 20 * 5 + 1
% 321 = 11 * 29 +2 , 321= 20 * 16 +1
\begin{exa}
A farmer has some pounds sugar. If the farmer
puts them into bags of 11 pounds the farmer can fit enough 
full bags but then is left with 2 spare pounds. 
If the farmer uses 20 pound bags then is
also left with 1 spare pound. How many pounds of sugar 
does the farmer have?
\end{exa}

Say the farmer has 321 pounds of sugar. This amount 
needs 29 11-pound bags and there 2 spare pounds. 
If the farmer uses 20-pound bags
the farmer can fill 16 bags and is left with 1 pound.
Is it a solution.

\begin{exa}
What if the farmer has 101 pounds of sugar ?
$101= 11\cdot  9 + 2 =  5 \cdot 20+1$.
\end{exa}

The only solution in $0 \ldots 219$ seems to be 101.
But beyond that range,
another solution is $101+220$, $100+2\cdot 220$, and so on.

\bigskip %THEOREM 45 %%
\begin{thm}[Chinese remainder theorem]
\label{crt}
Let $n_1 , \ldots , n_k$ are pairwise prime numbers
i.e. $\gcd(n_i , n_j )=1$ for $i\neq j$.
Let $a_1 , \ldots , a_k \in \mb{Z}$.
There is a unique $A \pmod {n_1 \ldots n_k}$ such that
it satisfies all of the modular equation below.
\begin{eqnarray*}
   A &\equiv& a_1 \pmod {n_1} \\
   A &\equiv& a_2 \pmod {n_2} \\
     & \ldots & \\
   A &\equiv& a_k \pmod {n_k} \\
\end{eqnarray*}
Moreover if there are two solutions $A, a$, then
\[
 A \equiv a \pmod{n_1 \ldots n_k}
\Rightarrow
 A \equiv a \pmod{N},
\]
where $N= n_1 n_2 \ldots n_k$.
\end{thm}

\begin{proof}
Let $N_j$, for $j=1, \ldots , k$ contain all $n_i$ except $n_j$.
That is
\[
  N_j = n_1 \ldots  n_{j-1} n_{j+1} \ldots n_k .
\]
We have $\dv{n_i}{N_j}$ for all $i \neq j$.
We have that $\gcd (N_j , n_j ) =1 $. This is because
$\gcd( n_j , n_i ) = 1$ for all $i \neq j$ as they are
pairwise prime . Since $\gcd( N_j , n_j)=1$ we have
that there exists an integer $x_j$ such that
\[
 N_j x_j \equiv 1 \pmod {n_j} ,
\]
for all $j =1 , \ldots , k$.
We then form the following expression.
\[
   A = a_1 N_1 x_1 + \ldots + a_i N_i x_i + \ldots + a_k N_k x_k .
\]
Consider $n_j$ and $N_j$.
It is $\dv{n_i}{N_j}$ for all $i \neq j$. Thus all
the terms $a_j N_j x_j$ are multiples of $n_i$ for $j \neq i$.
For the term $a_i N_i x_i$ this is note the case as
$\gcd(n_i , N_i ) =1 $. But we have that $N_i x_i \equiv 1 \pmod {n_i}$.
Thus $a_i N_i x_i \equiv a_i \pmod {n_i}$. The second equivalence
below follows
\[
  A \equiv a_1 \cdot 0 + a_2 \cdot 0 + \ldots + 
           a_i \cdot N_i \cdot x_i   + \ldots  + a_k \cdot 0  \pmod {n_i}
    \equiv  a_i \pmod {n_i} 
\]
This is true for all $i$ and the following applies.
\[
\forall 1 \leq i \leq k :\quad  A \equiv a_i \pmod {n_i}.
\]
Say that there is another solution $a$ i.e.
$a \equiv A \equiv a_i \pmod {n_i}$ or all $i$.
This would mean that $\dv{n_i}{A-a}$. Since $n_i$
are pairwise prime by Theorem~\ref{thm25} we have
\[
 \dv{n_1 n_2 \ldots n_k}{A-a}.
\]
This means $A-a \equiv \pmod {n_1 n_2 \ldots n_k}$.
\end{proof}

\begin{cor}
\noindent
Let $n_1 , \ldots , n_k$ be  pairwise prime numbers
i.e. $\gcd(n_i , n_j )=1$ for $i\neq j$, $i,j = 1, \ldots , k$.
Then $A \equiv a \pmod{n_i}$, $i=j = 1, \ldots , k$ if and only
if
\[
 A \equiv a \pmod{N},
\]
where $N= n_1 n_2 \ldots n_k$.
\end{cor}
\begin{proof}
It immediately follows from the previous theorem.
\end{proof}

\begin{cor}
\noindent
Let $n_1 , \ldots , n_k$ be  pairwise prime numbers
i.e. $\gcd(n_i , n_j )=1$ for $i\neq j$, $i,j = 1, \ldots , k$.
Let $a_1 , \ldots , a_k \in \mb{Z}$,
and let us assume that there is a solution for the system
of modular equations below.
\begin{eqnarray}
\label{crt2}
   A &\equiv& a_1 \pmod {n_1} \nonumber \\
   A &\equiv& a_2 \pmod {n_2} \nonumber \\
     & \ldots & \\
   A &\equiv& a_k \pmod {n_k} \nonumber
\end{eqnarray}
Furthermore, let $N= n_1 n_2 \ldots n_k$.
Show that the solution is as follows.
\[
A = \sum_{i=1}^{k}  a_i N_i x_i \pmod{N}
\]
where $N_i = N / n_i$ and
$N_i x_i \equiv 1 \pmod{n_i}$,
$i=1, 2, \ldots , k$.
\end{cor}
\begin{proof}
It follows directly from the proof of existence of
Eq.(\ref{crt}).
\end{proof}

\newpage

\begin{exa}      
Find all integers $A$ such that
\begin{eqnarray*}
    A &\equiv& 1 \pmod 2 \\ 
    A &\equiv& 2 \pmod 3 \\ 
    A &\equiv& 3 \pmod 5 \\ 
\end{eqnarray*}
\end{exa}      

This is the similar to the original
problem proposed by Sunzi
that appeared in the book  Sunzi Suanjing.
A solution though was not proposed there.
In the original problem, 2,3,5 were replaced by 3,5 and 7
respectively.

\begin{proof}
Let $n_1 =2 , n_2 = 3 , n_3 = 5$.
Then $N=30$.
Furthermore,
$N_1 = 3\cdot 5 = 15$,  $N_2 = 2\cdot 5 = 10$, $N_3 = 2\cdot 3 =  6$.

\begin{eqnarray*}
 N_1 x_1 &\equiv& 1 \pmod n_1 \Rightarrow 15 x_1 \equiv 1 \pmod 2 
                              \Rightarrow x_1 = 1  \\
 N_2 x_3 &\equiv& 1 \pmod n_2 \Rightarrow 10 x_2 \equiv 1 \pmod 3 
                              \Rightarrow x_2 = 1 \\
 N_3 x_4 &\equiv& 1 \pmod n_3 \Rightarrow  6 x_3 \equiv 1 \pmod 5
                              \Rightarrow x_3 = 1 \\
\end{eqnarray*}
Then $n_1 n_2 n_3 = 2 \cdot 3 \cdot 5= 30$, and
\[
A = a_1 N_1 x_1 + a_2 N_2 x_2 + a_3 N_3 x_3= 
   1\cdot  15 \cdot 1 + 2\cdot  10 \cdot 1 + 3\cdot   6 \cdot 1 =
    15+ 20 + 18 \pmod 30  = 23
\]
Thus one solution is 23. Another $23+30$, another $23+60$, and so on.
\end{proof}

\begin{cor}
\label{cor7}
Let $n_1 , n_2 >1$ be integers and let $a_1 , a_2 \in Z$.
Let $d=\gcd(n_1 , n_2 )$. If $\dv{d}{a_1 - a_2}$ then the equations
\begin{eqnarray*}
 A &\equiv& a_1 \pmod{n_1} \\
 A &\equiv& a_2 \pmod{n_2} \\
\end{eqnarray*}
have a unique solution $A \pmod \lcm(n_1 , n_2 )$.
If $\ndd{d}{a_1 - a_2}$ then the equations
have no solution.
\end{cor}

\bigskip\noindent
\begin{exa}
When a bit sequence is transmitted $a=(a_1 a_2  \ldots a_n )$, a
parity bit is computed and transmitted as well
where $p \equiv a_1 + \ldots + a_n \pmod 2$.
The {\bf even parity bit } e(a) os obtained by adding the
bits of $a$ and returning the value of the sum modulo two.
This sum is the number of ones in $a$; if it is an even
number $e(a)$ is 0 else it is 1.
\end{exa}

\begin{exa}
A 10-digit ISBN (International Standard Book Number) code
$a=(a_1 \ldots a_{10}$ where $a_{10}$ is a check digit.
The check digits is $\pmod 11$; an $X$ represents a 10.
The check digit computation involved is
\[
10a_1 +9 a_2 +8 a_3 +7 a_4 +6 a_5 +5 a_6 +
4 a_7 +3a_8 +2 a_9 +a_{10} 
\pmod{11}.
\]
If the checkdigit $a_{10}$ is valid this sum 
is equal to $0 \pmod 10$.
The weights can be an increasing left-to-right 
sequence as well.
\[
a_1 +2 a_2 +3 a_3 +4 a_4 +5 a_5 +6 a_6 +
7 a_7 +8a_8 +9 a_9 +10 a_{10} \pmod{11}.
\]
For a 13-digit ISBN $(a_1 \ldots a_{12}a_{13})$ a 
check digits is computed for 
$a_1 \ldots a_{12}$ where the weights are alternating 1 and 3s.
\[
a_1 +3 a_2 + a_3 +3 a_4 + a_5 +3 a_6 + a_7 +
3a_8 + a_9 +3 a_{10} +a_{11}+3a_{12}
\]
The sum $\pmod 10$ determines the check-digit after subtracting 
it from 10.
\end{exa}

\subsection{ CRT with two equations}

\begin{cor}[System of two equations]
\label{crt21}
Let $a_1 , a_2 \in Z$.
Let $n_1 , n_2 \in \mb{N}$ with $n_1 , n_2 > 1$,
Let $d$ be the g.c.d of $n_1 , n_2$ such that
$d=\gcd(n_1 , n_2 )$.
If $d=1$ and thus $n_1 , n_2$ are relatively prime,
then the system of two equations
\begin{eqnarray}
\label{cor7a}
\label{so2ecrt}
 x &\equiv& a_1 \pmod{n_1}\nonumber \\
 x &\equiv& a_2 \pmod{n_2}
\end{eqnarray}
has a unique solution $x \bmod (n_1 \cdot  n_2 ) $.
Moreover in $\mb{Z}$ any two solutions $x,y$ are such that
$x\equiv y \pmod{n_1 n_2}$.
\end{cor}

\begin{proof}
$ $ \\ $ $
Let $1=\gcd(n_1 , n_2 )$.
Then there exist
$m_1$ such that
$m_1 n_1 \equiv 1 \pmod{n_2}$,
and
$m_2$ such that
$m_2 n_2 \equiv 1 \pmod{n_1}$.
This follows from the fact that $\gcd(n_1 , n_2 )=1$,
and thefore
\[
m_1 n_1 -1 = n_2 k_2 , \quad
m_2 n_2 -1 = n_1 k_1
\]
have solution for $n_1, n_2$ by the extended GCD algorithm.
Consider then integer $X$ as follows
\[
 X = \left( a_1 m_1 n_1 + a_2 m_2 n_2 \right) \bmod{n_1 n_2} .
\]
It is obvious that
\[
X \pmod{n_1}\equiv 0 + a_2 m_2 n_2  \equiv a_2 \pmod{n_1},
\]
since $m_2 n_2 \equiv 1 \pmod n_1$.
Similarly,
\[
X \equiv a_1 \pmod{n_2}.
\]
It then suffices to show that there is no $Y \bmod (n_1 \cdot  n_2 )$
For if there is a $Y$ such that
\[
Y \equiv a_2 \pmod{n_1}, \quad
Y \equiv a_1 \pmod{n_2},
\]
then $\dv{n_1}{Y-X}$
and $\dv{n_2}{Y-X}$.
Since $\gcd(n_1 , n_2 )=1$, then
$\dv{n_1 n_2}{Y-X}$ which implies $n_1 n_2 \leq Y-X$ or
$X=Y$.
But $Y-X$ or $X-Y$ is between $0$ and $n_1 n_2 -1$
contradicting the conclusion $n_1 n_2 \leq Y-X$.
The only alternative left $X=Y$ or $X \equiv Y \pmod{n_1 n_2}$.
\end{proof}

\noindent
In the Chinese Remainder Theorem and
Corollary~\ref{crt21} we assume that $d=1$.
In Corollary~\ref{crt2d} to follow we allow  $d>1$.

\begin{cor}[System of two equations]
\label{crt2d}
Let $a_1 , a_2 \in Z$.
Let $n_1 , n_2 \in \mb{N}$ with $n_1 , n_2 > 1$,
Let $d$ be the g.c.d of $n_1 , n_2$ such that
$d=\gcd(n_1 , n_2 )$.
If $\dv{d}{a_1 - a_2}$ or equivalently
$a_1 \equiv a_2 \pmod{d}$,
then the system of two equations
\begin{eqnarray}
\label{cor7b}
\label{so2e}
 x &\equiv& a_1 \pmod{n_1}\nonumber \\
 x &\equiv& a_2 \pmod{n_2}
\end{eqnarray}
has a unique solution $x \pmod{\lcm(n_1 , n_2 )}$.
If $a_1 \not\equiv a_2 \pmod d$, or equivalently
$\ndv{d}{a_1 - a_2}$, then the system of equations
has no solution in $x$.
\end{cor}

\begin{proof}
$ $ \\ $ $
Let $d=\gcd(n_1 , n_2 )$.
Then $n_1 =d k_1 $ and $n_2 =d k_2$ for some
$k_1 , k_2 \in \mb{Z}$.
Let the system of Eq.(\ref{so2e}) has one solution $x$.
Then
$x \equiv a_1 \pmod n_1$
and
$x \equiv a_2 \pmod n_2$.
From the two we have for some
$l_1 , l_2 \in \mb{Z}$.
$ x- a_1 = n_1 l_1$ and
$ x- a_2 = n_2 l_2$.
From the former we obtain
\[
 x- a_1 = n_1 l_1 = (d k_1 ) l_1 = d (k_1 l_1 ),
\]
thus concluding $\dv{d}{x- a_1}$.
Similarly we obtain $\dv{d}{x- a_1}$.
Then $\dv{d}{a_1 - a_2}$ and therefore
\[
 a_1 \equiv a_2 \pmod d.
\]
Furthermore if
$a_1 \not\equiv a_2 \pmod d$, then
the system of Eq.(\ref{so2e}) has no solution in $x$,
since otherwise $ a_1 \equiv a_2 \pmod d$.
$ $ \\ $ $
Assume in the remainder that $ a_1 \equiv a_2 \pmod d$
and thus the system of Eq.(\ref{so2e}) has one
solution $x$.
Using the tools of the previous proof
\[
 x- a_1 = n_1 l_1  \wedge
%= (d k_1 ) l_1 = d (k_1 l_1 ) , \wedge
 x- a_2 = n_2 l_2
%= (d k_2 ) l_2 = d (k_2 l_2 )
\]
and by subtracting one from the other we obtain the following.
\[
a_1 - a_2 = n_2 l_2 - n_1 l_1 \Leftrightarrow
 n_1 l_1 = a_2 - a_1 + n_2 l_2  \Leftrightarrow
 n_1 l_1 \equiv a_2 - a_1  \pmod{n_2}
\]
The latter modular equation has a solution $l_1$
for $\gcd(n_1 , n_2) =d$ is such that $\dv{d}{a_2 -a_1}$
by way of $a_1 \equiv a_2 \pmod d$.
From the previous problem the solution $l_1$ is as follows.
\[
  l_1 \equiv  \frac{a_2 - a_1}{d}
              \left(\frac{n_1}{d}\right)^{\phi (\frac{n_2}{d}) -1 }
               \pmod{\frac{n_2}{d}}.
\]
We then have the following also using that $x - a_1 = n_1 l_1 $.
\begin{eqnarray*}
  l_1 &\equiv&  \frac{a_2 - a_1}{d}
              \left(\frac{n_1}{d}\right)^{\phi (\frac{n_2}{d}) -1 }
               \pmod{\frac{n_2}{d}} \Leftrightarrow \\
  l_1 &=&  \frac{a_2 - a_1}{d}
              \left(\frac{n_1}{d}\right)^{\phi (\frac{n_2}{d}) -1 }
               + m_2 \frac{n_2}{d} \Leftrightarrow \\
  x  &=& a_1 + n_1 \cdot  \frac{a_2 - a_1}{d}
              \left(\frac{n_1}{d}\right)^{\phi (\frac{n_2}{d}) -1 }
               + n_1 \cdot m_2 \frac{n_2}{d} \Leftrightarrow \\
  x  &=& a_1 + ( a_2 - a_1 )
              \left(\frac{n_1}{d}\right)^{\phi (\frac{n_2}{d}) }
               + m_2 \frac{n_1 \cdot n_2}{d} \Leftrightarrow \\
  x  &=& A +  m_2 \frac{n_1 \cdot n_2}{d} \Leftrightarrow \\
  x  &=& A +  m_2 \lcm(n_1 , n_2) \Leftrightarrow \\
  x  &\equiv& A  \pmod{\lcm(n_1 , n_2)} ,
\end{eqnarray*}
where $A$ is given by the following expression.
\[
A = a_1 + ( a_2 - a_1 )
              \left(\frac{n_1}{d}\right)^{\phi (\frac{n_2}{d}) }.
\]
\end{proof}

\begin{prp}
If for $m,m \in \mb{Z}$, $\gcd(m,n)=1$ then
\begin{equation}
\label{crtbyp}
a \equiv b \pmod{mn}
\Leftrightarrow
a \equiv b \pmod{m}
\quad
\wedge
\quad
a \equiv b \pmod{n}.
\end{equation}
\end{prp}
\begin{proof}
$\Rightarrow$.
If
$ a \equiv b \pmod{mn} $,
the there exists an $k$ such that $a-b=k mn$.
Since $m$ divides $kmn$ then $\dv{m}{a-b}$.
Likewise, $\dv{n}{a-b}$.
From the former we
have
$ a \equiv b \pmod{m}
$
and the latter
$ a \equiv b \pmod{n}$.
$ $ \\ $ $
$\Leftarrow$.
Let
$ a \equiv b \pmod{m}
\wedge
a \equiv b \pmod{n} $.
Then there exist $M,N \in \mb{Z}$ such that
$a-b = M m$ and $a-b = N n$ respectively.
Therefore $Mm=Nn$. Since $n$ divides $Nn$ it
should divide $Mm$ but because $\gcd(n,m)=1$ this
implies $\dv{n}{M}$. Therefore $M=l n$ for some
integer $l$. Equation $a-b=Mm= lnm = l(nm)$.
Therefore $\dv{nm}{a-b}$ and we conclude
$a-b \equiv 0 \pmod{mn}$
or equivalently,
$a \equiv b \pmod{mn}$.
$ $ \\ $ $
The result is also a by product of the chinese remainder
theorem.
\end{proof}

\newpage

\section{Modular system of polynomial equations}

\begin{thm}
Let $f(x)$ be a polynomial of degree $n$ of integer
coefficients. A polynomial congruence equation
is $f(x) \equiv 0 \pmod m$ for some integer $m$.
Let $n_1 , \ldots , n_k$ be  pairwise prime numbers
i.e. $\gcd(n_i , n_j )=1$ for $i\neq j$, $i,j = 1, \ldots , k$.
Let $N= n_1 n_2 \ldots n_k$.
Then integer $a$ is a solution of
\[
   f(x) \equiv 0 \pmod N ,
\]
if and only if $a$ is a solution of the system
of the following polynomial equations.
\begin{eqnarray}
\label{crt3}
f(a) &\equiv&  0  \pmod {n_1} \nonumber \\
f(a) &\equiv&  0  \pmod {n_2} \nonumber \\
     & \ldots & \\
f(a) &\equiv&  0  \pmod {n_k} \nonumber
\end{eqnarray}
Moreover if there are two solutions $a,b$, then
\[
 a \equiv b \pmod{n_1 \ldots n_k}
\Rightarrow
 a \equiv b \pmod{N},
\]
\end{thm}

\begin{proof}
$ $ \\ $ $
$\Rightarrow$.
If $a$ is a solution to the
congruence equation
\[
   f(x) \equiv 0 \pmod N ,
\]
then $f(a) \equiv 0 \pmod N$. Since $\dv{n_i}{N}$ for
all $i=1, \ldots , k$, the $f(a) \equiv 0 \pmod{n_i}$,
and ther result follows.
$ $ \\ $ $
$\Leftarrow$.
If $a$ is a solution of the sytem of equation~(\ref{crt3}),
then set $A=f(a)$
\begin{eqnarray}
\label{crt4}
A=f(a) &\equiv&  0  \pmod {n_1} \nonumber \\
A=f(a) &\equiv&  0  \pmod {n_2} \nonumber \\
     & \ldots & \\
A=f(a) &\equiv&  0  \pmod {n_k}, \nonumber
\end{eqnarray}
and by the way of Eq.(\ref{crt}) that is the
Chinese remainder theorem we obtain that
there is an $A=f(a)$, such that
\[
  A= f(a) \equiv 0 \pmod{N},
\]
as needed.
\end{proof}

\begin{prp}
The modular equation 
\[
x^2 \equiv a \pmod p
\]
where $p$ is a prime number has either 
0 solutions or two (congruence) solutions.
(Integer $a$ is such that $0 < a < p$.)
\end{prp}
\begin{proof}
If the equation has zero solutions, we are done.
Otherwise let it have one and let it be $x$.
Then we have the following
\[
 (p-x)^2 \equiv p^2 -2px + x^2 \equiv x^2 \equiv a \pmod p,
\]
and thus a second solution $p-x \neq x$ has been found.
We show that there are no more than them.
Let $x,y$ be two solutions.
We have the following.
\[
x^2 \equiv a \pmod p
, \quad
y^2 \equiv a \pmod p
\]
which implies
\[
x^2 - y^2 \equiv 0 \pmod p ,
\]
or equivalently $\dv{p}{x^2 -y^2}$ or
$\dv{p}{(x-y)(x+y)}$. For prime $p$ this
means $\dv{p}{x-y}$ or $\dv{p}{x+y}$ in
other words $x \equiv \pm y \pmod p$.
Our solution $p-x$ is such that
\[
p -x  \equiv - x \pmod p ,
\]
and is the complementary solution of $x$.
\end{proof}

\begin{dfn}[Quadratic residue]
For an $a$ such that
\[
x^2 \equiv a \pmod n
\]
has a solution we call $a$ a quadratic
residue mod $n$. If there is no
solution for $x$, then $a$ is called a
quadratic non-residue.
\end{dfn}

\noindent
In the previous proposition $n$ was a prime number and in 
fact we showed that if there is a solution then the number
of solutions is two; if $x$ is one the other is $n-x=p-x$.

\noindent
We use qr or qnr or q.r. and q.n.r. respectively.

\begin{prp}
Let $p$ be a prime number greather than 1 with
$p \equiv 3 \pmod 4$.
Then, for all $a,b$
\[
a^2 + b^2 \equiv 0 \pmod p \Longrightarrow
 a \equiv b \equiv 0 \pmod p.
\]
\end{prp}
\begin{proof}
We distinguish two cases: (a) $\dv{p}{a}$ and
(b) $\ndv{p}{a}$. We then show that the latter
is impossible, i.e. it is always the case
$\dv{p}{a}$.
$ $ \\ $ $
(a) Let $\dv{p}{a}$. Then $\dv{p}{a^2}$
and therefore $a^2 \equiv 0 \pmod p$ which
leads to $a \equiv 0 \pmod p$.
If $\dv{p}{a}$ given that $a^2 + b^2 \equiv 0 \pmod p$,
and from the former as before we obtain,
$\dv{p}{a^2}$ and then $a^2 \equiv 0 \pmod p$,
that leads to $b^2 \equiv 0 \pmod p$ i.e.
$b \equiv 0 \pmod p$.
$ $ \\ $ $
(b) Let $\ndv{p}{a}$.
From
\begin{eqnarray*}
a^2 + b^2 &\equiv& 0 \pmod p  \Leftrightarrow \\
a^2 b^{p-3} + b^2 b^{p-3} &\equiv& 0 \pmod p  \Leftrightarrow \\
a^2 b^{p-3} + b^{p-1} &\equiv& 0 \pmod p  \Leftrightarrow \\
a^2 b^{p-3} &\equiv& -b^{p-1} \pmod p    \Leftrightarrow \\
a^2 b^{p-3} &\equiv& -1 \pmod p  ,
\end{eqnarray*}
where in order to obtain the last step we used
Fermat's theorem.
Since $p \equiv 3 \pmod 4$, we have $p-3=4k$ for some
integer $k$ and $(p-3)/2 = 2k $ is an even integer.
For $A = a b^{(p-3)/2}$ we have
\[
a^2 b^{p-3} \equiv -1 \pmod p  \Leftrightarrow
A^2 \equiv -1 \pmod p          \Leftrightarrow
A^{2 \cdot \frac{(p-1)}{2}} \equiv (-1)^{\frac{(p-1)}{2}} \pmod p
                               \Leftrightarrow
A^{p-1} \equiv (-1) \pmod p
\]
where we derived the last expression by way of the fact
that $(p-1)/2 = 2k+1$ is an odd number.
The latter contradicts, by Fermat's theorem, the
\[
A^{p-1} \equiv 1 \pmod p  ,
\]
and therefore we can never have $\ndv{p}{a}$ of case (b).
\end{proof}

\section{Diophantine equations part ii}

\begin{prp}[Diophantine equation for $n > 2$]
Let $n \geq 2$. Show that the diophantine equation
\begin{equation}
\label{dioen}
 a_1  x_1 + a_2 x_2 + \ldots + a_n x_n = b
\end{equation}
has an integer solution
$(x_1 , x_2 , \ldots , x_n ) \in
\mb{Z} \times \ldots \times \mb{Z}=\mb{Z}^n$
provided that $( a_1 , a_2 , \ldots , a_n )=1$
that is, the gcd of $a_1, a_2, \ldots , a_n$ is equal to 1.
Moreover all ineger solutions can be expressed in terms of
$n-2$ integer parameters.
\end{prp}
\begin{proof}
The proof is by induction on $n$. The $n=1$ case is trivially true
as $(a_1 )=1$ implies $a_1 =1$ or $a_1 = -1$ and thus the integer
solution for $x_1$ is $x_1 = b $ or $x_1 = -b$.
$ $ \\ $ $
For $n \geq 2$ assume that it is true for $n-1$ and let $H(n-1)$
is the hypothesis and in the inductive step show that $H(n)$ is true.
Hypothesis (assumption) $H(n-1)$ is stated as follows.
\[
H(n-1):
 a_1  x_1 + a_2 x_2 + \ldots + a_{n-1} x_{n-1} = b
\]
has an integer solution
$(x_1 , x_2 , \ldots , x_{n-1} ) \in (\mb{Z})^{n-1}$
provided that $( a_1 , a_2 , \ldots , a_{n-1} )=1$,
and expressed in terms of $n-2$ integer parameters.
%such that $\dv{d}{b}$.
We are going to show $H(n)$ that is, Eq.(\ref{dioen})
for $n$ variables.
Let
\[
 a_1  x_1 + a_2 x_2 + \ldots + a_{n-1}x_{n-1} +  a_n x_n = b ,
\]
and by taking $\pmod{d}$ of both sides, we obtain the following.
\[
 a_1  x_1 + a_2 x_2 + \ldots + a_{n-1}x_{n-1} +  a_n x_n \equiv b \pmod d.
\]
By $H(n-1)$ we have that $d$ divides all $a_1 , \ldots a_{n-1}$ as their
common divisor. Therefore the following is obtained.
\[
  a_n x_n \equiv b \pmod d,
\]
We also have the following
\[
\gcd(d, a_n ) = (d, a_n) = (a_1 , \ldots , a_{n-1} , a_n )=1,
\]
by way of $( a_1 , a_2 , \ldots , a_n )=1$.
This implies by a prior problem (Diophantine equation $n=2$)
that $\gcd(d,a_n ) =1$ divides $b$ and thus there exists a solution
for $x_n$ of $ a_n x_n \equiv b \pmod d$.
Let this solution be
\[
a_n x_n \equiv b \pmod d  \Leftrightarrow
a_n A   \equiv b \pmod d  \Leftrightarrow
x_n     \equiv A \pmod d.
\]
The former one implies that there exist a $k \in \mb{Z}$ such
that $a_n x_n -b = k d$. This further implies that
$\dv{d}{a_n x_n -b}$, or $a_n x_n -b$ is a multiple of $d$
or equivalently $(a_n x_n -b)/d =k$ is an integer.
All solutions $x_n$ are such that $x_n     \equiv A \pmod d$
for $a_n A   \equiv b \pmod d$.
In equation Eq.(\ref{dioen}) we substitute as follows noting
$d \neq 0$.
\begin{eqnarray*}
 a_1  x_1 + a_2 x_2 + \ldots + a_{n-1}x_{n-1} +  a_n x_n &=& b \\
 a_1  x_1 + a_2 x_2 + \ldots + a_{n-1}x_{n-1}    &=& -d(a_n x_n -b)/d \\
 \frac{a_1}{d}  x_1 +
 \frac{a_2}{d}  x_2 + \ldots +
 \frac{a_{n-1}}{d} x_{n-1}    &=& -(a_n x_n -b)/d \\
    c_1         x_1 +
    c_2         x_2 + \ldots +
    c_{n-1}                  &=& -k ,
\end{eqnarray*}
where $c_i = a_i /d$.
We then have
\[
(c_1 , c_2 , \ldots c_{n-1} ) =
\frac{1}{d} \cdot (d c_1 , d c_2 , \ldots d c_{n-1} ) =
\frac{1}{d} \cdot (  a_1 ,   a_2 , \ldots   a_{n-1} ) =
\frac{1}{d} \cdot d = 1.
\]
By $H(n-1)$ the latter has integer solutions for $x_i$,
$i=1, \ldots , n-1$ expressed in terms of $n-2$ integer parameters.
Adding to it $x_n$ which is $x_n \equiv A \pmod d$ and thus there
is an $l \in \mb{Z}$ the $n-1$-st integer parameter
such that $x_n = A + d l$.
\end{proof}

\section{Wilson's theorem}

\begin{thm}[Wilson's theorem  onlyif]
\label{wilson1}
If $p$ is a prime number, then show that
\begin{equation}
(p-1)! \equiv -1 \pmod p.
\end{equation}
\end{thm}
\begin{proof}
$ $ \\ $ $
Consider $\{ 1, 2 , \ldots p-1 \}$.
No number $x$ among them is divisible by $p$ thus it is relatively
prime to $p$ and thus it has an inverse $\pmod p$.
$ $ \\ $ $
It is not possible that $x$ is its own inverse $x^{-1} \pmod p$
unless $x=1$ or $x \equiv p-1 \equiv -1 \pmod p $ from Theorem~\ref{thm47}.
Another way to say this is that if the inverse of $x$ for some
$x$ such that $1 \leq x < p$  is $x$ the
$x^2 \equiv 1 \pmod p$. Then $\dv{p}{(x-1)(x+1)}$. Given that $x<p$ this
is only possible for $x=1$ and $x=-1 \pmod p$ i.e. $x=p-1$.
$ $ \\ $ $
Thus for the remaining values $ x \in \{ 2, \ldots , p-2 \}$, we must have
$x \not\equiv x^{-1} \pmod p$. Every pair cancels each other
i.e. $x \cdot x^{-1} \equiv 1 \pmod p$.
Thus
\[
  2 \cdot 3 \cdot 4 \ldots (p-2) \equiv 1 \pmod p
\]
Restoring the missing $1$ and $p-1$  we have
\[
1\cdot  2 \cdot 3 \cdot 4 \ldots (p-2) \cdot (p-1) \equiv p-1 \equiv -1 \pmod p
\]
\end{proof}

\begin{thm}[Wilson's theorem]
\label{wilson2}
A natural number $p$ is a prime number,
if and only if
\begin{equation}
\label{thm59a}
(p-1)! \equiv -1 \pmod p.
\end{equation}
\end{thm}
\begin{proof}
$ $ \\ $ $
$\Leftarrow$.
If $p$ is not prime then $p=rs$ and $r,s<p$.
The term  $(p-1)!$ includes all integers $<p$ and thus $r$ and $s$.
This implies that $(p-1)! \equiv 0 \pmod p$.
There is one exception and that is $p= q^2 $.
For to have $q$ appearing in the
product twice it would mean that $q$ and $2q$ are part of the product,
i.e. $2q\leq p-1$. For this to be the case we need $p\geq 4$.
Thus we verify by hand exhaustively that for $p=2,3$ the Theorem is true.
$ $ \\ $ $
We can also claim that for  $p=q^2 > 2q > q$ the $(p-1)! \equiv 0 \pmod p$
is still the case. The latter is equivalent to $q^2 > 2q$ i.e. $q>2$.
For $q=2$ and thus $p=q^2 =4$ we have $(p-1)!= 1 \cdot 2 \cdot 3
\equiv 2 \pmod 4$ and thus $ (p-1)! \not\equiv -1 \pmod p$.
$ $ \\ $ $
$\Rightarrow$.
If $p$ is a prime number, then each $1,2, \ldots , p-1$ has an inverse,
and among them, 1 and $p-1$ are their own inverses.
\[
x \equiv x^{-1} \pmod p
\Rightarrow x \cdot x \equiv 1 \pmod p
\Rightarrow x \cdot x -1 \equiv 0 \pmod p
\Rightarrow (x-1)(x+1) \equiv 0 \pmod p
\]
Thus $\dv{p}{(x-1)(x+1)}$ and since $1 \leq x <p$
we have $\dv{p}{x-1}$ or $\dv{p}{x+1}$ or equivalently
 $x-1 \equiv 0 \pmod p$ or $x+1 \equiv 0 \pmod p$ or equivalently
 $x \equiv 1 \pmod p$ or $x \equiv -1 \pmod p$.
$ $ \\ $ $
Thus all the other integers invert  in pairs.
Thus $(p-1)! \equiv (p-1)\cdot (p-2) \cdot \ldots \cdot 2 \cdot 1
\equiv (p-1) \cdot 1 \equiv -1 \pmod p$.
\end{proof}

A shorter proof follows.

\begin{thm}[Wilson's Theorem]
\label{wilson3}
A natural number  $p$ is a prime prime 
if and only if $(p-1)! \equiv -1 \pmod p$.
\end{thm}
\begin{proof}
If $p$ is not prime then $p=rs$ and $r,s<p$. 
The term  $(p-1)!$ includes all integers $<p$ and thus $r$ and $s$.
This implies that $(p-1)! \equiv 0 \pmod p$.
There is one exception that $p= q^2 $. For to have $q$ appearing in the
product twice it would mean that $q$ and $2q$ are part of the product,
ie $2q\leq p-1$. For this to be the case we need $p\geq 4$.
Thus we verify by hand exhaustively that for $p=2,3$ the Thoerem is true.

If $p$ is a prime number, then $1,2, \ldots , p-1$ has a inverse,
and 1 and $p-1$ are their own inverses. Thus all the other integers
inverse in pairs. Thus $(p-1)! \equiv (p-1)\cdot (p-2) \cdot \ldots \cdot 2 \cdot 1
\equiv (p-1) \cdot 1 \equiv -1 \pmod p$.
\end{proof}

\newpage

\section{Fermat's several theorems}

As of now we have established that for a prime number $p$,
all integers $a$ with $1 \leq a \leq p-1$ are units 
that is, the modular equations  $ax \equiv 1 \pmod p$ has a 
solution for $x$.

\subsection{Fermat's (little) theorem}

\bigskip %THEOREM 46 %%
\begin{thm}[Fermat's  little  theorem]
\label{flt}
Let $p \in \mb{N}$ be a prime number, 
and let $a\in \mb{Z}$  with
$\gcd(a,p)=1$,
then the following applies.
\[
 a^{p-1} \equiv 1 \pmod p.
\]
\end{thm}
\begin{proof}
Consider $a, 2a, 3a, \ldots (p-1)a \pmod p$.
Since $\ndd{p}{a}$, all these values are non-zero
and distinct mod $p$.
This is because if $ i a \equiv j a \pmod p$, because
$\ndd{p}{a}$ it should be $\dv{p}{i-j}$. 
This means $p \leq |i-j|$.
But both $i,j$ are such that $0 \leq |i-j| \leq p-1$ and thus 
$p \leq p-1$ which is impossible!
The $p-1$ values $a, 2a , \ldots (p-1)a \pmod p$ can only be
the only $p-1$  available $\pmod p$ i.e.
$ 1, 2, 3, \ldots , p-1$ (possibly) rearranged.
Then taking their product one way or the other,
\[
 a\cdot 2a \cdot \ldots \cdot (p-1)a  \equiv
 a^{p-1} (p-1)!                       \equiv
 1\cdot  2 \cdot \ldots \cdot (p-1)   \equiv
  (p-1)!  \pmod p
\]
\[
 a^{p-1} (p-1)!   \equiv (p-1)!  \pmod p
\]
Since $p$ is relatively prime to $1,2,3, \ldots , (p-1)$ it is also
to $(p-1)!$. Thus
$ (a^{p-1}-1) (p-1)! \equiv 0 \pmod p$.
Thus it must be $a^{p-1} -1 \equiv 0 \pmod p$  or equivalently
 $a^{p-1}  \equiv 1 \pmod p$  as needed.
\end{proof}

The proof below is identical.
Note that if $p$ is a prime number, then $\gcd(a,p)= 1$
or $\gcd(a,p) =    p$ for any $a \in \mb{Z}$.
The latter is the case if $a$ is a multiple of p.
If $a$ is not a multiple of $p$ then $\ndv{p}{a}$
and  the only possibility is $\gcd(a,p)=1$.

\bigskip %THEOREM 46 %%
\begin{thm}[Fermat's  little  theorem restated]
\label{flt1}
Let $p \in \mb{N}$ be a prime number, 
and let $a\in \mb{Z}$  with
$\ndd{p}{a}$a,
then the following applies.
\[
 a^{p-1} \equiv 1 \pmod p.
\]
\end{thm}
\begin{proof}
$ $ \\ $ $
If $p$ is a prime number, then $\gcd(i,p)= 1 \vee p$.
The latter is the case if i is a multiple of p.
Say $a$ is not a multiple of $p$ and  $\gcd(a,p)=1$.
Consider
\[
 (1a)_p \cdot (2a)_p \cdot \ldots \cdot ((p-1)a)_p ,
\]
where
$(ia)_p != (0)_p$ and thus
$ia-0 != k p$.
This is so because otherwise $ia = k p$ would imply
that $p$ divides $i$ or $p$ divides $a$. The former
is not possible for prime $p$ since $0< i < p$,
and the latter by way of $ \gcd(a,p)=1$.
Furthermore $(ia)_p  \neq (ja)_p$ for $i \neq j$
since otherwise $(i-j)a \equiv 0 \pmod kp$ and
we can use the same consideration as before.
Therefore
\[
 (1a)_p ,     (2a)_p ,     \ldots ,     ((p-1)a)_p
\]
is a permutation of $1, 2, \ldots , p-1$ and
we obtain the following.
\[
 (1a)_p \cdot (2a)_p \cdot \ldots \cdot ((p-1)a)_p
 = ((p-1)! a^{p-1})_p = ((p-1)!)_p ,
\]
or equivalently,
\[
 ((p-1)! a^{p-1})_p = ((p-1)!)_p  \Leftrightarrow
 ((p-1)! a^{p-1}) \equiv (p-1)!   \Leftrightarrow
  a^{p-1} \equiv 1 \pmod p .
\]
The last one is obtained given $gcd(p,(p-1)!)=1$.
Completed.
\end{proof}

\bigskip %THEOREM 46 %%
\begin{thm}[Fermat's  little  theorem re-restated]
\label{flt2}
Let $p \in \mb{N}$ be a prime number. 
Let $f(x) = x^{p-1} -1$.
Then
\[
   f(x) \equiv 0 \pmod p ,
\]
has exactly $p-1$ distinct roots in $(\mb{Z}/p\mb{Z})^*$.
For a prime number $p$ the  latter set is 
$\{ 1, 2, \ldots , p-1 \}$.
\end{thm}

A direct consequence is the following Corollary.
\medskip
\begin{cor}
For $a\in \mb{Z}$, we have $a^p \equiv a \pmod p$.
\end{cor}

Given from Fermat's  theorem that 
$a^{p-1} \equiv a^{p-2} a$ we conclude the following.
\medskip
\begin{cor}
For $a\in \mb{Z}$ such that $\ndd{p}{a}$, 
we have $a^{-1} \equiv a^{p-2} \pmod p$.
\end{cor}

\begin{exa}
Consider the integers $\pmod 8$.
We have $3^2 \equiv 1 \pmod 8$.
We also have $5^2 \equiv 1 \pmod 8$.
There are two square roots of  $1 \pmod 8$.
Can you find others (e.g. 1, 7)?
Moreover $4 \equiv -4 \pmod 8$. 
\end{exa}

We can generalize the last observation as follows.

\begin{exa}
$a \equiv -a \pmod n$ is equivalent to
$2a \equiv 0 \pmod n$. 
If $n=2a$ this is trivially true. 
If $n$ is odd, then $\gcd(2,n)=1$ 
and thus $\dv{n}{a}$.
In the latter case $a \equiv 0 \pmod n$.
\end{exa}

\bigskip %THEOREM 47 %%
\begin{thm}
\label{thm47}
If $p$ is an odd prime ($p\neq 2$) and
$\ndd{p}{a}$ then the equation
\[
x^2 \equiv a \pmod p
\]
has either exactly two distinct roots or no roots at all.
\end{thm}

\begin{proof}
If there are roots to the modular equation the proof is
complete. Otherwise let $z$ be a solution i.e.
$z^2 \equiv a \pmod p$. 
Since $-z$ is such that
$(-z)^2 \equiv z^2 \equiv a \pmod p$, then
$-z$ is also a solution.
Is it $z \equiv -z \pmod p$? 
This is so if $p$ is an even number as it was
shown prior to the statement of Theorem~\ref{thm47}.
It is also possible that $p$ is odd but then it must
divide $a$. However the preconditions of the theorem
disallow the former (even number cannot be the case
as $p$ is an odd prime) and the latter (even number
and $\dv{p}{a}$ is not possible, since $p$ is
odd and $\ndd{p}{a}$).

Therefore there two distinct solutions $z, -z \pmod p$ if one
of them (say $z$) exists.
Does there exist a third (or fourth etc) solution?
Let us call it $w$. Then $w^2 \equiv z^2 \equiv a \pmod p$.
This implies $w^2 - z^2 \equiv 0 \pmod p$.
Then $(w-z)(w+z)\equiv 0 \pmod p$.
Then $\dv{p}{w-z}$ or $\dv{p}{w+z}$.
In other words $ w\equiv z$ or $w \equiv -z$.
There are two and only two solutions then. No third or
more!
\end{proof}

\subsection{Fermat's  theorem}

\begin{thm}[\bf{Fermat's theorem}]
For any prime number $p$ and any $a \in \mb{Z}$ we have
\begin{equation}
\label{ft1a}
 a^p \equiv a \pmod p .
\end{equation}
If in addition $\gcd(a,p)=1$, then
\begin{equation}
\label{ft1b}
a^{p-1} \equiv 1 \pmod p.
\end{equation}
\end{thm}

The second part condition $\gcd(a,p)=1$ of Eq.(\ref{ft1b})
is equivalent to $\ndv{p}{a}$ that eventually leads to
Fermat's Little theorem of Theorem~\ref{flt}.
\begin{proof}
$ $ \\ $ $
Let $a \geq 0$. We prove the first part by induction on $a$.
$ $ \\ $ $
Base case $a=0$. Then $a^p = 0^p = 0 = a$ and thus
$ a^p \equiv a \pmod p$ by default.
$ $ \\ $ $
We now assume that for $a \geq 1$ we have
\[
 (a-1)^p \equiv a-1 \pmod p
\]
We then have by utilizing the Binomial theorem, and the
induction hypothesis above the following.
\[
a^p = ((a-1)+1)^p \equiv (a-1)^p  + 1^p \equiv (a-1) +1
\equiv a \pmod p.
\]
$ $ \\ $ $
Now from $ a^p \equiv a \pmod p$
we have that there exists a $k \in \mb{Z}$ such that
\[
 a^p - a = k p \Leftrightarrow
a  ( a^{p-1} -1 )  = k p ,
\]
and given $\dv{p}{kp}$ we have
$\dv{p}{a  ( a^{p-1} -1 )}$. Since $\gcd(a,p)=1$,
then $\dv{p}{a^{p-1}-1}$ and therefore
\[
a^{p-1} \equiv 1 \pmod p.
\]
\end{proof}

\subsection{Some interesting results}

\begin{cor}[Primality testing]
Let $p \in \mb{Z}_+$. Show that $p$ is a prime
number if and only is $a^{p-1} \equiv 1 \pmod p$ for
every $a \neq 0$, and $a=1, \ldots , p-1$.
(We can  augment $a$ to be any integer $a \in \mb{Z}_+$
such that $\ndv{p}{a}$.)
\end{cor}
\begin{proof}
If $p$ is a prime number this follows from Fermat's
little theorem.
For the other direction, if the modular equation
is true for all $a$, and $p$ is not a prime number,
let $q$ be a divisor of $p$ that is $p=q k$ for
some integer $k$, and $q > 1$. Then $a^{p-1} \equiv 1 \pmod p$,
implies also $q^{p-1} \equiv 1 \pmod p$, and
since $\dv{q}{p}$ we also have
$q^{p-1} \equiv 1 \pmod q$. From the latter,
$\dv{q}{q^{p-1}-1}$ we have $\dv{q}{1}$ or
$q \leq 1$ which contradicts $q>1$. Thus
$q$ does not exist and $p$ is a prime number, as
needed.
\end{proof}

\begin{prp}
If $a,b \in \mb{Z}$
such that $a \equiv b \pmod{p^n}$ then
$a^p = b^p \pmod{p^{n+1}}$.
\end{prp}
\begin{proof}
Since $ a \equiv b \pmod{p^n}$ we have
$a-b = k p^n$ and thus $p$ divides $a-b$ and in fact any
$p^i$, $i \leq n$ divides $a-b$. Thus $a \equiv b \pmod p$.
Moreover $ a^j (a-b) = a^{j+1} - a^j b$ and thus
$a^{j+1} \equiv a^j b \pmod{p}$ and  also
$a^{p-1} \equiv a^j b^{p-1-j} \pmod{p}$.
Consider
\[
a^p - b^p = (a-b) (a^{p-1} + a^{p-2} b + \ldots + b^{p-1} ) .
\]
Moreover $a^{p-i}b \equiv a^{p-1} \pmod p$ and thus
\[
(a^{p-1} + a^{p-2} b + \ldots + b^{p-1} ) \equiv p a^{p-1}
\equiv 0 \pmod p.
\]
Then from $a \equiv b \pmod{p^n}$  we have
that $\dv{p^n}{a-b}$ and
$\dv{p}{(a^{p-1} + a^{p-2} b + \ldots + b^{p-1} )}$,
i.e. $\dv{p^{n+1}}{a^p - b^p}$.
\end{proof}

\begin{prp}
Ffor any prime number $p \in \mb{N}$
and any $a \in \mb{Z}$ with $\gcd(a,p)=1$ the following
applies.
\begin{equation}
\label{ft2}
  a^{(p-1) p^{n-1}} \equiv 1  \pmod{p^n} .
\end{equation}
\end{prp}
\begin{proof}
Use induction of $n \geq 1$.
For $n=1$ this follows from the previous problem.
Let by the induction hypothesis
\[
  a^{(p-1) p^{n-2}} \equiv 1  \pmod{p^{n-1}} .
\]
Then by part (a) we have
\[
a^{((p-1) p^{n-2})^{p}} \equiv 1^{p} \pmod{p^n}
\]
and then
\[
a^{(p-1) p^{n-1}} \equiv 1 \pmod{p^n} .
\]
\end{proof}

\begin{exa}
Is 511 a prime number?
\end{exa}
\begin{solution}
511 is not a prime number. We show that if
511 was a prime number by Fermat's Little Theorem,
for every $a$ such that $\gcd(a,p)=1$ we would have
have $a^{p-1} \equiv 1 \pmod{p}$ or
equivalently
$a^{510} \equiv 1 \pmod{511}$.
$ $ \\ $ $
We observe
\[
  2^{9} = 512 \equiv 1 \pmod{511}.
\]
Furthermore $2^6 \equiv 64 \pmod{511}$.
Moreover    $510 = 56 \cdot 9 + 6$.
Therefore
\[
2^{510} = 2^{56 \cdot 9 + 6} = (2^{9})^{56} \cdot 2^6
\equiv  1^{56} \cdot 64 \equiv 64 \pmod{511}.
\]
Give that $\gcd(2,511)=1$ this mean $511$ is not a prime
number.
\end{solution}

%HERE  Stopped at 2:01pm 6/12

\section{Euler's Theorem }

The units $\bmod p$ are all those
integers $a$ that they have an inverse $a^{-1} \pmod p$.
An integer $1\leq a <p$ is a unit if $\gcd(a,p) =1$.

\begin{dfn}[Set of units of $\mb{Z}_n$]
The set of units $\mb{U}_n$ of $\mb{Z}_n$, for $n >1$, 
is the set of units $\pmod n$, that is
the integers between 1 and $n-1$ that are 
relatively prime to $n$. 
\end{dfn}

\begin{exa}[$\mb{U}_3 , _4 , \ldots$]
Therefore $\mb{U}_3 = \{ 1,2,3 \}$;
$\mb{U}_4 = \{ 1,3 \}$, and
$\mb{U}_5 = \{ 1,2,3,4 \}$
and finally  $\mb{U}_6 = \{ 1,5 \}$.
\end{exa}

\bigskip %THEOREM 49 %%
\begin{thm}
For $a,b \in \mb{U}_n$ we have that 
$ab \in \mb{U}_n$ and also $a^{-1} \in \mb{U}_n$.
\end{thm}
\begin{proof}
Starting with the last result if $a \in \mb{U}_n$ it means that
$a a^{-1} \equiv 1 \pmod n$.
Moreover
$a^{-1} {a^{-1}}^{-1} \equiv 1 \pmod n$.
Thus $a^{-1} \in \mb{U}_n$.
For $a,b$ let their inverse be $a^{-1}  , b^{-1}$ respectively.
Consider $(ab)$. Since 
\[
(ab)(b^{-1}a^{-1} ) \equiv a \cdot 1 \cdot a^{-1}
\equiv 1 \pmod n
\]
it shows that $ ab \in \mb{U}_n$.
\end{proof}

Euler's totient function $\phi (n)$ denotes the
cardinality of $\mb{U}_n$ that is 
$\phi (n) = |\mb{U}_n| $, 
that is the number of units $\bmod n$.
Euler's theorem is an extension of Fermat's 
Little Theorem  where
the restriction of $n$ being a prime number 
has been relaxed.

\bigskip %THEOREM 50 %%
\begin{thm}[Euler's theorem]
\label{euler}
\label{thm50}
For any $n \in \mb{N}$, with $n>1$, and
for any $a \in Z$, $a>1$,
if $\gcd(a,n)=1$ then $a^{\phi(n)} \equiv 1 \pmod n$.
\end{thm}

\begin{proof}
$ $ \\ $ $
{\bf First proof.}
Let 
\[
a_1 , \ldots a_{\phi(n)} \pmod n 
\]
be the list of units.  Multiply each one of the elements
with $a$, where $\gcd(a,n)=1$. 
Since it is so $ax\equiv 1 \pmod n$.
More over $a_i$ are units thus $a_i x_i \equiv 1 \pmod n$ as well.
The $a a_i$ are such that 
$(aa_i)(xx_i) \equiv (ax)(a_i x_i) \equiv 1 \pmod n$.
That is all $aa_i$ are units.
\[
a a_1 , a a_2 , \ldots , a a_{\phi(n)}  \pmod n.
\]
Moreover $aa_i \not\equiv a a_j \pmod n$. 
This is because if $aa_i \equiv a a_j \pmod n$  would
implicate $xaa_i \equiv xaa_j \pmod n$ i.e. $a_i \equiv a_j \pmod n$.
If all of them are in the integer interval $[1,n-1]$, then this
implie $a_i = a_j$, i.e. $i=j$.
That is the two lists above are
the same up to a reordering of the same elements.
Thus
\[
a a_1  a a_2 , \ldots , a a_{\phi(n)}   \equiv a_1 \ldots a_{\phi(n)} \pmod n 
\]
This means
\[
a^{\phi(n)} a_1   a_2 \ldots  a_{\phi(n)} \equiv a_1 a_2 \ldots a_{\phi(n)}\pmod n
\]
Using cancellation (or multiplication with 
 $a_1^{-1} a_2^{-1} \ldots$) we once more  conclude that
$a^{\phi(n)} \equiv 1 \pmod n$.
$ $ \\ $ $
{\bf Second proof.}
A second slightly easier proof can be obtained by way of
Eq.(\ref{ft2}).
$ $ \\ $ $
Let $n=
  p_1^{n_1}
  p_2^{n_2}
  \ldots
  p_r^{n_r}
$
where $n_i \geq 0$ and $r>0$.
By way of Eq.(\ref{ft2}) we have the following for all $i=1, \ldots , r$,
for an $a \in \mb{N}$ and $\gcd(a,n)=1$.
\begin{equation}
\label{euler1a}
  a^{p_i^{n-1}(p_i -1)} \equiv 1 \pmod{p_i^{n_i}}.
\end{equation}
Raising equation Eq.(\ref{euler1a}) to the power
$\phi (n / p_i^{n_i})$ we obtain the following.
\begin{equation}
\label{euler1b}
  \left( a^{p_i^{n-1}(p_i -1)}\right)^{\phi (n / p_i^{n_i})}
\equiv 1 \pmod{p_i^{n_i}} \Leftrightarrow
   a^{\phi (n)} \equiv 1 \pmod{p_i^{n_i}}.
\end{equation}
Given that $p_i$ are relative prime to each other
the last one implies the following,
for $n =p_1^{n_1} p_2^{n_2} \ldots p_r^{n_r}$.
\begin{equation}
   a^{\phi (n)} \equiv 1 \pmod{n}.
\end{equation}
\end{proof}

\begin{cor}
If $p$ is prime $\phi(p)=p-1$ and $\mb{Z}_p = \{ 1, \ldots , p-1 \}$ and
Euler's theorem becomes Fermat's theorem.
\end{cor}

\chapter{Primitive roots mod $n$}

\section{Order of $a$ mod $n$}

$\mb{Z}$ is the set of integers (positive, negative or zero).
For a positive integer $n>0$ the set
$\mb{Z}_n$ sometimes denoted by $\mb{Z}/n\mb{Z}$ or $\mb{Z}/n$
is the set of integers modulo $n$, thus representing the $n$
equivalence classes that the integers of $\mb{Z}$ can be
split into depending on the remainder of their division by $n$.
$\mb{Z}_n$ is a cyclic group under addition, and a commutative
ring under multiplication and addition.
The ring is a field for a prime $n$.

From Euler's formula for $a \in \mb{Z}$ and
$\gcd(a,n)=1$, we have that
$a^{\phi(n)}\equiv 1 \pmod n$.
Is it possible that $a^k \equiv 1 \pmod n$ for smaller
$k < \phi (n)$?
We have already mentioned that
$(-1)^2 \equiv (n-1)^2 \equiv 1 \pmod n$.

\noindent
\begin{dfn}[Order $\bmod n$]
For $n \in \mb{N}$ and $n>1$,
and $a \in \mb{Z}_n$ with $\gcd(a,n)=1$,
 we define the order $\ord_n (a)$ to be the smallest
positive integer $k$ such that $a^k \equiv 1 \pmod n$.
\end{dfn}

By Euler's theorem $a^{\phi(n)} \equiv 1 \pmod n$ implies
that $\dv{k}{\phi(n)}$.
(If this was not the case $\phi(n) = k q +r $, $0 \leq r < k$
would imply $a^r \equiv 1 \pmod n$ for an $r < k$, an
impossibility given that $\ord_n (a)=k$.)
An element $a$ as defined above that has order
$\phi(n)$ is called a
primitive root mod $n$.

\bigskip %THEOREM 52 %%
\begin{thm}
\label{thm52}
Let $k=\ord_n (a)$ be as defined earlier. For all $m$ we have
$a^m \equiv 1 \pmod n$ if and only if $\dv{k}{m}$.
\end{thm}
\begin{proof}
$ $ \\ $ $
$\Rightarrow$.
$ $ \\ $ $
If $\dv{k}{m}$ we have $m=ks$ for some integer $s$.
Then
\[
a^{m} \equiv (a^{k})^s \equiv 1 \pmod n .
\]
$ $ \\ $ $
$\Leftarrow$.
$ $ \\ $ $
For the other way if $a^m \equiv 1 \pmod n$, Let $m=kq+r$,
where $0 \leq r < k$.  If $r=0$ we are done since then
$\dv{k}{m}$. Consider that $r \neq 0$.
Since $a^k \equiv 1 \pmod n$ and thus
$(a^k)^q \equiv 1 \pmod n$ and also
$a^m \equiv 1 \pmod n$, we have
\[
1 \equiv a^m \equiv a^{kq+r} 
  \equiv (a^k)^q \cdot a^r \equiv a^r \pmod n.
\]
For $a^r$ to be $\equiv 1 \pmod n$ given 
that $0< r <k$ is impossible
since $k$ is the smallest index for which this is true.
The only possibility is that $r=0$ that was put aside
and the result follows by  accepting the dismissed case $r=0$
as the only possibility.
\end{proof}

\begin{cor}
\label{cor52}
Let $k =\ord_n (a)$ be as defined earlier.
Then $a^{\phi(n)} \equiv 1 \pmod n$ implies that $\dv{k}{\phi(n)}$.
\end{cor}

\begin{cor}
Let $k =\ord_p (a)$, $a \in \mb{Z}$, where $p$ is a prime number.
Then $a^{\phi(p)} \equiv 1 \pmod p$ implies that $\dv{k}{p-1}$.
\end{cor}

\begin{prp}
\label{prp2}
If $k= \ord_n (a)$ be as defined earlier, then 
$a^m$ has order $k$ if and only if $\gcd(m,k)=1$.
\end{prp}

\begin{proof}
Let $d=\gcd(m,k)$ and let $q= \ord_n (a^m )$. 
$ $ \\ $ $ 
We need to show that $q=k$
if and only if $d=1$.
$ $ \\ $ $
$\Rightarrow$.
$ $ \\ $ $
Let $q=k$. We will show that $d=1$. Let 
$d>1$. Then $k=d x$ and $m=d y $,
for some integer $x,y$. Note that 
$x < k$ and $y<m$. Then we have
\[
  ( a^m)^x \equiv a^{mx} \equiv a^{dyx} 
           \equiv (a^{k})^y \equiv 1 \pmod n
\]
Therefore the order $q$ of $a^m$ is $x$, where 
$x< k$, is smaller than the order $q=k$ of $a^m$. 
This is impossible by the minimality of
$q=k$, $q= \ord_n (a^m )$.
Thus it should be $x=k=q$ which implies $d=1$. Result shown.
$ $ \\ $ $
$\Leftarrow$.
$ $ \\  $ $
If $d=1=\gcd(m,k)$ we show that for 
$q= \ord_n (a^m )$ we have $q=k$. 
Since $q=\ord_n (a^m )$ we have
\[
 a^{mq} \equiv (a^m)^q \equiv 1 \pmod n
\]
This means that $\dv{k}{mq}$ from Theorem~\ref{thm51}.
Since $d=\gcd(k,m)=1$, we have $\dv{k}{q}$. 
Thus $k\leq q$.  We also have 
\[
(a^m)^k \equiv (a^k)^m \equiv 1^m \equiv  \pmod n
\]
Thus $\dv{q}{k}$ i.e. $q\leq k$. 
From $k\leq q$ and $q\leq k$ the result $k=q$ follows.
\end{proof}

\begin{cor}
\label{cor17}
Let $n \in \mb{N}$ with $n>1$ and let $a,b \in \mb{Z}$
with $\gcd(a,n)=1$ and $\gcd(b,n)=1$,
For $\ord_n (a) =k$ and $\ord_n (b) = l$ 
if $\gcd(k,l)=1$ then $\ord_n(ab)= kl$.
\end{cor}

\begin{proof}
Let $\ord_n(ab)=m$. Then
\[
(ab)^{kl} \equiv (a^{k})^l (b^l)^k \equiv 1 \pmod n
\]
Therefore $\dv{m}{kl}$. Moreover
\[
1 \equiv ((ab)^{m})^k = (a^k)^m  (b^{km}) \equiv b^{km} \pmod n
\]
This means $\dv{l}{km}$. Since $\gcd(k,l)=1$ we have $\dv{l}{m}$.
Likewise, if we consider $((ab)^{m})^l$ instead, we conclude
$\dv{k}{m}$ instead. Since $\gcd(k,l)=1$ we have 
$\dv{kl}{m}$. This implies $kl \leq m$. But since $m$ is the
order of $ab$ we must also have, by the first derivation
above, $kl\geq m$. Thus $kl=m$.
\end{proof}

\begin{lem}
\label{lem52a}
For $n \in \mb{N}$, and $a \in \mb{Z}$ such
that $\gcd(a,n)=1$. Let $ord_n (a) = k$,
Then for all $b,c \in \mb{N}$ and $a$
the following applies.
\begin{equation}
\label{thm52a}
a^b \equiv a^c \pmod n  \Longleftrightarrow
b   \equiv c   \pmod k .
\end{equation}
\end{lem}
\begin{proof}
$\Leftarrow$.
If $b<c$ we swap, and thus
in the remainder $b \geq c$.
If $b \equiv c \pmod n$ then there exists an $l$
such that $b-c=k l$. Since $b\geq c$, then $l \geq 0$.
We then obtain the following.
\[
a^b = a^{kl+c} = a^c (a^k )^l \equiv a^c \cdot 1 \pmod n.
\]
This implies the following.
\[
a^b  \equiv a^c  \pmod n ,
\]
and the result is proven.
$ $ \\ $ $
$\Rightarrow$.
Let $
a^b  \equiv a^c  \pmod n $.
Then the following are applicable.
\begin{eqnarray*}
a^b              \equiv a^c  \pmod n  &\Leftrightarrow&
a^b - a^c        \equiv   0  \pmod n  \Leftrightarrow \\
(a^{b-c} -1) a^c \equiv   0  \pmod n  &\Leftrightarrow&
(a^{b-c} -1)     \equiv   0  \pmod n  \Leftrightarrow \\
 a^{b-c}         \equiv   1  \pmod n  ,
\end{eqnarray*}
where the last two equivalences is due to the fact
$\gcd(a,n)=1$ and thus $\ndv{n}{a^c}$.
Consider $b-c$. Then $b-c = k l + m$, where $0 \leq m < k$.
Then the following are applicable.
\begin{eqnarray*}
 a^{b-c}         \equiv   1  \pmod n   &\Leftrightarrow&
 a^{kl+m}        \equiv   1  \pmod n   \Leftrightarrow \\
 a^{kl} a^m      \equiv   1  \pmod n   &\Leftrightarrow&
 (a^{k})^l a^m   \equiv   1  \pmod n   \Leftrightarrow \\
           a^m   \equiv   1  \pmod n  .
\end{eqnarray*}
The last inequality derives from the fact $k$ is the
smallest positive integer $a^k \equiv 1 \pmod n$.
Note that $m$ can be $0$ or $1, \ldots , k-1$.
It cannot be $1, \ldots , k-1$ because the smallest
$m$ with $ a^m   \equiv   1  \pmod n$ that is a positive
integer is $k$ and nothing smaller such as $m$.
Thus the only other possibility is for $m=0$.
The $b-c =k l + 0 = kl$ therefore $\dv{k}{b-c}$ i.e
\[
  b -c \equiv 0 \pmod k  \Leftrightarrow
  b    \equiv c \pmod k,
\]
as needed. This completes the proof.
\end{proof}

\begin{exa}
For $n \in \mb{N}$, and $a \in \mb{Z}$ such
that $\gcd(a,n)=1$ and $ord_n (a) = k$,
the following applies.
This is Theorem~\ref{thm52}.
\begin{equation}
\label{thm52b}
 a^m \equiv 1 \pmod n  \Longleftrightarrow
 \dv{k}{m}.
\end{equation}
\end{exa}
\begin{solution}
 The choice of $k$ is such that $a^k \equiv 1 \pmod n$
and this is the smallest positive integer.
Consider an $m$ such that
$a^m \equiv 1 \pmod n$. Then
$1\equiv a^m \equiv a^k \pmod n$ and by the previous 
Lemma
we have
$m \equiv k \pmod k$. This implies $m-k \equiv 0 \pmod k$
and therefore $\dv{k}{m}$.
%$ $ \\ $ $
%(d) Immediate from part (c).
%$ $ \\ $ $
%{\bf Alternative proof for (c).}
%Let $b=k k_1 +r $, where $0 \leq r < k$.
%Then
%\[
%a^b \equiv a^{kk_1 +r} \equiv (a^k )^{k_1} a^r
%\equiv a^r \pmod n .
%\]
%Note that in $ a^b \equiv a^r \pmod n$,
%we have $r< k$.
%If $a^b \equiv 1 \pmod n$ then
%   $a^r \equiv 1 \pmod n$,
%and we have found an $0< r< k$, an impossibility to the
%positive minimality of $k$.
%Otherwise $r=0$ but then $b=k k_1$ which implies
%$\dv{k}{b}$ as needed and we are done.
\end{solution}

\begin{exa}
If $n \in \mb{N}$, and $a \in \mb{Z}$ such
that $\gcd(a,n)=1$ and $ord_n (a) = k$, 
then
 $\dv{k}{\phi (n)}$ or in other
words $\dv{ord_n(a)}{\phi(n)}$.
\end{exa}
\begin{solution}
%It follows directly from Theorem~\ref{thm52}
%and Lemma~\ref{lem52a}.
%Since by Euler's theorem
%\[
% a^{\phi (n)} \equiv 1 \pmod n,
%\]
%part (c), with $b=\phi (n)$
%concludes that $\dv{k}{\phi (n)}$.
This is Corollary~\ref{cor52}.
$ $ \\ $ $
A more direct alterantive approach follows.
Consider minimum $k=ord_n (a)$ such that
\[
 a^k \equiv 1 \pmod n ,
\]
\[
 a^{\phi(n)} \equiv 1 \pmod n.
\]
We have $k \leq \phi(n)$.
Then let $\phi (n) = A k + r $,
where $0 \leq r < k = ord_n (a)$.
We have the following
\[
a^{\phi(n)} \equiv 1 \pmod n  \Leftrightarrow
a^{Ak+r}    \equiv 1 \pmod n  \Leftrightarrow
(a^k)^A a^r \equiv 1 \pmod n  \Leftrightarrow
        a^r \equiv 1 \pmod n ,
\]
and the last one if $r>0$ then  
$r< k = ord_n (a)$, a contradiction
to the minimality of $k$. Thus $r=0$ is the only possibility
leading to $\phi (n) = A k $ and thus $\dv{ord_n(a)}{\phi(n)}$.
$ $ \\ $ $
\end{solution}

Below $p$ is a prime number $p \in \mb{N}$, 
and $a \in \mb{U}_p$ implies
$\gcd(a,p)=1$.

\begin{prp}
\label{prp3}
Let $a \in \mb{U}_p $,
where $p$ is an odd prime number.
Then we have the following
\begin{equation}
\label{ordpsquared}
ord_{p^2} (a) = ord_p (a)
\quad \vee \quad
ord_{p^2} (a) = p \cdot ord_p (a) .
\end{equation}
\end{prp}
\begin{proof}
Let
\begin{eqnarray}
\label{pp2a}
ord_p (a) = k &\Rightarrow&  a^k \equiv 1 \pmod{p}
\quad , \quad \\
\label{pp2b}
ord_{p^2} (a) =l   &\Rightarrow&  a^l \equiv 1 \mod{p^2} .
\end{eqnarray}
By way of Eq.(\ref{pp2a}) we have that there exists $A$
such that
\begin{eqnarray}
a^k - 1 = Ap &\Rightarrow& a^{kp} = (1+Ap)^p \equiv 1 \pmod{p^2}
\nonumber \\
             &\Rightarrow& \dv{ord_{p^2} (a)}{kp}
\nonumber \\
\label{pp2c}
             &\Rightarrow&  \dv{l}{kp} .
\end{eqnarray}
Furthermore,
\begin{eqnarray}
a^l - 1 = B p^2 &\Rightarrow& a^l \equiv 1 \pmod p
\nonumber \\
             &\Rightarrow& \dv{ord_{p} (a)}{l}
\nonumber \\
\label{pp2d}
             &\Rightarrow&  \dv{k}{l} .
\end{eqnarray}
By way of Eq.(\ref{pp2d}) we have
$l=K k $ for some integer $K$.
By way of Eq.(\ref{pp2c}) we have
\[
\dv{l}{kp}
\Rightarrow
\dv{Kk}{kp}
\Rightarrow
\dv{K}{p}
\Rightarrow
K=1 \quad \vee K=p,
\]
given that $p$ is a prime number.
Since $l=K k$ this leads to either $l=k$ or $l=pk$
and the result is proven.
\end{proof}

\begin{prp}
\label{prp4}
Let $p$ be a prime number.
Let $q$ be another prime number such that
$\dv{q^r}{p-1}$ for some positive integer $r$.
Then there exists
a $g \in \mb{U}_p$ such that
\begin{equation}
\label{ordgpqr}
 ord_p (g) = q^r .
\end{equation}
\end{prp}
\begin{proof}
Let us assume that there exists a $g$ such that
\begin{eqnarray}
\label{g1aa}
g^{q^r} \equiv 1 \pmod p .
\end{eqnarray}
Then
\[
 \dv{ord_p (g)}{q^r} .
\]
Then,
let $ord_p(g) = q^k$ for some $0 < k \leq r$.
We obtain the following.
\[
g^{q^k} \equiv 1 \pmod p
\Leftrightarrow
g^{q^j} \equiv 1 \pmod p , \quad \text{for} \quad j \geq k.
\]
The latter is because
\[
g^{q^j} \equiv g^{q^{j-k+k}}
        \equiv (g^{qk})^{q^{j-k}}
        \equiv (1)^{q^{j-k}} \equiv 1 \pmod p.
\]
If $ord_p(g) = q^k$ for some $0 < k  <   r$
(note that it is  $k < r$ not $k \leq r$),
then
\[
ord_p(g) = q^k \Rightarrow
 g^{q^{r-1}} \equiv 1 \pmod p  \quad \text{for} k < r.
\]
We address now the question.
\begin{que}
\label{g1q}
Does there exist a $g$ such that
\begin{eqnarray}
\label{g1b1}
g^{q^r} &\equiv& 1 \pmod p  \\
\label{g1b2}
g^{q^{r-1}} &\not\equiv& 1 \pmod p
\end{eqnarray}
\end{que}
Equation (\ref{g1b1}) is equivalent
to
$ g^{q^r} -1 \equiv 0 \pmod p  $.
Thus $g$ is a root of the
$ x^{q^r} -1 \equiv 0 \pmod p  $ and by
the previous Corollary to Lagrange's theorem
we know that there are $q^r$ roots provided that
$\dv{q^r}{p-1}$ which is indeed the case.
$ $ \\ $ $
Equation (\ref{g1b2}) generates polynomial
$ x^{q^{r-1}} - 1 \equiv 0 \pmod p  $.
Another application of the previous
Corollary to Lagrange's theorem indicates
that there are $q^{r-1}$ solutions to this
modular equation since $\dv{q^{r-1}}{q^r}$
and $\dv{q^r}{p-1}$ imply
$\dv{q^{r-1}}{p-1}$.
Thus there exist $q^r - q^{r-1}$ of the former
roots that satisfy the inequatlity
of Eq. (\ref{g1b2}).
Thus the way to find a $g$ as needed is to go
through the elements of $\mb{U}_p$ and answer
positively the question posed. Note
that $q^r - q^{r-1} >1$ and thus a $g$ can be
found, eventually.
\end{proof}

\section{Primitive roots}

A unit $g$ mod $n$ is a $g \in \mb{Z}_n$ such that $\gcd(g,n)=1$.
By Euler's theorem $g^{\phi (n)} \equiv 1 \pmod{n}$.

A unit $g \bmod n$ is a primitive root if its order is $\phi(n)$.

\begin{dfn}[Primitive root mod $n$]
Let $n \in \mb{N}$ with $n>1$. An integer $g \in \mb{Z}$ is
a primitive root modulo $n$ if $g \bmod n$ is a unit and thus
$\gcd(g,n)=1$ and by extension $g \bmod n \in  \mb{U}_n $ 
and has order $\phi(n)$.
\end{dfn}

A  unit $g \pmod n$ is a primitive root if it 
generates all $\mb{U}_n$.
For this reason $g$ is called a generator of   $\mb{U}_n$.
That is $\mb{U}_n = \{ 1, g , g^2 , \ldots , g^{\phi(n)-1} \}$.
This implies that $\ord_n (g) = \phi(n)$.
Moreover $\mb{U}_n$ is cyclic and $g$ is its generator.

\begin{exa}      
(Note that $ 3^2$ implies a $3^2 \pmod 7$.)
For $n=7$ we have 
\[
\mb{Z}_7 - \{ 0 \} =
\mb{U}_7 = \{ 1, 3^1 , 3^2 , 3^3 , 3^4 , 3^5  \}
    = \{ 1,  3  , 2   , 6   ,  4  , 5    \}.
\]
Thus $g=3$ is a primitive root.
\end{exa}

\begin{cor}
\label{nprimroots}
If $\mb{U}_n$ has a primitive root $g$, 
then all the primitive roots of $\mb{U}_n$ are 
those $g^q$ such that $\gcd(q,\phi(n))=1$.
In particular, there are $\phi(\phi(n))$ primitive roots 
$\bmod n$.
\end{cor}

\begin{proof}
If $g$ is a primitive root the $\ord_n (g) = \phi(n)$.
By Proposition~\ref{prp2} $ord_n (g^k ) =\phi(n)$ if and
only if $\gcd(k,\phi(n))=1$. The number of values $k$ such that
this is true is $\phi (\phi(n))$. All elements 
of $\mb{U}_n$ are of the form $g^i$ and will 
thus be found this way.
\end{proof}

\begin{exa}
If $g$ is a primitive root of $\mb{U}_n$ then
$g^m$ has order $\phi(n)$ if and only if 
$\gcd(m,\phi(n))=1$.
\end{exa}
\begin{solution}
This is a consequence of Proposition~\ref{prp2}
for $ord_n (g) = \phi = k$ there.
\end{solution}

\subsection{Some auxiliary results mod prime $p$}

\begin{lem}
If $p$ is a prime number and let $g$
be a primitive root mod $p$,
then $g^i \not\equiv g^j$ mod $p$ for
all $i,j$ such that $0 \leq i < j < p-1$.
\end{lem}
\begin{proof}
Suppose that $i,j$ are such that
\[
g^i \equiv g^j \pmod p .
\]
Then let $g^{-1}$ be the inverse of $g$ mod $p$,
that exists since $\gcd(g,p)=1$ and we then have
the following utilizing the previous assumption.
\[
(g) (g^{-1}) \equiv 1 \pmod p \Leftrightarrow
1 \equiv (g^i) (g^{-1})^i     \Leftrightarrow
1 \equiv (g^j) (g^{-1})^i     \Leftrightarrow
1 \equiv g^{j-i}.
\]
Then, since $ 0 < j-i \leq j < p-1$, $g$ can't be
a primitive root mod $p$.
\end{proof}

\begin{lem}
\label{ordgm}
If $p$ is a prime number and let $g$
be a primitive root mod $p$,
then $g^m$ has an  order as follows.
\begin{equation}
 ord_p (g^m ) = \frac{p-1}{\gcd(m,p-1)}.
\end{equation}
\end{lem}
\begin{proof}
Let $d=\gcd(m,p-1)$. Then there exists integers
$a,b$ such that
\[
m = d a , \quad p-1= db , \quad
\gcd(a,b)=\gcd(\frac{m}{d},\frac{p-1}{d}) =1.
\]
Let $ord_p (g^m ) = l$.
Consider $(g^m)^l$.
\[
g^{ml} \equiv (g^m)^l  \equiv 1 \pmod p ,
\]
by defition of $ord_p (g^m ) = l$. Then
$\dv{ord_p (g) }{ml}$ that is,
$\dv{p-1}{ml}$.
But then
$\dv{((p-1)/d)}{((m/d)l)}$,
and since $\gcd(\frac{m}{d},\frac{p-1}{d}) =1$
we conclude that $\dv{((p-1)/d}{l}$ or $(p-1)/d \leq l$.
$ $ \\ $ $
Furthemore consider $ (g^m)^{\frac{p-1}{d}}$.
\[
(g^m)^{\frac{p-1}{d}} \equiv
(g)^{m\frac{p-1}{d}} \equiv
(g)^{da\frac{p-1}{d}} \equiv
(g)^{(p-1)a} \equiv
(g^{p-1})^{a} \equiv
1 \pmod p,
\]
since $ord_p (g) = p-1$.
The latter implies $\dv{ord_p (g^m) }{\frac{p-1}{d}}$
and thus $l= ord_p (g^m) \leq (p-1)/d$.
From the prior $(p-1)/d \leq l$ and the current
$l= ord_p (g^m) \leq (p-1)/d$ we conclude
$l=ord_p (g^m) =(p-1)/d$, as needed.
\end{proof}

\subsection{Some auxiliary results mod $n$}

\begin{lem}
If $g$ is a primitive root mod $n$, for some
positive odd integer $n$, then
\[
 g^{\phi(n) /2} \equiv -1 \pmod n.
\]
\end{lem}
\begin{proof}
Let $n > 2$.
Then by Euler's theorem we have the following.
\[
  g^{\phi (n)} \equiv 1 \pmod n.
\]
Consider
\[
 x^2 \equiv 1 \pmod n.
\]
This implies the following.
\begin{eqnarray*}
 x^2 \equiv 1 \pmod n
&\Leftrightarrow&
 x^2 -1 \equiv 0 \pmod n \\
&\Leftrightarrow&
 (x-1)(x+1) \equiv 0 \pmod n \\
&\Leftrightarrow&
(x-1) \equiv 0 \pmod n
\quad \vee \quad
(x+1) \equiv 0 \pmod n \\
&\Leftrightarrow&
x \equiv 1 \pmod n
\quad \vee \quad
x \equiv -1 \pmod n .
\end{eqnarray*}
Set $x^2 = g^{\phi(n)}$ and
thus the following apply, given that
$\phi(n)$ is even and thus $\phi (n)/2$ is an
integer.
\[
g^{\frac{\phi(n)}{2}} \equiv 1 \pmod n
\quad \vee \quad
g^{\frac{\phi(n)}{2}} \equiv -1 \pmod n
\]
The former implies that the order of $g$
is no more than $ \frac{\phi(n)}{2}$
contradicting the fact that $g$ is a
primitive root.
Thus the latter applies.
\end{proof}

\begin{lem}
If $g$ is a primitive root mod $n$, for some positive
odd integer $n$,
then for $a \in \mb{Z}_n$ we have that
the modular congruence
\[
 a^2 \equiv 1 \pmod n,
\]
can have   two solutions mod $n$:
\[
a \equiv 1 \pmod n
\quad \text{and} \quad
a \equiv -1 \pmod n
\]
\end{lem}
\begin{proof}
We note that $ a^{-1} =a $ since
\[
 a a^{-1} \equiv a \cdot a \equiv 1 \pmod n .
\]
This means if the modular equation has a solution
$a$, then  $a$ is a unit and thus $\gcd(a,n)=1$.
This is because if $\gcd(a,n)=d >1$,
then
$ a^2 \equiv 1 \pmod n$ implies
$a^2 -1 = k n$ for some integer $k$.
We have $\dv{d}{a}$, thus
$\dv{d}{a^2}$ and we also have
$\dv{d}{n}$ and thus $\dv{d}{kn}$ and thus
$\dv{d}{a^2 - kn}$ which implies
$\dv{d}{1}$ i.e. $d=1$ contradicting $d>1$.
If $a \in \mb{U}_n$, and since $g$ be a primitive root
mod $n$, there exists a $k$ such that $0< k < \phi(n)$
such that
\[
  a \equiv g^k \pmod n
\Leftrightarrow
  a^2 \equiv g^{2k} \pmod n
\Leftrightarrow
  1   \equiv g^{2k} \pmod n
\Leftrightarrow
 g^{2k} \equiv 1 \pmod n
\]
This leads (see also the previous problem) to the following.
\[
 g^{k} \equiv 1 \pmod n  , k \neq 0
\quad \vee \quad
 g^{k} \equiv -1 \pmod n , k \neq 0,
\quad \vee \quad
 k=0.
\]
The  first is dismissed by $g$ being a primitive root
and $k < \phi (n)$. The latter two are the only
remaining possibilities. But then
\[
 a \equiv g^{k} \equiv -1 \pmod n ,
\]
or for $k=0$,
\[
  a \equiv g^k \equiv  1 \pmod n
\]
The result follows.
%Let $a$ be a unit mod $n$ (i.e. $a \in \mb{Z}_n$),
%which also implies $\gcd(a,n)=1$.
%Then
%\[
% a a^{-1} \equiv 1 \pmod n,
%\]
%Thus the assumption
%\[
% a^2 \equiv a a^{-1} \pmod n,
%\]
%leads to
%$aa^{-1} - a^2 \equiv 0 \pmod n$,
%and since $\gcd(a,n) =1 $,
%we conclude
%$a^{-1} - a \equiv 0 \pmod n$ or
%$a^{-1}  \equiv a \pmod n$.
\end{proof}

\begin{prp}
If $p,q$ are prime numbers $p \neq q$,
then $\mb{U}_{pq}$ has no primitive roots.
\end{prp}
\begin{proof}
It is the case that $\gcd(p,q)=1$ by $p \neq q$,
for prime $p,q$.
Pick an element $x$ of $\mb{Z}_{pq}^{x} = \mb{U}_{pq}$.
We have by Euler's theorem
\[
 x^{\phi (pq)} \equiv 1 \pmod{pq}
\Leftrightarrow
 x^{\phi (p) \phi (q)} \equiv 1 \pmod{pq}
\]
Note that $\phi (pq) = \phi (p) \phi (q)$
for prime $p,q$. It is easy to show that
$\phi (p)$ and $\phi (q)$ are even and thus
$\phi(p)/2$ and $\phi(q)/2$ are integer.
Consider
\[
k= \frac{\phi(p)\phi(q)}{2}.
\]
We have $\dv{\phi (p)}{k}$ and
        $\dv{\phi (q)}{k}$.
Let us assume that there exists a primitive root $g$ mod $pq$.
Then $ord_{pq}{g} = \phi (pq) = \phi (p) \phi (q)$.
Moreover,
\[
g^{k} \equiv 1 \mod p \quad , \quad
g^{k} \equiv 1 \mod q ,
\]
since for example
\[
g^{k} \equiv g^{\frac{\phi(p)\phi(q)}{2}}
      \equiv (g^{\phi(p)})^{\frac{\phi(q)}{2}}
      \equiv 1 \pmod p,
\]
with a similar and symmetric proof for a mod $q$ result.
Then
\[
 g^k \equiv 1 \pmod{pq},
\]
and given that $k < \phi(pq)$ since $k= \frac{\phi(p)\phi(q)}{2}$,
we have a contradiction to the  assumption that $g$ is a
primitive root mod $pq$.
\end{proof}

\section{Polynomials}

\begin{prp}
\label{poly1}
Let $(S, + , \cdot)$ be a field and
let $f \in S[x]$, where $f \neq 0$ be a polynomial
over $S$. Then $f$ has at most deg(f) roots.
\end{prp}
\begin{proof}
A proof follows by induction.
$ $ \\ $ $
{\bf Base case.} Let $n=deg(f)=0$ be the degree of $f$,
that is, $f(x) = a_0$. Since $f(x) \neq 0$ we have
$a_0 \neq 0$ and thus the polynomial has at most 0 roots
that is 0 roots obviously. The base case has been proven.
$ $ \\ $ $
{\bf Inductive step.}  Assume that the claim holds for
polynomials of degree up to $n-1$.
Consider $n=deg(f) \geq 1$.
If polynomial $f$ has at most $n$ roots, then we are done.
Consider $f(x)$ as follows.
\[
f(x) = a_n x^n + \ldots + a_1 x + a_0  \in S[x],
\]
with $f(x)$ being a polynomial of degree $n$.
If $f(x)$ has no roots in $S$, then we are also done.
Otherwise assume that $f(x)$ has at least one root,
and let it be $r$ that is $f(r)=0 \pmod n$.
Then we have the following.
\[
f(x) = f(x) - f(a) = a_n (x^n -r^n ) + \ldots + a_1 (x- r ) =
     \sum_{i=1}^{n} a_i (x^i - r^i ).
\]
We have that $x^i - r^i = (x-r) \sum_{j=0}^{i-1} x^j r^{i-1-j}$.
Then
\[
 f(x) = (x-r) g(x),
\]
where $deg(g) \leq n-1$.
Consider another root of $f$ and let it be $s$.
For $s \neq r$, we have $0=f(s) = (s-r) g(s)$.
Since $s \neq r$ this implies $g(s)=0$. This is because
element $s-r \neq 0$ has an inverse and multiplying
both sides of $0=(s-r) g(s)$ with that inverse we
conclude that $g(s)=0$ that is $s$ is a root of $g$.
By the induction hypothesis applied to
$g$ there can be no more than $n-1$ roots of $g$ such as
$s$. Adding to this root $r$ we can have at most $n$
roots for $f$.
$ $ \\ $ $
\bigskip
Consider $Z_p = \mb{Z}/p\mb{Z}$, where $p$ is a prime number.
Then $( Z_p , + , \cdot )$ is  a field.
Thus the following result due to Lagrange is applicable.
See the next problem, already proven for the more general
case of a field in this problem.
\end{proof}

\subsection{Lagrange's theorem}

\begin{thm}[Lagrange's theorem]
\label{poly2lag}
Consider a polynomial $f(x)$ of degree $n$ with
integer coefficients.
Let $p$ be a prime number.
\[
 f(x) = a_n x^n + a_{n-1} x^{n-1} + \ldots + a_1 x + a_0 .
\]
Consider $a_n \bmod p \neq 0$.
That is we exclude from
consideration the case that all coefficients
are multiples of $p$ or equivalently  exclude all
cases
$\dv{p}{a_i}$ for all $i=0, 1, \ldots , n$.
Then, $f(x)$ can have at most $n$ distinct zeroes mod $p$.
\end{thm}

In other words,
$f(x)\equiv 0 \pmod p$, has at most $deg(f)=n$
roots mod $p$, unless all
(presumed integer) coefficients of $f$ are divisible by $p$.
No more than $n$ elements of $\mb{U}_p$ can thus
be a root of $f(x)\equiv 0 \pmod p$.
The ``expression $n$ zeroes mod $p$'' is to
be read ``$n$ congruence classes of solutions (zeroes) mod $p$''.
Thus if $k$, an elements of $\mb{U}_p$ is
a root of $f(x)\equiv 0 \pmod p$ and $f(x) \equiv 0 \pmod p$
so do we expect for $k+p, k+2p, \ldots $.

\begin{proof}
A proof follows by induction.
$ $ \\ $ $
{\bf Base case.} For $n=deg(f)=0$ we have $f(x)=a_0$ that
has 0 roots mod $p$ for $f(x)\equiv a_0 \equiv 0 \pmod p$,
for $\ndv{p}{a_0}$. (Or for $n=0$,  $a_n \bmod p \neq 0$.)
$ $ \\ $ $
{\bf Inductive step.} Consider $n=deg(f) \geq 1$.
If $f$ has at most $n$ roots we are done.
By the inductive step $f$ has at least one root and let it be $r$.
Consider the polynomials
\[
  f(x) = g(x) (x-r) + R(x),
\]
with $deg(R) < deg(x-r) = 1$ i.e. $deg(R)=0$. We can write
$R(x)= c$ for some integer $c$.
Set $x=r$ and we have
\[
 ( f(r) \equiv 0 \pmod p)  \equiv g(r) \cdot 0 + c \pmod p
\]
Then
\[
 c \equiv 0 \pmod p.
\]
This means $R(x)=R=c=0 \pmod p$ and therefore
\[
 f(x) \equiv g(x) (x-r) \pmod p.
\]
Then, $deg(x-r)=1$ and $deg(g) \leq n-1$. By the inductive
hypothesis  the former has one root (and it is $r$) and the
latter has at most $n-1$ roots mod $p$. Combined $f(x)$ has
at most $n$ roots mod $p$, as needed.
\end{proof}

\begin{prp}
\label{poly3}
If $p$ is a prime number and $d$ is a divisor of $p-1$,
then the polynomial $f(x)=x^d -1 \in Z_p [x] $
has exactly $d$ roots.
\end{prp}
\begin{proof}
Since $\dv{d}{p-1}$ then $p-1=dk$.
Consider $x^{p-1}-1$. By Euler's (or FLT) all
non zero elements $1, \ldots , p-1$ are roots of it.
Therefore it has $p-1$ different roots.
Moreover
\[
(x^d)^k -1 = (x^d -1) ( (x^d)^{k-1} + \ldots + 
(x^d )^0  ) =f(x)g(x),
\]
where $g(x)$ is a polynomial of degree $dk -d = p-1-d$.
By the previous result $f$ has at most $d$ roots, and
$g$ at most $dk-d=p-1-d$. But the product has 
at least $p-1$ roots
all of them distinct. Thus $f(x)$ has exactly 
$d$ and $g(x)$ $p-1-d$
roots. The result follows.
\end{proof}

Consider the polynomial $f(x) = x^{p-1} -1$, where
$p$ is a prime number. By Fermat's Little theorem
we know. that
\[
 f(x) \equiv 0    \pmod p   \quad \Leftrightarrow \quad
 x^{p-1} \equiv 1 \pmod p   ,
\]
for any $x$ such that $\gcd(x,p)=1$. There are $p-1$
such $x$ and these are the elements of
$\mb{U}_p = \{ 1, 2, \ldots , p-1 \} $ since
$|\mb{U}_p|= \phi (p)=p-1$.
Lagrange's theorem confirms that there are at most
$p-1$ distinct solutions (roots) for $f(x) \equiv 0 \pmod p$.

%\begin{cor}[Corollary to Lagrange's theorem]
%\label{poly4lagc}
%Let $f(x) = x^{p-1} -1$, and let $p$ be a
%prime number. Consider $d$ a positive integer
%such that $\dv{d}{p-1}$.
%Then the modular equation below has at most $d$ distinct
%solutions mod $p$
%\begin{equation}
%\label{corLagrange}
%x^{d} \equiv 1 \pmod p \quad \quad
%\text{has at most d solutions mod p for}, \quad
%\dv{d}{p-1}.
%\end{equation}
%\end{cor}
%\begin{proof}
%Since $\dv{d}{p-1}$ then let $k = (p-1)/d$.
%We observe the following
%\[
%x^{p-1} -1 = x^{dk} -1 = (x^d -1) ( x^{(k-1)d} + \ldots + 1).
%\]
%Considering $\pmod p$ of both sides we conclude
%\[
%x^{p-1}-1 \equiv  (x^d -1) ( x^{(k-1)d} + \ldots + 1) 
%          \equiv 0 \pmod p .
%\]
%By the prior discussion $x^{p-1}-1 \equiv 0 \pmod p$ has
%$p-1$ distinct roots mod $p$.
%Say that $x^d -1 \equiv 0 \pmod p$ has LESS than 
%$d$ solutions (roots).
%By Lagrange's theorem
%$ x^{(k-1)d} + \ldots + 1) \equiv 0 \pmod p$
%has at most $(k-1)d$ distinct solutions mod $p$.
%Thus the product
%$(x^d -1) ( x^{(k-1)d} + \ldots + 1) \equiv 0 \pmod p$
%is going to have a number of solutions equal to the
%sum of (less than $d$) for $x^d -1 \equiv 0 \pmod p$
%and at most $(k-1)d$ for the second term a total
%of less than $d + (k-1)d = kd = p-1$.
%This ``less than'' contradicts the fact that
%$x^{p-1}-1 \equiv 0 \pmod p$ has exactly $p-1$ solutions
%mod $p$.
%\end{proof}

\begin{prp}
\label{poly5}
Let $p$ be a prime number.
Let
\[
 x^n \equiv 1 \pmod p,
\]
has $n$ distinct roots mod $p$.
Then $\mb{U}_p = \mb{Z}_p^{x}$ has exactly $\phi(n)$
elements of order $n$.
\end{prp}
\begin{proof}
Proof is by induction.
$ $ \\ $ $
{\bf Base case $n=1$.} If $x \equiv 1 \pmod p$ has one
solution mod $p$, and indeed the only solution is $x=1$,
then $\mb{U}_p$ has exactly $\phi(1)=1$ element of order $n=1$
and this is $1$ mod $p$. Case completed.
$ $ \\ $ $
{\bf Inductive step.} (Strong) Induction that is, assume that
result is true for all exponents $<n$ and we show the result
for $n$. That is if $x^k -1$ has exactly $k$ distinct roots mod
$p$, for all $k<n$ then $\mb{U}_k$ has exactly $\phi (k)$ elements
of order $k$. Consider  the case for $n$.
$ $ \\ $ $
Let $d$ be a divisor of $n$ i.e. there exists integer $N$ such
that $n=dN$. Just like in the previous problem we obtain the
following.
\[
x^n -1 = (x^d -1 ) ( x^{d(N-1)} + x^{d(N-2)} + \ldots  +x^d + 1).
\]
By Lagrange's theorem the last polynomial has at most $d(N-1)$
distinct roots mod $p$. The left-most polynomial 
has at most $n$ by Lagrange's theorem.
But we are told that $x^n -1$ has exactly $n$ distinct roots
mod $p$.
Then $x^d -1$ cannot have fewer than $d$ roots since then
\[
(=  n) \neq (<n ) =  (< d) + (\leq d(N-1)) ,
\]
$x^n -1$ would have fewer than $n$ distinct roots.
We note also that $d<n$. Thus $x^d -1$ must have exactly $d$
distinct solutions (roots). By the induction hypothesis since 
$d<n$ it has among the exactly  
$d$ distinct  roots $\phi (d)$ of order $d$.
Every root of $x^n -1$ is a root of $x^d -1$ or  the rest of the
polynomial (the other factor). Among the $n$ roots of $x^n-1$ ,
$d$ of them are roots of $x^d -1$ and distinct; moreover among
the $d$ latter roots $ \phi (d)$ are of order $d$ mod $p$.
Integer $d$ is an arbitrary divisor of $n$. One such divisor
is $n$ itself, and all the other divisors are less than $n$.
We know then the following
\[
n = \sum_{\dv{d}{n}} \phi (d ) =
    \sum_{\dv{d}{n}, d \neq n} \phi (d ) +
    \sum_{\dv{d}{n}, d=n}      \phi (d ) =
    \sum_{\dv{d}{n}, d \neq n} \phi (d ) +
                              \phi (n ).
\]
Moreover among the $n$ roots of $x^n -1$,  some of them
are of order $n$, and let that number be $a(n)$ and the
rest are of order $<$ than $n$. We concluded earlier that for
every divisor $d$ of $n$ there are $\phi (d)$ roots of $n$ of
order $d$. Adding up all of them we come up with the following.
\[
\text{(NumRoots of order } <  n)=
\sum_{d \neq n, \dv{d}{n}} = \phi (d),
\]
and we also have the following.
\[
\text{(NumRoots of order } < n) +
\text{(NumRoots of order } = n) =
\text{(NumRoots of order } < n) +
  a(n)                       =
n
\]
It is obvious that
\[
\text{(NumRoots of order }=n) = a(n)
=  n - \sum_{d \neq n, \dv{d}{n}} \phi(d)
=\phi (n).
\]
\end{proof}

\begin{cor}
\label{poly6}
\[
 x^{p-1} \equiv 1 \pmod p,
\]
has $\phi(p-1)=\phi(\phi(p))$ elements of order $p-1$ i.e.
that number of primitive roots mod $p$.
\end{cor}

\begin{prp}
\label{poly7}
A result similar to the previous one.
Let $p$ be a prime number.
Let there exists an $a$ such that $ord_p (a)=k$.
Then the number of elements mod $p$ of order $k$ is
$\phi (k)$.
\end{prp}
\begin{proof}
Since $ord_p (a)=k$, we have
\[
 a^k \equiv 1 \pmod p
\]
Consider the sequence
\[
a^0 , a^1 , \ldots , a^{k-1}
\]
are the $k$ roots of $x^k - 1\equiv 0 pmod p$,
or equivalently $ x^k \equiv 1 \pmod p$.
This is because for $i=0, 1, \ldots , k-1$, we have
the following.
\[
( a^i )^k -1 \equiv ( a^k)^i -1 \equiv 1-1 \equiv 0 \pmod p.
\]
We now ask the question: Out of the sequence
$a^0 , a^1 , \ldots , a^{k-1}$ how many of them  are
of order $k$?
The answer comes by way of Eq.(\ref{ordgm}) that states
\[
 ord_p(a^i )
  = \frac{ord_p(a)}{\gcd(i,ord_p(a))}
  = \frac{k}{\gcd(i,k)}.
\]
Thus
\[
 ord_p(a^i )
  = \frac{k}{\gcd(i,k)} = k,
\]
for all $i$ such that $\gcd(i,k)=1$ and thus for
$\phi (k)$ values $i$.
\end{proof}

\section{Primitive root existence}

\subsection{Primitive roots mod a prime}

\begin{prp}[Primitive root mod prime $p$]
\label{prrmodp}
For a prime number $p$,
$\mb{Z}_p$ has a primitive root  mod $p$,
or equivalently  $\mb{U}_p$ has a generator.
\begin{equation}
\label{prmodp}
\exists g \in \mb{U}_p : ord_p (g) = \phi(p).
\end{equation}
Furthermore, $\mb{U}_p$ is cyclic.
\end{prp}
\begin{proof}
$ $ \\ $ $
It is a consequence of Corollary~\ref{poly6}.
$ $ \\ $ $
A more direct proof follows.
$ $ \\ $ $
{\bf Case 1: $p=2$.}
If $p=2$, then $1$ is a primitive root of
$\mb{Z}_p = \{ 1 \}$.
$ $ \\ $ $
{\bf Case 2: $p>2$ is an odd prime.}
Let $p-1$ have a prime factorization as given below.
\[
 p-1= p_1^{a_1} \ldots  p_k^{a_k}
\]
where $p_1 < \ldots < p_k$.
Let us form
\[
 x^{p-1} -1 = (x^{{p_1}^{a_1}} - 1) f(x)
\]
The first factor has 
exactly $p_1^{a_1}$ distinct roots following
Proposition~\ref{poly3}.
From Proposition~\ref{prp4} there exists a $g_1$
such that $ord_p (g_1 ) = p_1^{a_1}$.
Repeating this argument for every $i=2, 3 \ldots, k$ and
$p_i$  we conclude that there exists  $g_i$ of order
$p_i^{a_i}$, and using  Corollary~\ref{cor17} we
conclude that $g_1 \cdot g_2 \cdot \ldots \cdot g_k$
has order $ p_1^{a_1} \cdot \ldots p_k^{a_k} = p-1$,
and thus $g_1 \ldots g_k$ is a primitive root mod $p$.
$ $ \\ $ $
A primitive root for $\mb{U}_p$ implies an element
of order $\phi (p)$. This mean $\mb{U}_p$ is cyclic.
%
%
%If $q$ is a root of $ (x^{{p_1}^{a_1}} - 1)$ then
%$ q^{{p_1}^{a_1}} \equiv 1 \pmod p $.
%Thus $\dv{\ord_p (q)}{{p_1}^{a_1}}$.
%Thus $\ord_p (q) = p_1^{b_1}$ , where $b_1 \leq a_1$.
%Every  root of $ (x^{{p_1}^{a_1}} - 1)$  cannot have
%order less that $p_1^{a_1}$ because then all those
%roots would also be roots of $(x^{{p_1}^{a_1-1}} - 1)$.
%The latter polynomial is of degree $1/p_1 \cdot {p_1}^{a_1}$ 
%of the original polynomial
%and it is going to have the same number of roots with the
%original one, an impossibility.
%Thus at least one of the roots is of
%order $p_1^{a_1}$. Let it be $g_1$.
%Repeating this argument for every $i=2, 3 \ldots, k$ and
%$p_i$  we conclude that there exists  $g_i$ of order
%$p_i^{a_i}$, and using  Corollary~\ref{cor17} we
%we conclude  that $g_1 \cdot g_2 \cdot \ldots \cdot g_k$
%has order $ p_1^{a_1} \cdot \ldots p_k^{a_k} = p-1$,
%and thus $g_1 \ldots d_k$ is a primitive root mod $p$.
%A primitive root for $\mb{U}_p$ implies an element
%of order $\phi (p)$. This mean $\mb{U}_p$ is cyclic.
\end{proof}

\begin{lem}
\label{prrmodpp}
Let $p$ be a positive prime number.
Suppose that $\phi(p)=p-1=p_1^{a_1} \ldots p_k^{a_k}$
are the prime factors of $p-1$, where
$p_1 < \ldots < p_k$ with $a_i > 0$.
Then $g$ is a primititve root
$\pmod p$ if and only if
\[
g^{\frac{p-1}{p_i}} \not\equiv 1 \pmod p
\]
for every $p_i$.
%\[
%g \; \text{is primitive root mod } p \quad \Leftrightarrow \quad
%g^{\frac{p-1}{p_i}} \not\equiv 1 \pmod p , \forall p_i .
%\]
\end{lem}
\begin{proof}
This is derived from Proposition~\ref{prp4}
and its Eq.(\ref{ordgpqr})
and the answer to its Question (\ref{g1q}).
\end{proof}

\noindent
We extend the existence of primitive roots beyond
prime numbers with the following question.

\begin{prp}[Conditions for primitive roots mod $n$ and $2n$]
\label{prrmod2n}
Let $n$ be equal to two ($n=2$) or to an odd prime
integer power $n=p^a$, $a \geq 1$.
Then there exists a primitive root  mod $n$ if and
only if there exists a primitive root  mod $2n$.
\end{prp}
\begin{proof}
Let $g$ be any odd integer. Then
\[
  g \equiv 1 \pmod 2.
\]
Then for any $k \geq 1$ we have the following.
\[
  g^k \equiv 1 \pmod 2.
\]
If $ord_n (g)= k$ then we have the following.
\[
  g^k \equiv 1 \pmod n.
\]
Combining the two using Equation (\ref{crtbyp}) we
have the following.
\[
  g^k \equiv 1 \pmod{2 n}.
\]
Thus $g$ is a primitive root mod $n$ if and only if
it is a primitive root mod $2n$.
\end{proof}

\subsection{Primitive roots mod $2$ and $4$} 

\begin{thm}
\label{prrmod2or4}
There are primitive roots mod $n$, where $n=2$ or $n=4$.
\begin{equation}
\label{prmod2a4}
\exists g_2 , \quad
g_2 \ \text{is a primitive root mod} \quad n=2 \  \wedge \
\exists g_4 , \quad
g_4 \ \text{is a primitive root mod} \quad n=4 \  .
\end{equation}
\end{thm}
\begin{proof}
Direct inspection.
For the former case $|U_2|=1$ and contains 1. The primitive
root is $1$.
For the latter case $|U_4|=2$ and contains 1,3.
The primitive root is $3$.
\end{proof}

\begin{cor}[No primitive roots mod $8$]
\label{prrmod8}
 There are no primitive roots mod $n$, where $n=8$.
\end{cor}
\begin{proof}
 For $n=8$, we have $|U_8|=4$ and contains 1,3,5,7.
For all these value $x=1,3,5,7$ we have
$x^2 \equiv 1 \pmod 8$.
For an $x$ to be a primitive root we require
$x^4 \equiv 1 \pmod 8$ i.e. the minimal power/order
of $x$ to be 4 and not 2, since $\phi(8)=4$.
\end{proof}

\subsection{No primitive roots mod $2^a$, $a \geq 3$}

\begin{prp}[No primitive roots mod $2^a$, $a \geq 3$]
\label{prrmod2p}
Let $n=2^a$,  where $a\geq 3$.
Then, there are NO primitive roots mod $n$
for $n=2^a$ as defined.
\begin{equation}
\label{nprmod2twoa}
\nexists g , \quad
g \ \text{is a primitive root mod} \quad n=2^a , \ a \geq 3.
\end{equation}
\end{prp}
\begin{proof}
We prove the result by induction on $a$.
Base case $a=3$. There are no primitive roots
mod $n=2^3 = 8$. This is part (b) of the
previous problem and can use a proof by inspection.
$ $ \\ $ $
For the inductive step we
we then show by induction that
\begin{equation}
\label{n2a1}
   g^{2^{a-2}} \equiv 1 \pmod{2^a}
\end{equation}
for all $g$ such that $\gcd(g,n)=\gcd(g, 2^a )=1$. The
latter implies an odd $g$.
$ $ \\ $ $
The base case is $a=3$ and Equation (\ref{n2a1})
becomes $g^2 \equiv 1 \pmod{2^3}$, which is true for
every odd number $g$ and thus every element of $U_8$.
$ $ \\ $ $
For the inductive step,
assuming
   $ g^{2^{a-2}} \equiv 1 \pmod{2^a} $
we shall show the following.
\[
    g^{2^{a-1}} \equiv 1 \pmod{2^{a+1}} .
\]
Equation (\ref{n2a1}) implies that $g$ can never have
order $\phi (2^a ) = 2^{a-1}$, as $2^{a-2} < 2^{a-1}$.
From Equation (\ref{n2a1}) we obtain that there exists
an $A$ such that
\[
  g^{2^{a-2}} -1 = A 2^a \Leftrightarrow
  g^{2^{a-2}}  = 1+ A 2^a \Leftrightarrow
  g^{2^{a-1}}  = (1+ A 2^a)^2  \Leftrightarrow
  g^{2^{a-1}}  = 1+ A 2^{a+1} + A^2 A^{2a} \equiv 1 \pmod{2^{a+1}}.
\]
This proves the inductive step, induction is incomplete and
$g^{2^{a-2}} \equiv 1 \pmod{2^a}$ implies that any $g$ of
$U_n$, $n=2^a$, $a\geq 3$ cannot be a primitive root mod $n$.
\end{proof}

\begin{thm}[Conditions for primitive roots mod $p^a$ or $2p^a$]
\label{prrconpa2pa}
Let $p$ be odd prime (positive) number.
Let $n$ be equal to two ($n=2$) or $n$ is equal
to an odd prime integer power $n=p^a$, $a \geq 1$.
Then if there exists a primitive root  mod $n$,
then either $n=p^a$ or $n=2p^a$,
with $n$ as defined.
\end{thm}
\begin{proof}
Let $p$ be an odd prime number.
Let $n=N p^a $ for some integer $a\geq 1$.
$N$ is such that $\gcd(N, p^a ) = \gcd(N, p)=1$
as otherwise we increase $a$ and reduce $N$ to $N/p$.
Let $N \geq 3$. We will show that that there are no
primitive roots mod $n$ for $N\geq 3$.
There will be then for $N=1$ and $N=2$.
Consider the following.
\[
\phi (n) = \phi (N p^a ) = \phi(N) \phi(p^a ).
\]
Both $\phi$ are even numbers. For example
$\phi (p^a) = p^{a-1} (p-1)$ and $p$ is odd.
Similar considerations apply to $N$.
Let $b \in \mb{Z}$ and $\gcd(b,n)=1$ and let us
assume it is a primitive root mod $n$. Moreover
$\gcd(b,N)=1$ and then $b^{\phi(N)}\equiv 1 \pmod N$.
Additionally $\gcd(b,p^a)=1$ and then
$b^{\phi(p^a)}\equiv 1 \pmod p^a $.
We first obtain the following.
\[
b^{\frac{\phi(n)}{2}}
\equiv
b^{\frac{\phi(N) \phi(p^a )}{2}}
\equiv
( b^{\phi(N)})^{\frac{\phi(p^a )}{2}}
\equiv
 1 \pmod{N}.
\]
We then  obtain the following.
\[
b^{\frac{\phi(n)}{2}}
\equiv
b^{\frac{\phi(N) \phi(p^a )}{2}}
\equiv
( b^{\phi(p^a)})^{\frac{\phi(N   )}{2}}
\equiv
 1 \pmod{p^a}.
\]
Combining the two using Equation (\ref{crtbyp}) we
have the following.
\[
b^{\frac{\phi(n)}{2}}
\equiv 1 \pmod{Np^a}
\Rightarrow
b^{\frac{\phi(n)}{2}}
\equiv 1 \pmod{n}.
\]
But $b$ is a primitive root mod $n$ and thus the following
applies.
\[
 a^{\phi(n)} \equiv 1 \pmod{n}
\]
For a primitive root with order $\phi(n)$,
the equality $a^{\frac{\phi(n)}{2}}
\equiv 1 \pmod{n}$ contradicts it as it implies
an order $\frac{\phi(n)}{2} < \phi(n)$.
Thus there can be no primitive root for
$Np^a$ if $N \geq 3$.
The only possible case(s) is/are $p^a$ and $2p^a$.
\end{proof}

\noindent
So far we have proved the existence
of primitive roots mod $n$, where $n$ is an (odd)
prime number, or $n=2$ or $n=4$.  
We also proved that there are
no primitive roots for $n=2^a$, where $a \geq 3$.

\subsection{Primitive roots mod $p^2$}

\begin{thm}[Primitive roots mod $p^2$]
\label{prrmodp2}
If $p$ is an odd prime number and thus
$p>2$, then  $\mb{Z}_{p^2}$ has a primitive root.
\begin{equation}
\label{prmodptwo2}
\exists g , \quad
g \ \text{is a primitive root mod} \quad p^2 , \
p>2 \ \text{is prime} .
\end{equation}
\end{thm}
\begin{proof}
Let $g$ be a primitive root mod $p$.
Since $g$ is a primitive root we have
$g^{p-1} \equiv 1 \pmod {p}$.
Let $k=ord_{p^2}(g)$.
Since $g^k \equiv 1 \pmod {p^2}$ we also have
      $g^k \equiv 1 \pmod {p}$.
Since $g$ is a primitive root mod $p$,
and $\phi(p)=p-1$ we also have that $\dv{p-1}{k}$.
Thus $k=(p-1)r$.
Moreover $\dv{k}{\phi( p^2)}$ since $k=ord_{p^2}(g)$ that
is, $\dv{k}{p(p-1)}$. That is $p(p-1) = (p-1)r s$.
This means $\dv{r}{p}$. Since $p$ is prime this means
$r=1$ or $r=p$. (This by itself is also a previous problem.)
$ $ \\ $ $
If $r=p$, $k=(p-1)r =p(p-1)$, then
$\ord_{p^2} (g) = k= p(p-1)=\phi(p^2)$.
This means {\bf $g$ is a primitive root mod $p^2$}.
$ $ \\ $ $
If $r=1$, $k=(p-1)r=(p-1)$, then $ord_{p^2} (g) = k = p-1$
and this means $g^{p-1} \equiv 1 \pmod {p^2}$.
$ $ \\ $ $
Consider $g_1 =g +p$. $g_1$ is also a primitive root $\pmod p$.
This is because
\[
 g_1^{p-1} \equiv (g+p)^{p-1}
           \equiv g^{p-1} + \lambda p
           \equiv 1 + 0
           \equiv 1 \pmod p
\]
implies $g_1^{p-1} \equiv 1 \pmod p$.
Similarly as before $\ord_{p^2}{g_1} = r_1 (p-1)$ where
$r_1$ is either 1 or $p$.
$ $ \\ $ $
Consider that $g^{p-1} \equiv 1 \pmod {p^2}$ implies that
\[
g_1^{p-1} \equiv (g+p)^{p-1}
          \equiv g^{p-1} +p(p-1)g^{p-2} + t p^2 g^{p-3}
          \equiv g^{p-1} -p g^{p-2} + (t+g^{p-2}) p^2
          \equiv 1 - p g^{p-2} \pmod {p^2}
\]
If $r_1 =1$ then $\ord_{p^2}(g_1)=(p-1)r_1 =1$, and thus
\[
1 \equiv g_1^{p-1} \equiv 1 -pg^{p-2} \pmod p^2
\]
which implies
$pg^{p-2} \equiv 0 \pmod p^2$ i.e. $\dv{p}{g}$ i.e. $p \leq g$.
This contradicts the fact that $g$ chosen as primitive root
mod $p$ implies $g < p$. That is {\bf it can't be that $r_1 =1$.}
$ $ \\ $ $
Thus $r_1 =p$  and 
$\ord_{p^2}(g_1 )=(p-1)r_1 =p(p-1)=\phi (p^2)$, and thus
{\bf $g_1$ is primitive root $\pmod {p^2}$}.
$ $ \\ $ $
To conclude if $r=p$ then $g$ is a primitive root $\pmod {p^2}$.
If $r=1$ then $g_1 = g+p$ implies that 
$g_1$ is a primitive root $\pmod {p^2}$.
One way or the other there is a primitive root in $\mb{Z}_p^2$,
and this is either a primitive root $g$ of $\mb{Z}_p$ or  the sum
of the primitive root plus p i.e. $g+p$!
\end{proof}

In the next result, we use this step as the base case or
stepping stone of an inductive proof.

\subsection{Primitive roots mod $p^a$}

In the next few problems
that there are primitive roots mod a prime power
$n=p^a$, $a \geq 1$, or for
$n=2p^a$.

\begin{lem}
Let $g$ be  a primitive root mod $p$, where
$p$ is an odd prime.
Let $g$ be such that
$ g^{p-1} \not\equiv 1 \pmod{p^2}$.
Show then the following.
\begin{equation}
\label{ptwo2prop}
g^{p-1} \not\equiv 1 \pmod{p^2}
\Rightarrow
g^{\phi (p^a )} \not\equiv 1 \pmod{p^{a+1}} \quad \forall a \geq 1.
\end{equation}
\end{lem}
\begin{proof}
Proof by induction on $a$.
Base case is $a=1$ that is
\[
g^{\phi (p   )} \not\equiv 1 \pmod{p^{2}},
\]
which is another writing of
\[
g^{p-1} \not\equiv 1 \pmod{p^2},
\]
since $\phi(p)=p-1$.
$ $ \\ $ $
{\bf Inductive step.} We show the result for $a\geq 2$.
If
\[
g^{\phi (p^a )} \not\equiv 1 \pmod{p^{a+1}} ,
\]
we show then
\[
g^{\phi (p^{a+1} )} \not\equiv 1 \pmod{p^{a+2}} ,
\]
By Euler's theorem we have the following.
\[
g^{\phi (p^a )} \equiv 1 \pmod{p^{a}}
\Rightarrow
g^{\phi (p^a )} -  1 = A {p^{a}}
\Rightarrow
g^{\phi (p^a )} =  1 + A {p^{a}} .
\]
The induction hypothesis
\[
g^{\phi (p^a )} - 1 \neq B \pmod{p^{a+1}} ,
\]
implies that $\ndv{p}{A}$, since otherwise
\[
g^{\phi (p^a )} - 1 = B \pmod{p^{a+1}} .
\]
Consider
\[
\phi(p^{a+1} ) = p^a (p-1) = p \times \phi(p^a ).
\]
We then obtain
\begin{eqnarray*}
g^{\phi (p^{a+1} )} &=& g^{p \phi (p^{a} )} \\
                    &=& (1+ A{p^{a}} )^p \\
                    &=& 1+A p^{a+1} \\
                    &\not\equiv& 1 \mod{p^{a+2}},
\end{eqnarray*}
as needed to conclude the inductive step,
where in the last derivation above one
used the binomial theorem, for $x= A p^a$, as shown below.
\begin{eqnarray*}
(1+Ap^{a})^p = (1+x)^p &=& \sum_{i=0}^{p} 
                             {p \choose i} x^i 1^{p-1} \\
                       &=& 1+ p x + x^2 \cdot S \\
                       &=& 1+ Ap p^a  + p^{2a} \cdot S \\
                       &=& 1+ A p^{a+1}  + 
                                p^{a+2} \cdot S p^{a-2}.
\end{eqnarray*}
\end{proof}

\begin{prp}[Primitive roots mod $p^a$, $a\geq 2$]
\label{prrmodp2a}
Let $g$ be a primitive root mod an odd positive prime
number $p$. Then either $g$ or $g+p$ is
a primitive root mod $p^a$ for all $a \geq 2$.
\end{prp}

We have already shown the result for $a=2$.
We need to show existence of primitive roots for $a>2$.

\begin{proof}
Let $g$ be a primitive root mod $p$. Then we have
the following.
\[
  g^{p-1} \equiv  1 \pmod p,
\]
with $\gcd(g,p)=1$.
We distinguish two cases.
\begin{equation}
\begin{split}
\text{Case 1:}& \quad g^{p-1} \not\equiv 1 \pmod{p^2}   \\
\text{Case 2:}& \quad g^{p-1}     \equiv 1 \pmod{p^2}  \Rightarrow
                \quad (p+g)^{p-1} \not\equiv 1 \pmod{p^2}.
\end{split}
\end{equation}
$ $ \\ $ $
We will then show by induction that  for all
$a \geq 1$, we have the following.
\begin{equation}
\label{p2a1}
ord_{p^a}(g)= \phi(p^a ) = p^a - p^{a-1},
\end{equation}
that is, $g$ is a primitive root mod $p^a$.
$ $ \\ $ $
{\bf Case 1, base case $a=1$.}
Equation (\ref{p2a1}) is true for $a=1$ as a by product
of Fermat's little theorem which also establishes
that $g$ is a primitive root mod $p$.
This is the basis of the inductive proof.
$ $ \\ $ $
{\bf Case 1, inductive step from $a$ to $a+1$.}
We assume that Equation
(\ref{p2a1}) is true for $a$, we shall show it is
true for $a+1$.
That is, the following will be shown.
\[
ord_p(g) = \phi (p) \quad \wedge \quad
ord_{p^a}(g) = \phi(p^a )  \Rightarrow
ord_{p^{a+1}}(g) = \phi(p^{a+1} ) .
\]
We will show that
\[
  k = ord_{p^{a+1}} (g)  \quad \wedge \quad
   g^k \equiv 1 \pmod{p^{a+1}}
\Rightarrow k = \phi (p^{a+1}).
\]
Since $g^k \equiv 1 \pmod{p^{a+1}}$
this implies the following
\[
g^k - 1 = K p^{a+1}
\Rightarrow
\dv{p^{a}}{g^k -1}
\Rightarrow
   g^k \equiv 1 \pmod{p^{a}}
\Rightarrow
\Rightarrow
   \dv{ord_{p^a}(g)}{k}
\Rightarrow
   \dv{\phi({p^a})}{k}
\]
Furthermore,
\[
 g^k \equiv 1 \pmod{p^{a+1}}
\Leftrightarrow
\dv{ord_{p^{a+1}}(g)}{\phi (p^{a+1})}
\Leftrightarrow
\dv{k}{\phi (p^{a+1})}
\]
From the two derivations
$\dv{\phi({p^a})}{k}$
and
$\dv{k}{\phi (p^{a+1})}$
we conclude that
(a)
either $k=\phi (p^{a+1}) =p^a (p-1)$
(b)
or     $k=\phi({p^a}))   =p^{a-1} (p-1)$.
$ $ \\  $ $
{\bf Case (b) leads to contradiction.}
If it is not the former case (a), then it is
the latter case (b) and  $k=\phi({p^a}))   =p^{a-1} (p-1)$.
The latter implies
\[
  k = ord_{p^{a+1}} (g) =\phi({p^a}))  \quad \wedge \quad
   g^{\phi({p^a})} \equiv 1 \pmod{p^{a+1}},
\]
which contradicts  Eq.(\ref{ptwo2prop})
that from
\[
g^{p-1} \not\equiv 1 \pmod{p^2},
\]
derives a
\[
   g^{\phi({p^a})} \not\equiv 1 \pmod{p^{a+1}}
\]
instead.
Thus case (a) is applicable
and we conclude that $g$
has order $k=\phi (p^{a+1})$ thus completing the inductive
step. As a conclusion
$g$ is a primitive root mod $p^a$ for all $a \geq 1$.
$ $ \\ $ $
{\bf Case 2, base case $a=1$.}
Equation (\ref{p2a1}) is true for $a=1$ as a by product
of Fermat's little theorem which also establishes
that $g$ is a primitive root mod $p$.
This is the basis of the inductive proof.
Furthermore, by way of Case 2 we have the following.
\[
g^{p-1} \equiv 1 \pmod{p^2}.
\]
$ $ \\ $ $
{\bf Case 2, inductive step from $a$ to $a+1$.}
We then consider $g+p$ as a candidate for a primitive
root. We note that since $g$ is a primitive root mod
$p$ so is $p+g$. We use the binomial theorem for
\begin{eqnarray*}
(p+g)^{p-1} &=& \sum_{i=0}^{p-1} {p-1 \choose i} p^i g^{p-1-i} \\
            &=& p^0 g^{p-1} + (p-1)p g^{p-2} + p^2 P \\
            &\equiv& g^{p-1} +p^2 g^{p-2} -p g^{p-2} \pmod{p^2} \\
            &\equiv& g^{p-1}  -p g^{p-2} \pmod{p^2} \\
            &\equiv&  1       -p g^{p-2} \pmod{p^2} .
\end{eqnarray*}
where we used the $ g^{p-1} \equiv 1 \pmod{p^2}$ to obtain the
last derivation above. Moreover $\gcd(p,g)=1$ and
further more we obtain
\[
(p+g)^{p-1} \equiv  1       -p g^{p-2} \pmod{p^2}
            \not\equiv 1 \pmod{p^2} .
\]
This is a condition symmetric to the one of Case 1.
Using similar arguments we conclude that
$p+g$ is primitive root for this case mod $p^a$,
for every $a\geq 1$.
\end{proof}

\begin{thm}[Primitive roots mod $p^a$, $a \geq 1$]
\label{prrmodpa}
There are primitive roots mod $p^a$ for all
$a \geq 1$, where $p$ is an odd prime number.
\end{thm}

\begin{proof}
From Proposition~\ref{prrmodp2a}, 
if  $g$ is a primitive root
mod $p$, then either $g$ or $g+p$ is a primitive root
mod $p^a$ for all $a \geq 1$.
By Proposition~\ref{prrmodp}, there exists
a primitive root $g$ mod $p$. Thus there exist primitive
roots mod $p^a$.
Furthermore if there exists one primitive root mod $p$,
from Eq. (\ref{nprimroots}) there are $\phi (\phi (p))= \phi (p-1)$
such primitive roots mod $p$ and by extension mod $p^a$.
Every primitive root mod $p$ is a generator
of $\mb{U}_p$ and thus $\mb{U}_p$ is cyclic.
%This can be extended to $\mb{U}_n$ for any $n$.
\end{proof}

\begin{thm}[Primitive roots mod $2p^a$, $a \geq 1$]
\label{prrmod2pa}
There are primitive roots mod $2p^a$ for all
$a \geq 1$, where $p$ is an odd prime number.
\end{thm}
\begin{proof}
Existentially,
this follows by way of Corollary~\ref{prrmodpa}
and
Proposition~\ref{prrmod2n} and
Proposition~\ref{prrconpa2pa}.
This is going to be proved indirectly.
By the previous Proposition~\ref{prrmodpa} 
we know that there is a
primitive root, call it $g$, mod $p^a$, $a \geq 1$.
$ $\\ $ $
We are going to show that if $g$ is odd then $g$
is also a primitive root mod $2p^a$. Otherwise,
$g+p^a$ is odd, and in this case we can show that
$g+p^a$ is a primitive root mod $2p^a$.
$ $ \\ $ $
{\bf Case 1: $g$ is odd.}
Then, the sequence,
\[
 g, g^2 , g^3 , \ldots g^{\phi(p^a)} ,
\]
is a sequence of odd integer and are distinct mod $p^a$.
Then they are also distinct mod $2p^a$ as well.
% say x notequiv y mod p^a
% but say x  equiv y mod 2p^a
% then x-y =2p^a K => x-y = p^a L => x equv y mod p^2 contradicting
% x notequiv y mod p^a.
Note that
\[
 \phi (2 p^a) = \phi (2) \phi (p^a) = \phi (p^2).
\]
Thus the sequence above enumerates all unit not only of
$\mb{U}_{p^a}$ but also $\mb{U}_{2p^a}$ since
\[
|\mb{U}_{p^a}| = |\mb{U}_{2p^a}| = \phi (p^a ).
\]
$ $ \\ $ $
{\bf Case 2 $g$ is even.}
The $g+p^a$ is odd and proceed likewise.
$ $ \\ $ $
This concludes the theorem.
\end{proof}

\bigskip

The conclusion leads to the following summarizing theorem.

\begin{thm}
There are primitive roots mod $n$ where $n=2,4, p^a , 2p^a$
for $a \geq 1$. The number of primitive roots mod $n$
is $\phi (\phi (n))$.
\end{thm}
\begin{proof}
{\bf Case 1: $n=2,4$.} By way of Theorem~(\ref{prrmod2or4} 
there exist primitive roots mod $n=2$ and mod $n=4$. 
They are by direct inspection 1 and 3 respectively.
$ $ \\ $ $
{\bf Case 2: $n=p^a$, $a \geq 1$.} This is a by product of
Theorem~\ref{prrmodpa}.
$ $ \\ $ $
{\bf Case 3: $n=2p^a$, $a \geq 1$.} This is a by product of
Theorem~\ref{prrmod2pa}.
\end{proof}

\section{Auxiliaries for Legendre and Jacobi symbols}

\begin{lem}
\label{jlaux1}
(a)
If $p,q$  are odd integers, then the following
applies.
\begin{equation}
\label{jacu1}
\frac{p-1}{2}
+
\frac{q-1}{2}
\equiv
\frac{pq-1}{2} \mod 2.
\end{equation}
This can be generalized for $p_1 , p_2 , \ldots p_k$,
with $p= p_1 \cdot p_2 \cdot \ldots \cdot p_k$.

\noindent
(b)  If $p_i$ are odd integer then the following
applies.
\begin{equation}
\label{jacu2}
\frac{p_1 -1}{2}
+
\frac{p_2 -1}{2}
+
\ldots
+
\frac{p_k -1}{2}
\equiv
\frac{p-1}{2} \mod 2.
\end{equation}
\end{lem}
\begin{proof}
(a)
Let $p=2P+1$, $q=2Q+1$ for odd $p,q$.
Then
\[
\frac{pq-1}{2}
=
\frac{(2P+1)(2Q+1)-1}{2}
=
2PQ + P + Q
\equiv
 0 + P + Q \pmod 2
\equiv
 0 + \frac{p-1}{2} +  \frac{q-1}{2}\pmod 2 .
\]
$ $ \\ $ $
(b) Follows by induction from part (a).
\end{proof}

\begin{lem}
\label{jlaux2}
(a)
If $p,q$  are odd integers, then the following
applies.
\begin{equation}
\label{jacu3}
\frac{p^2 -1}{8}
+
\frac{q^2 -1}{8}
\equiv
\frac{p^2 q^2 -1}{8} \mod 2.
\end{equation}

\noindent
(b)  If $p_i$ are odd integer then the following
applies.
\begin{equation}
\label{jacu4}
\frac{p_1^2 -1}{8}
+
\frac{p_2^2 -1}{8}
+
\ldots
+
\frac{p_k^2 -1}{8}
\equiv
\frac{p^2 -1}{8} \mod 2.
\end{equation}
\end{lem}

%% USe also ((pq)^2 -1 )/8 - (p^2 -1)/8 - (q^2 -1 )/8 = (p^2 -1) (q^2 -1)/8
%% with the latter being even (case analysis) p=2P+1, q=2Q+1.

\begin{proof}
$ $ \\ $ $
(a)
Let $p=2P+1$, $q=2Q+1$ for odd $p,q$.
Then
\[
p^2 - 1 =  4P(P+1), \quad
q^2 - 1 =  4Q(Q+1) \Rightarrow
\frac{p^2 - 1}{8} =  \frac{P(P+1)}{2}, \quad
\frac{q^2 - 1}{8} =  \frac{Q(Q+1)}{2},
\]
and the products $P(P+1)$, $Q(Q+1)$ are even integers.
Moreover
\[
\frac{p^2 - 1}{8} +
\frac{q^2 - 1}{8} =
 \frac{P(P+1)}{2} +
 \frac{Q(Q+1)}{2}.
\]
Similarly
\[
p^2 q^2 = 16P^2 Q^2 + 16P^2 Q+ 4P^2 + 16PQ^2 + 16PQ + 4P +4Q^2 + 4Q+1,
\]
and therefore
\[
\frac{p^2 q^2 -1}{8} =
   (2P^2 Q^2 + 2 P^2 Q + 2PQ^2 + 2PQ )
+\frac{P(P+1)}{2}
+\frac{Q(Q+1)}{2}.
\]
which, given that the parenthesized term is an even number,
leads to the following.
\[
\frac{p^2 q^2 -1}{8} \equiv
 0
+\frac{P(P+1)}{2}
+\frac{Q(Q+1)}{2} \pmod 2 ,
\]
\[
\frac{p^2 q^2 -1}{8} \equiv
 0
+\frac{P(P+1)}{2}
+\frac{Q(Q+1)}{2}
=
\frac{p^2 - 1}{8} +
\frac{q^2 - 1}{8}
\pmod 2 .
\]
$ $ \\ $ $
(b) Follows by induction from part (a).
\end{proof}

\section{Legendre symbol}

We write  q.r. to indicate that an $a$ is a quadratic
residue and q.nr.  to indicate that it is not a quadratic
residue i.e. it is a quadratic non-residue.
The set of units $(\mb{Z}/n\mb{Z})^x$ for prime $n$
will be represented by $\mb{U}_n$. Then $U_n$ becomes
a field. For $g \in \mb{U}_n$ it is $\gcd(g,n)=1$
and $g$  is invertible and thus $gx \equiv 1 \pmod n$
exists, and it is known as a unit.
Moreover $\mb{U}_n$ for prime $n$ it is a cyclic group.
The Legendre symbol $\leg{a}{p}$ is defined to be 1 if
$a$ is a quadratic residue $\pmod p$.
It is -1 if $a$ is a quadratic non-residue
(i.e. it is not a quadratic residue mod $p$).
Therefore for those two cases $\ndv{p}{a}$.
Furthermore, it is 0 if $\dv{p}{a}$.

\begin{dfn}[\bf{Legendre symbol}]
\label{legsym}
Let $p \in \mb{N}$ be an odd prime number. For any
$a \in \mb{Z}$ the Legendre symbol is defined as
follows.
\[
\leg{a}{p}  =
\begin{cases}
 0 & \text{if} \  \dv{p}{a} \\
 1 & \text{if a is quadratic residue mod p}\\
-1 & \text{otherwise}.
\end{cases}
\]
\end{dfn}

Viewing the Legendre symbol as a function we conclude the
following.
\[
\leg{a+kp}{p}
=
\leg{a}{p}.
\]

\begin{fct}[Legendre symbol property summary]
Let $p \in \mb{N}$ be an odd prime number ($p>2$).
For any $a,b \in \mb{Z}$ the following apply.
\begin{align}
a \equiv b \pmod p  &\Rightarrow  \leg{a}{p}=\leg{b}{p} . \\
\gcd(a,p)=1         &\Rightarrow   \leg{a^2}{p} =1. \\
\leg{ab}{p}         &= \leg{a}{p} \leg{b}{p} .\\
\gcd(a,p)=1, a \text{\ is q.r.}  &\Leftrightarrow
  a^{\frac{p-1}{2}} \equiv 1 \pmod p  \\
\leg{a}{p}          &\equiv a^{\frac{p-1}{2}} \pmod p. \\
\leg{-1}{p} =1   &\Leftrightarrow p \equiv 1 \pmod 4 . \\
\leg{-1}{p} =-1  &\Leftrightarrow p \equiv 3 \pmod 4 . \\
\sum_{a=0}^{p-1} \leg{a}{p} &= 0. \\
\leg{a}{p}=\leg{b}{p}=1  &\Rightarrow  \leg{ab}{p}=1. \\
\leg{a}{p}=\leg{b}{p}=-1 &\Rightarrow  \leg{ab}{p}=1. \\
\end{align}
\end{fct}

\begin{fct}[Legendre symbol property summary continued]
Let $p ,q \in \mb{N}$ be  odd prime numbers ($p,q>2$).
The following apply.
\begin{align}
 \leg{p}{q} \leg{q}{p}  &= (-1)^{\frac{p-1}{2} \frac{q-1}{2}}. \\
\leg{p}{q}              &= \begin{cases}
- \leg{q}{p}  & p \equiv q \equiv 3 \pmod 4  \\
\leg{q}{p}    & \text{otherwise}
\end{cases} \\
\leg{2}{p}  &=  (-1)^{\frac{p^2 -1}{8}} . \\
\leg{2}{p}=1 &\Leftarrow  p \equiv \pm 1 \pmod 8 . \\
\leg{2}{p} =-1  &\Leftarrow p \equiv \pm 3 \pmod 8 .
\end{align}
\end{fct}

\begin{lem}[More Legendre symbol property summary]
Let $p \in \mb{N}$ be an odd prime number ($p>2$).
Let $a\in \mb{U}_p$ such that $a>0$ and  
$\gcd(a,p)=1$.
Consider
\[
I (a) = \{ 1 \leq i \leq (p-1)/2 : ia  \pmod p \}
      = \{ 1 \leq i \leq (p-1)/2 : R(i) \},
\]
and
\[
T = \{ 1 \leq i \leq (p-1)/2 : R(i) < 0 \},
\]
where $R(i)$ is  the absolute least residue of $ai \pmod p$,
and thus $ -(p-1)/2 \leq r(i) \leq (p-1)/2$,
with
\[
 R(i) \equiv i \cdot a \equiv a(i) r(i) \pmod p,
\]
and $a(i) \equiv \pm 1 \pmod p$.
\begin{align}
\leg{a}{p}     &= (-1)^{|T|} & \text{(Gauss)}. \\
a \neq p, a \bmod 2 =1 :
\leg{a}{p} &= (-1)^{\sum_{i=1}^{\frac{p-1}{2}} \floor{\frac{ia}{p}}}
           & \text{(Eisenstein)} .
\end{align}
\end{lem}
\begin{proof}
\end{proof}

\begin{lem}
Show that if $a \equiv b \pmod p$ for an odd prime number
$p$, then
\[
 \leg{a}{p}=\leg{b}{p}
\]
\end{lem}
\begin{proof}
$ $ \\ $ $
{\bf Case 1: $\dv{p}{a}$.}
Note that if $\dv{p}{a}$ then since $a \equiv b \pmod p$ we
have $\dv{p}{b}$ and thus in that case
\[
 \leg{a}{p}=\leg{b}{p} =0
\]
$ $ \\ $  $
{\bf Case 2: a is a q.r.}
$ $ \\ $  $
Let $a$ otherwise be a q.r. mod $p$ i.e.
\[
\leg{a}{p}=1.
\]
Then there exists an $x$ such that
$x^2 \equiv a \pmod p$ and thus $\dv{p}{x^2-a}$. Moreover
$a \equiv b \pmod p$ implies $\dv{p}{a-b}$ and by addition
$ \dv{p}{x^2-b}$ i.e. $x^2 \equiv b\pmod p$ i.e. $b$ is a q.r.
mod $p$.
Same if $b$ is a q.r. mod $p$.
$ $ \\ $ $
$ $ \\ $  $
{\bf Case 3: a is a q.nr.}
$ $ \\ $  $
Let $a$ be a q.nr. mod $p$. Then there is NO $x$ such
that $x^2 \equiv a \pmod p$. This implies that $b$ is also a
q. nr. mod $p$. Why ? Otherwise if $b$ is a q.r. mod $p$
there exists a $y$ such that $y^2 \equiv b\pmod p$
and by using $\dv{p}{a-b}$ we would then conclude $a$ is a q.r.
mod $p$, a contradiction.
Case completed, cases completed, problem completed.
\end{proof}

\begin{lem}
Let $p$ be an odd prime number $>2$. Then for every $a \in \mb{Z}$ with
$\gcd (a,p)=1$ then show the following.
\[
 \leg{a^2}{p} =1.
\]
\end{lem}
\begin{proof}
Integer $a^2$ is a quadratic residue mod $p$ since there
exists an $x$ such that $x^2 \equiv a^2 \mod p$.
This $x$ is $x=a$ or $a \mod p$ in general if $a \geq p$.
\end{proof}

\begin{lem}
\label{thm61}
Let $p$ be an odd prime number $>2$. For all $a,b \in \mb{Z}_p$ we have
\begin{itemize}
\item[(a)] If $a\equiv b \pmod p$ then $\leg{a}{p} = \leg{b}{p}$.
\item[(b)] $\leg{a^2}{p}=1$.
\item[(c)] $\leg{ab}{p} = \leg{a}{p} \leg{b}{p}$.
\end{itemize}
\end{lem}
\begin{proof}
(a) It is immediate (and proven earlier).
$ $ \\ $ $
(b)
$ (a^2)^{(p-1)/2} \equiv 1 \pmod p$ by Fermat's Little Theorem and
Euler's criterion.
Moreover, $a^2$ is such that $\leg{a^2}{p}=1$ obviously by way of part (c).
$ $ \\ $ $
(c) $\leg{ab}{p}   \equiv  (ab)^{(p-1)/2}
                   \equiv  \leg{a}{p} \leg{b}{p} \pmod p$.
Integer $p$ is a prime number; furthermore the Legendre symbols
on the left and right side of the equivalence  are of a diference
between $-2$ and $2$. Give that $p$ divides that difference the
only possibility is for them to be equal to each other and thus
\[
\leg{ab}{p} = \leg{a}{p} \leg{b}{p}.
\]
The result is true for all $a,b>0$ as long as $\gcd(a,p)=\gcd(b,p)=1$.
\end{proof}

\begin{exa}
Let $p$ be an odd prime.
Show the following: If $a$ is a quadratic residue mod $p$
and $b$ is a quadratic residue mod $p$ then $ab$ is a
quadratic residue mod $p$ as well.
\end{exa}
\begin{proof}
If $a$ is a q.r. mod $p$ then there exists $x$ such
that $x^2 \equiv a \pmod p$.
If $b$ is a q.r. mod $p$ then there exists $y$ such
that $y^2 \equiv b \pmod p$.
Then obviously
$x^2 y^2 = (xy)^2 \equiv ab \pmod p$, implying that
$ab$ is a
quadratic residue mod $p$ as well.

\end{proof}

\begin{lem}
Let $p$ be an odd prime number.
The congruence relation
\begin{equation}
\label{x2modp}
 x^2 \equiv a \pmod p
\end{equation}
has the following properties.

\noindent
(1) If $a=0$ then it has only one solution $x=0$ mod $p$.

\noindent
(2) If $\ndv{p}{a}$, then it has zero (0) or two (2) solutions
mod $p$.
\end{lem}
\begin{proof}
{\bf Case (1).} If $a=0$, which implies $\dv{p}{a}$, then
the congruence relation becomes as follows.
\[
 x^2 \equiv 0 \pmod p ,
\]
which implies $\dv{p}{x^2}$ or equivalently
$\dv{p}{x}$ and thus $x \equiv 0 \pmod p$, and the result
follows.
$ $ \\ $ $
{\bf Case (2).} Suppose $\ndv{p}{a}$. If the modular equation
has 0 solutions we are done.
Otherwise it has at least one solution.
One can use Lagrange's theorem to conclude that
modular Eq.(\ref{x2modp}) for prime $p$ has at most two
distinct solutions.
Consider $x$ the ``at least one solution'' such that
\[
 x^2 \equiv a \pmod p
\]
We observe then that $p-x$ is another solutions as follows.
\[
(p-x)^2 \equiv p^2 -2xp + x^2 \equiv 0+ x^2 \equiv x^2 \pmod p.
\]
Thus if $x^2 \equiv a \pmod p$ then $(p-x)^2 \equiv a \pmod p$.
Therefore we found a second solution mod $p$.
It then suffices to show that those two solutions are distinct
mod $p$ to tie with the Lagrange upper-bound of two solutions.
In order to prove $x \not\equiv p-x \pmod p$ let us assume
that it is $x \equiv p-x \pmod p$. Then equivalently
\[
x \equiv p-x \pmod p
\Leftrightarrow
2x \equiv 0 \pmod p
\Leftrightarrow
\dv{p}{x}
\Leftrightarrow
\dv{p}{x^2}
\Leftrightarrow
\dv{p}{a},
\]
since $p$ is an odd prime and thus it can't divide two or be divisible
by two. The conclusion $\dv{p}{a}$ contradicts $\ndv{p}{a}$.
Thus $x$ and $p-x$ are distinct mod $p$, if one of them exists.
There are no more by Lagrange's theorem.
$ $ \\ $ $
(We can sideswipe references to Lagrange's theorem by using
an alternative arguement. Consider
another solutions $z$ to $z^2 \equiv a \pmod p$.
Then $z^2 \equiv x^2 \equiv a \pmod p$ would lead
to $(z-x)(z+x) \equiv 0 \pmod p$ and thus any other solution
is either $z\equiv x \pmod p$ or $z\equiv -x \pmod p$ i.e.
already included in the (set of) two solutions already obtained:
$x$ and $p-x$ mod $p$.
We also show that if $x+y \neq p$ or $x+y \not\equiv 0 \pmod p$
then $x^2 \not\equiv y^2 \pmod p$ and thus the two solutions
$x$ and $p-x$ is all that we can have.
If $x^2 \equiv y^2 \mod p$ we have as before $x-y \equiv 0 \pmod p$
or $x+y \equiv 0 \pmod p$.
The latter contradicts $x+y \not\equiv 0 \pmod p$, and is dismissed.
The former leads to $\dv{p}{x-y}$. But $-p < x-y < p$ and $p$
primes leads to $\ndv{p}{x-y}$ which contradicts the former
$x-y \equiv 0 \pmod p$, and is also dismissed.
Thus we have covered all possible solutions.
)
$ $ \\ $ $
Furthermore, consider the set
\[
 \{ 1, 2, \ldots , i , \ldots , \frac{p-1}{2}, \frac{p+1}{2} , \ldots , p-1 \}
\]
and then, the set of
\[
 1^2 , 2^2 , \ldots , i^2 \pmod p , \ldots , (\frac{(p-1)}{2})^2  \pmod p
\]
describes   all $\frac{p-1}{2}$ quadratic residues mod $p$,
with the remainder  $\frac{p-1}{2}$  elements of the set forming
the quadratic non-residues mod $p$.
\end{proof}

\section{Euler's criterion}

\begin{thm}[Simple form of Euler's criterion]
Let $p$ be an odd prime number, for every $a \in \mb{Z}$
with $\gcd(a,p)=1$,
%$a$ is a quadratic residue mod $p$
%if and only if
%\[
%\text{a is q.r. mod $p$} \Leftrightarrow
%a^{\frac{p-1}{2}} \equiv 1 \pmod p .
%\]
\[
a^{\frac{p-1}{2}} =
\begin{cases}
 1 & \text{if a is quadratic residue mod p}\\
-1 & \text{if a is not a quadratic residue mod p}.
\end{cases}
\]
\end{thm}
\begin{proof}
Let $b= a^{\frac{p-1}{2}}$.
Then
\[
 b^2 = a^{p-1} \equiv 1 \pmod p,
\]
by Fermat's Little theorem. The polynomial on (with indeterminate)
$b$, by Lagrange's theorem has at most two roots. We
recognize them as $1, p-1$ mod $p$ i.e. $1, -1$ mod $p$.
$ $ \\ $ $
{\bf Case 1.}
Let $a$ be a q.r. mod $p$. Then
\[
\exists x :  x^2 \equiv a \pmod p.
\]
Moreover for prime $p$, $\gcd(x,p)=1$ since
$1 \leq x < p$, and by Fermat's Little theorem
we have $x^{p-1} \equiv 1 \pmod p$.
Then
\[
a^{\frac{p-1}{2}}
=
{x^2}^{\frac{p-1}{2}}
\equiv
{x}^{p-1}
\equiv 1 \pmod p,
\]
We proved the top part that
$a^{\frac{p-1}{2}} \equiv 1 \pmod p$ for a quadratic residue $a$
mod $p$.
$ $ \\ $ $
{\bf Case 2.}
Let $a$ be a q.nr. mod $p$. Then
set $z = a^{\frac{p-1}{2}}$.
We know by the discussion prior to case 1, (rename $x$ there
to $z$ or the other way around) that $z$ is either
$+1$ or $-1$ mod $p$.
The equation $y^{\frac{p-1}{2}} \equiv 1 \pmod p$
has at most $(p-1)/2$ solutions, after noting $(p-1)/2$ is
an integer for odd $p$. Each such solution is a q.r. mod $p$
by way of case 1. Thus the only possibility for
$a$ is the $-1$ case and thus
$z = a^{\frac{p-1}{2}} \equiv -1 \pmod p$.
We thus proved the bottom part that
$a^{\frac{p-1}{2}} \equiv -1 \pmod p$ for a non
quadratic residue $a$ mod $p$.
\end{proof}
% Case 2 alternative proof
%
%Let $a$ be a quadratic non-residue. For every $b=1,2, \ldots , p-1$ the
%congruence $bx \equiv a \pmod p$ is such that $\gcd(b,p)=1$ and
%thus the congruence has a unique solution in $1, 2, \ldots , p-1$.
%Since $a$ is quadratic non-residue we can't have $x=b$ since then
%$b\cdot b \equiv a \pmod p$. Thus the integers $1,2, \ldots , p-1$ can
%be broken into pairs whose products are all equal to $a$. There are
%$(p-1)/2$ such pairs. Then $(p-1)! \equiv a^{(p-1)/2} \pmod p$. The latter
%is $-1$ by Wilson't theorem.

\begin{thm}[Simpler form of Euler's criterion]
Let $p$ be an odd prime number, for every $a \in \mb{Z}$
with $\gcd(a,p)=1$,
$a$ is a quadratic residue mod $p$
if and only if
\[
\text{a is q.r. mod $p$} \Leftrightarrow
a^{\frac{p-1}{2}} \equiv 1 \pmod p .
\]
\end{thm}
\begin{proof}
$\Rightarrow$.
Let $a$ be a q.r. mod $p$. Then
\[
\exists x :  x^2 \equiv a \pmod p.
\]
Moreover for prime $p$, $\gcd(x,p)=1$ since
$1 \leq x < p$, and by Fermat's Little theorem
we have $x^{p-1} \equiv 1 \pmod p$.
Then
\[
a^{\frac{p-1}{2}}
=
{x^2}^{\frac{p-1}{2}}
\equiv
{x}^{p-1}
\equiv 1 \pmod p,
\]
and the only-if has been proved.
$ $ \\ $ $
$\Leftarrow$.
Let $a^{\frac{p-1}{2}} \equiv 1 \pmod p$.
We shall show that $a$ is q.r. mod $p$.
$ $  \\ $ $
If $g$ is a primitive root mod $p$,
the $ord_p (g) = p-1$ and therefore
\[
\exists k : g^k \equiv a \pmod p
\Rightarrow
            g^{k(p-1)/2} \equiv a^{(p-1)/2} \pmod p
\Rightarrow
            g^{k(p-1)/2} \equiv 1 \pmod p.
\]
The latter implies $k(p-1)/2 \equiv 0 \pmod{p-1}$
or in other words $k=2l$ i.e. $k$ is even and thus
$k/2$ an integer.
Then
\[
 ( g^{k/2} )^2 \equiv a \pmod p ,
\]
implying that $a$ is a quadratic residue and the
$x$ of $x^2 \equiv a \pmod p$ is $x= g^{k/2} \pmod p$.
\end{proof}

\begin{thm}[Euler's criterion for q.r. mod $p$]
Let $p$ be an odd prime number, for every $a \in \mb{Z}$.
The following holds.
\begin{equation}
\label{legeq0}
\leg{a}{p} \equiv  a^{\frac{p-1}{2}} \pmod p.
\end{equation}
\end{thm}
\begin{proof}
We skip the possibility $\dv{p}{a}$ for which
$\leg{a}{p}=0$ by definition.
For odd prime $p$ then $(p-1)/2$ is
an integer, and thus $\dv{p}{a}$ implies
$\dv{p}{a^{(p-1)/2}}$ from which
$a^{(p-1)/2} \equiv 0 \pmod p$ and thus
$a^{(p-1)/2} \equiv \leg{a}{p} \pmod p$.
$ $ \\ $ $
Consider $x= a^{(p-1)/2}$, for some $a$ such that
$1\leq a < p$. By Fermat's Little theorem
$x^2 = a^{p-1} \equiv 1 \pmod p$. By Lagrange's
theorem $x^2 \equiv 1 \pmod p$, for prime $p$ has
at most two solutions mod $p$. We can verify $+1$ and
$-1 \equiv p-1 \pmod p$ as those two mod $p$.
$ $ \\ $ $
We show that if $a$ is a quadratic residue mod $p$,
then
\[
a^{(p-1)/2} \equiv 1 \pmod p .
\]
If $a$ is a quadratic residue, there exists $b$
such that $1 \leq b < p$ such that
$b^2 \equiv a \pmod p$.
Then
\[
a^{(p-1)/2} = (b^2)^{(p-1)/2} =b^{p-1} \equiv 1 \pmod p,
\]
by Fermat's Little theorem.
$ $ \\ $ $
Let now $a$ is NOT a quadratic residue mod $p$.
The
\[
 x^{(p-1)/2} \equiv 1  \pmod p ,
\]
has at most $(p-1)/2$ solutions. All these are the
quadratic residues mod $p$. Note that for odd prime
$p$ then $(p-1)/2$ is an integer. Since $a$ now is not
a quadratic residue, it is not one of them,
and thus
\[
a^{(p-1)/2}\not\equiv  1  \pmod p .
\]
Consider
\[
  y = a^{(p-1)/2},
\]
As derived previously $y^2 \equiv 1 \pmod p$,
by Fermat's Little theorem, and the only
possibilities of $y^2 \equiv 1 \pmod p$
by Lagrange's theorem are $y \equiv 1 $ or
$y \equiv -1$ mod $p$. The latter is  not
possible for a quadratic non-residue,
thus  the only possibility left is
$y \equiv -1$  mod $p$ i.e.
\[
y= a^{(p-1)/2}\equiv  -1  \pmod p .
\]

\end{proof}

\begin{lem}
Let $p$ be an odd prime number, for every $a \in \mb{Z}$.
The following holds.
\end{lem}
\begin{proof}
{\bf Case 1: $\dv{p}{a}$.}
Since $p$ is an odd prime number, $(p-1)/2$ is an integer.
Then $\dv{p}{a}$ implies $\dv{p}{a^{\frac{p-1}{2}}}$.
The latter implies
\[
a^{\frac{p-1}{2}} \equiv 0 \pmod p,
\]
and thus for $\dv{p}{a}$
it is indeed
\[
\leg{a}{p} \equiv  a^{\frac{p-1}{2}} \pmod p
\Rightarrow
\leg{a}{p}  = 0.
\]
$ $ \\ $ $
{\bf Case 2: $\ndv{p}{a}$.}
In the remaining cases $\ndv{p}{a}$.
Then $\gcd(a,p)=1$ and by Fermat's little theorem
we have
\[
a^{p-1} \equiv 1 \pmod p.
\]
By factoring we obtain the following
\[
 ( a^{\frac{p-1}{2}} - 1) ( a^{\frac{p-1}{2}}  + 1) \equiv 0 \pmod p
\]
The (congruence) equation
\[
x^2 \equiv 1 \pmod p
\]
has at most two distinct solutions by Lagrange's theorem mod $p$:
$\pm 1 \pmod p$ and there are no more mod $p$.
Furthermore
\[
 ( a^{\frac{p-1}{2}} - 1) ( a^{\frac{p-1}{2}}  + 1) \equiv 0 \pmod{p}
\Rightarrow \quad
a^{\frac{p-1}{2}} \equiv 1 \pmod{p}
\vee
 a^{\frac{p-1}{2}}  \equiv -1 \pmod{p}
 \Rightarrow \quad
\leg{a}{p} \equiv a^{\frac{p-1}{2}}  \pmod{p}.
\]
We need to distinguish which case maps to
residuosity and which one to non-residuosity.
$ $ \\ $ $
{\bf Case 2a: a is a quadratic residue mod $p$, mapping to
$ a^{\frac{p-1}{2}} \equiv 1 \pmod p$.}
Then there exists an $0< x < p$ such that
\[
x^2 \equiv a \pmod p .
\]
We then have
\[
a^{\frac{p-1}{2}} \equiv
(x^2)^{\frac{p-1}{2}} \equiv
x^{p-1} \equiv
1 \pmod p.
\]
This follows from Fermat's little Theorem as
$0 < x < p$, $p$ prime and thus $\gcd(x,p)=1$.
(Note also that if $\dv{p}{x}$ then $\dv{p}{x^2}$ and
since $\dv{p}{x^2 -a}$ then
$\dv{p}{a}$ contradicting $\ndv{p}{a}$.)
$ $ \\ $ $
Moreover $z^{\frac{p-1}{2}} \equiv 1 \pmod p$
has at most $\frac{p-1}{2}$ distinct solutions mod $p$,
$a$ is one of them and there are no more than $(p-1)/2$,
and by the previous problem we know there are
exactly $(p-1)/2$ distinct solutions, the quadratic
residues mod $p$,
which are $1^2, 2^2,  \ldots , ((p-1)/2)^2 \bmod p $.
$ $ \\ $ $
{\bf Case 2b: a is a quadratic non residue mod $p$, mapping to
$ a^{\frac{p-1}{2}} \equiv -1 \pmod p$.}
Let $a$ be a quadratic non-residue.
Since the $(p-1)/2$ quadratic residues satisfy
$z^{\frac{p-1}{2}} \equiv 1 \pmod p$
then since it is still
$a^{p-1} \equiv 1 \pmod p$ this can only be satisfied by
way of $ a^{\frac{p-1}{2}}  \equiv -1 \pmod{p}$
and thus the remaining $(p-1)/2$ quadratic non-residues will
satisfy
$z^{\frac{p-1}{2}} \equiv -1 \pmod p$ and this
concludes the problem.
$ $ \\ $ $
One can also argue as follows.
For every $b=1,2, \ldots , p-1$ the
congruence $bx \equiv a \pmod p$ is such that $\gcd(b,p)=1$ and
thus the congruence has a unique solution in $1, 2, \ldots , p-1$.
Since $a$ is a quadratic non residue we can't have $x=b$
since then $b\cdot b \equiv a \pmod p$ that would make $a$ a
quadratic residue.
Thus the integers $1,2, \ldots , p-1$ can
be broken into pairs whose products are all equal to $a$. There are
$(p-1)/2$ such pairs. Then by multiplying all of them we
get  $(p-1)! \equiv a^{(p-1)/2} \pmod p$. The latter $(p-1)!$
is $-1$ by Wilson't theorem.
Therefore for a quadratic non residue $a$ we have the following
\[
a^{(p-1)/2} \equiv -1 \pmod p .
\]
\end{proof}

\subsection{Applications of Euler's criterion: Legendre symbol}

\begin{lem}
Let $p$ be an odd prime number number greater than two.
Then $-1$ is a quadratic residue $\pmod p$ if and only
if $p \equiv 1 \pmod 4$.
\[
\leg{-1}{p} =1  \Leftrightarrow p \equiv 1 \pmod 4
\]
Moreover
\[
\leg{-1}{p} =-1  \Leftrightarrow p \equiv 3 \pmod 4
\]

\end{lem}
\begin{proof}
{\bf Case 1. $p$ is prime and $p=2$.}
Then $-1 \bmod 2 = 1 \mod 2$.
We have that $1^2 \equiv 1 \pmod p$. Thus
$-1$ as in $ -1 \bmod 2$ is q.r. mod $2$.
$ $ \\ $ $
{\bf Case 2. $p$ is an odd prime.}
If $p=4k+1$ then by Euler's criterion, Eq.(\ref{legeq0})
we have
\[
a^{\frac{p-1}{2}} \equiv (-1)^{(p-1)/2}
                   \equiv (-1)^{2k} \equiv 1 \pmod p.
\]
Alternatively, $(p-1)/2$ must be even.
Thus $-1$ is a q.r.
$ $ \\ $ $
For a $p=4k+3$ we conclude $-1$ is a q.nr by Euler's criterion
as.
\[
a^{\frac{p-1}{2}} \equiv (-1)^{(p-1)/2}
                   \equiv (-1)^{2k+1} \equiv -1 \pmod p.
\]
\end{proof}

\begin{lem}
Let $p$ be an odd prime number number greater than two
Then
\[
\sum_{a=0}^{p-1} \leg{a}{p} = 0.
\]

\end{lem}
\begin{proof}
Half of the integers $1, \ldots , p-1$ are
q.r. mod $p$ and thus
\[
\leg{a}{p}=1
\]
and
half are q.nr. mod $p$ and thus
\[
\leg{a}{p}=-1
\]
For $a=0$ we have $\dv{p}{a}$ and thus
\[
\leg{a}{p}=0.
\]
The number of $1$s is equal to the number of $-1$.
\end{proof}

\begin{lem}
The following apply for a $p$ that is an odd prime number.
(a)
\[
\leg{a}{p}=\leg{b}{p}=1 \quad \Rightarrow \quad \leg{ab}{p}=1.
\]
\noindent
(b)
\[
\leg{a}{p}=\leg{b}{p}=-1 \quad \Rightarrow \quad \leg{ab}{p}=1.
\]

\end{lem}
\begin{proof}
$ $ \\ $ $
(a) It has been proven earlier (e.g. previous problem part (c)).
$ $ \\ $ $
(b) $\mb{U}_p$ has a primitive root $g$ and therefore
\[
 a =g^{k} \pmod p , \quad
 b =g^{l} \pmod p , \quad
\leg{a}{p}=-1=\leg{b}{p}.
\]
We conclude that $k$ and $l$ are odd numbers since otherwise
$k/2, l/2$ is an integer and thus
$g^{k/2}$ is such that
\[
 (g^{k/2})^2 \equiv g^k \equiv a \pmod p,
\]
and likewise
$g^{l/2}$ is such that
\[
 (g^{l/2})^2 \equiv g^l \equiv b \pmod p,
\]
which would implye that $a,b$ are q.r. mod $p$ contradicting
the corresponding assumptions i.e.
$\leg{a}{p}=-1=\leg{b}{p}$.
Thus for odd $k,l$ we have that $(k+l)/2$ is an integer
Therefore
\[
( g^{\frac{k+l}{2}} )^2 \equiv g^k g^l \equiv ab \pmod p,
\]
implies that $ ab$ is a q.r. mod $p$.
\end{proof}

%HERE
\section{Jacobi symbol}

In all cases below, $n$ is odd and positive.
Note that for the Jacobi symbol, we have $\gcd(a,n)=1$.
If $a$ is such that $\gcd(a,n) \neq 1$, because
of the prime decomposition of $n$ we would have
$\dv{p_i}{a}$ for some $i$. Then
$\leg{a}{p_i}=0$ and consequently
$\leg{a}{n}=0$.

\begin{dfn}[\bf{Jacobi symbol}]
\label{jacsym}
Let $n \in \mb{N}$ and $a \in \mb{Z}$ such
that $\gcd(a,n)=1$. Let $n$ be a product of
odd prime numbers, not necessarily distinct,
\[
 n = p_1 p_1 \ldots p_k ,
\]
then the Jacobi symbol is defined as follows.
\[
\leg{a}{n}  =
\leg{a}{p_1} \cdot
\leg{a}{p_2} \cdot
\ldots       \cdot
\leg{a}{p_k} \cdot .
\]
\end{dfn}

The Jacobi symbol for a prime number $n$ is the
Legendre symbol.

\begin{cor}
Let $n \in \mb{N}$ and $a \in \mb{Z}$ such
that $\gcd(a,n)=1$. Let $n$ be a product of
odd prime numbers, not necessarily distinct,
\[
 n = p_1 p_1 \ldots p_k ,
\]
and $a$ is a quadratic residue $\mod n$.
Then the following applies.
\[
\leg{a}{n} =1.
\]
\end{cor}
\begin{proof}
If
$
 n = p_1 p_1 \ldots p_k
$ and given $\gcd(a,n)=1$,
for $a$ a quadratic residue we have
$x^2 \equiv a \pmod n$ for some $x$.
Since $\dv{p_i}{n}$ and  $\gcd(a,n)=1$ it
follows then that $\gcd(a,p_i )=1$.
Moreover $x^2 \equiv a \pmod{p_i}$.
Therefore
\[
\leg{a}{p_i} = 1 , \forall i=1, \ldots , k.
\]
Therefore
\[
\leg{a}{n}  =
\leg{a}{p_1} \cdot
\leg{a}{p_2} \cdot
\ldots
\leg{a}{p_k}
= 1 \cdot 1 \cdot \ldots \cdot 1 = 1.
\]
\end{proof}

\noindent
Note that that if $ \leg{b}{n}  =1$ it is not necessarily
true that $b$ is a quadratic residue $\mod n$.

\begin{fct}[Jacobi symbol properties]
Let $n,m \in \mb{N}$, $n > 0, m > 0$ be
odd integers. Let $a \in \mb{Z}$ such
that $\gcd(a,n)=1$.
Furthermore,  let $a,b \in \mb{Z}$ such that
$\gcd(nm,ab)=1$.
The following then apply.
\begin{align}
a \text{\ is a q.r. mod n} &\Rightarrow \leg{a}{n}=1 . \\
a \equiv b \pmod n &\Rightarrow \leg{a}{n}  = \leg{b}{n} . \\
\leg{a}{n} \leg{a}{m} &= \leg{a}{nm}. \\
\leg{a}{n} \leg{b}{n}  &= \leg{ab}{n}. \\
 \leg{a^2}{n}  = \leg{a}{n^2} &= 1. \\
\leg{a^2 b}{n^2 m} &= \leg{b}{m  }. \\
\leg{-1}{n}  &= (-1)^{\frac{n-1}{2}} . \\
\leg{ 2}{n} &= (-1)^{\frac{n^2 -1}{8}} . \\
\gcd(n,m)=1, \leg{n}{m} \leg{m}{n} &= (-1)^{\frac{n-1}{2} \frac{m-1}{2} }.
\end{align}
Let $m,m$ be odd integer numbers such that $\gcd(n,m)=1$,
where
$n=n_1 n_2 \ldots n_k$,
$m=m_1 m_2 \ldots m_l$, for $k \geq i \geq 1 $,
$l \geq j \geq 1$
and $n_i , m_j$ prime numbers.
\begin{align}
\leg{n}{m} &=  \prod_i \prod_j \leg{n_i}{m_j} .
%
%\leg{p}{q}  \leg{q}{p} &= (-1)^{\frac{p-1}{2} \frac{q-1}{2}} .
\end{align}

\end{fct}

\begin{lem}
Let $n,m \in \mb{Z}$, $n > 0 , m>0$ be odd (positive)
integers. Furthermore let $N,M \in \mb{Z}$ such that
$\gcd(nm,NM)=1$.
Then the following apply.
\[
(a) \leg{N}{n} \leg{N}{m} = \leg{N}{nm},
\]
\[
(b) \leg{N}{n} \leg{M}{n} = \leg{NM}{n},
\]
\[
(c) \leg{N^2}{n}  = \leg{N}{n^2} = 1,
\]
\[
(d) \leg{N^2 M}{n^2 m}  = \leg{M}{m  },
\]
and if $N \equiv M \pmod n$, then
\[
(e) \leg{N}{n}  = \leg{M}{n}.
\]
\end{lem}
\begin{proof}
(a) Write $n,m$ as a product of  odd primes using the unique
factorization theorem.
\[
 n= p_1 p_2 \ldots p_k  , \quad
 m= q_1 q_2 \ldots q_l  .
\]
Then we have the following
\[
\leg{N}{n} \leg{N}{m} =
\leg{N}{p_1}
\leg{N}{p_2}
\ldots
\leg{N}{p_k}
\leg{N}{q_1}
\leg{N}{q_2}
\ldots
\leg{N}{q_k}
=
\leg{N}{p_1 p_2 \ldots p_k q_1 q_2 \ldots q_l}
=
\leg{N}{nm}.
\]
$ $ \\ $ $
(b)
Write $n$ as a product of  odd primes using the unique
factorization theorem.
\[
 n= p_1 p_2 \ldots p_k  .
\]
Then
\[
\leg{N}{n} \leg{M}{n} =
\leg{N}{p_1}
\leg{N}{p_2}
\ldots
\leg{N}{p_k}
\leg{M}{p_1}
\leg{M}{p_2}
\ldots
\leg{M}{p_k}
=
\leg{NM}{p_1}
\leg{NM}{p_2}
\ldots
\leg{NM}{p_k}
=
\leg{NM}{p}.
\]
by way of $x^2 \equiv N \pmod n$ and
by way of $y^2 \equiv M \pmod n$ we have
          $(xy)^2 \equiv NM \pmod n$ and thus
for example
\[
\leg{N}{p_1} \leg{M}{p_1} = \leg{NM}{p_1}.
\]
$ $ \\ $ $
(c)
Write $n$ as a product of  odd primes using the unique
factorization theorem.
\[
 n= p_1 p_2 \ldots p_k  .
\]
Given that each $p_i$ is an odd prime using the Legendre
symbol
\[
\leg{N^2}{p_i} =1,
\]
for all $i=1, \ldots , k$.
Therefore the following applies.
\[
\leg{N^2}{n} =
\leg{N^2}{p_1}
\leg{N^2}{p_2}
\ldots
\leg{N^2}{p_k} = 1.
\]
But
\[
\leg{N}{p_1}
\leg{N}{p_1}
= 1
=
\leg{N}{p_1^2},
\]
for all $i=1, \ldots , k$.
Then we have the following.
$ $ \\ $ $
\[
\leg{N}{n^2} =
\leg{N}{p_1}
\leg{N}{p_1}
\leg{N}{p_2}
\leg{N}{p_2}
\ldots
\leg{N}{p_k}
\leg{N}{p_k}
= 1 \cdot 1 \cdot \ldots \cdot 1 = 1.
\]
$ $ \\ $ $
(d)
\begin{eqnarray*}
\leg{N^2 M}{n^2 m}  &=&
\leg{N^2 }{n^2 m}  \leg{M   }{n^2 m}  \\
                    &=&  1\cdot \leg{M   }{n^2 m}  \\
                    &=&  \leg{M   }{n^2}  \leg{M}{m} \\
                    &=&   1 \cdot         \leg{M}{m} \\
                    &=&  \leg{M}{m  }.
\end{eqnarray*}
$ $ \\ $ $
(e)
If $N \equiv M \pmod n$, since $n=p_1 \ldots p_k $ we have
$N \equiv M \pmod{p_i}$ for every $i=1 , \ldots , k$.
Then using Legendre symbol properties as $p_i$ are prime and odd,
$\leg{N}{p_i} = \leg{M}{p_i}$ for all $i$.
We then conclude
\[
\leg{N}{n} =
\leg{N}{p_1}
\leg{N}{p_2}
\ldots
\leg{N}{p_k}
=
\leg{M}{p_1}
\leg{M}{p_2}
\ldots
\leg{M}{p_k}
=
\leg{M}{n},
\]
as needed.
\end{proof}

\begin{lem}
Let $n,m \in \mb{Z}$, $n > 0 , m>0$ be odd (positive)
integers.
%Furthermore let $N,M \in \mb{Z}$ such that
%$\gcd(nm,NM)=1$.
Then the following apply.
\begin{itemize}
\item[(a)] if $a\equiv b \pmod n$ then $\leg{a}{n} = \leg{b}{n}$.
\item[(b)] $\leg{-1}{n} = (-1)^{\frac{n-1}{2}}$.
\item[(c)] $\leg{ 2}{n} = (-1)^{\frac{n^2 -1}{8}}$.
\item[(d)] For $\gcd(n,m)=1$
\[
\leg{n}{m} \leg{m}{n} = (-1)^{\frac{n-1}{2} \frac{m-1}{2} }.
\]
\end{itemize}
\end{lem}
\begin{proof}
(a) The result holds for a Legendre symbol i.e. when $n$ is an
odd prime number. In our case $n$ is an odd integer.
If it is a prime number it holds by way of the Legendre symbol
property. Otherwise it is the product of odd prime numbers,
and let $n=p_1 \ldots p_k$.
Moreover $a \equiv b \mod n$ implies
$a \equiv b \mod{p_i}$ for all $i=1, \ldots , k$.
Then from a Legendre property
$\leg{a}{p_i} = \leg{b}{p_i}$.
Then
\[
\leg{a}{n} = \leg{a}{p_1}
             \leg{a}{p_2}  \ldots
             \leg{a}{p_k}
           =
             \leg{b}{p_1}
             \leg{b}{p_2}  \ldots
             \leg{b}{p_k}
           = \leg{b}{n}.
\]
$ $ \\ $ $
(b) Again recalling Legendre symbol properties,
the result holds for a Legendre symbol i.e. when $n$ is an
odd prime number. In our case $n$ is an odd integer.
If it is a prime number it holds by way of the Legendre symbol
property. Otherwise it is the product of odd prime numbers,
and let $n=p_1 \ldots p_k$.
\begin{eqnarray*}
\leg{-1}{n}&=& \leg{-1}{p_1} \leg{-1}{p_2}  \ldots \leg{-1}{p_k} \\
           &=&
             (-1)^{\frac{p_1 -1}{2}}
             (-1)^{\frac{p_2 -1}{2}} \ldots
             (-1)^{\frac{p_k -1}{2}} \\
          &=&(-1)^{\frac{p_1 -1}{2} + \frac{p_2 -1}{2} + \ldots + \frac{p_k -1}{2}} \\
          &=&
(-1)^{\frac{p_1 -1}{2} + \frac{p_2 -1}{2} + \ldots + \frac{p_k -1}{2} \bmod 2} \\
          &=& (-1)^{\frac{p_1 p_2 \ldots p_k -1}{2}}         \\
          &=& (-1)^{\frac{n -1}{2}} .
\end{eqnarray*}
$ $ \\ $ $
(c) Similarly to  part (b). Use induction with base case $n=p_1 p_2$ so that
a claim similar to the one used above can be shown i.e.
\[
\frac{p_1^2 -1}{8} +
\frac{p_2^2 -1}{8} \bmod 2
=
\frac{p_1^2 p_2^2 -1}{8}
=
\frac{    n^2     -1}{8}.
\]
$ $ \\ $ $
(d) Omitted.
\end{proof}

 The symbol $\leg{p}{q}$ is the Jacobi symbol as
$q$ might not be a prime number and in fact it is
a compositive number in the general case.
The symbol $\leg{p_i}{q_j}$ is a Legendre symbol though as
all $q_j$ are prime numbers.

\begin{lem}
(a)
Let $p,q$ be odd integer numbers such that
$\gcd(p,q)=1$.
The following applies.
\begin{equation}
\label{nreci}
\leg{p}{q}  =  \prod_i \prod_j \leg{p_i}{q_j},
\end{equation}
where
$p=p_1 p_2 \ldots p_k$,
$q=q_1 q_2 \ldots q_l$, for $k \geq i \geq 1 $,
$l \geq j \geq 1$
and $p_i , q_j$ prime numbers.
\end{lem}
Furthermore,

\noindent
(b) \begin{equation}
\label{nreci2}
\leg{p}{q}  \leg{q}{p} =
(-1)^{\frac{p-1}{2} \frac{q-1}{2}}.
\end{equation}

\begin{proof}
Note that $p_i \neq q_j$ since $\gcd(p,q)=1$.
\begin{eqnarray*}
\leg{p}{q}  &=& \leg{p_1 p_2 \ldots p_k}{q}
             =  \prod_{i=1}^{k} \leg{p_i}{q}
             =  \prod_{i=1}^{k} \leg{p_i}{q_1 q_2 \ldots q_l}  \\
            &=& \prod_{i=1}^{k}  \prod_{j=1}^{l} \leg{p_i}{q_j}.
\end{eqnarray*}
Likewise one can show the following.
\begin{eqnarray*}
\leg{q}{p}  &=&
            \prod_{i=1}^{k}  \prod_{j=1}^{l} \leg{q_j}{p_i}.
\end{eqnarray*}
Later on, these two results are multiplied together, and
the following is derived.
\begin{eqnarray*}
\leg{p}{q}  \leg{q}{p}   &=&
            \prod_{i=1}^{k}  \prod_{j=1}^{l} \leg{p_i}{q_j} \leg{q_j}{p_i},
\end{eqnarray*}
and the quadratic reciporicity result is derived using Legendre
symbol properties and in particular
\[
 \leg{p_i}{q_j}
 \leg{q_j}{p_i} =
(-1)^{\frac{p_i -1}{2} \frac{q_j -1}{2}}
\Leftrightarrow
 \leg{p_i}{q_j} =
 \leg{q_j}{p_i}
(-1)^{\frac{p_i -1}{2} \frac{q_j -1}{2}}.
\]
and then
\[
\leg{p}{q}  =\prod_{i=1}^{k}  \prod_{j=1}^{l} \leg{p_i}{q_j}
            =\prod_{i=1}^{k}  \prod_{j=1}^{l}
 \leg{q_j}{p_i}
(-1)^{\frac{p_i -1}{2} \frac{q_j -1}{2}} =
\leg{q}{p}
\prod_{i=1}^{k}  \prod_{j=1}^{l}
(-1)^{\frac{p_i -1}{2} \frac{q_j -1}{2}} =
\leg{q}{p}
 (-1)^{\sum_i \sum_j \frac{p_i -1}{2} \frac{q_j -1}{2}},
\]
and then utilizing the previous problem we have.
\begin{eqnarray*}
\sum_i \sum_j \frac{p_i -1}{2} \frac{q_j -1}{2}
&=& \sum_i  \frac{p_i -1}{2} \sum_j \frac{q_j -1}{2}  \\
&=& \frac{p_1 p_2 \ldots p_k}{2} \frac{q_1 q_2 \ldots q_l}{2} \\
&=& \frac{p-1}{2} \frac{q-1}{2} \pmod 2 .
\end{eqnarray*}
\end{proof}

\begin{lem}
\label{jac1}
Function $\text{Jacobi}(a,n)$ calculates
\[
\leg{a}{n}
\]
in polynomial time for  $n \in \mb{N}$ and $a \in \mb{Z}$.
It is a recursive algorithm.
\end{lem}
\begin{proof}
We distinguish multiple cases.
$ $ \\ $ $
{\bf Case 1.} Say $a$ is even. Then $a=2^k b$, where
$b$ is odd. Therefore we use the following
\[
\leg{a}{n} = \leg{\frac{a}{2}}{n} \leg{2}{n}
           = \ldots \leg{b}{n} \cdot \left(\leg{2}{n} \right)^k .
\]
$ $ \\ $ $
{\bf Case 2.} If $a$ is odd and $a<n$
use $a \equiv b \pmod n \Rightarrow \leg{a}{n}=\leg{b}{n}$
and for $a,n$ such that
$\gcd(a,n)=1$
$\leg{a}{n} \leg{n}{a} = (-1)^{\frac{n-1}{2} \frac{a-1}{2} }$,
to obtain
\[
\leg{a}{n} = \ldots \leg{n^\prime}{a^\prime} , n^\prime >a^\prime ,
\]
for some $n^\prime >a^\prime$ dependent on $a,n$, and this becomes
Case 3.
$ $ \\ $ $
{\bf Case 3.} If $a$ is odd and $a>n$
\[
\leg{a}{n} = \leg{a \bmod n}{n},
\]
and continue until $a$ is  1 or 2.
$ $ \\ $ $
{\bf Recursive formulation.}  It combines properties
of the Jacobi symbol. Running time is similar to that
of Euclid's gcd algorithm i.e. $O(\lg{n})$ steps,
a polynomial time algorithm.
%\newpage
\SetKwComment{Comment}{/* }{ */}
%\SetKwRepeat{Do}{do}{while}
%  \Return{$\mathbf{PseudoPrime}$}
%  \Comment*[r]{$n$ is either prime or composite}
% \eIf(\tcc*[f]{$a \geq n$}){$a^{n-1} \not\equiv \pmod{n}$} {
\begin{algorithm}
\KwIn{$a,n$, where $n \in \mb{N}$ and $a \in \mb{Z}$}
\KwOut{$\leg{a}{n}$}
 \eIf{$a \geq n$} {
  $A = a \bmod n $;
 }{
  $A = a $;
 }
 $N = n ; s=1$;

 \If{$N  == 1$} {
   \Return{$1$};
 }
 \If{$(A  == 0) || (A == 1)$} {
   \Return{$A$};
 }
 \If{$N \bmod 2 == 0$} {
   \Return{$0$};
 }
 \If{$A < 0$}{
   \Return{$\mathrm{Jacobi}(A \bmod N, N)$}
 }
 \If{$A == 2$}{
   \eIf{$((N \bmod 8 ==3) || (N \bmod 8 == 5))$} {
      \Return{$-1$};
   }{
      \Return{$1$};
   }
 }
 \If{$A \bmod 2 == 0$}{
    $i=0$;

    \While{$A \bmod 2 == 0$} {
         $i++$;

         $A=A/2$;
    }
    \Return{$\mathrm{Jacobi}(2,N)^{i} \cdot \mathrm{Jacobi}(A,N)$};
 }
\caption{Jacobi(a,n) : Jacobi symbol recursive calculation}
\label{jsymrec}
\end{algorithm}
\end{proof}

\newpage
\begin{lem}
\label{jac2}
Function $\text{Jacobi}(a,n)$ calculates
\[
\leg{a}{n}
\]
in polynomial time for  $n \in \mb{N}$ and $a \in \mb{Z}$.
It is an iterative algorithm.
\end{lem}
\vspace{-0.2cm}
\begin{proof}
{\bf Non-recursive formulation.} See next page.
\newpage
{\small
\begin{algorithm}[H]
\KwIn{$a,n$, where $n \in \mb{N}$ and $a \in \mb{Z}$}
\KwOut{$\leg{a}{n}$}
 \eIf{$a<0$} {
   $A= a \bmod n$;
 }{
   $A = a $;
 }
 \If{$A \geq n$} {
  $A = A \bmod n $;
 }
 $N = n$; $s = 1$;

 \If{$N  == 1$} {
   \Return{$1$};
 }
 \If{$(A  == 0) ||(A == 1)$} {
   \Return{$A$};
 }
 \If{$N \bmod 2 == 0$} {
   \Return{$0$};
 }
 \If{$A == 2$}{
   \eIf{$((N \bmod 8 ==3) || (N \bmod 8 == 5))$} {
      \Return{$-1$};
   }{
      \Return{$1$};
   }
 }
 \While{$A \geq 2$}{

    \While{$A \bmod 4 == 0$} {
         $A=A/4$;
    }
    \If{$A \bmod 2 == 0$}{
        \If{$((N \bmod 8 ==3) || (N \bmod 8 == 5))$} {
           $s = -s$;
%        }{
%           $s = s$;
        }

        $A = A /2$;
    }
    \If{$A==1$}{
        \Return{$1$};
    }
    \If{$N>A$}{
        \If{$((N \bmod 4 ==3) \& \& (A \bmod 4 == 3))$} {
           $s = -s$;
        }
        $A = N \bmod A ; N=A $;
    }
 }
    \Return{$s \cdot A$};
\caption{Jacobi(a,n) : Jacobi symbol iterative calculation}
\label{jsymiter}
\end{algorithm}
}
$ $ \\ $ $
The result follows.
\end{proof}

\newpage
\begin{lem}
\label{jac3}
Function $\text{Jacobi}(a,n)$ 
(\cite{RC2e}, page 98)
calculates
\[
\leg{a}{n}
\]
in polynomial time for  $n \in \mb{N}$ and $a \in \mb{Z}$.
It is equivalent to the code of Jacobi(a,n) of 
Lemma~\ref{jac2}.
It is an iterative algorithm.
\end{lem}
\vspace{-0.2cm}
\begin{proof}
{\bf Alternative non-recursive formulation (\cite{RC2e}).} 
$ $ \\ $ $
{\small
\begin{algorithm}[H]
\KwIn{$a,n$, where $n \in \mb{N}$ and $a \in \mb{Z}$}
\KwOut{$\leg{a}{n}$}
 \eIf{$a<0$} {
   $A= a \bmod n$;
 }{
   $A = a $;
 }
 \If{$A \geq n$} {
  $A = A \bmod n $;
 }
 $N = n$; $s = 1$;

 \While{$A \neq 0$}{
     \While{$A \bmod 2 == 0$}{
        $A=A/2$;

        \If{$((N \bmod 8 ==3) || (N \bmod 8 == 5))$} {
          $ s = -s $;
        }
     }
     temp=A ;
     A=N ;
     N=temp;

     \If{$((N \bmod 4 ==3) \& \& (A \bmod 4 == 3))$} {
           $s = -s$;
     }
     $A = N \bmod A$ ;
 }
    \If{$N=1$}{
       \Return{$s$};
    }
    \Return{$0$};
\caption{Jacobi(a,n) : Jacobi symbol iterative calculation}
\label{jsymiter1}
\end{algorithm}
}
\end{proof}

\newpage

\section{Quadratic residues}

\begin{dfn}[\bf{Quadratic residue}]
For $p\in \mb{N}$, where $p$ is an odd prime number,
and for any $a \in \mb{Z}_p$ 
we say that $a$ is a quadratic residue $\pmod p$ if
and only if the congruence
\[
 x^2 \equiv a \pmod p
\]
has a solution for $0 < x < p$.
%For $p\in \mb{N}$, where $p$ is an odd prime number
%and for any $a \in \mb{Z}$ such that $\ndv{p}{a}$
%we say that $a$ is a quadratic residue $\pmod p$ if
%and only if the congruence
%\[
% x^2 \equiv a \pmod p
%\]
%has a solution.
\end{dfn}

The definition extends to the general case of an integer $n$
rather than of odd prime $p$.

\begin{dfn}[\bf{Quadratic residue}]
For $n\in \mb{N}$,
and for any $a \in \mb{Z}_n$ 
we say that $a$ is a quadratic residue $\pmod n$ if
and only if the congruence
\[
 x^2 \equiv a \pmod n
\]
has a solution for $0 < x < n$.
%For $n\in \mb{N}$,
%and for any $a \in \mb{Z}$ such that $\gcd(a,n)=1$,
%we say that $a$ is a quadratic residue $\pmod n$ if
%and only if the congruence
%\[
% x^2 \equiv a \pmod n
%\]
%has a solution.
\end{dfn}

Let $a \equiv b \pmod p$, for $p$ an odd prime number,
then
$x^2 \equiv b  \pmod p$
has a solution if and only if
$x^2 \equiv a  \pmod p$ does also.
%Moreover $\dv{p}{b}$ if and only if $\dv{p}{a}$.
Quadratic residuocity (with respect to odd prime number $p$) is
relevant to the introduction of the Legendre symbol.
Later we extend it to the introduction of the Jacobi symbol.

\begin{exa}
For $\mb{Z}_7$ find the q.r (quadratic residues) 
and the q.nr. (quadratic non-residues) mod 7.
\end{exa}

\begin{solution}
$1^2 , 2^2 , 3^2 , 4^2 , 5^2 , 6^2 \pmod 7$
are respectively $1, 4, 2, 2, 4, 1 \pmod 7$.
Thus $1, 2, 4$ are the quadratic residues, 
and $3, 5, 6$
are the quadratic non-residues and $0$ sometimes
is counted, sometimes not!
\end{solution}

\begin{lem}
Let $m,n \in \mb{Z}$ such that $\gcd (m,m)=1$.
The following applies.
\begin{equation}
\label{qrmn}
a \text{ is  q.r.  mod}\  mn
\quad
\Leftrightarrow
\quad
a \text{ is  q.r. mod} \  m
\quad \wedge \quad
b \text{ is  q.r. mod} \ n.
\end{equation}
\end{lem}
\begin{proof}
$\Rightarrow$.
Let $a$ be a q.r. mod $mn$.
Then there exists an $x$ such that
\[
x^2 \equiv a \pmod{mn}
\Rightarrow
\dv{mn}{x^2 -a}.
\]
The latter implies $\dv{m}{x^2 -a}$ and $\dv{n}{x^2 -a}$
since  $\gcd (m,m)=1$.
This proves the $\Rightarrow$ part.
$ $ \\ $ $
Let $a$ be q.r. mod $m$ and $n$.respectively.
Then, there exist $x,y$ such that
\[
x^2 \equiv a \pmod{m}
\quad \wedge \quad
y^2 \equiv a \pmod{n} ,
\]
and therefore there exist $M,N$ integer such that
$x^2 - a = m M$
and
$y^2 - a = n N$.
By the extension of Euclid's theorem by way of
 $\gcd (m,m)=1$ we have integer $A,B$ such that
\[
 A m + B n =1
\]
Then we derive the following.
\begin{eqnarray*}
 A m + B n &=& 1 \Rightarrow \\
  Am (x-y) + Bn (x-y) &=&x-y \Rightarrow \\
   C m     +  Dn      &=&x-y \Rightarrow \\
              Dn +y   &=&x -Cm \Rightarrow \\
\end{eqnarray*}
where $C= A(x-y)$ and $D=B(x-y)$.
Consider now $X= x- Cm$ and $Y=y + Dn$.
They are $X=Y$.
Furthermore,
\[
X^2 = (x-Cm)^2 \equiv x^2 \equiv a \pmod m ,
\]
and
\[
Y^2 = (y+Dn)^2 \equiv y^2 \equiv a \pmod n .
\]
Since $X=Y$ and thus $X^2 =Y^2$ we have shown
that
\[
\dv{m}{X^2 -a}
\quad \wedge \quad
\dv{n}{X^2 -a},
\]
and since $\gcd(m,n)=1$, this implies
\[
\dv{nm}{X^2 -a},
\]
i.e. $X^2 \equiv a \pmod{mn}$, and thus $a$ is
q.r. mod $mn$ as needed.
\end{proof}

%FIRST
\bigskip %THEOREM 57 %%
\begin{thm}
\label{thm57}
If $p$ is an odd prime then $(p-1)/2$ of the units
$\pmod p$ are quadratic residues, $(p-1)/2$ are quadratic
non-residues and there is nothing left unaccounted for.
\end{thm}

\begin{proof}
Consider  $\pm 1 , \pm 2, \ldots , \pm (p-1)/2 $ and
take the square of those
elements. These elements account for all the units $\pmod p$.
If $b^2 \equiv a \pmod p$, then $(-b)^2 \equiv a \pmod p$ as well.
The $(p-1)/2$ distinct values (of the squares) are the
quadratic residues. Everything else is a quadratic non-residue or 0.
\end{proof}

\begin{lem}
Let $n=pq$ where $p,q$ are prime numbers and $p\neq q$.
If a is a quadratic residue mod $n$, then $a$
has four square roots in $\mb{Z}_p^{x}$ and thus one quarter
of the elements of  $\mb{Z}_p^{x}$ are quadratic residues
mod $n$.
\end{lem}
\begin{proof}
Let $\mb{QR}_n$ be
the set of quadratic residues mod $n$.
We define function $\mb{Z}^{x} \mapsto \mb{QR}_n$
with $ x \mapsto x^2 \pmod n$.
$ $ \\ $ $
If $a \in \mb{QR}_n$ then there exists an $x$ such
that
\[
x^2 \equiv a \pmod n.
\]
Then since $n=pq$ as defined, we also have the following.
\[
x^2 \equiv a \pmod p,
\]
\[
x^2 \equiv a \pmod q.
\]
Thus $x$ is also a square root of $a$ mod $p$ and mod $q$.
Consider now a $b,c$ such that
\[
b^2 \equiv a \pmod p,
\]
\[
c^2 \equiv a \pmod q,
\]
i.e. one of the square roots of $a$ mod $p$,
and
 one of the square roots of $a$ mod $q$.
By the Chinese Remainder Theorem,
there is a $A$ such that
\[
A   \equiv B \pmod p,
\]
\[
A   \equiv C \pmod q,
\]
which imply
\[
A^2 \equiv a \pmod{pq}
\]
since $\gcd(p,q)=1$. (See also previous problem.)
$ $ \\ $ $
Integer $a$ has two square roots mod $p$.
Integer $a$ has two square roots mod $q$.
Thus combining the choices to denote $B$, and $C$ we have
four possibilities for a pair $(B,C)$ above.
Thus $a$ has four square roots mod $pq=n$.

\end{proof}

\begin{prp}
Let $p>2$ be an odd prime number. Let $a<p$ or in general
$\ndv{p}{a}$ and thus $\gcd(p,a)=1$.
Then for $k \geq 2$, $a$ is a q.r
mod $p^k$  if and only if $a$ is a q.r. mod $p$.
\end{prp}
\begin{proof}
$\Rightarrow$.
If $a$ is a q.r. mod $p^k$, then  there exists a $b$ such
that
$b^2 \equiv a \pmod {p^k}$. That is,
$\dv{p^k}{b^2 -a}$. Then $\dv{p}{b^2 -a}$ and the
result follows for the forward
direction:  $a$ is then a q.r $\pmod p$ as well.
$ $ \\ $ $
$\Leftarrow$.
For the converse, let $b^2 \equiv a \pmod p$. The
proof resembles a prior proof and is by induction on $k$.
$ $ \\ $ $
{\bf Base case.} Obviously true for $k=1$.
$ $ \\ $ $
{\bf Inductive step: from $k$ to $k+1$.}
Let $a$ be a q.r $\pmod {p^k}$ for
$k \geq 1$. We shall show that $a$ is a q.r. $\pmod {p^{k+1}}$.
$ $ \\ $ $
By the induction hypothesis (inductive assumption) he have
$b^2 \equiv a \pmod{p^k}$ i.e.
$b^2 -a = p^k r$, for some integer $r$.
Let  us form $c=b+d p^k$, where $c,d$ are yet to be
determined in full. We observe first that $2k \geq k+1$ for
$k \geq 1$. (It will be used to obtain the third derivation
below.)
We then have the following.
\begin{eqnarray*}
 c^2 -a  &\equiv& (b+d p^k)^2 -a  \\
         &\equiv& b^2 + 2bd p^k + d^2 p^{2k} -a \pmod {p^{k+1}} \\
         &\equiv& r p^k + 2bd p^k   \\
         &\equiv& p^k (r+2bd) \pmod {p^{k+1}} \\
\end{eqnarray*}
For $c^2 -a \equiv 0 \pmod {p^{k+1}}$ we
need $\dv{p}{r+2bd}$ in other words
$(2b) d \equiv -r \pmod p$. Since $p$ is an
odd prime $b<p$ and $\gcd(2b,p)=1$,
there is a solution for $d$ of the modular equation
$(2b) d \equiv -r \pmod p$.
Thus $c^2 -a  \equiv 0 \pmod {p^{k+1}}$ as
needed and this completes the
inductive step.
$ $ \\ $ $
In conclusion if $p$ is a prime and $a$ is a q.r $\pmod p$.
Let $b^2 \equiv a \pmod {p^k}$ or $r= (b^2 -a )/p^k$, for some
integer $r$.
The $c=b+dp^k$ is such that $c^2 \equiv a \pmod {p^{k+1}}$
if and only if $c$ is defined as follows after defining $d$.
Let $t\equiv (2b)^{-1} \pmod p$ and thus $d \equiv -r t \pmod p$.
We have from above $c=b+d p^k = b  - \frac{b^2- a}{p^k} p^k t $,
and thus 
$c^2 \equiv \left( b - t (b^2 - a) \right)^2 
     \equiv a \pmod {p^{k+1}}$.
Thus a square root $b$ of $a$ mod $p^k$ can be used to
construct a square root $c$ of $a$ mod $p^{k+1}$.
\end{proof}

\begin{prp}
For $k \geq 3$, $a>0$ is a q.r
mod $2^k$  if and only if $a$ is an odd integer number
and $a \equiv 1 \pmod 8$.
\end{prp}

In fact if $a \equiv 1 \pmod 8$ there are four square
roots of $a$ mod $2^k$: if $b$ is one of them, so is
$-b$ mod $2^k$, so is $b+2^{k-1}$ and also $-(b+2^{k-1})$.

\begin{proof}
$ $ \\ $ $
{\bf Base case $k=3$.}
We establish the base case $k=3$ i.e. show that an
odd integer $a$ with $a \equiv 1 \pmod 2^3$ is a quadratic
residue. For $1,2, \ldots , 2^3 -1$, there is only one integer
$a=1$ such that $a \equiv 1 \pmod 8$, and this is obviously $a=1$.
It is straightforward to confirm that $1$ is a q.r. mod $8$.
One has four square roots: 1,3,5, 7 mod $8$.
We can describe the four square roots as $b=1 \pmod 8$,
$-b \equiv 7 \pmod 8$, $b+4 =5$, and $-(b+4) \equiv 3 \pmod 8$.
$ $ \\ $ $
{\bf Inductive step from $2^k$ to $2^{k+1}$, $k \geq 3$.}
$ $ \\ $  $
By the inductive hypothesis,
say $a$ is a q.r mod $2^{k}$, $k\geq 3$ and also $a \equiv 1 \pmod 8$.
% Note that $2^{2k} \geq 2^{k+1}$ for $k\geq 3$,
% and also $2^{2k} \geq 2^{k+2} $.
Then there exists a $b$ such that
\begin{equation}
\label{twok1}
b^2 \equiv a \pmod{2^{k}}
\end{equation}
or equivalently $\dv{2^{k}}{b^2 -a}$,
or equivalently $b^2 -a =q 2^k$, for some integer $k$.
We will find a square root $c$ of $a$ mod $2^{k+1}$ then.
\begin{equation}
\label{twok2}
c^2 \equiv a \pmod{2^{k+1}}
\end{equation}
$ $ \\ $ $
{\bf Case 1: $b^2 \equiv a \pmod{2^{k+1}}$}. If it is
so we are done: $b$ is a square root of $a$ mod $2^k$ and
also mod $2^{k+1}$.
$ $ \\ $ $
{\bf Case 2: $b^2 \not\equiv \pmod{2^{k+1}}$}.
We have $a$ is odd; this mean $b$ is also odd and thus $b+1$
is an even number and multiple of two.
Just like a prior problem we form $c$ from $b$ as follows
\[
 c = b + q 2^{k-1} .
\]
We note that $2k-2 \geq k+1$ for $k \geq 3$.
Also $b+1$ is even i.e. $b+1=2B$ for some integer $B$.
We then obtain the following.
\begin{eqnarray*}
 c     &=& b + q 2^{k-1}  \\
 c^2   &=& (b + q 2^{k-1})^2  \\
 c^2 -a   &=& b^2 -a + q^2 2^{2k-2} + 2bq 2^{k-1} \\
 c^2 -a   &\equiv& b^2 -a + q^2 2^{2k-2} + 2bq 2^{k-1} \pmod{2^{k+1}}\\
 c^2 -a   &\equiv& b^2 -a +  0           + 2bq 2^{k-1} \pmod{2^{k+1}}\\
 c^2 -a   &\equiv& 2^k q  +  0           + bq 2^{k}    \pmod{2^{k+1}}\\
 c^2 -a   &\equiv& 2^k q  \cdot (b+1)                  \pmod{2^{k+1}}\\
 c^2 -a   &\equiv& 2^{k+1} q  B                        \pmod{2^{k+1}}\\
 c^2 -a   &\equiv&            0                        \pmod{2^{k+1}}\\
 c^2      &\equiv& a \pmod{2^{k+1}} \\
\end{eqnarray*}
A square root mod $2^{k+1}$ of $a$ has been found.
It is straightforward to show that $-c$ is also a square root mod $2^{k+1}$.
Furthermore,
\[
(c+2^{k} )^2 = c^2 + 2^{k+1} c + 2^{2k} \equiv c^2 \equiv a \pmod{2^{k+1}} ,
\]
indicates that $c+2^{k}$ is a third square root of $a$, and likewise
$-(c+2^{k})$ is a fourth square root of $a$ mod $2^{k+1}$.
\end{proof}

\begin{lem}
Let $p$ be an odd prime number.
Let $g$ be a primitive root of $\mb{U}_p$ and let $k$
be an even and positive natural number.
Then $g^k$ is a q.r. if and only if $k$ is an even
number.
\end{lem}
\begin{proof}
$ $ \\ $ $
$\Leftarrow$.
$ $ \\ $ $
Since $k$ is an even number, then $k/2$ is an integer.
Therefore
\[
 g^k \equiv ( g^{\frac{k}{2}} )^2 \pmod p,
\]
after which we conclude $g^k$ is a quadratic residue
mod $p$.
$ $ \\ $ $
$\Rightarrow $
$ $ \\ $ $
If $g^k$ is a qudratic residue then there exists
an $a$ such that
\[
 a^2 \equiv g^k \pmod p .
\]
Since $g$ is a generator there exists an $i$ such that
$ a \equiv g^i \pmod p$.
Therefore
\[
  g^k \equiv (g^i )^2 \equiv g^{2i} \pmod p .
\]
Then $k-2i \equiv 0 \mod{\phi(p)}$,
or equivalently
$\dv{p-1}{k-2i}$.
Since $p$ is an odd prime, $p-1$ is even.
So is $2i$ i.e. $k$ must also be even.
\end{proof}

\begin{cor}
\label{thm58}
 If $g$ is a primitive root $\pmod p$ the $g^k$ is a quadratic
residue if $k$ is even.
\end{cor}
\begin{proof}
The $g^2 , g^4 , g^6, \ldots , g^{p-1}$ are the $(p-1)/2$ quadratic
residues of Theorem~\ref{thm57}.
\end{proof}

\begin{prp}
Let $p$ be an odd prime number.
Let $g$ be a generator of $\mb{Z}_p^*$ and let $k$
be an even and positive natural number.
As shown earlier, $\mb{Z}_p^*$ has $(p-1)/2$
q.r. and thus $\mb{Z} = \mb{Z}_p^* \cup \{ 0 \}$
has $(p+1)/2$. Since $\mb{Z}_p^*$ is cyclic it
has a generator $g$.
The following then apply.
\begin{itemize}
\item[(a)] Show that $g$ is a q.nr.
\item[(b)] Show that $g^2 , g^4 , \ldots , g^{p-1}$ 
are   q.r and distinct.
\item[(c)] Show that $g^1 , g^3 , \ldots , g^{p-2}$ 
are   q.nr and distinct.
\end{itemize}

\end{prp}
\begin{proof}
(a)
Say that $g$ is a q.r. and let $a$ be such that
$\gcd(a,p)=1$ and $ g \equiv a^2 \pmod p$.
It is the case that $a^{p-1} \equiv 1 \pmod p$
by Fermat's Little Theorem.
Therefore
\[
g^{\frac{p-1}{2}} \equiv (a^2 )^{\frac{p-1}{2}}
                  \equiv a^{p-1} \equiv 1 \pmod p .
\]
The latter implies that $g$ can be a generator
since $ord_p (g) \leq \frac{p-1}{2}$ instea of
$ord_p (g) = \phi (p) = p-1$ as it should be.
$ $ \\ $ $
(b) It is immediate from the previous problem.
The fact that they are distinct is also obvious.
Let
$g^{2i} \equiv g^{2j} \pmod p$ for $2i \neq 2j \leq p-1$,
and $i-j < p-1$ leads to
$2i - 2j \equiv 0 \pmod{ \phi (p)} \equiv 0 \pmod{p-1}$
This implies $\dv{p-1}{2i-2j}$ i.e. $p-1 \leq |2i-2j|$.
The only possibility for these to happen is $i=j$.
$ $ \\ $ $
(c) From (b) and the previous problem.
By Euler's criterion we know that the number of
q.r. is at most $(p-1)/2$ and by (b) it is exactly
$(p-1)/2$. Thus the remaining are the q.nr.
\end{proof}

\section{Gauss lemma}

\begin{thm}[Gauss lemma]
Let $p>2$ be an odd prime number number.
Let $a \in \mb{Z}_p^*$ be such that $a>0$ and $\gcd(a,p)=1$.
Consider
\[
I (a) = \{ 1 \leq i \leq (p-1)/2 : ia  \pmod p \}
      = \{ 1 \leq i \leq (p-1)/2 : R(i) \},
\]
and
\[
T = \{ 1 \leq i \leq (p-1)/2 : R(i) < 0 \},
\]
where $R(i)$ is  the absolute least residue of $ai \pmod p$,
and thus $ -(p-1)/2 \leq r(i) \leq (p-1)/2$,
with
\[
 R(i) \equiv i \cdot a \equiv a(i) r(i) \pmod p,
\]
and $a(i) \equiv \pm 1 \pmod p$.
Then the following holds
\[
 \leg{a}{p} = (-1)^{|T|}.
\]
\end{thm}
\begin{proof}
{\bf Diversion first.}
In the remainder, $ 1 \leq i,j \leq (p-1)/2$ and thus
$2 \leq i+j \leq p-1$, and $i-j$ or $j-i$ is at least
$-(p-1)/2$ and at most $(p-1)/2$.
Consider set $A$ and set $B$ defined as follows.
\[
  A= \{ 1, 2, \ldots , \frac{p-1}{2} \},
\]
\[
  B= \mb{Z}_p^{*} -A =
\{ \frac{p-1}{2} +1 , \frac{p-1}{2} +2 , \ldots , p-1 \}=
\{ -1, -2, \ldots , -\frac{p-1}{2} \}.
\]
Let $a \in \mb{Z}_p^{*}$ per problem's statement. Consider the set
\[
I(a) =  a \cdot  A= \{ a, 2a, \ldots ,  \frac{p-1}{2} a \} \pmod p .
\]
For every $i \in A$, i.e. $1 \leq i \leq (p-1)/2$ we
have the following
\[
 a \cdot i \equiv R(i) = a(i) \cdot r(i) \pmod p
\]
where $a(i) = \pm 1$ and  $r(i)$ defines
a mapping $ A \mapsto A$ that is an injective function (injection)
and
$r(i)     \equiv r(j) \pmod p \Rightarrow i \equiv j \pmod p$
(or equivalently,
$i \not\equiv j \pmod p     \Rightarrow r(i) \not\equiv r(j) \pmod p$
).
% $r(i) \not\equiv r(j) \pmod p$  for $ i \not\equiv j \pmod p$,
$ $ \\ $ $
The
$r(i) \equiv r(j) \pmod p \Rightarrow i \equiv j \pmod p$,
is derived by the following observations,
noting that
if $r(i) \equiv r(j) \pmod p$, then $ai \equiv aj \pmod p$,
or $ai \equiv -aj \pmod p$
$ $ \\ $ $
{\bf Observation 1:
$ai \equiv aj \pmod p$ implies $i=j$.}
$ $ \\ $ $
If $R(i) \equiv R(j) \pmod p$ with $ai \equiv aj \pmod p$,
then $i \equiv j \pmod p$
since $ai \equiv aj \pmod p$ implies $(i-j)a = k p$ for some $k$ and
since $\gcd (p,a)=1$ we have $\dv{p}{i-j}$; for the range
of $i,j$ this can be true only for $i=j$.
$ $ \\ $ $
{\bf Observation 2:
$ai \not\equiv -aj \pmod p$.}
If $ai \equiv -aj \pmod p$ then $a(i+j) = kp$ for some $k$
and since $\gcd (p,a)=1$ we have $\dv{p}{i+j}$, We know
however that $i+j \leq p-1$ and thus $\ndv{p}{i+j}$.
$ $ \\ $ $
Therefore $r(.)$ is a permutation of $A$.
We then have the following
\begin{eqnarray*}
 a \cdot i &\equiv&  a(i) r(i) \pmod p \Leftrightarrow \\
 \prod_i a \cdot i &\equiv&  \prod_i a(i) r(i) \pmod p \Leftrightarrow \\
 a^{\frac{p-1}{2}} \left( \frac{p-1}{2} \right)!  &\equiv&
\prod_i a(i) \prod_i r(i) \pmod p \Leftrightarrow \\
 a^{\frac{p-1}{2}} \left( \frac{p-1}{2} \right) !  &\equiv&
\prod_i a(i)  \left( \frac{p-1}{2} \right) ! \pmod p .
\end{eqnarray*}
Since $\gcd (p,  \frac{p-1}{2}! )=1$ and by
Euler's criterion $\leg{a}{p}  \equiv  a^{\frac{p-1}{2}} \pmod p$,
we obtain the following.
\begin{eqnarray*}
 a^{\frac{p-1}{2}} \left( \frac{p-1}{2}\right) !  &\equiv&
\prod_i a(i)  \left( \frac{p-1}{2}\right) ! \pmod p \Leftrightarrow  \\
 \leg{a}{p}         &  =   &  \prod_i a(i)   \pmod p \Leftrightarrow
 \leg{a}{p}            =      (-1)^{|T|}     \bmod p  \Leftrightarrow
 \leg{a}{p}            =      (-1)^{|T|}             .
\end{eqnarray*}
\end{proof}

\begin{thm}[Gauss lemma version two]
Let $p>2$ be an odd prime number.
Let $a \in \mb{Z}_p^*$  be such that $a>0$ and
$\gcd(a,p)=1$. 
Consider
\[
I (a) = \{ 1 \leq i \leq (p-1)/2 : ia  \pmod p \}.
      = \{ 1 \leq i \leq (p-1)/2 : R(i) \},
\]
and
\[
T = \{ 1 \leq i \leq (p-1)/2 : R(i) < 0 \},
\]
where $R(i)$ is  the absolute least residue of $ai \pmod p$,
and thus $ -(p-1)/2 \leq r(i) \leq (p-1)/2$,
with 
\[
 R(i) \equiv i \cdot a \equiv a(i) r(i) \pmod p,
\]
and $a(i) \equiv \pm 1 \pmod p$.
Then the following holds.
\[
 \leg{a}{p} = (-1)^{|T|}.
\]
\end{thm}
\begin{proof}
Pick two elements of $I(a)$.
If two multiples of $a$ say $ia$ and $ja$, $1 \leq i \neq j \leq (p-1)/2$,
 are congruent $\pmod p$ i.e.
$ia \equiv  ja \pmod p$  then $i \equiv  j \pmod p$, which leads
to $i=j$.
Likewise if $ia \equiv -ja \pmod p$  then $i \equiv -j \pmod p$,
or $(i+j) \equiv 0 \pmod p$, which leads to $\dv{p}{i+j}$ which
is impossible since $p$ is prime and $i+j \leq p-1 < p$.
Thus the absolute values of the multiples should be distinct.
That is $|a|, |2a|, \ldots , |(p-1)/2 a|$ are distinct, and form
a permutation of $1, \ldots , \frac{p-1}{2}$.
Multiplying the multiples the first way we get $a^{(p-1)/2} ((p-1)/2)!$.
Multiplying them together the other way we get $ (-1)^{|T|} ((p-1)/2)!$.
Equating the two we get $a^{(p-1)/2} \equiv (-1)^{|T|} \pmod p$, i.e.
$\leg{a}{p} =(-1)^{|T|}$.
\end{proof}

\begin{exa}
Let $a=2 \in \mb{Z}_7^*$ and $p=7$. Use Gauss's lemma
to calclulate $\leg{a}{p}$.
\end{exa}
\begin{solution}
\[
I(a) = \{ 2 \cdot 1, 2 \cdot 2 , 2 \cdot 3  \}
     = \{ 2  , 4 , 6   \}
     = \{  2, 4-7, 6-7 \pmod 7 \}
     = \{  2, -3 , -1  \pmod 7 \}
\]
\[
T =   \{   -3 , -1 \pmod 7         \} =
      \{  R(i) : R(i) < 0 , 1 \leq i \leq \frac{p-1}{2} \}
\]
and $|T|=2$.
Furthermore,
\[
\prod_{x \in I(a)} R(i) =
\prod_{1 \leq i \leq \frac{p-1}{2}} a i =
   a^{\frac{p-1}{2}}   \cdot  \frac{p-1}{2}!  = 2^3 \cdot 3!  =
   =2^{\frac{p-1}{2}} \cdot 3 !
   = \leg{2}{p} \cdot 3! .
\]
Moreover,
\[
\prod_{x \in I(a)} R(i)   \equiv
\prod_{x \in I(a)} a(i) r(i)  \equiv
   (-1)^{|T|}  \frac{p-1}{2}!
  = (-1)^{2  } \cdot 3!
  = (-1)^{|T|} \cdot 3!
\]
Equating the two products and cancelling
$ \frac{p-1}{2}!$ since $\gcd (p, \frac{p-1}{2}! )=1$, the
result follows.
\end{solution}

\begin{prp}[Gauss lemma reformulation]
\label{thm62}
Let $p$ be an odd prime.
For $a \in \mb{Z}_p$ consider $M (a) = \{ a, 2a, \ldots , ((p-1)/2) a \}$.
Let $q$ be the number of values of $M(a)$ that are greater than $p/2$.
Then
\[
 \leg{a}{p} = (-1)^q
\]
\end{prp}

The original set of values of $M(a)$ can be reduced to belong to an interval
$(-p/2,p/2)$. Then the number $q$ of values greater than $p/2$ becomes equal
to the number of negative values.

\begin{proof}
If two multiples of $a$ say $ia$ and $ja$ are congruent $\pmod p$ i.e.
$ia \equiv  ja \pmod p$  then $i \equiv  j \pmod p$. Likewise if
$ia \equiv -ja \pmod p$  then $i \equiv -j \pmod p$.
Thus the absolute values of the multiples should be distinct.
That is $|a|, |2a|, \ldots , |(p-1)/2 a|$ are distinct.
Multiplying the multiples the first way we get $a^{(p-1)/2} ((p-1)/2)!$.
Multiplying them together the other way we get $ (-1)^q ((p-1)/2)!$.
Equating the two we get $a^{(p-1)/2} \equiv (-1)^q \pmod p$, i.e.
$\leg{a}{p} =(-1)^q$.
\end{proof}

\begin{exa}
Let $a=2 \in \mb{Z}_7^*$ and $p=7$. Use Gauss's lemma (other form)
to calclulate $\leg{a}{p}$.
\end{exa}
\begin{solution}
For $\mb{Z}_7 $, all the multiples of 2 are
$\{ 2,4,6, 1, 3, 5 \}$ and $M(2) =\{ 2,4,6 \}$.
Two of them are greater than $p/2$ i.e. greater than or equal to 4.
This leads to $\leg{a}{p} = (-1)^2 =1 $.
%If we multiply the first three multiples of the former or all the
%elements of the latter, we get $2^3 \cdot 3! $.
%Using the equivalent form
%of utilizing negatives, we get $\{ 2,-3,-1,1,3,-2 \}$ or
%$M(2) =\{ 2,-3,-1 \}$.
%Notice that each one of $1, \ldots , (p-1)/2$ appears once in its positive
%or negative form among the first $(p-1)/2$ numbers and once in the next batch
%of numbers.
%If we multiply the first three multiples, we get $(-1)^2 \cdot 3! $.
%Thus $2^3 \equiv (-1)^2 \pmod 7$.
%The $2^3$ (Euler's identity) is an indicator
% of the quadratic residuosity of $2$.
%We can find this by counting the negatives in the first 3 multiples of 2.
\end{solution}

A more comprhensive proof is provided later.
\bigskip %THEOREM 63 %%
\begin{thm}
\label{thm63}
If $p$ is an odd prime then
\[
\leg{2}{p} = 1 \mathbf{\; if \;} p \equiv \pm 1 \pmod 8 , \quad \quad \quad
\leg{2}{p} =-1 \mathbf{\; if \;} p \equiv \pm 3 \pmod 8 \\
\]
\end{thm}

\begin{proof}
Let $a=2$ and consider $M(2)$.
There are $(p-1)/2$ multiples and $\floor{(p-1)/4}$ are less than $p/2$
and thus $(p-1)/2 - \floor{(p-1)/4}$ are greater than $p/2$ or negative.
If $p \equiv 1 \pmod 8$ i.e.
$p=8k+1$ the $(p-1)/2 =4k$. Then $0 < 2i \leq (p-1)/2$ if and only if
$0 < 2i \leq 4k$ i.e. $0 < i \leq 2k$. Then $q=2k$.
So $\leg{2}{p} = (-1)^{2k} =1$.
Other cases are proven similarly.
\end{proof}

\subsection{Eisenstein's theorem and the Legendre symbol}

\begin{thm}[Eisenstein]
If $p>2$ is an odd prime number and $a \neq p$ is odd and
thus $\gcd(a,p)=1$, then show
\begin{equation}
%was thm64
\label{eisenstein}
 \leg{a}{p} = (-1)^{\sum_{i=1}^{\frac{p-1}{2}} \floor{\frac{ia}{p}}} ,
\end{equation}
using Gauss's Lemma.
\end{thm}
\begin{proof}
For each $i=1,2, \ldots , (p-1)/2$
we have by the division theorem the following.
\begin{equation}
\label{eisena}
 ia = q(i) p + r(i),
\end{equation}
where
$q(i) = \floor{ia/p}$ and $0 \leq r(i) < p$.
If $0 \leq r(i) \leq (p-1)/2$ then let $s(i)=r(i)$.
If $(p-1)/2 < r(i) <p $ then let $s(i)=r(i)-p$ or
equivalently $r(i) = s(i)+p$.
Then, $R(i)$ of Gauss's Lemma is $R(i)=s(i)$ for
all $i=1,2, \ldots , (p-1)/2$.
By Gauss's Lemma we obtain the following.
\[
\leg{a}{p} = (-1)^{|T|} ,
\]
where $|T|$ is the number of $R(i) < 0$ or equivalently
the number of $r(i)> (p-1)/2$.
We now add-up Eq.(\ref{eisena}) for all relevant $i$.
\begin{eqnarray}
\label{eisen1}
 \sum_{i=1}^{\frac{p-1}{2}} (ia) &=&
  p \sum_{i=1}^{\frac{p-1}{2}} \floor{\frac{ia}{p}} +
    \sum_{i=1}^{\frac{p-1}{2}} r(i) \Leftrightarrow \\
 a \cdot \sum_{i=1}^{\frac{p-1}{2}} i &=&
 p \cdot \sum_{i=1}^{\frac{p-1}{2}} \floor{\frac{ia}{p}} +
         \sum_{i=1}^{\frac{p-1}{2}}  s(i) +p |T|
\Leftrightarrow \nonumber\\
 a \cdot \sum_{i=1}^{\frac{p-1}{2}} i \pmod 2 &\equiv&
 p \cdot \sum_{i=1}^{\frac{p-1}{2}} \floor{\frac{ia}{p}} +
         \sum_{i=1}^{\frac{p-1}{2}}  s(i) +p |T|  \pmod 2
\Leftrightarrow \nonumber\\
\label{eisen2}
         \sum_{i=1}^{\frac{p-1}{2}} i \pmod 2 &\equiv&
         \sum_{i=1}^{\frac{p-1}{2}} \floor{\frac{ia}{p}} +
         \sum_{i=1}^{\frac{p-1}{2}}  s(i) +  |T|  \pmod 2
\end{eqnarray}
We note above that $a,p$ are odd and thus $a \equiv 1 \pmod 2$
and $p \equiv 1 \pmod 2$.
Furthermore, in Eq.(\ref{eisen2}) the $\sum_i s(i)$ values are
$\sum_i (\pm i)$. But mod $2$ we have $i \equiv i \pmod 2$ and
also $i \equiv -i \pmod 2$. Thus the two sums cancel out mod $2$.
So we can rewrite Eq.(ref{eisen2}) as follows.
\begin{eqnarray}
\label{eisen3}
 \sum_{i=1}^{\frac{p-1}{2}} (ia) &=&
  p \sum_{i=1}^{\frac{p-1}{2}} \floor{\frac{ia}{p}} +
    \sum_{i=1}^{\frac{p-1}{2}} r(i) \Leftrightarrow \\
\label{eisen4}
         \sum_{i=1}^{\frac{p-1}{2}} i \pmod 2 &\equiv&
         \sum_{i=1}^{\frac{p-1}{2}} \floor{\frac{ia}{p}} +
         \sum_{i=1}^{\frac{p-1}{2}}  s(i) +  |T|  \pmod 2
\Leftrightarrow \\
                                            0 &\equiv&
         \sum_{i=1}^{\frac{p-1}{2}} \floor{\frac{ia}{p}} +
                                        0 +  |T|  \pmod 2
\Leftrightarrow \nonumber\\
\label{eisen5}
                                       |T|    &\equiv&
         \sum_{i=1}^{\frac{p-1}{2}} \floor{\frac{ia}{p}} \pmod 2 .
\end{eqnarray}
We plus the right hand expression of Eq.(\ref{eisen5}) into Gauss's
Lemma and the result follows.
\end{proof}

The sum of the exponent has a nice
geometric interpretation.
It is the number of lattice points under the
line $y=\frac{a}{p}x$ that are
over the $x$ axis between
$x=0$ and $x=p/2$.

\subsection{Applications of Eisenstein's theorem: Legendre symbol}

\begin{prp}
Let $p,q$ be odd primes and let $p \neq q > 0$, then the
following holds.
\[
 \leg{p}{q} \leg{q}{p} = (-1)^{\frac{p-1}{2} \frac{q-1}{2}}
\]
\end{prp}
\begin{proof}
The proof uses Eisenstein theorem and its lattice point
interpretation of the sum of its exponent.
$ $ \\ $ $
Let $p,q$ be odd primes.
Let $r =  \sum_{i=1}^{(p-1)/2} \floor{\frac{iq}{p}}$
be the number of lattice points below
$y=\frac{q}{p}x$ and over the $x$ axis and between
$x=0$ and $x=p/2$.
$ $ \\ $ $
Similarly,
let $s =  \sum_{i=1}^{(q-1)/2} \floor{\frac{ip}{q}}$
be the number of lattice points below
$x=\frac{p}{q}y$ and over to the right of the $y$ axis
and between $y=0$ and $y=q/2$.
$ $ \\ $ $
The line  $y=\frac{q}{p}x$ and $x=\frac{p}{q}x$
are the same.
$ $ \\ $ $
None of the two set of points are
double counted as they lie on different
areas of the dividing line.
No point lies on the line as then $xq=py$
and thus $x$ is a multiple
of $p$ and $y$ a multiple of $q$.
$ $ \\ $ $
The number of points are all inside the rectangle
defined by $x=p/2$ and $y=q/2$
and the two axes. The total number of points
is $(p-1)/2 \cdot (q-1)/2$.
Thus $r+s = (p-1)/2 \cdot (q-1)/2$.
$ $ \\ $ $
By Eisenstein's Theorem
$\leg{p}{q} = (-1)^r$ and $\leg{q}{q}= (-1)^s$.
Thus
\[
\leg{p}{q} \leg{q}{q} = (-1)^r (-1)^s = (-1)^{r+s} =
 (-1)^{\frac{p-1}{2} \frac{q-1}{2}}
\]
\end{proof}

\begin{cor}
The following conjectured by Euler was proven by Gauss.
\begin{equation}
\label{gauss1a}
\leg{p}{q} =
\begin{cases}
- \leg{q}{p}  & p \equiv q \equiv 3 \pmod 4  \\
\leg{q}{p}    & \text{otherwise}
\end{cases}
\end{equation}
\end{cor}

\begin{proof}
Cases analysis for $p,q$.
If $p=4k+3$ and $q=4l+3$, then the product
$\frac{p-1}{2} \frac{q-1}{2}$ is an odd number.
In any other case, if one or both of $p,q$ is $4k+1$
the product is an even number. The result then follows.
\end{proof}

Equivalently, if either $p$ or $q$ is of the form $4k+1$ then
$\leg{p}{q}= \leg{q}{p}$, otherwise $\leg{p}{q}=- \leg{q}{p}$.

\begin{lem}
For $p$ an odd prime $>2$ the following applies.
\begin{equation}
\label{leg2p}
\leg{2}{p} = (-1)^{\frac{p^2 -1}{8}} .
\end{equation}

An Eisenshtein-based proof involving the $n$-th
roots of unity and in particular $8$-th roots of unity.
$w_8^8 =1$, $z= \sqrt{i} =w_8 $, where $i^2 =- 1$,
and $z^2 =i$ and $z^4 = i^2 =-1$, and thus $z^8 = w_8^8 =1$.
Moreover $w = z + 1/z$ is such that $w^2 = z^2 + 2 + 1/z^2 = 2+i-i=2$,
and thus $\sqrt{w}=\sqrt{2}$.

\end{lem}
\begin{proof}
By the reciprocity theorem (Euler criterion), we have the following;
note that in the third step we utilize the binomial theorem $(a+b)^n$
for $a=z$, $b=1/z$ and $n=p$.
\begin{eqnarray*}
\leg{2}{p} &=&  2^{\frac{p-1}{2}} = ( \sqrt{2} )^{p-1} = w^{p-1}  \\
           &=& w^p w^{-1} = (z+\frac{1}{z})^p w^{-1} \\
           &\equiv& (z^p + z^{-p} ) w^{-1} \pmod p.
\end{eqnarray*}
From the exponent $ \frac{p^2 -1}{8}$ we distinguish the following
cases: (1) $p = \pm 1 \pmod 8$ and
       (2) $p = \pm 3 \pmod 8$.
$ $ \\ $ $
{\bf Case 1: $p=8k \pm 1$.}
Let $p=8k+1$. Noting that $z^8 = z^{8k} =1$ we
have the following.
\[
z^p + z^{-p} = z^{8k+1} +z^{-8k-1} = z + 1/z =w.
\]
{\bf Case 1b.} Let $p=8k-1$.
Similarly we have the following.
\[
z^p + z^{-p} = z^{8k-1} +z^{-8k+1} = 1/z + z =w.
\]
We then conclude the following
\[
\leg{2}{p} \equiv (z^p + z^{-p} ) w^{-1}
           \equiv  w \cdot w^{-1}
           \equiv  1
           \equiv (-1)^{\frac{p^2 -1}{8}} \pmod p,
\]
as needed.
$ $ \\ $ $
{\bf Case 2: $p=8k \pm 3$.}
Let $p=8k+3$. Noting that $z^8 = z^{8k} =1$ we
have the following.
\[
z^p + z^{-p} = z^{8k+3} +z^{-8k-3}  =  z^3 + 1/z^3 = -w.
\]
This is because
\[
w=z + 1/z \Rightarrow w^3 = (z+1/z)^3 = z^3 + 1/z^3 + 3z +3/z
          \Rightarrow z^3 + 1/z^3 = w^3 - 3w = w^2 w -3w = 2w-3w=-w.
\]
The $p=8k-3$ is proven similarly.
We then conclude the following
\[
\leg{2}{p} \equiv (z^p + z^{-p} ) w^{-1}
           \equiv  -w \cdot w^{-1}
           \equiv  -1
           \equiv (-1)^{\frac{p^2 -1}{8}} \pmod p,
\]
as needed.
\end{proof}

\section{Pythagorean triplets}

Let $a^2 + b^2 = c^2$ be the pythagorean identity with $a \leq b \leq c$.
A triplet $(a,b,c)$ is a solution to the identity. It is called primitive
if $\gcd(a,b)=1$. If $(a,b,c)$ is a triplet so is $(ma,mb,mc)$ for every
integer $m$.

One way to generate triplets is to use the identity
$(x+y)^2 = x^2 +2xy + y^2$
and
$(x-y)^2 = x^2 -2xy + y^2$
which implye
$(x+y)^2 - (x-y)^2 = 4 xy$.
Set $x=X^2$ and $y=Y^2$ we have

Thus
\bigskip %THEOREM 70 %%
\begin{thm}
\label{thm70}
Thus $a=2XY$, $b=X^2 -Y^2$ and $c=X^2 +Y^2$ or equivalently
$(a,b,c)=(2XY,X^2 - Y^2 , X^2 + Y^2 )$ is a pythagorean triplet,
since
\[
 (X^2 +Y^2 )^2 = (2XY)^2 + (X^2 -Y^2 )^2
\]
Moreover $Z(a,b,c)$ are other triplets, for $X,Y,Z \in \mb{Z}$.
\end{thm}

\bigskip %THEOREM 71 %%
\begin{thm}
\label{thm71}
If $p$ is a prime factor of $n$ with $p \equiv 3 \pmod 4$ and
$p$ divides $n$ an odd number of times. Then $n$ cannot be expressed
as a sum of squares.
\end{thm}

\begin{proof}
Let us assume the theorem is false. Let $n$ be the smallest $n$ that
is a counterexample and thus $n=a^2 +b^2 $.
Since $\dv{p}{n}$ we have $a^2 + b^2 \equiv 0 \pmod p$ i.e.
$b^2 \equiv -a^2 \pmod p$.
If $\ndd{p}{a}$ then $\ndd{p}{b}$ and thus $a^{-1}$ and $b^{-1}$ exist.Thus
$(ba^{-1})^2 \equiv -1 \pmod p$. This means $-1$ is a q.r $\pmod p$.
Which contradicts the non-being so since $p \equiv 3 \pmod 4$.

Thus $\dv{p}{a}$ and $\dv{p}{b}$ and thus $\dv{p^2}{n}$.
The $a=p a_1 , b=p b_1$ $n = p^2 n_1$. We get
that $n_1^2 = a_1^2 + b_1^2$.
\end{proof}
A linear congruential generator (LCG) is one such that
$x_{i+1} \equiv a x_i + b \pmod  n$, where $\gcd(a,n)=1$.
If $\gcd(a-1,n)=1$ the $x_0$ should be chosen so that
$\gcd(x_0 - b (1-a)^-1 , n)=1$

The period of LCG is its modulus $n$ if and only if
$\gcd(b,n)=1$, $a\equiv 1 \pmod p$ for every
prime $p$ such that $\dv{p}{n}$,
and $a \equiv 1 \pmod 4$ if $\dv{4}{n}$.

A Blum-Blum-Shub (BBS) sequence is one where $n=pq$ and $p,q$ are primes
such that $p\equiv q \equiv 3 \pmod p$. A seed $x_0$ is chosen so that
$\gcd( x_0 , n)=1$. Then $x^2_{i+1} \equiv x^2_i \pmod n$.
The output is $x_i \pmod 2$. For a long
period $\gcd( \phi (p-1), \phi (q-1))$ is
small compared to $n$.

\section{Public Key Cryptography}

A very brief overview of applications of number theory
is given in this section

\subsection{Diffie-Hellman key exchange}

It uses $\mb{U}_p$ exponentiation.
Choose a large prime $p$, and an element $g\in \mb{U}_p$,
where $g$ is preferably a primitive root.
This information is  public.
The following two pieces of information
are secret for each party involved.
Alice chooses a secret exponent $1\leq a \leq p-1$.
Bod   chooses a secret exponent $1\leq b \leq p-1$.
Alice publishes $g^a \pmod p $ and
Bob $g^b$ $\pmod p$.
The other party picks the other's published info and
compute $g^a g^b \pmod p$. Only they know both
multiplicands and thus their product.
% 5 x 3 ie multiply 5  three times
% 5 is multiplicand , 3 is multiplier
The only way to retrieve from $g^a , g^b , g , p$ the
$a$ or $b$ is by a slow discrete logarithm process.

\subsection{RSA}

Let $p,q$ are two large primes $p\neq q$. Let $n=pq$ and choose
an $e$ such that
\[
\gcd( e, \phi(n)) = \gcd ( e , (p-1)(q-1) ) =1
\]

$ $\\ $ $
{\bf 1. Public key.} It is the pair $(e,n)$.
The modular  equation
\[
  ed \equiv 1 \pmod  {(p-1)(q-1)}
\]
is then solved. Because
$ \gcd ( e , (p-1)(q-1) ) =1$, there exists one and only one
solution $\bmod \phi(n)$.
$ $\\ $ $
{\bf 2. Private key.} It is the  pair $(p,q)$ or for convenience the
triplet  $(d,p,q)$
($d$ is not needed as it can be recomputed from the public key
and private $p$,$q$).

The following information is public or private for a a party, say
Alice. (Similarly for the other part, Bob.)
$ $ \\  $ $
\bigskip\noindent
{\bf Alice's Public Information:}  $(e_A ,n_A)$
$ $ \\  $ $
\bigskip\noindent
{\bf Alices' Private Information:} $(d_A ,p_A ,q_A )$.

\medskip\noindent
\noindent
{\bf Communication from Alice to Bob: sending encrypted message $M$.}
In order for Alice to send Bob a message $M$, Alice must
retrieve Bob's Public information $(e_B , n_B )$.
Then for  message $M$ Alice
computes $C \equiv M^{e_B} \pmod n_B$.
It transmits $C$ to Bob, not $M$.

\medskip\noindent
{\bf Communication received by Bob: recovery of $M$}
Bob received $C$ from Alice.
He retrieves his private information $(d_B , p_B , q_B )$.
He then performs the following computation (note $n_B = p_B q_B$).
\[
 C^{d_B} \equiv (M^{e_B})^{d_B} \equiv M^{e_B d_B} \equiv M \pmod{n_B}
\]

The correctness of the decruption process $B$ follows by way of
Euler's theorem.
Bob chose $n_B = p_B q_B$ where
\[
  e_B d_B \equiv 1 \pmod  {(p_B -1)(q_B -1)}
\Leftrightarrow
  e_B d_B \equiv 1 \pmod  {\phi (n_B )}
\]
Therefore there exists a $k \in \mb{Z}$ such that
\[
 e_B d_B -1 =  k \phi (n_B ) .
\]
Then from the decryption process, we have
\[
M^{e_B d_B} \equiv
M^{e_B d_B -1 } \cdot M \equiv
M^{k \phi (n_B ) } \cdot M \equiv
(M^{\phi (n_B ) })^k \cdot M \equiv
(1)^k \cdot M \equiv M \pmod{n_B},
\]
where we used Euler's theorem for
\[
M^{\phi (n_B )} \equiv 1 \pmod{n_B} .
\]

\bigskip
RSA's difficulty relies on the perceived difficulty of factoring $n$
into $p,q$ and thus computing $d$. Equivalently on computing $d$ from
$e,n$ alone without factoring $n$.

For RSA message $M$ must be close to the size of $\phi (n)$.
Thus padding may need to be performed if $M$ is small (or an attacker
may rely on brute force techniques).
Because of these, RSA is primarily being used to transmit secret keys,
and  other methods are used for transmitting messages such as $M$.

\chapter{Primality testing}

\section{Carmichael numbers}

\begin{dfn}[Carmichael numbers]
A composite natural number $n>1$, $n \in \mb{N}$,
is a Carmichael number 
if it satisfies the following.
\[
 a^{n} \equiv a \pmod n, \quad \forall  a \in \mb{Z} .
\]
\end{dfn}

\begin{dfn}[Carmichael numbers; an alternative definition]
A composite natural number $n>1$, $n \in \mb{N}$,
is a Carmichael number 
if it satisfies the following.
\[
 a^{n-1} \equiv 1 \pmod n, \quad \forall  a \in \mb{Z} : \gcd(a,n)=1.
\]
\end{dfn}

\begin{exa}
Integer $n=561$ is a Carmichael number.
Show that $n=561$ is not a prime number by finding an $a$
such that
\[
 a^{n-1} \not\equiv 1 \pmod n, 
\]
that's. a violation of Fermat's little theorem is established.
How much is 
\[
 a^{n-1} \pmod n, 
\]
then?
\end{exa}

\begin{solution}
$ $ \\ $ $
Consider $n=561$. It is obviously a composite number
as $n$ is divisible by 3. Thus $n$ is not a prime number.
For a prime number $n$,
$a^{n-1} \equiv 1 \pmod n$ for all $a=1,2, \ldots , n-1$,
by way of Fermat's little theorem.
$ $ \\ $ $
Consider $a=2$. It can be shown that 
$2^{560} \equiv 1 \pmod{561}$.
Obviously then
$2^{561} \equiv 2 \pmod{561}$.
$ $ \\ $ $
Consider now $a=3$.
$3^{560} \equiv 375 \pmod{561} \not\equiv 1 \pmod{561}$,
and a  violation of Fermat's little
Theorem is obtained for the premise $n=561$ is
a prime number.
Note however that $3^{561} \equiv 3 \pmod{561}$.
Even though $561$ is a Carmichael number and for
every $a$ with  $\gcd(a,561)=1$, we have
$ a^{560} \equiv  1 \pmod{561}$, note that
$a=3$ has $\gcd( 3, 561) =3 \neq 1$.
%%%%%%%%%%%%%%%%%%%%%%%%%%%%%%%%%%%%%%%%%%%
\end{solution}

\subsection{Korselt's theorem}

It is also known as Korselt's criterion.
It characterizes Carmichael numbers.
A positive (composite) integer $n$ is a
Carmichael number if and only if it is
square free and for all prime divisors $p$ of
$n$ we have $\dv{p-1}{n-1}$.
One can directly and obviously use the Korselt theorem 
to show that all Carmichael numbers are odd.

The proof is deferred after the proof of two technical
lemmas.

\begin{lem}
All Carmichael numbers are odd.
\end{lem}
\begin{proof}
$ $ \\ $ $
(Using Korselt's theorem.)
$ $ \\ $ $
Say $n$ is a Carmichael number and it is even.
Then $n=2m$, where $m$ is odd since a 
Carmichael number is square-free and thus $m$
can not be even. Then $m$ has at least one
prime factor, let it that be $p$, and it must be odd
(including other ones). We thus have 
$\dv{p}{m} \Rightarrow \dv{p}{n}$. Then by Korselt's
theorem $\dv{p-1}{n-1}$ would imply that an
even number $p-1$ divides an odd number $n-1$, a contradiction.
$ $ \\ $ $
(Another proof without using Korselt's theorem.)
$ $ \\ $ $
Say $n$ is even.
Consider $a=(-1)$ and by Carmichael numbers definition
we have the following.
\[
(-1)^n \equiv (-1) \pmod n
\Rightarrow
1 \equiv (-1) \pmod n
\Rightarrow
2 \equiv 0 \pmod n
\]
This implies $\dv{n}{2}$. Thus $n \leq 2$. For a positive
$n$ this means $n=1$ or $n=2$. A Carmichael number is $n>1$.
Thus the only case left is $n=2$. But a Carmichael number
is composite and two is prime. Thus $n$ cannot be even,
and must be odd.
\end{proof}

\begin{lem}
Let $n= p_1 p_2 $,  $p_1 \neq p_2$,
where $p_1 , p_2 \in \mb{Z}_+$ are prime greater than 2.
Let $ a\in \mb{Z}$ such that
\[
a^{n-1} \equiv 1 \pmod{p_1}
\]
and
\[
a^{n-1} \equiv 1 \pmod{p_2}.
\]
Then,  the following applies.
\[
a^{n-1} \equiv 1 \pmod{p_1 p_2} .
\]
\end{lem}

The result can be generalized by using induction on  integer
$n$ having more than two prime factors.

\begin{proof}
$ $ \\ $ $
If $a$ is as given then there exist $K,L$ such that
\[
a^{n-1} - 1 = K p_1 , \quad
a^{n-1} - 1 = L p_2 .
\]
Since $p_1 , p_2$ are prime different from each
other $\gcd( p_1 ,  p_2 ) =1$, and therefore,
there exists $x,y$ integer such that
\[
 x p_1  + y p_2 = 1.
\]
Multiply both sides by $ a^{n-1} - 1$.
We derive the following.
\[
 ( a^{n-1} - 1) ( x p_1  + y p_2 ) = a^{n-1} - 1 
\Rightarrow
 xL p_1 p_2 + yK p_1 p_2 = a^{n-1} - 1,
\]
from which we conclude
$a^{n-1} \equiv 1 \pmod{p_1 p_2}$.
\end{proof}

\begin{thm}[Korselt]
\label{korseltT}
A positive (composite) integer $n$ is a
Carmichael number if and only if it is
square free and for all prime divisors $p$ of
$n$ we have $\dv{p-1}{n-1}$.
\end{thm}
\begin{proof}
$ $ \\ $ $
$\Rightarrow$. We first prove the sufficient condition.
Let $n \in \mb{Z}_+$, with $n > 1$ that is compositive
and a Carmichael number. The following then apply
by definition
\[
 a^{n-1} \equiv 1 \pmod n, \quad \forall  a \in \mb{Z} : \gcd(a,n)=1,
\]
or
\[
 a^{n} \equiv a \pmod n, \quad \forall  a \in \mb{Z} .
\]
We will then show the following:
\begin{itemize}
\item[(1)]  $n$ is odd,
\item[(2a)] for every prime divisor $p$ of $n$  we have 
$\ndv{p^2}{n}$, that is, $n$ is square-free, and
\item[(2b)] for every prime divisor $p$ of $n$  we have 
$\dv{p-1}{n-1}$.
\end{itemize}
$ $ \\ $ $
Part (1) has been proven in another problem. Set $a=-1$ for
example.
We then show (2a) and (2b) as follows.
Let by the fundamental theorem of arithmetic
\begin{equation}
\label{korn}
  n = p_{1}^{a_1} p_{2}^{a_2} \ldots p_{k}^{a_k},
\end{equation}
be the prime decomposition of $n$, where $a_i \geq 1$ for
all $i=1, \ldots , k$. We will show that
$a_i =1 $ for all $i=1, 2, \ldots , k$. 
In the remainder $i$ runs  $i=1, 2, \ldots , k$. 
So does $j \neq i $ if needed.
Note that 
$\gcd( p_i , p_j )=1$ for all $i,j$.
By prior results on primitive roots, there exist a primitive
root $b_i$ mod $p_i^{a_i}$ for all $i$.
Then 
\begin{equation}
\label{kor0}
 ord_{p_i^{a_i}} ( b_i ) = \phi (p_i^{a_i} ) = p_i^{a_i -1} (p_i -1 ),
\end{equation}
where $\gcd (b_i , p_i )=1$.
Since $\gcd( p_i , p_j )=1$ by the C.R.T. we have that there exists
an $a$ such that
\begin{equation}
\label{kor1}
  a \equiv b_i \pmod{p_i^{a_i}}  \ \forall i
\end{equation}
Note that $\gcd( a, p_i )=1$ since otherwise
 $\gcd (b_i , p_i ) >1$, a contradiction to $\gcd (b_i , p_i )=1$.
Since all $p_i$ are prime numbers and relatively prime to each
other we have the following by a prior problem.
\begin{equation}
\label{kor2}
\gcd (a,n)=1.
\end{equation}
This will be used to
derive the second equation below.
By the Carmichael number definition we have the following for 
all $a \in \mb{Z}$ and all $i$.
\begin{eqnarray}
a^n &\equiv& a \pmod{n} \xRightarrow{\text{by Eq.(\ref{kor2})}} \nonumber \\
a^{n-1} &\equiv& 1 \pmod{n} \xRightarrow{\dv{p_i^{a_i}}{n}} \nonumber \\
a^{n-1} &\equiv& 1 \pmod{p_i^{a_i}} \xRightarrow{\dv{p_i}{n}} \nonumber \\
\label{kor3}
b_i^{n-1} &\equiv& 1 \pmod{p_i^{a_i}} 
\end{eqnarray}
By way of Eq.~(\ref{kor0}) and Eq.~(\ref{kor3}) we 
obtain the following.
\begin{equation}
\label{kor4}
\dv{ p_i^{a_i -1} (p_i -1 )}{n-1} .
\end{equation}
From Eq.(\ref{korn}) we have  $\dv{p_i^{a_i}}{n}$
and given $a_i \geq 1$,
$\dv{p_i^{a_i -  1}}{n}$.
Furthermore from Eq.(\ref{kor4}) we have
$\dv{p_i^{a_i -  1}}{n-1}$.
Combining the two we obtain
$\dv{p_i^{a_i -  1}}{1}$.
This is only possible for $a_i =1$ and this is true for all $i$.
This establishes the square-free consition (2a).
$ $ \\ $ $
Continuing, $a_i = 1$ and Eq.(\ref{kor4} imply
$\dv{p_i -1}{n-1}$ for all $i$, conditional on $\dv{p_i}{n}$
and this establishes condition (2b).
The sufficient part of Korstelt's theorem has been shown.
$ $ \\ $ $
$\Leftarrow$ We now show the necessary part.
Assuming (1), (2a) and (2b) we proceed to show that $n$ is
a Carmichael number.
$ $ \\ $ $
By the fundamental theorem of arithmetic,  since 
$\ndv{p^2}{n}$ for every prime factor $p$ of $n$ we have
the following
\begin{equation}
\label{kor5}
 n = p_1 \ldots p_k ,
\end{equation}
and of course $p_i > 2$ for all $i=1 , \ldots  , k$.
Moreover $\dv{p_i -1}{n-1}$ by way of condition~(2b).
We consider two cases.
$ $ \\ $ $
{\bf Case 1.} Pick $a \in \mb{Z}$ with $\gcd (a,p_i )=1$.
By Fermat's little Theorem 
\begin{eqnarray}
a^{p_i -1 } &\equiv& 1 \pmod{p_i} \quad \forall i 
\xLeftrightarrow{\dv{p_i -1}{n-1} : (n-1)=(p_i -1)K} \nonumber \\
a^{(p_i -1)K } &\equiv& 1 \pmod{p_i} \quad \forall i 
                                                     \nonumber \\
a^{n-1} &\equiv& 1 \pmod{p_i} \quad \forall i 
\xLeftrightarrow{\gcd(a, p_i )=1} \nonumber \\
a^{n} &\equiv& a \pmod{p_i} \quad \forall i 
 \nonumber
\end{eqnarray}
$ $ \\ $ $
{\bf Case 2.} Pick $a \in \mb{Z}$ with $\gcd (a,p_i )=d > 1$.
Since $\dv{d}{a}$ and $\dv{d}{p_i}$ for all $i$, we from
the latter $d \leq p_i$ and since $d >1$ it can only be
$d=p_i$ for a prime $p_i$ for all $i$. Then
$\dv{p_i}{a}$ for all $i$. Trivially then,
\[
a^n \equiv p_i^n \equiv p_i \equiv a \equiv 0 \pmod{p_i}
\]
Therefore $a^n \equiv a \pmod{p_i}$ forall $i$ using
a prior problem leads to
$a^n \equiv a \pmod{n}$ since all $p_i$ are relatively prime
to each other by way each one b eing a prime number.
$ $ \\ $ $
{\bf An alternative proof of 
the sufficient condition $\Rightarrow$ follows.}
$ $ \\ $ $
For a Carmichael number $n$ we have for all a
$ a^n \equiv a \pmod n$.
Pick a $p$, such that $\dv{p}{n}$. Then
by way of $n$ being a Carmichael number
$ p^n \equiv p \pmod n$, or
$\dv{n}{p^n -p}$.
Since $\dv{p}{n}$ the prior result shows that
$\dv{p}{p^n -p}$.
Say $\dv{p^2}{n}$. Then
$\dv{p^2}{p^n -p}$ which would imply
$\dv{p}{p^{n-1}-1}$.
Obviously $\dv{p}{p^{n-1}}$, $n>1$.
The last two imply $\dv{p}{1}$ and 
$p=1$ a contradiction
to the primality of $p$.
Therefore $\ndv{p^2}{n}$ for $\dv{p}{n}$.
$ $ \\ $ $
In order to show that for $\dv{p}{n}$ we
have $\dv{p-1}{n-1}$ we start with the
definition of a Carmichael number.
For $a \in \mb{Z}$ we have $a^n \equiv a \pmod n$,
and for $a$ such that $\gcd(a,n)=1$ we
further have than $a^{n-1} \equiv 1 \pmod n$.
For any $p$ such that $\dv{p}{n}$ we have
as a result $a^{n-1} \equiv 1 \pmod p$.
Consider a generator $a$ of $\mb{Z}_p^*$.
We have $a^{p-1} \equiv 1 \pmod p$. From this
equivalence and the $a^{n-1} \equiv 1 \pmod n$
we conclude $(n-1)-(p-1) = k \phi (p) = k (p-1)$
that results into $\dv{p-1}{n-1}$ as needed.
\end{proof}

\begin{cor}
Show that the two definitions of Carmichael numbers
are equivalent.
\end{cor}

\begin{proof}
$ $ \\ $ $
$\Leftarrow$.
$ $ \\ $ $
Let
\[
 a^{n} \equiv a \pmod n, \quad \forall  a \in \mb{Z} .
\]
If $\gcd(a,n)=1$ we obtain
the following.
\[
a^{n} \equiv a \pmod n
\Rightarrow
a (a^{n-1} -1) \equiv 0 \pmod n
\Rightarrow
\dv{n}{a(a^{n-1} -1)} 
\xRightarrow{\gcd(a,n)=1}
\dv{n}{(a^{n-1} -1)} 
\Rightarrow
a^{n-1}  \equiv 1 \pmod n
\]
$\Rightarrow$.
$ $ \\ $ $
Let $n$ be a Carmichael number. By Korselt's theorem it
is squarefree. We need to show 
$a^{n} \equiv a \pmod n$ for all $a \in  \mb{Z}$.
Since $n$ is squarefree it is the products of mutually
distinct primes. It thus suffices to
show $ a^{n} \equiv a \pmod p$ for all prime divisors
$p$ of $n$.
$ $ \\ $ $
{\bf Case 1: $a$ is divided by $p$.}
The $a\equiv 0 \pmod p$ and  $ a^{n} \equiv a \pmod p$  is
obviously true.
$ $ \\ $ $
{\bf Case 2: $\ndv{p}{a}$.}
Since $p$ is a prime number we use Fermat's little theorem.
$a^{p-1} \equiv 1 \pmod p$.
Furthermore by Korselt's theorem $\dv{p-1}{n-1}$ and therefore
$a^{n-1} \equiv 1 \pmod p$.
\end{proof}

\begin{cor}
(a)
Show that a Carmichael composite odd number
$n$ is the product of at least three prime numbers.
$ $ \\ $ $
\noindent
(b) Moreover every prime factor is less than $\sqrt{n}$.
\end{cor}

\begin{proof}
$ $ \\ $ $
(a)
{\bf Proof 1.}
If $n$ is a Carmichael number then by Korselt's
theorem it is square free and for all prime divisors
$p$ of $n$ we have $\dv{p-1}{n-1}$.
All $p,q,n$ below are odd integer numbers.
$ $ \\ $ $
Let $n=pq$ be the product of just two primes,
where $p < q$ and $p \neq q$ since $n$ is 
squarefree. By Carmichael properties
$\dv{p-1}{n-1}$ and
$\dv{q-1}{n-1}$.
\begin{eqnarray*}
 q    &\equiv& 1         \pmod{q-1} \Rightarrow\\
 n=pq &\equiv& p \cdot 1 \pmod{q-1} \\
 n-1  &\equiv& p  -    1 \pmod{q-1} 
\end{eqnarray*}
Since $p<q$ we have $p-1 < q-1$. This means
that $p-1$ is the remainder of the division of 
$n-1$ by $q-1$ i.e. $\ndv{q-1}{n-1}$ that
contradicts the stated
$\dv{q-1}{n-1}$.
$ $ \\ $ $
(b) From the previous case
$\dv{q-1}{p-1}$ implies, since $n=pq$, that
$\dv{q-1}{n/q-1}$.
Therefore $q \leq n/q$. If $q=n/q$ it means
$n=q^2$ contradicting $n$ bein square free.
Therefore $q < \sqrt{n}$.
Furthermore $p< q < \sqrt{n}$ as well.
Note that then $pq < n$ and therefore there must
be a third prime factor $r$ for $pqr=n$, providing
sort of a third proof to (a).
\end{proof}

\begin{exa}
Prove part (a) of the Corollary above using another method.
\end{exa}
\begin{solution}
Reach the same conclusion as before, assuming $p<q$ 
by noting the following.
\begin{eqnarray*}
 n    & =    & pq               \\
 n-1  & =    & p(q-1) + (p-1)   \\
 \frac{n-1}{q-1}  & =    & p + \frac{p-1}{q-1}
\end{eqnarray*}
The $ \frac{n-1}{q-1}$ is an integer
since $\dv{q-1}{n-1}$, $p$ is obviously an
integer, and therefore $  \frac{p-1}{q-1}$
must be an integer i.e. $\dv{q-1}{p-1}$.
This would imply $q-1 \leq p-1$ i.e. $q \leq p$
that contradicts $p<q$.
\end{solution}

\newpage

\section{Lucas theorem}

\begin{thm}[Lucas theorem]
Show that $n>2$ is a prime number if
and only if the following holds for some integer $a$
such that $1 < a < n$.
\begin{eqnarray}
\label{lucas1}
 a^{n-1}           &\equiv&     1 \pmod n  \quad \wedge \quad \\
\label{lucas2}
 a^{\frac{n-1}{p}} &\not\equiv& 1 \pmod n,
\end{eqnarray}
for every prime divisor $p$ of $n-1$.
\end{thm}
\begin{proof}
\end{proof}

If no such $a$ exists then $n$ is composite or 1 or 2;
the latter two cases are eliminated by way of $n>2$.

\begin{proof}
$ $ \\ $ $
$\Rightarrow$
If $n$ is a prime number then the  $\mb{U}_n$  multiplicative 
group is a cyclic one with $n-1$ elements. It then has a generator $g$. 
Therefore $g^{n-1} \equiv 1 \pmod{n}$, but $g^i \not\equiv 1 \pmod{n}$ for
$i<n-1$. Every element $a$ of $\mb{U}_n$ can be written in the form
$a \equiv g^k \pmod n$. Then $a^{n-1} \equiv (g^{n-1})^k \equiv 1 \pmod n$.
Then $ord_n (a)$ divides $n-1$ and of course $ord_n (a) < n-1$ unless
$a$ is a generator. Thus we focus on generators such as $g$ in
the remainder.  Moreover, consider a prime divisor $p$ of $n-1$.
We must have for a generator $g$, $g^p \not\equiv 1 \pmod n$ for every $p$
since otherwise if $g^p \equiv 1 \pmod n$, $p \neq n-1$,
then we have the following.
\[
\dv{p}{n-1} \Rightarrow
\frac{n-1}{p} >1  \text{and is integer} \Rightarrow
g^{p} \equiv 1 \pmod p \Rightarrow
ord_n (g) = p < n-1,
\]
a contradiction to $g$ being a generator and thus
$ord_n (g) = n-1$.
$ $ \\ $ $
$\Leftarrow$
Consider that there exists an element $a$ such that
\[
 a^{n-1} \equiv           1 \pmod n .
\]
Then $\dv{ord_n (a)}{n-1}$. 
Moreover, $ \gcd(a,n)=d$ shows $d=1$.
Otherwise if $d>1$, $\dv{d}{a^{n-1}}$ and $\dv{d}{n}$ and
therefore $\dv{d}{1}$ and $d \leq 1$. This ($ \gcd(a,n)=1$)
implies by Euler's theorem
\[
 a^{\phi (n)} \equiv 1 \pmod n ,
\]
and thus $\dv{ord_n(a)}{\phi(n)}$.
Let $ord_n (a)=k$.
Then from the first equation above we have
the following: $\dv{k}{n-1}$.  There are two possibilities:
$k=n-1$ and $n$ is prime, or $k < n-1$.
For the latter case,  $(n-1)/k =p K >1$ is an integer 
and either a prime $p$ (and $K=1$) or the product of 
a prime $p$ and a composite $K>1$. 
Consider then
\[
a^{\frac{n-1}{p}} .
\]
By the second condition of the problem we have that for every
prime divisor of $n-1$ and thus for the prime $p$
\[
a^{\frac{n-1}{p}}  \not\equiv 1 \pmod{n} .
\]
However we also have the following.
\[
a^{\frac{n-1}{p}}   \equiv
a^{k \cdot K}       \equiv
(a^{k})^K           \equiv
      1^K           \equiv
      1  \pmod n ,
\]
leading to a contradiction.
The only case left is $k=n-1$. If the order of
an element is  $n-1$ this means $n$ is prime.
This is because it implies
$\dv{n-1}{\phi(n)}$ and $\phi(n) \leq n-1$, i.e.
$\phi(n) = n-1$ and thus $n$ is a prime number.
\end{proof}

\begin{thm}
Let $n>5$ be an odd composite number. The following
statements are equivalent.
\begin{enumerate}
\item 
\[
\forall a \in \mb{Z} , \gcd(a,n)=1 \Rightarrow
a^{\frac{n-1}{2}} \equiv \pm 1 \pmod n .
\]
\item $n$ is square-free and for every prime number
$p$ such that $\dv{p}{n}$ we have $\dv{p-1}{\frac{n-1}{2}}$ which
is equivalent to
\[
Composite(n), SquareFee(n), \forall p, Prime(p) \wedge \dv{p}{n}
\Rightarrow \dv{p-1}{\frac{n-1}{2}}.
\]
\end{enumerate}
\end{thm}

\begin{proof}
$ $ \\ $ $
{\bf (1) $\Rightarrow$ (2).} 
$ $ \\ $ $
If $ \gcd(a,n)=1$ and $a^{\frac{n-1}{2}} \equiv \pm 1 \pmod n$
then $a^{n-1} \equiv 1 \pmod n$ and $n$ is a Carmichael number.
By Korselt's theorem $n$ is square-free. Moreover for every
prime divisor $p$ of $n$ we have by Korselt's theorem also
that $\dv{p-1}{n-1}$. We thus need to show that
$\dv{p-1}{\frac{n-1}{2}}$ to complete the proof.
$ $ \\ $ $
{\bf Show $\dv{p-1}{\frac{n-1}{2}}$.}
Let $n-1=(p-1)q$ by Korselt's theorem and the fact that
$\dv{p-1}{n-1}$.
There are two possibilies: (a) $q$ is even, or (b) $q$ is
odd. 
$ $ \\ $ $
{\bf Case (a): $q$ is even.} Then $q/2$ is integer
and 
\[
n-1=(p-1) q  \Rightarrow
\frac{n-1}{2} = (p-1) \frac{q}{2}
\Rightarrow
\dv{p-1}{\frac{n-1}{2}},
\]
since $q/2$ is an integer.
$ $ \\ $ $
{\bf Case (b): $q$ is odd.}
We will show that this case is not possible. Therefore
case (a) is the only possible case that leads to the
desired outcome : $\dv{p-1}{\frac{n-1}{2}}$ as proven.
$ $ \\ $ $
If $q$ is odd,
then $n-1=(p-1)q$ and $n-1$ is even for odd $n$,
and thus we have the following.
\[
a^{\frac{n-1}{2}} \equiv a^{\frac{(p-1)q}{2}} 
                  \equiv \left( a^{\frac{p-1}{2}} \right)^q 
                  \equiv \left( \leg{a}{p}         \right)^q  \pmod p .
\]
Since $q$ is odd, the latter part works as follows.
\[
\leg{a^2}{p} = 1 \Rightarrow
\left( \leg{a}{p} \right)^q  \equiv \leg{a}{p} \pmod p
\]
By transitivity from the previous derivations we have
\[
a^{\frac{n-1}{2}} \equiv \leg{a}{p}  \pmod p .
\]
For a generator $g$ of $\mb{U}_p$ we have 
$g^{\phi(p)}  \equiv g^{p-1} \equiv 1 \pmod p$,
and also
$ g^{\frac{p-1}{2}} \equiv -1 \pmod p$,
by way of Lagrange's theorem.
Then, by the Chinese Remainder Theorem for $n=pP$ and noting
because $n$ is square-free $\gcd(p,P)=\gcd(p,\frac{n}{p})=1$ 
we can find an $a$
such that the following hold.
\[
  a \equiv g \pmod p ,
\]
and
\[
  a \equiv 1 \pmod{\frac{n}{p}} .
\]
Then
\begin{equation}
\label{sse8a}
a^{\frac{n-1}{2}} \equiv \leg{a}{p} \equiv \leg{g}{p} 
 \equiv g^{\frac{p-1}{2}} \equiv -1 \pmod p .
\end{equation}
Moreover $a \equiv 1 \pmod{\frac{n}{p}}$ and therefore
\begin{equation}
\label{sse8b}
a^{\frac{n-1}{2}} \equiv 1 \pmod{\frac{n}{p}}.
\end{equation}
By way of the consequent of part (1) we have
two possibilities: (i) $a^{\frac{n-1}{2}} \equiv 1 \pmod n$,
(ii) $ a^{\frac{n-1}{2}} \equiv -1 \pmod n$.
$ $ \\ $ $
{\bf Case (b1). $a^{\frac{n-1}{2}} \equiv 1 \pmod n$.}
We have by part (1) $a^{\frac{n-1}{2}} \equiv 1 \pmod n$ as
one possibility. If it is true  then since $\dv{p}{n}$ that
$a^{\frac{n-1}{2}} \equiv 1 \pmod p$.
By way of Eq.(\ref{sse8a}) we then conclude that
$\dv{p}{2}$ that leads to $p=2$ an impossibility since $n$
is odd, and $p$ is a factor of $n$.
The other case left from part (1) is the following one.
$ $ \\ $ $
{\bf Case (b2). $a^{\frac{n-1}{2}} \equiv -1 \pmod n$.}
Then considering $\frac{n}{p}=P$,  since $\dv{P}{n}$ we 
then conclude that
$a^{\frac{n-1}{2}} \equiv  -1  \pmod P$.
By way of Eq.(\ref{sse8b}) we also conclude that
$\dv{P}{2}$ which leads to $P=2$, an impossibility since $n$
is odd.
Therefore case (b) is impossible only case (a) is possible
and by case (a) we have already concluded $\dv{p-1}{\frac{n-1}{2}}$.
Proof completed.
$ $ \\ $ $ 
{\bf (2) $\Rightarrow$ (1).}
$ $ \\ $ $
Since $n$ is square free let $n=p_1 p_2 \ldots p_k$.
By Fermat's little theorem for every $a$ such that 
$\gcd(a,p_i )=1$ we have $a^{p_i -1} \equiv 1 \pmod{p_i}$.
Furthermore $\dv{p_i}{n}$ implies $\dv{p_i -1}{((n-1)/2)}$
or equivalently $(n-1)/2 = (p_i -1 ) k_i$ for some integer
$k_i$.
Then
\[
a^{\frac{n-1}{2}} \equiv
a^{(p_i -1 ) k_i} \equiv
(a^{p_i -1})^{k_i} \equiv
1 \pmod{p_i}.
\]
By the CRT given that $\gcd(p_i , p_j )=1$ we have
that
\[
a^{\frac{n-1}{2}} \equiv 1 \pmod{n}.
\]
\end{proof}

\newpage
\section{Preliminary results for primality testing}

\noindent
Let $\mb{Z}_p = \mb{Z}/p\mb{Z} = \left\{ 0, 1, \ldots , n-1 \right\} $.
Let $\mb{U}_p$ be the units of $\mb{Z}_p$, where $p$
is a prime number greater than two.
It is also represented by $\mb{Z}_p^x$ or
$(\mb{Z}/p\mb{Z})^x$.
Formally
\[
\mb{U}_p = \left\{ a : a \in \mb{Z}_p , a \text{\ has an inverse} \right\} .
\]

\begin{lem}
\label{uplem1}
$\mb{U}_p$ is a group under multiplication mod $p$.
\end{lem}

\begin{proof}
$ $ \\ $ $
It suffices to show that $\mb{U}_p$ is closed under
under multiplication mod $p$.
Consider $a,b \in \mb{U}_p$.
For prime $p$ and $ 1 \leq a,b <p$ we have
$\gcd(a,p)=\gcd(b,p)=1$ and thus $a^{-1} , b^{-1}$ exist.
Then
\[
(ab) \in \mb{U}_p ,
\]
since $(ab)^{-1}$ exists and it is $b^{-1} a^{-1}$ 
by way of
\[
(ab) (b^{-1} a^{-1} ) =a (b  (b^{-1} a^{-1} ) ) =a a^{-1} = 1,
\]
and
\[
(b^{-1} a^{-1} ) (ab) = (b^{-1} ( a^{-1} a ) b = b^{-1} b = 1.
\]
All other properties (e.g. associativity) hold for 
$\mb{U}_p$ by way of holding for $\mb{Z}_p$.
\end{proof}

\begin{cor}
The result is true for composite $n$ as well;
$\mb{U}_n$ is then all $i$ such that $\gcd(i,n)=1$.
\end{cor}

\begin{cor}
Use Lemma~\ref{uplem1}  to prove Fermat's little theorem.
\end{cor}
\begin{proof}
If $p$ is a prime number then
\[
\mb{U}_p = \left\{ 1, 2, \ldots , p-1 \right\} .
\]
If $\ndv{p}{a}$ then if $a > p-1$, consider $b$ such that
\[
 b= a \bmod p \Leftrightarrow  a \equiv b \pmod p,
\text{\ where\ } b \in \{ 1,2, \ldots ,p-1 \} .
\]
By Lagrange's theorem 
$\dv{ord_p (b)}{ord( \mb{U}_p )}$
i.e.
$\dv{ord_p (b)}{p-1}$, or $p-1= ord_p (b) l$.
If $ord_p (b)=k <p-1$, then
\[
b^k \equiv 1 \pmod p
\Rightarrow
b^{p-1} = b^{kl} = (b^{k})^l \equiv 1 \pmod p .
\]
Since $ a \equiv b \pmod p$, we have
\[
a^{p-1} \equiv b^{p-1} \equiv 1 \pmod p,
\]
and the result follows.
\end{proof}

\begin{lem}
\label{lemA}
Let $\mb{U}_p$ be the units of $\mb{Z}_p$, where $p$ is a prime
number greater than 2.
Let 
\[
A = \left\{ a \in \mb{U}_p : a^{\frac{p-1}{2}} \equiv 1 \pmod p \right\} .
\]
Show that $A$ is a subgroup of $\mb{U}_p$ under modular multiplication
mod $p$.
\end{lem}

\begin{proof}
$ $ \\ $ $
Consider $c,d \in A$.
It is easy to show that $cd \bmod p \in A$.
Thus the order of $A$ divides the order of $\mb{U}_p$ which is
$p-1$. Since $p$ is a prime number greater than two, the $p$ is
odd and thus $(p-1)/2$ is an integer.
We have shown that the order of $A$ is a divisor of $p-1$.
We will show that the order of $A$ can either be $p-1$ or $(p-1)/2$
first by showing that $A$ has at least $(p-1)/2$ elements.
We can then conclude that the order of $A$ is $(p-1)/2$ by showing
that it cannot be $p-1$.
$ $ \\ $ $
Consider then $b=a^2 \pmod p$. 
We show that $b^2 \in A$.
In order to show inclusion to $A$ we must show that
$ b^{\frac{p-1}{2}} \equiv 1 \pmod p$.
This is shown as follows.
\[
b^{\frac{p-1}{2}} = (a^2)^{\frac{p-1}{2}} = a^{p-1} \equiv 1 \pmod p ,
\]
by Fermat's little theorem as $\gcd(a,p)=1$ for prime $p$ and
$ 1 \leq a < p$.
Furthermore, $b$ is a quadratic residue mod $p$ since there exists an
$a$ such that $a^2 \equiv b \pmod p$.
Thus all quadratic residues are in $A$. There are $(p-1)/2$ of them.
The order of $A$ is thus at least $(p-1)/2$. It can thus
be $(p-1)/2$ or $p-1$.
We proceed to dismissing $p-1$ as a possibility.
$ $ \\ $ $.
A generator $g$ of $\mb{U}_p$ is not in $A$ since
$g^{\phi(p)} = g^{p-1} \equiv 1 \pmod p$ and
$g^{(p-1)/2} \not\equiv 1 \pmod p$. Thus $|A| < p-1$,
and therefore it must be $(p-1)/2$ if it can't be $p-1$.
%
%The set
%\[
%B = \left\{ a^2 \bmod p : a \in A \right\} ,
%\]
%contains quadratic residues mod $p$.
%Note that if $a \in A$ then $(a-p) \in A$ by using the
%binomial theorem.
%Furthermore $b \equiv a^2 \pmod p$ we have
%$(p-a)^2 \equiv a^2 \equiv b  \pmod p$
Moreover consider $t \not\in A$.
We have $t^{p-1} \equiv 1 \pmod p$ but since
$t \not\in A$, we have
$ t^{\frac{p-1}{2}} \not\equiv 1 \pmod p $,
and it should be
$ t^{\frac{p-1}{2}} \equiv -1 \pmod p $,
as the only two square roots of 
$x^2 \equiv 1 \pmod $ by Lagrange's theorem for
prime $p$ are $+1$ and $-1$ mod $p$.
Note that
$ u = t^{\frac{p-1}{2}} $ is such that
$u^2 \equiv 1 \pmod p$ and thus if $u  \not\equiv 1 \pmod p $
then it must be $u=t^{\frac{p-1}{2}} \equiv -1 \pmod p $.
\end{proof}

\subsection{Probabilistic Turing machines}

\begin{dfn}[Probabilistic Turing machine]
A probabilistic Turing machine is a non deterministic
polynomial time Turing machine $T$ such that
for each configuration of $T$ there are at most two
possible follow-up configurations (non-deterministic steps), 
and $T$ chooses which 
one of the two to pursue (take) by flipping a coin.
\end{dfn}

Thus the paths  of different configurations resembles
a binary tree. For $k$ coin flips and paths of length $k$
a given path is pursued with probability equal $1/2^k$.

\subsection{Class BPP}

\begin{dfn}[Class BPP]
Let $0 \leq \epsilon < 1/2$. A (probabilistic polynomial
time) Turing Machine $M$ decides
language $L$ with error probability $\epsilon$ if
\begin{itemize}
\item for a $w \in L$ we have 
      $Pr[\mbox{M\ accepts\ w\ }] \geq 1-\epsilon$,
\item for a $w \not\in L$ we have 
      $Pr[\mbox{M\ rejects\ w\ }] \geq 1-\epsilon$,
      or equivalently
      $Pr[\mbox{M\ accept \ w\ }] \leq \epsilon$.
\end{itemize}
BPP is the class of languages that are decided by a probabilistic
polynomial time Turing machine with error probability $\epsilon$,
as defined.
\end{dfn}

\newpage

\section{Fermat primality testing}

\subsection{A Fermat little theorem-based compositeness test}

\begin{thm}
\label{tfct1a}
Algorithm~\ref{fct1a} is a probabilistic algorithm
for testing whether $n$ is composite.
\end{thm}

\begin{proof}
$ $ \\ $ $
We note that in Algorithm~\ref{fct1a} we check for the
compositeness of $n$ for a uniformly at random drawn $a$
in the range $2,3, \ldots , n-2$ in two ways:
(a) we first perform the trivial check of whether $\gcd(a,n) >1$, 
and
(b) then check whether $a^{n-1} \not\equiv 1 \pmod n$.
If any of the two checks succeeds we not only conclude that
$n$ is composite but $a$ serves as a witness (proof) of the
compositeness of $n$.
The latter (b) is a consequence of Fermat's little theorem.
If $n$ is a prime number, then any $a$ picked through line 1 is
such that $\gcd(a,n)=1$. Therefore by Fermat's little theorem
we expect $a^{n-1} \equiv 1 \pmod n$. 
If the control expression of  line 2 evaluates to true i.e.  
there is a common divisor other than one between $a$, which is
$2 \leq a < n-1$ and $n$ indeed $n$ is a composite number.
If the control expression of  line 4 evaluates to true i.e.  
$a^{n-1} \not\equiv 1 \pmod n$ this also guarantees by Fermat's
little theorem  that the answer Composite is the correct answer. 
However, an
answer through line 5 does not guarantee that $n$ is prime:
we checked only for one $a$ in Algorithm~\ref{fct1a} and not for all
$a$. 
We prefer to call this algorithm FermatCompositenessTest
for the following reason: when the algorithm declares $n$ as
Composite the answer is always correct. Otherwise it declares
$n$ as PseudoPrime. This means $n$ can be Prime or Composite.
One test for one value of $a$ is not enough to determine
reliably the primeness of $a$.
Note that neither $1$ nor $n-1$ are picked as choices of $a$.
For both such cases the answer for $a^{n-1}$ is equivalent to
$1$ mod $n$ and provides no useful information.

\SetKwComment{Comment}{/* }{ */}
\SetKwRepeat{Do}{do}{while}
\begin{algorithm}
\KwIn{$n>4$ is odd; $a$ is in $\{ 2, 3, \ldots , n-2 \}$}
\KwOut{$n$ is composite or pseudoprime}
 Pick $ a \in \{ 2, 3, \ldots , n-2 \}$ uniformly at random  \;
 \If{$\gcd(a,n) \neq 1$} {
  \Return{$\mathbf{Composite}$}
 }
 \eIf(\tcc*[f]{Compositeness check}){$a^{n-1} \not\equiv 1 \pmod{n}$} {
  \Return{$\mathbf{Composite}$} 
  \Comment*[r]{$a$ is a witness of $n$'s compositeness}
 }{
  \Return{$\mathbf{PseudoPrime}$}
  \Comment*[r]{$n$ is either prime or composite}
 }
%\While{ $\mathbf{NotDone}$}{
%        $ t \gets t+1$ \;
%        \For(\tcc*[f]{Synchronous update step}){$i=1,2, \ldots , n$}{
%          $PR^{(t)}_i =\frac{1-d}{n} +
%               d \cdot \sum\limits_{j\rightarrow i}
%                \frac{PR^{(t-1)}_j }{outd(j)} $
%        }
%        \If(\tcc*[f]{Termination check}){Done}{
%           $\mathbf{Break}$ \;
%        }
%}
 \caption{FermatCompositenessTest(n)}
 \label{fct1a}
\end{algorithm}
\end{proof}

We strengthen the capability and reliability 
of Algorithm~\ref{fct1a} by
introducing Algorithm~\ref{fct1b} a Fermat little theorem
based primality testing algorithm.

\newpage
\subsection{A Fermat little theorem-based primality test}

\begin{thm}
\label{tfct1b}
Algorithm~\ref{fct1b} is a probabilistic algorithm
for testing whether $n$ is composite or prime.
\end{thm}
\begin{proof}
Algorithm~\ref{fct1b} is in fact $t$ runs of
Algorithm~\ref{fct1a}; the $t$ runs are  independent of
each other and each run picks potentially (but not necessarily)
a different value for $a$.
Whereas the previous algorithm picks one $a$ this one
picks $t$ $a$'s from $\{ 2, 3, \ldots , n-2 \}$ uniformly
at random. 

\begin{algorithm}
\SetKwComment{Comment}{/* }{ */}
\SetKwRepeat{Do}{do}{while}
\KwIn{$n >4$ is odd; number of runs is $t$}
\KwOut{$n$ is composite or pseudoprime}
$ i = 0 $ \;
\Do{$i<t$}{
 $i=i+1$ \;
 Pick $ a \in \{ 2, 3, \ldots , n-2 \}$ uniformly at random  \;
 \If{$\gcd(a,n) \neq 1$} {
  \Return{$\mathbf{Composite}$}
 }
 \If(\tcc*[f]{Compositeness check}){$a^{n-1} \not\equiv 1 \pmod{n}$} {
  \Return{$\mathbf{Composite}$}
  \Comment*[r]{$a$ is a witness of $n$'s compositeness}
 }
}
  \Return{$\mathbf{PseudoPrime}$}
  \Comment*[r]{$n$ is either prime or composite}
 \caption{FermatPrimalityTest}
 \label{fct1b}
\end{algorithm}
If $n$ is a prime, then any $1 \leq a < n$ has
$\gcd(a,n)=1$. Therefore by Fermat's tittle theorem
we have $a^{n-1} \equiv 1 \pmod n$.
If an $a$ can be found such that
$a^{n-1} \not\equiv 1 \pmod n$, then the $a$ becomes
a witness of the compositeness of $n$.
We shall denote it as $\text{Fw}(a,n)$ or say
$a \in \text{Fw}(n)$ to denote that $a$ is one
of Fermat witnesses of the compositeness of $n$.
Line 6 of Algorithm~\ref{fct1a} or
Line 9 of Algorithm~\ref{fct1b} report correctly this, 
if such an $a$ has been picked.
Moreover, LIne 3 of Algorithm~\ref{fct1a} or
Line 6 of Algorithm~\ref{fct1b} also report the 
case where $a$ is such that $\gcd(a,n)\neq 1$, an obvious
proof of the compositeness of $n$.
When either algorithm reports PseudoPrime this means
one of two things: (a)
that $n$ is indeed a prime number or  
(b) the test is not comprehensive enough to 
determine that $n$ is indeed a composite number.
Sometimes we call such an $a$ for which 
$a^{n-1} \equiv 1 \pmod n$
a Fermat non-witness of the compositeness of $n$.
We shall denote it with $\text{Fnw}(a,n)$ or say
$a \in \text{Fnw}(n)$.
Therefore we have the following.
\[
a \in \text{Fw}(n) :   a^{n-1} \not\equiv 1 \pmod n , \gcd(a,n)=1,
\]
and
\[
a \in \text{Fnw}(n) :   a^{n-1} \equiv 1 \pmod n , \gcd(a,n)=1
\]
For a Carmichael number n, which is a composite number, 
we have the following from a prior discussion.
%If we denote with $\text{Carmichael}(n)$ the set of $a$
%in $\{ 2,3, \ldots , n-2 \}$ that satisfy the Carmichael
If we denote with $\text{Carmichael}(n)$ the set of $a$
that satisfy the Carmichael
condition $a^{n-1} \equiv 1 \pmod n$ for $a$ such that
$\gcd(a,n)=1$ we conclude that the Carmichael(n) set
contains alll those integers.
\[
\forall a \in \mb{Z}, a \in \text{Carmichael}(n) :  
a^{n-1} \equiv 1 \pmod n,  \gcd(a,n)=1.
\]
Thus for a composite Carmichael number $n$  every $a$ such
that $\gcd(a,n)=1$ is a Fermat non-witness Fnw(a,n).
A prime number $n$ will be reported through line 12 as PseudoPrime.
A composite number $n$ could be reported as composite immediately
through line 6 or through FermatLittleTheorem  testing  in Line 9 for
a properly chosen $a$.
But it is possible that the values $a$ used would never lead
to determining that that $n$ were a Composite number through line 6.
Furthermore the choice of a given $a$ (or all $a$) might not allow
the determination of $n$ as Composite through line 9 because $n$
is a Carmichael number.
In that case $n$ would be reported as PseudoPrime through line 12.
Im summary,
If an $a$ is picked with $\gcd(a,n)\neq 1$,
and $a^{n-1} \not\equiv 1 \pmod n$.
then  $a$ is a witness of the compositeness of $n$ or an  
Fw(a,n) either because of Line 5 or Line 8 of
Algorithm~\ref{fct1b}.
Algorithm~\ref{fct1b} checks in lines 5-6 the obvious possibility
that $n$ is a composite number. Thus the test of lines 8-10
is applicable only to $a$ such that $\gcd(a,n)=1$, and therefore
it can fail for Carmichael composite numbers only or a prime number
$n$.
\end{proof}

Algorithm~\ref{fct1b} is a probabilistic algorithm
for testing whether $n$ is prime or composite.

\begin{prp}[Probabibility of successful reporting]
\label{fctp}
The probability that the output of Algorithm~\ref{fct1b} is
Composite  given that $n$ is a composite number but not
a Carmichael number is at least $1 - 2^{-t}$.
\end{prp}

\begin{proof}
$ $ \\ $ $
In other words, if $n$ is a prime number Algorithm~\ref{fct1b}
returns Pseudoprime with probability one.
If $n$ is a composite number  but not
a Carmichael number then
Algorithm~\ref{fct1b} return composite 
with probability at least $1 - 2^{-t}$.
Suppose that $\exists a  \in \mb{U}_n$ such that
\[
 a^{n-1} \not\equiv 1 \pmod n .
\]
Consider the following two sets $B_n$ and $C_n$ as follows.
\[
B_n = \left\{ 
        a : a  \in \mb{U}_n , a^{n-1} \not\equiv 1 \pmod n 
       \right\} ,
\]
\[
C_n = \left\{ 
        a : a  \in \mb{U}_n , a^{n-1}     \equiv 1 \pmod n 
       \right\} ,
\]
%\[
%D_n = \left\{ 
%        \gcd(a,n) \neq 1
%       \right\} ,
%\]
We have $ |B_n | + |C_n| = \phi (n) = |\mb{U}_n |$.
For a composite $n$, we have $|D_n| >1$ as it contains
the prime factors of $n$, minimally.
Note that $a\in \mb{U}_n$ implies $\gcd(a,n)=1$.
Set $C_n$ is a subgroup. Consider $k,l \in C_n$ with
$k^{n-1} \equiv 1 \pmod n$
and
$l^{n-1} \equiv 1 \pmod n$.
We conclude easily that
$(kl)^{n-1} \equiv 1 \pmod n$ to prove closure.
Sidelining the claims about the obvious properties,
consider $K=k^{-1}$ the multiplicative inverse of $k$
and
consider $L=l^{-1}$ the multiplicative inverse of $l$.
We then have
\[
K k \equiv 1 \pmod n
\Rightarrow
(K k)^{n-1} \equiv 1 \pmod n ,
\]
from which we conclude $K \in C_n$.
Similarly for $l \in C_n$ we can show that $L \in C_n$.
Moreover, we can show $(kl)^{-1} \in C_n$.
Thus $C_n$ is a subgroup of $\mb{U}_n$ and the order
of $C_n$ divides the order of $\mb{U}_n$ which is $\phi (n)$.
Thus if $C_n \neq \mb{U}_n$ (e.g. $n$ is not a Carmichael 
number), then $|C_n| \leq |\mb{U}_n|/2$
and  $|B_n| \geq |\mb{U}_n|/2$.
\end{proof}

\begin{lem}[Running time of algorithm~\ref{fct1b}]
The running time of 
Algorithm~\ref{fct1b} is 
$O(t \cdot \lg{n} \cdot M(n) )$, where $M(n)$
is the cost of multiplying $n$-bit integers.
\end{lem}
\begin{proof}
Exponentiation involves $O(\lg{n})$ multiplication. 
Depending on how we
implement integer multiplication the overall time complexity
is $O(t \cdot \lg{n} \cdot M(n) )$, where $M(n)$
is the computational cost (bit model) of 
multiplying $n$-bit integers.
\end{proof}

\begin{cor}[FermatPrimalityTest $\in$ BPP]
Algorithm~\ref{fct1b} belongs to class BPP,
known as Bounded-error Probabilistic Polynomial
time class.
\end{cor}
\begin{proof}
We shall map (call) $L$ and $w$ to  
$\text{Prime}$ and $n$ respectively.
If $n$ is a prime number, it is decided by Algorithm~\ref{fct1b}
that it is a Prime (and output PseudoPrime is printed) with
probability 1.
$ $ \\ $ $
If $n$ is a composite number, it is decided by 
Algorithm~\ref{fct1b}
that it is a  Composite number  (and output Composite is printed)
and thus rejected being a Prime/PseudoPrime with probability 
at least $1- 1/ 2^t$ by way of Proposition~\ref{fctp}.
$ $ \\ $ $
In other words, $n$ can be  wrongly decided by Algorithm~\ref{fct1b}
that it is a Prime (and output PseudoPrime is printed)
with probability at most $1/ 2^t$.
\end{proof}

\newpage
\section{Solovay-Strassen primality testing}

\noindent
The primality testing algorithm by Solovay-Strassen uses
Euler's property for the Jacobi symbol (composite $n$) 
or Legendre symbol (prime $n$) rather than Fermat's
tittle theorem.
Euler's theorem for an odd prime number $p> 2$ 
states the  following (Legendre symbol use).
\[
 \leg{a}{p} \equiv a^{\frac{p-1}{2}} \pmod p , \forall a : 1 \leq a < p.
\]

Therefore  were we to test if $p$ was a prime number
by finding an $a$ that violates the condition above that would
prove the following: 
(a) $p$ is not a prime number, and 
(b) $a$ is a witness of the compositeness of $p$.
Note that for an odd composite number $n>2$ the symbol below is a Jacobi
symbol.
\[
 \leg{a}{n} \equiv a^{\frac{n-1}{2}} \pmod n , \forall a : 1 \leq a < n.
\]

The way we introduced in the previous section the concept
of a Fermat witness or a Fermat non-witness, we do the
same in this section by introducing the concept
of an Euler witness and  Euler non-witness.
We present an example to highlight the concepts.

\begin{exa}
Consider $n=1387$, $n=561$ and $n=63$.
Pick $a=2$. Examine whether $a=2$ is a Fermat witness
or an Euler witness (for the Solovay-Strassen algorithm to follow) 
of the compositeness of $n$.
\end{exa}

\begin{solution}
$ $ \\ $ $
{\bf (a) $n=1387 = 19 \cdot 73$.} Pick $a=2$, as instructed.
$ $ \\ $ $
It is
\[
 a^{1387-1} \equiv 1 \pmod{1387} ,
\]
and
\[
 a^{\frac{1387-1}{2}} \equiv 51 \pmod{1387} .
\]
In the former case $a$ is NOT a Fermat witness;
the latter case indicates that $a$ is an Euler witness.
This is because $\leg{a}{n}$ should be $1$, $-1$ or $0$.
$ $ \\ $ $
{\bf (b) $n= 561 = 3 \cdot 11 \cdot 17$.} Pick $a=2$, as instructed.
$ $ \\ $ $
It is
\[
 a^{561 -1} \equiv 1 \pmod{561}  ,
\]
and
\[
 a^{\frac{561-1}{2}} \equiv 1 \pmod{561} .  
\]
In the former case $a$ is NOT a Fermat witness;
the latter case indicates that $a$ is also NOT an Euler witness.
$ $ \\ $ $
{\bf (c) $n= 63 = 3^2 \cdot 7$.} Pick $a=2$, as instructed.
$ $ \\ $ $
It is
\[
 a^{63  -1} \equiv 4 \pmod{63}   ,
\]
and
\[
 a^{\frac{63 -1}{2}} \equiv 2 \pmod{63}  .  
\]
In the former case $a$ is a Fermat witness;
the latter case indicates that $a$ is also an Euler witness.
\end{solution}

\subsection{The Solovay-Strassen primality test}

The primality testing algorithm by Solovay-Strassen uses
Euler's theorem (property or criterion).

Therefore, as we mentioned earlier,  
were we to test if $p$ was a prime number
by finding an $a$ that violates the condition above that would
prove the following: 
(a) $p$ is not a prime number, and 
(b) $a$ is a witness of the compositeness of $p$.
Note that for an odd composite number $n>2$ the symbol below is a Jacobi
symbol.
\[
 \leg{a}{n} \equiv a^{\frac{n-1}{2}} \pmod n , \forall a : 1 \leq a < n.
\]
Algorithm~\ref{ss1b} 
performs $t$ times the tasks of 
Algorithm~\ref{ss1a} which is a simpler
 probabilistic algorithm.
Algorithm~\ref{ss1a}  checks Euler's theorem for
a single random value $a$.
Algorithm~\ref{ss1b}  checks Euler's theorem for
a $t$ uniformly at  random selected values $a$.
The former is weak algorithm and for this we call
it a compositeness test algorithm.
The latter is a more reliable one and is referred
to as a primality test algorithm.

\begin{thm}
\label{tss1b}
Algorithm~\ref{ss1b} is a probabilistic algorithm
for testing whether $n$ is composite or not.
\end{thm}

\begin{proof}
$ $ \\ $ $
The algorithm tests for two conditions a number of times. We denote
as $t$ the numer of times these conditions are tested.
Every time these conditions are tested we pick a random $a$
from the set      $\{ 2, 3, \ldots , n-2 \}$ of cardinality $n-1$.
The first condition $C_1 (a,n)$ checks whether $a$ has a 
common divisor with $n$. A positive answer serves makes $a$
a witness of the  compositeness of $n$. We call
then $a$ an Euler (test) compositeness witness and denote it by
Ew(a,n) or the set of witnesses for $n$ is denoted as Ew(n).
\[
C_1 (a,n) : \gcd(a,n) > 1.
\]
The second condition is Euler's criterion.
\[
C_2 (a,n) : \leg{a}{n} \not\equiv a^{\frac{n-1}{2}} \pmod n .
\]
Therefore $a$ becomes a Euler test compositeness witness if
and only if the following applies.
\[
a \in Ew(n) \Leftrightarrow C_1 (a,n) \vee C_2 (a,n) .
\]
By way of 
$\overline{(C_1 (a,n)  \vee C_2 (a,n)  )} = 
\overline{C_1 (a,n) } \wedge \overline{C_2 (a,n) }$
we obtain the following condition for Euler (test) 
compositeness non-witnesses Enw(a,n) or the set Enw(n).
\[
a \in Enw(n) \Leftrightarrow 
 \overline{C_1 (a,n)} \wedge \overline{C_2 (a,n)}.
\]
The proof of correctness of both algorithms follows from the
prior discussion and thus omitted.
\SetKwComment{Comment}{/* }{ */}
\SetKwRepeat{Do}{do}{while}
\begin{algorithm}
\KwIn{$n >4$ is odd; $a$ is in $\{ 2, 3, \ldots , n-2 \}$}
\KwOut{$n$ is composite or pseudoprime}
 Pick $ a \in \{ 2, 3, \ldots , n-2 \}$ uniformly at random  \;
 \If(\tcc*[f]{$C_1 (a,n)$}){$\gcd(a,n) >1$}{
  \Return{$\mathbf{Composite}$} 
  \Comment*[r]{$a$ is a witness of $n$'s compositeness}
 }
 $\text{Calculate}$ $\leg{a}{n}$ , $a^{\frac{n-1}{2}} \pmod n$;

 \eIf(\tcc*[f]{$C_2 (a,n)$}){$\leg{a}{n}\not\equiv a^{\frac{n-1}{2}} \pmod n$}{
  \Return{$\mathbf{Composite}$} 
  \Comment*[r]{$a$ is a witness of $n$'s compositeness}
 }{
  \Return{$\mathbf{PseudoPrime}$}
  \Comment*[r]{$n$ is either prime or composite}
 }
 \caption{SolovayStrassenCompositenessTest}
 \label{ss1a}
\end{algorithm}

\begin{algorithm}[H]
\KwIn{$n >4$ is odd; number of runs is $t$}
\KwOut{$n$ is composite or pseudoprime}
$ i = 0 $ \;

\Do{$i<t$}{
 $i=i+1$ \;
 Pick $ a \in \{ 2, 3, \ldots , n-2 \}$ uniformly at random  \;
 \If(\tcc*[f]{$C_1 (a,n)$}){$\gcd(a,n) >1$}{
  \Return{$\mathbf{Composite}$} 
  \Comment*[r]{$a$ is a witness of $n$'s compositeness}
 }
 $\text{Calculate}$ $\leg{a}{n}$ , $a^{\frac{n-1}{2}} \pmod n$;

 \If(\tcc*[f]{$C_2 (a,n)$}){$\leg{a}{n}\not\equiv a^{\frac{n-1}{2}} \pmod n$}{
  \Return{$\mathbf{Composite}$} 
  \Comment*[r]{$a$ is a witness of $n$'s compositeness}
 }
 }
  \Return{$\mathbf{PseudoPrime}$}
  \Comment*[r]{$n$ is either prime or composite}
 \caption{SolovayStrassenPrimalityTest}
 \label{ss1b}
\end{algorithm}
\end{proof}

\begin{prp}[Probability of successful reporting]
\label{ssp1}
If $n$ is an odd composite number greater than two,
then Algorithm~\ref{ss1a} returns Composite with
probability at least $1/2$.
The probability that the output of Algorithm~\ref{ss1b} is
Composite  given that $n$ is a composite number
is at least $1 - 2^{-t}$.
\end{prp}

\begin{proof}
$ $ \\ $ $
Consider the following two sets $B_n$ and $C_n$ as follows,
describing in fact Ew(n) and Enw(n).
\[
Ew(n) = B_n = \left\{ 
        a : a  \in \mb{U}_n , 
        a^{\frac{n-1}{2}}\not\equiv \leg{a}{n} \pmod n 
      \right\} ,
\]
\[
Enw(n) = C_n = \left\{ 
        a : a  \in \mb{U}_n , 
        a^{\frac{n-1}{2}}    \equiv \leg{a}{n} \pmod n 
      \right\} ,
\]
Consider also set $D_n$ defined as follows.
\[
D_n = \left\{
       a:  1\leq a < n \wedge \gcd(a,n) \neq 1 
      \right\} .
\]
We have that if $a \in B_n$ then $n$ is composite
and of course $a \in Ew(n)$.
We have that if $a \in C_n$ then Solovay-Strassen
can't figure out whether $n$ is prime or composite
and of course $a \in Enw(n)$.
We prove a sequence of claims to derive our result.
$ $ \\ $ $
{\bf (a) Show that $C_n$ is a subgroup of $\mb{U}_n$.}
$ $ \\ $ $
If $a,b \in C_n$ then it is trivial to show that
$ab \in C_n$ as well.
\begin{eqnarray*}
a \in C_n \Rightarrow  a^{\frac{n-1}{2}}    \equiv \leg{a}{n} \pmod n ,
&\wedge&
b \in C_n \Rightarrow  b^{\frac{n-1}{2}}    \equiv \leg{b}{n} \pmod n , \\
&\Rightarrow&
(ab)^{\frac{n-1}{2}} =  a^{\frac{n-1}{2}} b^{\frac{n-1}{2}}
\equiv  \leg{a}{n} \leg{b}{n} \equiv \leg{ab}{n} \pmod n \\
&\Rightarrow&
(ab) \in C_n .
\end{eqnarray*}
Let $A=a^{-1}$ that is $A a = a A =1 \mod n$. We show that
$A \in C_n$.
\begin{eqnarray}
\label{sseq1}
\leg{a}{n} \leg{A}{n} \equiv \leg{aA}{n} \equiv \leg{1}{n} \equiv 1 \pmod n
\Rightarrow
\leg{a}{n} \equiv \leg{A}{n} \pmod n
\end{eqnarray}
Furthermore,
\begin{eqnarray}
a^{\frac{n-1}{2}}    
\equiv \leg{a}{n} \pmod n 
&\Rightarrow&
(Aa)^{\frac{n-1}{2}}    \equiv A^{\frac{n-1}{2}}    \leg{a}{n} \pmod n 
\nonumber \\
&\Rightarrow&
 1  \equiv A^{\frac{n-1}{2}}    \leg{a}{n} \pmod n 
\nonumber \\
\label{sseq2}
&\Rightarrow&
 \leg{a}{n}           \equiv A^{\frac{n-1}{2}}    \pmod n \nonumber \\
&\xRightarrow{\text{by Eq.~(\ref{sseq1})}}&
 \leg{A}{n}           \equiv A^{\frac{n-1}{2}}    \pmod n  ,
\end{eqnarray}
with the latter implying $A \in C_n$.
Given that $C_n$  is a subgroup of $\mb{U}_n$,
the order of $C_n$ divides that of $\mb{U}_n$.
If the two are not equal then
\[
  |C_n | \leq \frac{|\mb{U}_n|}{2}.
\]
We dismiss the possibility $C_n = \mb{U}_n$.
$ $ \\ $ $
{\bf (b) Show $C_n \neq \mb{U}_n$.}
The proof is by contradiction (contrapositive).
Let us assume that  $C_n   =  \mb{U}_n$.
Moreover, let $n= p_1^{a_1} p_2^{a_2} \ldots p_k^{a_k}$.
We distinguish two cases for $n$.
$ $ \\ $ $
{\bf Case 1: $a_1 =1 $.}
For convenience rewrite
\[
n= p_1^{a_1} p_2^{a_2} \ldots p_k^{a_k} =
   p_1 p_2^{a_2}       \ldots p_k^{a_k} =
   p_1 Q,
\]
where $\gcd(p_1 , Q) =1$.
$ $ \\ $ $
Let $g$ be a generator of 
$\mb{Z}_{p_1}^* = \mb{U}_{p_1}$. From prior results
we know it exists for prime $p_1$. (Note $n$ is odd.)
Moreover $g^{\phi(p_1)} = g^{p_1 -1} \equiv 1 \pmod{p_1}$
and therefore the following is also true
\[
g^{\frac{p_1 -1}{2}} \equiv -1 \pmod{p_1} ,
\]
by Lagrange's theorem.
Then by the Chinese Remainder Theorem (CRT) there exists
an $a$ such that
\begin{eqnarray*}
 a &\equiv& g \pmod{p_1},  \\
 a &\equiv& 1 \pmod{Q}.
\end{eqnarray*}
We then obtain the following.
\begin{eqnarray*}
\leg{a}{n} &=& \leg{a}{p_1^{a_1} p_2^{a_2} \ldots p_k^{a_k}} \\
           &=& \leg{a}{p_1       p_2^{a_2} \ldots p_k^{a_k}} \\
           &=& \leg{a}{p_1}\leg{a}{p_2^{a_2} \ldots p_k^{a_k}} \\
           &=& \leg{a}{p_1}\leg{a}{Q}.
\end{eqnarray*}
For $\leg{a}{p_1}$ since  $a \equiv g \pmod{p_1}$ we have
\[
\leg{a}{p_1} \equiv \leg{g}{p_1} \equiv g^{\frac{p_1 -1}{2}}
\equiv -1 \pmod{p_1}.
\]
For $\leg{a}{Q}$ we have by way of $ a\equiv 1 \pmod{Q}$
the following
\[
\leg{a}{Q}  \equiv \leg{1}{Q} \equiv 1 \pmod{Q}.
\]
Therefore we conclude the following
\[
\leg{a}{n} =  \leg{a}{p_1}\leg{a}{Q} = (-1) (+1) = -1.
\]
If $a \in C_n$ and $\leg{a}{n} = -1$ this implies
\[
a^{\frac{n-1}{2}}    \equiv \leg{a}{n} \equiv -1 \pmod n ,
\]
Then, we have the following, noting that $n=p_1 Q$ and thus
$\dv{Q}{n}$ and $\gcd(p_1 , Q)=1$.
$ $ \\ $ $
\begin{eqnarray}
a^{\frac{n-1}{2}}    \equiv -1 \pmod n
&\Rightarrow& \dv{n}{a^{\frac{n-1}{2}}+1} \nonumber \\
&\Rightarrow& \dv{Q}{a^{\frac{n-1}{2}}+1} \nonumber \\
\label{sse3}
&\Rightarrow& a^{\frac{n-1}{2}} \equiv -1 \pmod Q
\end{eqnarray}
By way of 
 $ a \equiv 1 \pmod{Q} $ we also have the following.
\begin{equation}
\label{sse4}
 a^{\frac{n-1}{2}} \equiv 1 \pmod{Q},
\end{equation}
and therefore combining Eq.~\ref{sse3} and Eq.~\ref{sse4}
we obtain the following
\[
2 \equiv 0 \pmod Q,
\]
which is only possible for $Q$ being equal to 2, but this
contradicts the fact that $n=p_1 Q$ is an odd integer.
$ $ \\ $ $
{\bf Case 2: $a_1 \geq 2 $.}
$ $ \\ $ $
If $a_1 \geq 2$ it means $a_1 -1 \geq 1$. 
For $C_n   =  \mb{U}_n$,
it means
\[
 a^{\frac{n-1}{2}} \equiv \leg{a}{n} \pmod{n}
\Rightarrow
 a^{\frac{n-1}{2}} \equiv \leg{a}{n} \equiv \pm 1 \pmod{n} ,
\]
and by squaring we obtain the following
\[
a^{n-1} \equiv 1 \pmod n.
\]
This last equation implies $\dv{n}{a^{n-1}-1}$.
Since $p_1^{a_1}$ is a factor of $n$ we also have
that $\dv{p_1^{a_1}}{a^{n-1}-1}$,
or equivalently
\[
a^{n-1} \equiv 1 \pmod{p_1^{a_1}}.
\]
By prior results the last one implies
\[
\dv{\phi(p_1^{a_1} )}{n-1}
\Rightarrow
\dv{p_1^{a_1-1}(p_1 -1) )}{n-1}
\xRightarrow{\text{by $a_1 > 1$}}
\dv{p_1}{n-1}.
\]
In addition to $\dv{p_1}{n-1}$ we have $\dv{p_1}{n}$ as
$p_1$ is a factor of $n$. The two imply
$\dv{p_1}{1}$ i.e. $p_1 =1$ which contradicts the fact 
that $p_1$ is a prime number, and an
odd integer and thus $p_1 \geq 3$.
$ $ \\ $ $
{\bf (c) Conclude $C_n   \neq  \mb{U}_n$.}
By case 1 and case 2 we conclude that
$C_n   \neq  \mb{U}_n$.
Then 
\[
 |C_n| \leq |\mb{U}_n| /2 \leq (n-1) / 2
\]
Note that for prime $N$
$|\mb{U}_n | = \phi (n) = n-1$.
$ $ \\ $ $
The probability that the output of Algorithm~\ref{ss1b} is
Composite  given that $n$ is a composite number
is at least $1 - 2^{-t}$ as  Algorithm~\ref{ss1b} 
repeats Algorithm~\ref{ss1a} a number of $t$ times.
The result then follows.
\end{proof}

Conrad \cite{ConradSS} establishes Proposition~\ref{ssp2}
which is equivalent to Proposition~\ref{ssp1}.

\begin{prp}[Probability of successful reporting 
alternative of \cite{ConradSS}]
\label{ssp2}
If $n$ is an odd composite number greater than two,
then Algorithm~\ref{ss1a} returns Composite with
probability at least $1/2$.
\end{prp}

\begin{proof}
$ $ \\ $ $
Consider the following three sets $B_n$, $C_n$, $D_n$ 
defined as follows, with the first two 
describing in fact Ew(n) and Enw(n).
\[
Ew(n) = B_n = \left\{ 
        a : a  \in \mb{U}_n , 
        a^{\frac{n-1}{2}}\not\equiv \leg{a}{n} \pmod n 
      \right\} ,
\]
\[
Enw(n) = C_n = \left\{ 
        a : a  \in \mb{U}_n , 
        a^{\frac{n-1}{2}}    \equiv \leg{a}{n} \pmod n 
      \right\} ,
\]
Consider also set $D_n$ defined as follows.
\[
D_n = \left\{
       a:  1\leq a < n \wedge \gcd(a,n) \neq 1 
      \right\} .
\]
We have $|B_n | \neq 0$, $1 \in C_n $ and 
$n-1 \equiv -1 \pmod n \in C_n$ as well.
$ $ \\ $ $
We  show through group theory
$C_n   \neq  \mb{U}_n$,
then implying
\[
 |C_n| \leq |\mb{U}_n| /2 \leq (n-1) / 2
\]
The part that show that $C_n$ is a subgroup
of $\mb{U}_n$ is borrowed from the previous proof.
$ $ \\ $ $
{\bf Step 1: Pick a $b \in B_n$ and show that $C_n b \subseteq  B_n$.}
Given that $|B_n | \neq 0$, pick $b \in B_n$ with $\gcd(b,n)=1$
and consider set
\[
 C_n b = \left\{
         c \cdot b : c \in C_n
         \right\}. 
\]
We will show below that $C_n \subseteq  B_n$.
$ $ \\ $ $
Let $c \in C_n$ with $\gcd(c,n)=1$.
Then
\[
  a^{\frac{n-1}{2}} \equiv \leg{a}{n} \pmod{n}.
\]
Consider the $b \in B_n$ with $\gcd(b,n)=1$.
Then we have $\gcd(bc,n)=1$.
We now prove that $bc \not\in C_n$.
We prove it by contrapositive/contradiction as
follows.
$ $  \\ $ $
Let $bc \in C_n$. Then 
\[
 (bc)^{\frac{n-1}{2}} \equiv \leg{bc}{n} \pmod n.
\]
We then derive the following
\begin{eqnarray}
(bc)^{\frac{n-1}{2}}& \equiv& \leg{bc}{n} \pmod n \nonumber \\
\label{sse7a}
                    & \equiv& \leg{b}{n} \leg{c}{n} \pmod n 
\end{eqnarray}
Moreover, we have the following.
\begin{eqnarray}
(bc)^{\frac{n-1}{2}}& \equiv& 
  b^{\frac{n-1}{2}} c^{\frac{n-1}{2}} \pmod n \nonumber \\
\label{sse7b}
                    & \equiv& b^{\frac{n-1}{2}} \leg{c}{n} \pmod n 
\end{eqnarray}
Since $\gcd(c,n)=1$ we have $\leg{c}{n} \neq 0$ and thus
by way of Eq.~\ref{sse7a} and Eq.~\ref{sse7b} we obtain
\[
 b^{\frac{n-1}{2}} \equiv \leg{b}{n} \pmod n ,
\]
which implies $ b \in C_n$ contradicting the assumption
that $b \in B_n$ instead.
Therefore it can' be $bc \in C_n$. 
Thus it must be $bc \in B_n$ for all $c \in C_n$.
Therefore we have proved
that
\[
 C_n b  \subseteq B_n .
\]
We know pick two different $c_1 \neq c_2 \in C_n$.
\[
c_1 b \equiv c_2 b \pmod n 
\Rightarrow
(c_1 - c_2 \equiv 0 \pmod n ,
\]
given that $\gcd(b,n)=1$. Moreover for $c_1 , c_2 < n$,
we conclude that $c_1 = c_2 $.
Therefore we have the following.
\[
 |C_n | = |C_n b | \leq |B_n |.
\]
$ $ \\ $ $
If $n$ is compositive for an odd $n$ as the statement 
emphatically states, then $n$ has a prime factor $p$
for which $\gcd(p,n) >1$. Then $p \in D_n$ and thus
$|D_n | \geq 1$.
We then obtain the following
\begin{eqnarray}
n-1 &=&    |B_n | + |C_n | + |D_n |  \\
    &\geq& |C_n | + |C_n | + |D_n |  \\
    &\geq& |C_n | + |C_n | +  1      \\
    &\geq& 2|C_n |  +  1      \\
|C_n| &\leq& \frac{n-2}{2}.
\end{eqnarray}
The latter $|C_n| \leq (n-2)/2 \leq (n-1)/2$, as needed.
\end{proof}

\begin{lem}[Running time of algorithm~\ref{ss1b}]
The running time of
Algorithm~\ref{ss1b} is
$O(t \cdot \lg{n} \cdot M(n) + t \lg^3{n} )$, where $M(n)$
is the cost of multiplying $n$-bit integers, and 
$O(\lg^3{n})$ is the cost of computing the Jacobi symbol
$\leg{a}{n}$ for $n$ and an $a<n$.
\end{lem}
\begin{proof}
Exponentiation involves $O(\lg{n})$ multiplications. 
Depending on how we
implement integer multiplication the overall time complexity
is $O(t \cdot \lg{n} \cdot M(n) +t \lg^3{n} )$, where $M(n)$
is the computational cost (bit model) of
multiplying $n$-bit integers, and $\lg^3{n}$ is the 
cost of computing the Jacobi symbol
$\leg{a}{n}$ for $n$ and an $a<n$.
\end{proof}

\begin{cor}[SolovayStrassenPrimalityTest $\in$ BPP]
Algorithm~\ref{ss1b} belongs to class BPP.
\end{cor}
\begin{proof}
It follows from Proposition~\ref{ssp1} or Proposition~\ref{ssp2}.
We shall map (call) $L$ and $w$ to
$\text{Prime}$ and $n$ respectively.
If $n$ is a prime number, it is decided by Algorithm~\ref{ss1b}
that it is a Prime (and output PseudoPrime is printed) with
probability 1.
$ $ \\ $ $
If $n$ is a composite number, it is decided by
Algorithm~\ref{ss1b}
that it is a  Composite number  (and output Composite is printed)
and thus rejected being a Prime/PseudoPrime with probability
at least $1- 1/ 2^t$ by way of Proposition~\ref{ssp1}.
$ $ \\ $ $
In other words, $n$ can be  wrongly decided by Algorithm~\ref{fct1b}
that it is a Prime (and output PseudoPrime is printed)
with probability at most $1/ 2^t$.
\end{proof}

\newpage

\section{Riemann hypothesis associated primality tests}

\subsection{Perfect powers}

\noindent
One of the first steps in Miller's algorithm it is
to detect whether an integer $n$ is a perfect power  or
not. We say that natural number  $n$ is a perfect power 
if there exist positive integers $x,m$ such that $n= x^m$.

\begin{dfn}[Perfect power]
Given a natural integer number $n$  we say
$n$ is a perfect power if there exist
a natural integer number $x$ and
a natural integer number $m>1$
such that
\[
       n = x^m
\]
\end{dfn}

\begin{lem}[Perfect power algorithm]
One can use binary search and thus time polylogarithmic
in $n$ to determine $x,m$. This is NaivePerfectPower(n).
\end{lem}
\begin{proof}
$ $ \\ $ $
Since $n=x^m \geq 2^m $ we observe
that $m \leq \lg{n}$ or
$m \leq \lfloor \lg{n} \rfloor$.
Thus the integer candidate values for $m$
are $2,3, \ldots , \lfloor \lg{n} \rfloor$.
If $A=\lg{n}$ the number of values $m$ we
need to search for is at most $A-1$.

\SetKwComment{Comment}{/* }{ */}
\SetKwRepeat{Do}{do}{while}
\begin{algorithm}[H]
\KwIn{$n$ greater than 2}
\KwOut{$x,m$ such that $x^m = n$ or $\mathbf{NotaPPower}$}
$ A = \lg{n} $ 
 \Comment*[r]{Bound for $m$ is $A$}

\For{$i=2,3, \ldots , \lfloor  A \rfloor$}{

  $ m = i $ ;
 
  $x=BinarySearch([1,y],m)$
  \Comment*[r]{Determine $1< x \leq n$ such that $x^m <n$, $(x+1)^m \geq n$}

 \If{$x^{ m } == n$}{
  \Return{$\mathbf{(x, m )}$} 
  \Comment*[r]{Perfect Power}}

 \If{$(x+1)^{ m } == n$}{
  \Return{$\mathbf{(x+1, m )}$} 
  \Comment*[r]{Perfect Power}}

 }

 \Return{$\mathbf{NotaPPower}$}

 \caption{NaivePerfectPower(n) : Determine if $n$ is a perfect power}
 \label{ppalg0}
\end{algorithm}
%With the next problem the binary search operation of Line 4 is replaced
%with an arithmetic operation that finds $\floor{y^{1/m}}$ for a given
%$m$. That, is then set to $x$.
\end{proof}

\newpage

\subsection{Miller primality test under GRH : Miller1}

Miller (\cite{M76}) describes two deterministic
algorithms that work correctly under the Generalized
Riemann Hypothesis (GRH).  Miller cites the
Extended Riemann Hypothesis for a Dirichlet function;
the Generalized Riemann Hypothesis for Dirichlet $L$-functions
is equivalent to the Extended Riemann hypothesis for the
problem in hand.
The first algorithm, that will be referred to as Miller1,
appears on page 303 of \cite{M76}.
The second algorithm, that will be referred to as Miller2,
appears on page 308 of \cite{M76}.
Miller1 is the deterministic algorithm 
turned by Rabin \cite{R80} into a probabilistic algorithm, 
and analyzed in \cite{R80}.
This probabilistic algorithm will be referred to as 
Rabin-Miller.
Monier \cite{Monier} proved that the conditions used
in Miller1 and Miller2 are equivalent.
The probabilistic version of Miller2 will be referred
to as the Miller-Rabin algorithm. Because of the equivalence
of the conditions of Miller1 and Miller2, the Rabin  analysis
originally for the probabilistic algorithm of the modified
Miller1 applies also to the modified Miller2.
Miller1 and also Algorithm~\ref{Miller2} known as 
Miller2, if the  GRH is incorrect, might generate an 
incorrect answer.
This uncertainty is not present in 
Miller-Rabin or Rabin-Miller.

\begin{thm}[Miller \cite{M76}, pages 303-304]
\label{tMiller1}
Algorithm~\ref{Miller1}, Miller1, is a deterministic algorithm
for testing whether $n$ is prime or composite, under the GRH.
\end{thm}

% Dirichlet functions (Bach result)  ==> GRH
% Dedekind  functions                ==> ERH
\begin{proof}
$ $ \\ $ $
Algorithm~\ref{Miller1} appears on pages 303-304 of \cite{M76}.
Miller1 and also Algorithm~\ref{Miller2} known as 
Miller2, if the  GRH is incorrect, might generate an 
incorrect answer.
Furthermore, a constant $c$ cited by Miller in \cite{M76} is
never explicitly calculated. Under GRH, Bach \cite{Bach}
calculated the constant $c$ to be (upper bounded by) $2$.
In Miller1 and Miller2 we thus use $c=2$ from Bach \cite{Bach}.

\SetKwComment{Comment}{/* }{ */}
\SetKwRepeat{Do}{do}{while}
\begin{algorithm}[H]
\KwIn{$n>4$ is odd; $n-1=2^k l$, $l$ odd}
\KwOut{$n$ is $\mathbf{Composite}$ or $\mathbf{Prime}$}
$ n-1 = 2^k l $ ;
$ f(n) = c (\ln{n})^2 $ 
 \Comment*[r]{Bach \cite{Bach} calculated $c=2$}

\If(\tcc*[f]{Line 1}){$\text{PerfectPower}(n)$}{
 \Return{$\mathbf{Composite}$} 
 \Comment*[r]{$n$ is a perfect power $n=p^s$, $s \geq 2$}
}

\For{$a=2,3, \ldots , 2(\ln{n})^2$}{

 \If(\tcc*[f]{Line 2 (i)}){$(\dv{a}{n}) || (\gcd(a,n) >1)$}{
  \Return{$\mathbf{Composite}$} 
  \Comment*[r]{Just $\dv{a}{n}$ in \cite{M76}}
 }

 \If(\tcc*[f]{Line 2 (ii)}){$a^{n-1} \not\equiv 1 \pmod n$}{
  \Return{$\mathbf{Composite}$} \; 
 }

 \If(\tcc*[f]{Line 2 (iii)}){$\exists m, 0\leq m <k: \gcd(( a^{2^k l} \bmod n)-1, n) \neq 1, n$}{
  \Return{$\mathbf{Composite}$}  \;
 }

 }

 \Return{$\mathbf{Prime}$}
 \Comment*[r]{Line 3}

 \caption{Miller1(n):  primality test algorithm}
 \label{Miller1}
\end{algorithm}
Line references are to the pseudocode of \cite{M76}, pages 303-304.
The proof of correctness is in \cite{M76}.
Note that the line 16 return statement of Algorithm~\ref{Miller1} 
should be read as follows:
either $n$ is a prime number or the Generalized 
Riemann Hypothesis is false.
\end{proof}

\newpage

\subsection{A second Miller primality test under GRH: Miller2}

\begin{thm}[Miller \cite{M76}, page 308]
\label{tMiller2}
Algorithm~\ref{Miller2}, Miller2, is a deterministic algorithm
for testing whether $n$ is prime or composite, under the GRH.
\end{thm}

\noindent
\begin{proof}
$ $ \\ $ $
\SetKwComment{Comment}{/* }{ */}
\SetKwRepeat{Do}{do}{while}
\begin{algorithm}[H]
\KwIn{$n >4$ is odd $n-1=2^k l$, $l$ odd}
\KwOut{$n$ is $\mathbf{Composite}$ or $\mathbf{Prime}$}
$ n-1 = 2^k l $ ;

$ f(n) = c (\ln{n})^2 $ 
\Comment*[r]{Bach \cite{Bach} calculated $c=2$}

\If(\tcc*[f]{Line 1}){$\text{PerfectPower}$(n)}{
 \Return{$\mathbf{Composite}$} 
 \Comment*[r]{$n$ is a perfect power $n=p^s$, $s \geq 2$}
}

$\text{Generate primes}\ p_1 , \ldots , p_t \leq f(n)$ ;
 \Comment*[r]{Line 2}

\For(\tcc*[f]{Line 2(i)}){$i=1,2, \ldots , t$}{
 $a = p_i$ ;

 \If(\tcc*[f]{Line 2(ii); or say $\dv{a}{n}$}){$\gcd(a,n) >1$}{
  \Return{$\mathbf{Composite}$} 
  \Comment*[r]{$a$ is a witness of $n$'s compositeness}
 }
  Calculate $a^{2^i \cdot l} \pmod n$ for $0\leq i \leq k$.

 \If(\tcc*[f]{Line 2(iii)}){$a^{2^k l} \equiv a^{n-1} \not\equiv 1 \pmod n$}{
  \Return{$\mathbf{Composite}$} 
  \Comment*[r]{$a$ is a witness of $n$'s compositeness}
 }

 \If(\tcc*[f]{Line 2(iv)}){$a^{l} \equiv \equiv 1 \pmod n$}{
   continue;
 }

 Find max $j = \max \left\{ i: a^{2^i \cdot l} \pmod n \neq 1 \right\}$ ;

 \If(\tcc*[f]{Test (v)}){$a^{2^j l} \equiv \equiv -1 \pmod n$}{
   continue;
 }

 \Return{$\mathbf{Composite}$}
 \Comment*[r]{$a$ is a witness of $n$'s compositeness}

 }
 \Return{$\mathbf{Prime}$}
 \caption{Miller2(n) :  primality test algorithm}
 \label{Miller2}
\end{algorithm}
The proof of correctness is in \cite{M76}.
The cases (ii)-(v) and corresponding line references
are from \cite{Bach}, page 308.
Note that the line 25 return statement should be read as follows:
either $n$ is a prime number or the Generalized 
Riemann Hypothesis is false.
% Dirichlet functions (Bach result)  ==> GRH
% Dedeking  functions                ==> ERH
\end{proof}

\newpage

\section{Probabilistic primality tests}

\subsection{The Miller-Rabin primality test}

The following lemma is established. It is used
to establish Proposition~\ref{mrp2} and
prove the correctness of Algorithm~\ref{MillerRabin}.
A condition expressed by Eq.(\ref{mre0}) will have
cases $P_{2a} , P_{2b}$ merged into Case $P_2$.
Therefore the condition of Eq.(\ref{mre0}) 
will appear as in Eq.(\ref{mrepri})
in Lemma~\ref{mrl1}.
\begin{equation}
\label{mre0}
\text{n is prime} \Rightarrow \forall x ,  1 \leq x < n :
\begin{cases}
 x^l       \equiv 1 \pmod n & \text{Case } P_1(x,n) \\
     \vee             &                  \\
 x^l       \equiv -1 \pmod n & \text{Case } P_{2a} (x,n,m) \\
     \vee             &                  \\
 x^{2^m l} \equiv -1 \pmod n ,\quad  
\exists m: 0 <m< k &\text{Case } P_{2b}\\
\end{cases}
\end{equation}

\begin{lem}
\label{mrl1}
Let $n \in \mb{Z}_+$ be an odd integer with $n>2$.
Let $n-1 = 2^k \cdot l$, where $k\geq 1$, and $l$ odd.
For all $x$ such that $1\leq  x < n$
we have the following.
\begin{equation}
\label{mrepri}
\text{n is prime} \Rightarrow \forall x , 1 \leq x < n :
\begin{cases}
 x^l       \equiv 1 \pmod n  & \text{Case } P_1 (x,n) \\
            \vee             &                  \\
 x^{2^m l} \equiv -1 \pmod n ,\quad  \exists m: 0 \leq m < k 
                             & \text{Case } P_2 (x,n,m) \\
\end{cases}
\end{equation}
\end{lem}

\noindent

\begin{proof}
$ $ \\ $ $
If $n$ is a prime number then Fermat's little theorem is invoked.
For any $x$ such that $1\leq x < n$ and we note that then
$\gcd(x,n)=1$ we have the following.
\[
 x^{n-1} \equiv 1 \pmod n  \Rightarrow  x^{n-1} -1 \equiv 0 \pmod n .
\]
We note that $n-1=2^k l$ as stated therefore
\[
  n-1 = \frac{n-1}{2} \cdot 2 = 2^{k-1} l  + 2^{k-1} l 
                              = 2 \cdot ( 2^{k-1} l ) .
\]
Therefore we have the following
\begin{eqnarray}
\label{mre1a} 
x^{n-1} -1  = x^{2^k l} -1       =  x^{2 \cdot 2^{k-1} l} -1  
%                                 &=&  \left( x^{2^{k-1} l} \right)^2 - 1^2
%  \nonumber \\
                                 &=&\left( x^{2^{k-1} l} + 1 \right)
                                    \left( x^{2^{k-1} l} - 1 \right)
\end{eqnarray}
We continue unrolling the right-most term a total of $k$ times.
\begin{eqnarray}
\label{mre1b} 
x^{n-1} -1  &=&\left( x^{2^{k-1} l} + 1 \right)
               \left( x^{2^{k-1} l} - 1 \right)
  \nonumber \\
            &=&\left( x^{2^{k-1} l} + 1 \right)
               \left( x^{2^{k-2} l} + 1 \right)
               \left( x^{2^{k-2} l} - 1 \right)
  \nonumber \\
            &=&\left( x^{2^{k-1} l} + 1 \right)
               \left( x^{2^{k-2} l} + 1 \right)
               \left( x^{2^{k-3} l} + 1 \right)
               \left( x^{2^{k-3} l} - 1 \right)
  \nonumber \\
           &\ldots&
  \nonumber \\
            &=&\left( x^{2^{k-1} l} + 1 \right)
               \left( x^{2^{k-2} l} + 1 \right)
                   \ldots
               \left( x^{2 l} + 1 \right)
               \left( x^{l} + 1 \right)
               \left( x^{l} - 1 \right) .
\end{eqnarray}
Moreover, $x^{n-1} -1 \equiv 0 \pmod n$ implies the following.
\begin{eqnarray}
\label{mre1c} 
x^{n-1} -1                         &\equiv&0 \pmod n \Leftrightarrow 
\nonumber\\
  \left( x^{2^{k-1} l} + 1 \right)
  \left( x^{2^{k-2} l} + 1 \right)
         \ldots
  \left( x^{2 l} + 1 \right)
  \left( x^{l} + 1 \right)
  \left( x^{l} - 1 \right)         &\equiv&0 \pmod n 
\end{eqnarray}
We perform a case analysis.
$ $ \\ $ $
One possibility that generates $P_1$ is that 
$\dv{n}{ x^{l} - 1}$, the right-most term. 
Then $x^{l}\equiv 1\pmod n$.
If this is not the case, then $n$  might divide the second
term from the right that is,
$\dv{n}{x^{l} + 1 }$. Then $x^{l}\equiv -1 \pmod n$.
This is what was called case $P_{2a}$ or it is the first
case of $P_{2}$ for $m=0$.
If this is not the case, then $n$ might divide the third
term from the right that is,
$\dv{n}{x^{2l} + 1 }$. Then $x^{2l}\equiv -1 \pmod n$.
This is the first case of $P_{2b}$ with $m=1$.
Continuing like this, if it is not the case that
$n$ divides $x^{2^{k-2} l} + 1$, then
$\dv{n}{x^{2^{k-1}l} + 1 }$. Then $x^{2^{k-1}l}\equiv -1 \pmod n$.
This is the last case of case $P_{2b}$ with $m=k-1$.

Note that $x^{2^{k-1}l}$ is the square of $x^{2^{k-2}l}$
and so on, and $x^{2l}$ is the square of $x^{l}$ mod $n$.
Thus the sequence implied is as follows if listed
right to left.
\begin{equation}
\label{millers}
x^l          \bmod n ,
x^{2l}       \bmod n ,
x^{2^2 l}    \bmod n ,
             \ldots
x^{2^{k-1}l} \bmod n
\end{equation}
The terms as shown in this order can be generated by squaring
mod $n$ starting with the left-most term $x^l$ mod $n$.
If $x^l$ is equivalent to one mod $n$ so are all subsequent terms.
If $x^{2^m l}$ is equivalent to $1$ mod $n$ so are all 
subsequent terms $x^{2^j l}$, $m \leq j < k$ because of squaring.
If $x^{2^m l}$ is $-1$ mod $n$        all subsequent terms
$x^{2^j l}$, $m <    j < k$, are $1$ mod $n$.
For the sequence
This implies that in the sequence of Eq.~(\ref{millers})
there is only one $-1$ followed by (a suffix) of ones to the end
of it or the whole sequence is a sequence of ones..

\noindent
In conclusion if $n$ is prime then we have $P_1$ or $P_2$ that is
\begin{equation}
\label{mrcpri}
   \text{n is prime} \Rightarrow \forall x, 1\leq x < n: P_1 (x,n) \vee \exists
0 \leq m < k :P_2 (x,n,m) .
\end{equation}

\noindent
If $n$ is a composite number  then we have 
$\overline{ P_1 \vee P_2 } = \overline{P_1} \wedge \overline{P_2}$.
\begin{equation}
\label{mrccom}
   \text{n is composite} \Rightarrow  \exists  x , 1 \leq x < n :
\overline{P_1 (x,n)} \wedge \forall m,  0 \leq m < k :\overline{P_2 (x,n,m)} .
\end{equation}

This is equivalent to the following one.
\begin{equation}
\label{mrecom}
\text{n is composite} \Leftrightarrow
\exists x , 1 \leq x < n  :
\begin{cases}
 x^l       \not\equiv 1 \pmod n  & \text{Case } \overline{P_1 (x,n)} \\
           \wedge             &                  \\
 x^{2^m l} \not\equiv -1 \pmod n ,\quad  \forall m, 0 \leq m < k : 
                             & \text{Case } \overline{P_2 (x,n,m)} \\
\end{cases}
\end{equation}
Later one we will rename 
\begin{eqnarray}
\label{cp1}
C_1 (x,n)   &=& \overline{P_1 (x,n)} \\
\label{cp2}
C_2 (x,n,m) &=& \overline{P_2 (x,n,m)} .
\end{eqnarray}
\end{proof}

We refine   Lemma~\ref{mrl1} to match the corresponding
lines of Algorithm~\ref{MillerRabin}.

\begin{prp}
\label{mrp2}
Let $n \in \mb{Z}_+$ be an odd integer with $n>4$.
Let $n-1 = 2^k \cdot l$, where $k\geq 1$, and $l$ odd.
Show then, 
%that in Algorithm~\ref{MillerRabin}, 
integer $n$ is declared a $\textrm{PseudoPrime}$ number if and only
if for every $a$ such that $2 \leq a < n-1$ and  $\ndv{n}{a}$
either
\begin{eqnarray}
\label{mrcp1}
 P_1 (a,n) = \overline{C_1 (a,n)} &:& a^{l} \equiv 1 \pmod n , 
\end{eqnarray}
or
\begin{eqnarray}
\label{mrcp2}
 P_2 (a,n,m) = \overline{C_2 (a,n,m)} &:& a^{2^m l} \equiv -1 \pmod n ,
\end{eqnarray}
for some integer $m$ such that $0 \leq m  <   k$.
$\textrm{PseudoPrime}$ means it is either a prime number or
$a^{n-1} \equiv 1 \pmod n$.
\end{prp}
We note that the conditions of Proposition~\ref{mrp2}
are Eq.~(\ref{mrcpri}) of       Lemma~\ref{mrl1} for
$x=a$.
Therefore the proof of correctness
directly follows from            Lemma~\ref{mrl1}.
The proof of correctness of Proposition~\ref{mrp2}
deals with the necessary and sufficient part, 
whereas   Lemma~\ref{mrl1} did no elaborate on the latter.

\begin{proof}
$ $ \\ $ $
$\Rightarrow$ (Sufficient condition.)
$ $ \\ $ $
This is   Lemma~\ref{mrl1} summarized.
Let $n$ be a prime number. Then by Fermat's
little theorem we have the following
\[
a^{n-1} \equiv 1 \pmod n ,
\]
for all $a \in \mb{Z}$ such that $\ndv{n}{a}$. 
Thefore we have the following.
\begin{eqnarray*}
a^{n-1}    &\equiv&   1 \pmod{n} \Leftrightarrow \\
a^{2^k l } &\equiv&   1 \pmod{n} \Leftrightarrow \\
( a^{2^{k-1}l} -1 ) ( a^{2^{k-1}l} +1 ) &\equiv& 0 \pmod{n} \Rightarrow \\
 a^{2^{k-1}l} &\equiv& \pm 1 \pmod{n} .
\end{eqnarray*}
The last modular equation gives one of the following.
$\dv{n}{a^{2^{k-1}l} -1 }$ or
$\dv{n}{a^{2^{k-1}l} +1 }$, for a prime number $p$.
We distinguish and examine the two cases separately.
$ $ \\ $ $
{\bf Case 1.} If $a^{2^{k-1}l} \equiv -1 \pmod n$, we are done.
This is the second case in the statement.
$ $ \\ $ $
{\bf Case 2.} If $a^{2^{k-1}l} \not\equiv -1 \pmod n$, then
$a^{2^{k-1}l} \equiv 1 \pmod n$, and we repeat the same
argument to conclude
$\dv{n}{a^{2^{k-2}l} -1 }$ or
$\dv{n}{a^{2^{k-2}l} +1 }$. 
$ $ \\ $ $
Likewise, at some point $a^{2^{m}l} \equiv -1 \pmod p$, for 
some $m$ such that $0 \leq m  <   k$, and we stop by
way of the second case of the statement again, or we
exhaust the $2^{m}$ to reach
$a^{l} \equiv -1 \pmod n$, 
or
$a^{l} \equiv 1 \pmod n$. 
In the former case, we reach the second case of the statement
for $m=0$; in the latter case we reach the first case of
the statement. In either case we are done. Sufficient
condition proved.
$ $ \\ $ $
$\Leftarrow$ (Necessary condition.)
$ $ \\ $ $
We distinguish two cases.
$ $ \\ $ $
{\bf Case 1.}
Let
$ a^{l}      \equiv   1 \pmod{n}$. Then we have the following.
\begin{eqnarray*}
a^{l}      &\equiv&   1 \pmod{n} \Rightarrow \\
a^{2^k l}  &\equiv&   1 \pmod{n} \Rightarrow \\
a^{n-1}    &\equiv&   1 \pmod{n}.
\end{eqnarray*}
$ $ \\ $ $
{\bf Case 2.}
Let
$ a^{2^m l}      \equiv   -1 \pmod{n}$. 
Then we also observe the following.
$ a^{2^m l}  \not\equiv   1 \pmod{n}$, 
since otherwise $n=2$, and $n$ is trivially prime.
We then have the following.
\begin{eqnarray*}
a^{2^m l}  &\equiv&   -1 \pmod{n} \Leftrightarrow \\
a^{n-1} = \left( a^{2^m l} \right)^{2^{k-m}} =
          \left( -1        \right)^{2^{k-m}}  
           &\equiv&   1 \pmod{n} ,
\end{eqnarray*}
since $m<k$ and thus $k-m>0$ and $2^{k-m} \geq 2$.
\end{proof}

\noindent
Algorithm~\ref{MillerRabin}
performs $t$ times the tasks of
Algorithm~\ref{MRtest} which is a simpler
 probabilistic algorithm.
Algorithm~\ref{MRtest}  checks for a given $a$
condition $C_1 (a,n)$ and for all values $m$
condition $C_2 (a,n,m)$. The range of values $m$
depends on $n$ since $n-1 = 2^k \cdot l$, where $k\geq 1$
and $l$ odd determines the range of $m$ as
follows: $0 \leq m < k$.
Algorithm~\ref{MillerRabin}  repeats 
Algorithm~\ref{MRtest} a number of times equal
to $t$ by drawing uniformly at random an $a$
from the range $\{ 2,3 , \ldots , n-2 \}$.
Algorithm~\ref{MRtest} is weak algorithm and 
for this we call
it a compositeness test algorithm.
The latter Algorithm~\ref{MillerRabin} is a more reliable 
one and is referred to as a primality test algorithm.

\begin{thm}
\label{tMillerRabin}
Algorithm~\ref{MillerRabin}
is a probabilistic algorithm
for testing whether $n$ is composite or not.
\end{thm}

%For MillerRabin
%An $a$ is drawn uniformly at random from
%$\{ 2,3, \ldots , n-2 \}$.
%In other words we repeat the following experiment as many times as needed.
%returns Composite with
%If $n$ is a composite number $n>4$ and
%odd, then Algorithm~\ref{MRtest} returns Composite with
%probability, approximately, at least $3/4$.

\begin{proof}
$ $ \\ $ $
Let $P_1 (a,n)$ and $P_2 (a,n,m)$ be the two conditions
of Proposition~\ref{mrp2} as indicated in
Eq.~(\ref{mrcpri}) of Lemma~\ref{mrl1} for $x=a$.
Let $C_1 (a,n)$ and $C_2 (a,n,m)$ be the two conditions of
Eq.(\ref{cp1}) and Eq.(\ref{cp2}) established respectively 
from Eq.~(\ref{mrecom}) of Lemma~\ref{mrl1} for $x=a$.
They are also in negation form part of
Eq.(\ref{mrcp1}) and Eq.(\ref{mrcp2}) in Proposition~\ref{mrp2}.
$ $ \\ $ $ 
In line 2 of Algorithm~\ref{MRtest},
or line 4 of Algorithm~\ref{MillerRabin},
an $a$  is picked uniformly at random such that $2 \leq a < n-1$.
$ $ \\ $ $
{\bf Step 1.}
The condition  $\ndv{n}{a}$ of Proposition~\ref{mrp2}
translates into $\gcd(a,n)>1$ since $a<n$,
and this maps to lines 3-5 of Algorithm~\ref{MRtest},
or lines 5-7 of Algorithm~\ref{MillerRabin}.
If $n$ is composite the algorithm terminates with the
correct answer Composite and $a$ becomes a witness of
the compositeness of $n$.
$ $ \\ $ $
We are left with a sequence of steps to determine other
wise the compositeness of $n$ for an $a$ such that
$\gcd(a,n)=1$.
$ $ \\ $ $
{\bf Step 2.} 
If 
\[
C_1 (a,n) \equiv \overline{P_1 (a,n)}:  \quad a^{l} \not\equiv 1 \pmod n ,
\]
and
\[
C_2 (a,n,m) \equiv \overline{P_2 (a,n,m)}:
 a^{2^m l} \not\equiv -1 \pmod n , \  \forall m, 0\leq m < k, 
\]
report $n$ to be composite and $a$ to be the witness of its
compositeness. This is a consequence of Eq.(\ref{mrecom})
of Lemma~\ref{mrl1}.
$ $ \\ $ $
{\bf Step 3.} Otherwise $n$ is either prime or composite.
$ $ \\ $ $
A single execution of Steps 1, 2, and 3 is 
Algorithm~\ref{MRtest}.
Repeat Step 1,2, and 3 a total of $t$ times and this
leads to Algorithm~\ref{MillerRabin}.
$ $ \\ $ $
Algorithm~\ref{MillerRabin} tests for the two conditions
$C_1 (a,n)$ and $C_2 (a,n,m)$ a number of times. 
We denote the number of times these conditions are tested as $t$.
Every time these conditions are tested we pick a random $a$
from the interval/set $\{ 2, 3, \ldots , n-2 \}$ of cardinality $n-1$.
If for the chosen $a$, we have $\overline{P_1 (a,n)}$ and
$\overline{P_2 (a,n,m)}$ for all applicable $m$, 
then $n$ is composite and $a$ is
the Rabin-Miller witness of the compositeness of $n$ and
denoted as MRw(a,n), or the set of witnesses for $n$ is 
denoted as MRw(n). Otherwise $a$ is a non-witness denoted
MRnw(a,n) and the corresponding set MRnw(n).
To simplify things we used
$C_1 (a,n) = \overline{P_1 (a,n)}$
and
$C_2 (a,n,m) = \overline{P_2 (a,n,m)}$.
\[
MRw(a,n) \equiv a \in MRw(n) \Leftrightarrow C_1 (a,n) \wedge 
   \left( \wedge_{m=0}^{k-1} C_2 (a,n,m) \right) .
\]
We obtain the following condition for 
non witnesses MRnw(a,n) or  set MRnw(n).
\[
MRnw(a,n) \equiv
a \in MRnw(n) 
\Leftrightarrow 
 \overline{C_1 (a,n)} \vee \left(
                \vee_{m=0}^{k-1} \overline{C_2 (a,n,m)} \right)
\Leftrightarrow 
           P_1 (a,n)  \vee \left(
                \vee_{m=0}^{k-1}           P_2 (a,n,m)  \right) .
\]

\SetKwComment{Comment}{/* }{ */}
\SetKwRepeat{Do}{do}{while}

\begin{algorithm}
\KwIn{$n >4$ is odd $n-1=2^k l$, $l$ odd}
\KwOut{$n$ is $\textbf{Composite}$ or $\textbf{Pseudoprime}$}
 $n-1 = 2^k \cdot l$ \;

 Pick $ a \in \{ 2, 3, \ldots , n-2 \}$ uniformly at random  \;

 \If{$\gcd(a,n) >1$}{
  \Return{$\mathbf{Composite}$} 
  \Comment*[r]{$a$ is a witness of $n$'s compositeness}
 }

 $ x = a^{l} \pmod n$;

 \eIf(\tcc*[f]{$C_1 (a)$}){$a^{l} \not\equiv 1 \pmod n$}{
   \;
 }{
  \Return{$\mathbf{PseudoPrime}$}
 }

 Condition = TRUE \;

 \For(\tcc*[f]{$C_2(a,n,m)$}){$m=0,1, \ldots , k-1$}{
  \eIf{$a^{2^m l} \not\equiv -1$}{
     Condition = Condition $\wedge$ TRUE;
  }{
     Condition = Condition $\wedge$  FALSE;
   
%     \Return{$\mathbf{PseudoPrime}$}
     break
     \Comment*[r]{\text{or\ } return($\mathbf{PseudoPrime}$)}
  }
 }
 \eIf{$\text{Condition}$}{
  \Return{$\mathbf{Composite}$}
  \Comment*[r]{$a$ is a witness of $n$'s compositeness}
 }{
  \Return{$\mathbf{PseudoPrime}$}
 }
 \caption{MillerRabinCTest (n): Miller-Rabin compositeness test}
 \label{MRtest}
\end{algorithm}

%Note that line 20 of Algorithm~\ref{MillerRabin}
%is equivalent to a return of PseudoPrime; whether 
%it is done through line 20 or line 29 is the difference
%of executing the conditional of line 23.
%Similarly, line 18 of Algorithm~\ref{MRtest}
Line 18 of Algorithm~\ref{MRtest}
is equivalent to a return of PseudoPrime; whether 
it is done through line 18 or line 24 is the difference
of executing the conditional of line 21.

There are two conditions that can be checked together:
these are $C_1 (a,n)$ and $C_2 (a,n,0)$ that map
to $a^l \not\equiv 1 \pmod n$ and
to $a^l \not\equiv -1  \pmod n$ respectively.
The former  is checked in line 9 of Algorithm~\ref{MillerRabin}
and the latter in line 15, $m=0$ of Algorithm~\ref{MillerRabin}
The condition $C_2 (a,n,m)$ for all $m$
such that $0\leq m < k$ is checked in lines 15-22
of Algorithm~\ref{MillerRabin}. It includes the check 
$C_2 (a,n,0)$. Similarly for Algorithm~\ref{MRtest}.

If $a^l \not\equiv 1 \pmod n$ is not true i.e. $C_1 (a,n)$
fails, there is no reason to continue with $C_2 (a,n,m)$ checking.
Algorithm~\ref{MRtest} exits through line 10.
Algorithm~\ref{MillerRabin} continues its execution. 
%one could have exited similarly to line 10 of Algorithm~\ref{MRtest}.
If $a^l \not\equiv 1 \pmod n$ is true i.e. $C_1 (a,n)$ succeeds
we continue with $C_2 (a,n,m)$. For a composite $n$, $C_2 (a,n,m)$ must
check true for all $m = 0, \ldots, k-1$. Thus if Condition is true,
this is the case and Algorithm~\ref{MillerRabin} exits through line 24,
otherwise an $m=j$ that results in a failure of $C_2 (a,n,j)$,
forces an exit through line 20 of the for loop for lines 15-22;
in that case Condition is false, and the algorithm continues
until the limit of $t$ iterations is reached. 
Therefore, if the algorithm is unsuccessful in proving the
compositeness of $n$, Algorithm~\ref{MillerRabin} exits through line 29.
Pseudoprime then means that $n$ can be composite (no proof
of the compositeness of $n$ was found) or it can 
ble prime (and this is the reason no
proof of the compositeness of $n$ was found).
$ $ \\ $ $

\begin{algorithm}[H]
\KwIn{$n >4$ is odd $n-1=2^k l$, $l$ odd; number of runs is $t$}
\KwOut{$n$ is $\textbf{Composite}$ or $\textbf{Pseudoprime}$}
$ n-1 = 2^k l $ ;
$ i = 0 $ ;

\Do{$i<t$}{
 $i=i+1$ \;
 Pick $ a \in \{ 2, 3, \ldots , n-2 \}$ uniformly at random  \;

 \If{$\gcd(a,n) >1$}{
  \Return{$\mathbf{Composite}$} 
  \Comment*[r]{$a$ is a witness of $n$'s compositeness}
 }

 $ x = a^{l} \pmod n$;

 \eIf(\tcc*[f]{$C_1 (a,n)$}){$a^{l} \not\equiv 1 \pmod n$}{
    \;
 }{
    continue;
 }

 Condition = TRUE \;

 \For(\tcc*[f]{$C_2(a,n,m)$}){$m=0,1, \ldots , k-1$}{
  \eIf(\tcc*[f]{$x=x^2 \bmod n; x \not\equiv -1 \pmod n$}){$a^{2^m l} \not\equiv -1$}{
     Condition = Condition $\wedge$ TRUE;
  }{
     Condition = Condition $\wedge$  FALSE;
   
     break;
  }
 }
 \eIf{$\text{Condition}$}{
  \Return{$\mathbf{Composite}$}
 }{
  continue \;
 }

 }
 \Return{$\mathbf{PseudoPrime}$}
 \Comment*[r]{$n$ is either prime or composite}
 \caption{Miller-Rabin primality test algorithm}
 \label{MillerRabin}
\end{algorithm}
Moreover, note that the powers of line 16 are generated
from $a^l$ i.e. the statement of line 8 through repeated doubling
as indicated in the comment of line 16. We do not need to generate
powers from scratch (starting from $a$) every time. Powers
are mod $n$ always. Therefore a $-1$ mod $n$ can be rewritten
$n-1$ mod $n$.
\end{proof}

\newpage

\begin{prp}[Probability of successful reporting]
\label{mrp}
If $n$ is an odd composite number greater than four
then Algorithm~\ref{MillerRabin} returns Composite with
probability at least $3/4$.
The probability that the output of Algorithm~\ref{MillerRabin} is
Composite  given that $n$ is a composite number
is at least $1 - 4^{-t}$.
\end{prp}

\begin{proof}
It was proved by Rabin \cite{R80} and independently
by Monier \cite{Monier} that for a composite odd $n$,
$|MRw(a)|$ is at least $3\phi (n)/4$ and $|MRnw(a)|$ 
is thus at most $\phi(n)/4$.
But the result is applicable to Miller1.
Algorithm~\ref{MillerRabin} uses the conditions of
Miller2.
Monier \cite{Monier} proved that the conditions used
in Miller1 and Miller2 are equivalent.
So we can use for this algorithm the claims of
Rabin \cite{R80} and also of Monier \cite{Monier}.
Later, we introduce RabinMiller which was
analyzed by Rabin \cite{R80}.
\end{proof}

\begin{lem}[Running time of algorithm~\ref{MillerRabin}]
The running time of
Algorithm~\ref{MillerRabin} is
$O(t \cdot \lg{n} \cdot M(n) )$, where $M(n)$
is the cost of multiplying $n$-bit integers.
\end{lem}
\begin{proof}
Exponentiation involves $O(\lg{n})$ multiplications.
Depending on how we
implement integer multiplication the overall time complexity
is $O(t \cdot \lg{n} \cdot M(n) )$, where $M(n)$
is the computational cost (bit model) of
multiplying $n$-bit integers.
\end{proof}

\newpage

\subsection{A reintepretation of the Miller-Rabin test under GRH}

\begin{thm}
\label{tBachMillerRabin}
Algorithm~\ref{BachMillerRabin}, BachMillerRabin, 
is a deterministic algorithm
for testing whether $n$ is prime or composite, under the GRH.
\end{thm}

Bach \cite{Bach} showed that for a Dirichlet function and thus
under the Extended/Generalized Riemann Hypothesis (ERH/GRH)
there is a MRw(a,n) witness $a$ which is at most $2 (\log{n})^2$ for
a composite $n$. ($\log{n}$ is $\ln{n}$; in the remainder we might
use as an upper bound $\lg{n}$ or $\ln{n}$ interchangeably.)

\begin{proof}
$ $ \\ $ $
\SetKwComment{Comment}{/* }{ */}
\SetKwRepeat{Do}{do}{while}
\begin{algorithm}[H]
\KwIn{$n >4$ is odd $n-1=2^k l$, $l$ odd; number of runs is $t$}
\KwOut{$n$ is composite or pseudoprime}
$ n-1 = 2^k l $ ;
$ i = 0 $ ;

\For{$a=2,3, \ldots , 2(\ln{n})^2$}{

 \If{$\gcd(a,n) >1$}{
  \Return{$\mathbf{Composite}$} 
  \Comment*[r]{$a$ is a witness of $n$'s compositeness}
 }

 \eIf(\tcc*[f]{$C_1 (a,n)$}){$a^{l} \not\equiv 1 \pmod n$}{
    \;
 }{
    continue;
 }

 Condition = TRUE \;

 \For(\tcc*[f]{$C_2(a,n,m)$}){$m=0,1, \ldots , k-1$}{
  \eIf{$a^{2^m l} \not\equiv -1$}{
     Condition = Condition $\wedge$ TRUE;
  }{
     Condition = Condition $\wedge$  FALSE;
   
     break;
  }
 }
 \eIf{$\text{Condition}$}{
  \Return{$\mathbf{Composite}$} \;
 }{
  continue \;
 }

 }
 \Return{$\mathbf{Prime}$}
 \Comment*[r]{$n$ is either prime or composite}
 \caption{BachMillerRabin(n) : 
Bach-based GRH Rabin-Miller primality test algorithm}
 \label{BachMillerRabin}
\end{algorithm}
Note that the line 25 return statement should be read as follows:
either $n$ is a prime number or the Generalized 
Riemann Hypothesis is false.
\end{proof}

\subsection{Equivalence of Miller1 and Miller2 conditions}

Algorithm~\ref{Miller1} was utilized by Rabin 
to derive the probabilistic primality testing
algorithm of \cite{R80} that  we will refer to as Rabin-Miller to
distinguish it from Miller-Rabin.

\noindent
Algorithm~\ref{Miller1} does not use the 
conditions of Eq.(\ref{mrcp1}) and Eq.(\ref{mrcp2}),
where
\begin{eqnarray}
\label{dmrcp1} % d for duplicate !
 P_1 (a,n) = \overline{C_1 (a,n)} &:& a^{l} \equiv 1 \pmod n , 
\end{eqnarray}
\begin{eqnarray}
\label{dmrcp2}
 P_2 (a,n,m) = \overline{C_2 (a,n,m)} &:& a^{2^m l} \equiv -1 \pmod n ,
\end{eqnarray}

\noindent
It uses instead the following conditions of Eq.(\ref{rmecom})
and Eq.(\ref{rmepri}), where
\begin{eqnarray}
\label{xmre2a}
 R_1 (a,n) = \overline{M_1 (a,n)} &:& a^{n-1} \not\equiv 1 \pmod n , 
\end{eqnarray}
\begin{eqnarray}
\label{xmre2b}
 R_2 (a,n,m) = \overline{M_2 (a,n,m)} &:& d_m \neq 1 \wedge d_n \neq n.
\end{eqnarray}
where $d_m = \gcd( a^{2^m \cdot l}-1 , n)$, $m \in \mb{N}$.

Then instead of using Eq.(\ref{mrepri}) and Eq.(\ref{mrecom})
it establishes  a new set of conditions.
We first rewrite Eq.(\ref{mrepri}) and Eq.(\ref{mrecom}), the
conditions used by Miller 2.
\begin{equation}
%\label{mrepri}
\text{n is prime} \Rightarrow \forall a , 1 \leq a < n :
{\small
\begin{cases}
 a^l       \equiv 1 \pmod n  & \text{Case } P_1 (a,n) \\
            \vee             &                  \\
 a^{2^m l} \equiv -1 \pmod n ,\quad  \exists m: 0 \leq m < k
                             & \text{Case } P_2 (a,n,m) \\
\end{cases} \tag{\ref{mrepri}} \nonumber
}
\end{equation}

\begin{equation}
%\label{mrecom}
\text{n is composite} \Leftrightarrow
\exists a , 1 \leq a < n  :
{\small
\begin{cases}
 a^l       \not\equiv 1 \pmod n  & \text{Case } \overline{P_1 (a,n)} \\
           \wedge             &                  \\
 a^{2^m l} \not\equiv -1 \pmod n ,\quad  \forall m, 0 \leq m < k :
                             & \text{Case } \overline{P_2 (a,n,m)} \\
\end{cases} \tag{\ref{mrecom}} \nonumber
}
\end{equation}

Then we write the new set of conditions used by Miller1.
\begin{equation}
\label{rmepri}
\text{n is prime} \Rightarrow
\forall a: \quad \ndv{n}{a} \wedge \quad
{\small
\begin{cases}
 a^{n-1}   \equiv 1 \pmod n  & 
           \text{Case } M_1(a,n) = \overline{R_1 (a,n)} \\
           \wedge            &                  \\
 \forall m, 0 \leq m < k: d_m = 1 \vee d_m = n  
                             & 
           \text{Case } M_2 (a,n,m) = \overline{R_2 (a,n,m)} \\
\end{cases}
}
\end{equation}

\begin{equation}
\label{rmecom}
\text{n is composite} \Leftrightarrow
\exists a : \quad \dv{n}{a} \quad \vee   \quad
{\small
\begin{cases}
 a^{n-1}\not\equiv 1 \pmod n  & \text{Case } R_1 (a,n) \\
           \vee              &                  \\
 \exists m, 0 \leq m < k: d_m \neq 1 \wedge d_m \neq n 
                             & \text{Case } R_2 (a,n,m) \\
\end{cases}
}
\end{equation}

\begin{prp}[Equivalence of condition in Miller1 and Miller 2]
Let $n$ be an odd number $n>4$ and that
$n-1=2^k \cdot l$, $k \geq 1$ and $l$ is odd.
The conditions of Eq.(\ref{rmepri}) and Eq.(\ref{rmecom})
utilized by MillerRabin are equivalent to 
the conditions of Eq.(\ref{mrepri}) and Eq.(\ref{mrecom})
utilized by RabinMiller.
\end{prp}

\begin{proof}
$ $ \\ $ $
For both set of conditions we prove the equivalence
for $n$ being a prime number. The case $n$ is composite
can be proved by symmetry using De Morgan's Law.
Consider the sequence of the $d$ values forming the following
row vector in that case.
\[
 ( d_0 , d_1 , \ldots , d_{k-1}, d_k ).
\]
We will show first that if a $d_j$ value is $n$ then all
the values to its right will also be $n$ that is
$d_i = n$ for all $i>j$.
Thus $1$s can extend only to the left of the leftmost $n$.
As a conclusion the row vector can be: (a) all $n$, or
(b) a sequence of consecutive 1s  followed by consecutive 
$n$s,
and (c) the rightmost position is $n$.
%% Moreover $d_j = n$ is equivalent to $a^{2^j l} \equiv 1  \pmod n$.
%% Moreover $d_j = 1$ is equivalent to $a^{2^j l} \equiv -1 \pmod n$.
%
%
%
$ $ \\ $ $
{\bf Condition $M_1 (a,n)$.} By way of $M_1 (a,n)$ we have the
following
\begin{eqnarray}
\label{mpc2a}
a^{n-1}   \equiv 1 \pmod n
&\Rightarrow&
\dv{n}{a^{n-1}-1}
\Rightarrow
\dv{n}{a^{2^k l}-1} \nonumber \\
&\Rightarrow&
a^{2^k l} -1 \equiv 0 \pmod n
\Rightarrow
\gcd(a^{2^k l}-1,n)=n\nonumber \\
&\Rightarrow &
M_a(a,k) : d_k = n.
\end{eqnarray}
Therefore we obtain that the right-most position
of the row vector has $d_k = n$.
\[
 ( d_0 , d_1 , \ldots , d_{k-1}, n=d_k ).
\]
$ $ \\ $ $
{\bf Condition $M_2 (a,n,m)$ case (i): 
for some  $0\leq j<k$ : $d_j =n$.} 
We prove that a $n$ in some position $j$ is followed
to the right with a block of $n$s.
Consider that for some $m$,
call it $j$, we have for $M_2 (a,n,j)$ the following:
$d_j = n$.
This implies
$d_j = n = \gcd(a^{2^j l}-1,n)$, and also
$a^{2^j l} \equiv 1 \pmod n$.
Since $a^{2^{j+1}l} = ( a^{2^j l})^2$,
$a^{2^{j+1} l} \equiv 1 \pmod n$,
and therefore
$\gcd(a^{2^{j+1} l}-1,n)=d_{j+1} = n$, thus having
$d_{j+1} =n$.
By induction, all $i  >   j$ have $d_i =n $ as well.
Therefore we obtain the following refinement for case (i).
\[
 ( d_0 , d_1 , \ldots , d_{j-1}, d_j =n, d_{j+1}=n, \ldots , d_{k-1}=n, d_k =n ).
\]
If $j=0$ we have an all $n$ row-vector.
\[
 ( d_0 =n , d_1 = n  , \ldots ,  d_{k-1}=n, d_k =n ).
\]
We conclude that $d_j = n$ is equivalent to 
$a^{2^j l} \equiv 1  \pmod n$.
$ $ \\ $ $
{\bf Condition $M_2 (a,n,m)$ case (ii): if the left-most
$n$ is at position $j>0$ i.e.  $d_j =n$ there is a 
$d_{j-1} =1$.}
Consider that for some $m$,
call it $j$, we have for $M_2 (a,n,j)$ the following:
$d_j = n$ and it is the leftmost $n$ entry of the row vector.
The latter implies
$\gcd(a^{2^j l}-1,n)=d_j = n$ or $a^{2^j l} \equiv 1 \pmod n$.
Let us consider $m=j-1$. We have $d_{j-1}=1$.
$ $ \\ $ $
Let us set $x= a^{2^{j-1} l}$.
Then $x^2 = a^{2^{j}l}$. Since $a^{2^j l} \equiv 1 \pmod n$,
we have $x^2 \equiv 1 \pmod n$.
\begin{eqnarray*}
x^2         &\equiv& 1 \pmod n \\
(x-1)(x+1)  &\equiv& 0 \pmod n \\
(x-1)(x+1)  &\equiv& 0 \pmod{p_i^{a_i}} \\
(x+1)       &\equiv& 0 \pmod{p_i^{a_i}} \\
x           &\equiv&-1 \pmod{p_i^{a_i}},
\end{eqnarray*}
and given $\gcd(p_u^{a_u}, p_v^{a_v})=1$, we have
$ x        \equiv-1 \pmod{n}$.
We need to explain the step that led to the dismisal
of $(x-1)$. We claimed that
$\gcd(p_i^{a_i}, x-1)=1$.
This is because of the following.
\begin{eqnarray*}
1 = d_{j-1}  &=& \gcd(a^{2^{j-1}l}-1,n) \\
             &=& \gcd(x-1,n)            \\
             &=& \gcd(x-1,p_i^{a_i}) \forall i  \\
\end{eqnarray*}
In conclusion $a^{2^{j-1}l} \equiv -1 \pmod n$.
On the left of a 1 we can only have 1. If there
was an $n$ then to its right there should have been
an $n$ and not an $1$.
We conclude that $d_j = 1$ is equivalent to 
$a^{2^j l} \equiv -1 \pmod n$.
Now that we examined the cases for $M_1 (a,n)$ and
$M_2 (a,n,m)$, we move to $P_1 (a,n)$ and $P_2 (a,n,m)$.
$ $ \\ $ $
{\bf Mapping $M_1 , M_2$ to $P_1 , P_2$.}
$ $ \\ $ $
A row vector that has $n$ in the left-most position
means
$d_0 = n$ which is equivalent to 
$a^{2^0 l} \equiv a^{l} \equiv 1  \pmod n$.
This is case $P_1 (a,n)$.
A row vector that has $1$ in the left-most position $m=0$
and $n$ in position $m=1$, by prior discussion (case (ii))
$a^{l} \equiv -1 \pmod n$ which is $P_2 (a,n,0)$.
If the right-most 1 is in position $j-1$ then
$a^{2^{j-1}l} \equiv -1 \pmod n$ which is $P_2(a,n, j-1)$.
The rightmost $m$ for a $d_j =1$ is $j=k-1$. Since $d_k =n$
this maps to $a^{2^{k-1}l} \equiv -1 \pmod n$ which is $P_2 (a,n, k-1)$.
There is a 1-1 map of $M_1 (a,n) \mapsto P_1 (a,n)$
and $M_2 (a,n,m) \mapsto P_2 (a,n,m)$.
This concludes the proof.
$ $ \\ $ $
{\bf Supplement 1:} Show $a^{2^{j+1}l} \equiv 1 \pmod n$
if $a^{2^{j}l} \equiv 1 \pmod n$.
Let $n=p_1^{a_1} \ldots p_r^{a_r}$ be the prime decomposition
of $n$.
$a^{2^{j}l} \equiv 1 \pmod n$ implies for every $i=1,\ldots , r$
that
$a^{2^{j}l} \equiv 1 \pmod{p_i^{a_i}}$.
\begin{eqnarray*}
a^{2^{j}l}        &\equiv& 1 \pmod{p_i^{a_i}} \\
(a^{2^{j}l}-1)^2  &\equiv& 0 \pmod{p_i^{a_i}} \\
a^{2^{j+1}l}+1 -2a^{2^{j+1}}l  &\equiv& 0 \pmod{p_i^{a_i}} \\
a^{2^{j+1}l}+1 -2  &\equiv& 0 \pmod{p_i^{a_i}} \\
a^{2^{j+1}l}  &\equiv& 1 \pmod{p_i^{a_i}},
\end{eqnarray*}
and given $\gcd(p_u^{a_u}, p_v^{a_v})=1$, we have
$a^{2^{j+1}l}  \equiv 1 \pmod{n}$ i.e. $d_{j+1}=n$.
$ $ \\ $ $
% a/c b/c (a,b)=1
%(a,b)=1 => ax+by=1
% a/c => c=aA; b/c => c=bB;
% ax+by=1 => cax+cby = c =>
%  bBax +aAby=c => ab/c.
\end{proof}

\subsection{The Rabin-Miller primality test}

Rabin \cite{R80} converted  Algorithm~\ref{Miller1} 
into a  probabilistic primality test algorithm.

\begin{prp}[Rabin \cite{R80}]
Algorithm~\ref{RabinMiller} is a probabilistic algorithm
for testing whether $n$ is prime or composite.
\end{prp}

\begin{proof}
$ $ \\ $ $
Let $W(a,n)$ \cite{R80} indicate that $a$ is a witness of the compositeness
of $n$. Then denote by $W(n)$ the set of all $W(a,n)$ witnesses.
Later on we we will use the terminology $MRw(a,n)$ and $MRw(n)$.
for $W(a,n)$ and $W(n)$ respectively.
\[
  W(n) = \left\{ a : W(a,n) \right\} .
\]
It was proved in \cite{R80} that for $n>4$ and $n$ composite
\[
  |W(n)| \geq \frac{3(n-1)}{4}
\]
In a theorem to follow, it will be shown that
\[
 |MRnw(n)| \leq \frac{\phi(n)}{4} \leq \frac{n-1}{4},
\]
that implies
\[
 |MRw(n)| \geq \frac{3(n-1)}{4}.
\]
The proof of correctness appears in \cite{M76} and the
derivations above \cite{R80}.
We present later a  proof that is based on \cite{RC2e},
with the latter originating from the paper of \cite{Monier}.

\SetKwComment{Comment}{/* }{ */}
\SetKwRepeat{Do}{do}{while}
\begin{algorithm}[H]
\KwIn{$n >4$ is odd $n-1=2^k l$, $l$ odd; number of runs is $t$}
\KwOut{$n$ is $\mathbf{Composite}$ or $\mathbf{PseudoPrime}$}
$ n-1 = 2^k l $ ;
$ i = 0 $ ;

\If(\tcc*[f]{Line 1}){PerfectPower(n)}{
 \Return{$\mathbf{Composite}$}
 \Comment*[r]{\cite{R80} is unclear about this test}
}

\Do{$i<t$}{
 $i=i+1$ \;
 Pick $ a \in \{ 2, 3, \ldots , n-2 \}$ uniformly at random  \;

 \If(\tcc*[f]{Line 2 (i)}){$(\dv{a}{n}) || (\gcd(a,n) >1)$}{
  \Return{$\mathbf{Composite}$}
  \Comment*[r]{\cite{R80} is unclear about this test}
 }

 $ x = a^{l} \pmod n$;

 Calculate using repeated doubling $a^{2^i l} \bmod n$, for $i=1 , \ldots , k$
  \Comment*[r]{$a^{2^k l} \bmod n$ is $a^{n-1} \bmod n$}
 
 \If(\tcc*[f]{Line 2 (ii)}){$a^{n-1} \not\equiv 1 \pmod n$}{
  \Return{$\mathbf{Composite}$} 
  \Comment*[r]{$a$ is a $W(a,n)$}
 }

\If(\tcc*[f]{Line 2 (iii)}){$\exists m, 0\leq m <k: \gcd(( a^{2^k l} \bmod n)-1, n) \neq 1, n$}{
  \Return{$\mathbf{Composite}$} 
  \Comment*[r]{$a$ is a $W(a,n)$}
 }

 }
 \Return{$\mathbf{PseudoPrime}$}
 \Comment*[r]{Line 3}

 \caption{Rabin-Miller primality test algorithm}
 \label{RabinMiller}
\end{algorithm}

\end{proof}

\begin{prp}[Probability of successful reporting]
\label{rmp}
If $n$ is an odd composite number greater than four
then Algorithm~\ref{RabinMiller} returns Composite with
probability at least $3/4$.
The probability that the output of Algorithm~\ref{RabinMiller} is
Composite  given that $n$ is a composite number
is at least $1 - 4^{-t}$.
\end{prp}

\begin{proof}
It was proved by Rabin \cite{R80} and independently
by Monier \cite{Monier} that for a composite odd $n$,
$|MRw(a)|$ is at least $3\phi (n)/4$ and $|MRnw(a)|$ 
is thus at most $\phi(n)/4$.
\end{proof}

\begin{lem}[Running time of algorithm~\ref{RabinMiller}]
The running time of
Algorithm~\ref{RabinMiller} is
$O(t \cdot \lg{n} \cdot M(n) )$, where $M(n)$
is the cost of multiplying $n$-bit integers.
\end{lem}
\begin{proof}
Exponentiation involves $O(\lg{n})$ multiplications.
Depending on how we
implement integer multiplication the overall time complexity
is $O(t \cdot \lg{n} \cdot M(n) )$, where $M(n)$
is the computational cost (bit model) of
multiplying $n$-bit integers.
\end{proof}

\newpage

\subsection{Examples on witnesses}

\begin{exa}
Consider $n=9$, a composite and odd integer $n>4$.
It is $n-1 = 2^k \cdot l =  2^3 \cdot 1$, with $l$
being an odd integer. 
Determine the MRw(n) and MRnw(n).
\end{exa}

\begin{solution}
We note that $n=9$, we have $n-1=2^k \cdot l$, where
$k=3$ and $l=1$.
The first column is $a$.
The second column shows     $a^{l} \bmod n$.
The third  column indicates whether $C_1 (a,n)$ is the case i.e.
$a^{l} \not\equiv 1 \pmod n$. 
A $+$ indicates so, a $-$ indicates the opposite i.e.
$a^{l}     \equiv 1 \pmod n$. 
The fourth column indicates whether $C_2 (a,n,0)$ is the case i.e.
$a^{l} \not\equiv -1 \pmod n$.
A $+$ indicates so, a $-$ indicates the opposite i.e.
$a^{l}     \equiv -1 \pmod n$. 
$ $ \\ $ $
The fifth  column shows   $a^{2^1 \cdot l} \bmod n$.
The sixth  column indicates with a $+$
whether $C_2 (a,n,1)$ is the case i.e.
$a^{2^1 \cdot l} \not\equiv -1 \pmod n$;
otherwise a $-$ is shown.
The seventh column shows   $a^{2^2 \cdot l} \bmod n$.
The eighth  column indicates with a $+$ 
whether $C_2 (a,n,2)$ is the case i.e.  
$a^{2^2 \cdot l} \not\equiv -1 \pmod n$; 
otherwise a $-$ is shown.
In the ninth column a $+$ indicates that $a$
is a witness, and an $-$ that $a$ is a non-witness 
to the compositeness of $n$.
$ $ \\ $ $
For $a \in MRw(n)$ we need only $+$'s for the corresponding row.
For $a \in MRnw(n)$ we need at least one $-$ or equivalently
an $n$ in the MRw(n) column.
\[
\begin{array}{llcclclcc}
a  &a^{    l}:      &C_1(a,n)   &C_2(a,n,0)
   &a^{2^1 \cdot 1}:&C_2(a,n,1)
   &a^{2^2 \cdot 1}:&C_2(a,n,2) & MRw(n)\\
1 & 1& - & + & 1 & +  & 1 & + & n     \\
2 & 2& + & + & 4 & +  & 7 & + & y     \\
3 & 3& + & + & 0 & +  & 0 & + & y     \\
4 & 4& + & + & 7 & +  & 4 & + & y     \\
5 & 5& + & + & 7 & +  & 4 & + & y     \\
6 & 6& + & + & 0 & +  & 0 & + & y     \\
7 & 7& + & + & 4 & +  & 7 & + & y     \\
8 & 8& + & - & 1 & +  & 1 & + & n     \\
\end{array}
\]

We observe the number of MRw(a,n) is $\geq \frac{3(n-1)}{4} = 6$.
The number of MRnw(a,n) is $\leq \frac{(n-1)}{4} = 2$.
\end{solution}

\begin{exa}
Consider $n=13$, a non-composite (!!) and odd integer $n>2$.
It is $n-1 = 2^2 \cdot 3$. Using the prior definition
we have $k=2$ and $l=3$, with $l$ being a odd number.
Determine the MRw(n) and MRnw(n).
\end{exa}

\begin{solution}
$ $ \\ $ $
We note that $n=13$, we have $n-1=2^2 \cdot 3$, where
$k=2$ and $l=3$.
The first column is $a$.
The second column shows     $a^{l} \bmod n$.
The third  column indicates whether $C_1 (a,n)$ is the case i.e.
$a^{l} \not\equiv 1 \pmod n$. 
A $+$ indicates so, a $-$ indicates the opposite i.e.
$a^{l}     \equiv 1 \pmod n$. 
The fourth column indicates whether $C_2 (a,n,0)$ is the case i.e.
$a^{l} \not\equiv -1 \pmod n$.
A $+$ indicates so, a $-$ indicates the opposite i.e.
$a^{l}     \equiv -1 \pmod n$. 
$ $ \\ $ $
The fifth  column shows   $a^{2^1 \cdot l} \bmod n$.
The sixth  column indicates with a $+$
whether $C_2 (a,n,1)$ is the case i.e.
$a^{2^1 \cdot l} \not\equiv -1 \pmod n$;
otherwise a $-$ is shown.
%The seventh column shows   $a^{2^2 \cdot l} \bmod n$.
%The eighth  column indicates with a $+$ 
%whether $C_2 (a,n,2)$ is the case i.e.  
%$a^{2^2 \cdot l} \not\equiv -1 \pmod n$; 
%otherwise a $-$ is shown.
%In the nhe ninth column a $+$ indicates that $a$
In the seventh column a $+$ indicates that $a$
is a witness, and an $-$ that $a$ is a non-witness 
to the compositeness of $n$.
$ $ \\ $ $
For $a \in MRw(n)$ we need only $+$'s for the corresponding row.
For $a \in MRnw(n)$ we need at least one $-$ or equivalently
an $n$ in the MRw(n) column.
There are no witnesses to the compositeness of $13$; we know
that $13$ is a prime number not a composite number!
We observe the number of MRw(a,n) is $0$.
The number of MRnw(a,n) is thus $n-1=12$.
\[
\begin{array}{llcclcc}
a  &a^{    l}:      &C_1(a,n)   &C_2(a,n,0)
   &a^{2^1 \cdot 1}:&C_2(a,n,1)
                              & MRw(n)\\
1 & 1& - & + & 1 & +  & n     \\
2 & 8& + & + &12 & -  & n     \\
3 & 1& - & + & 1 & +  & n     \\
4 &12& + & - & 1 & +  & n     \\
5 & 8& + & + &12 & -  & n     \\
6 & 8& + & + &12 & -  & n     \\
7 & 5& + & + &12 & -  & n     \\
8 & 5& + & + &12 & -  & n     \\
9 & 1& - & + & 1 & +  & n     \\
10&12& + & - & 1 & +  & n     \\
11& 5& + & + &12 & -  & n     \\
12&12& + & - & 1 & +  & n     \\
\end{array}
\]
\end{solution}

\newpage
\subsection{Solovay-Strassen and Miller-Rabin failure probability bounds}

Both in  the $t$-round
probabilistic version of the Solovay-Strassen test 
and
probabilistic version of the Rabin-Miller     test 
we derive bounds on the probability of
\[
P_{SolStr} (E_1 / E_2 ) \leq \frac{1}{2^t},
\]
\[
P_{RabMil} (E_1 / E_2 ) \leq \frac{1}{4^t},
\]
where
\[
E_1 : \text{event that $t$ runs of the algorithm produce NO witnesses},
\]
\[
E_2 : \text{$n$ is a composite number},
\]
and of course
\[
E_2^\prime : \text{$n$ is a prime number}.
\]
However we are interested in the probability
\[
P_{SolStr} (E_2 / E_1 ) ,
\]
and
\[
P_{RabMil} (E_2 / E_1 ) ,
\]
instead.

\begin{prp}[\cite{ConradSS}]
For terms as previously defined,
by using Bayes theorem we conclude (\cite{ConradSS}) the following.
\[
P_{SolStr} (E_2^\prime / E_1 ) \geq  1 - \frac{\ln{n}-1}{2^t}
\]
and
\[
P_{RabMil} (E_2^\prime / E_1 )  \geq  1 - \frac{\ln{n}-1}{4^t} , 
\]
or equivalently
\[
P_{SolStr} (E_2 / E_1 ) \leq  \frac{\ln{n}-1}{2^t}
\]
and
\[
P_{RabMil} (E_2 / E_1 )  \leq \frac{\ln{n}-1}{4^t} , 
\]
\end{prp}

\begin{proof}
$ $ \\ $ $
{\bf Bayes theorem simple form.}
The simple form of Bayes theorem for two event $A$ and $B$ with $P(B) \neq 0$ takes the
following form.
\[
  P(A/B) = \frac{P(B/A)P(A)}{P(B)}.
\]
A proof is based on conditional probability. Therefore for events $A,B$ we
have the following.
\[
P(A \cap B) = P(A/B) P(B), \quad \quad P(A \cap B) = P(B/A) P(A).
\]
Equating the right-hand sides and solving for $P(A/B)$ given $P(B) \neq 0$ gives
the desired result.
$ $ \\ $ $
{\bf Bayes theorem for two mutually exclusive and exhaustive events $A$ and $A^\prime$.}
We denote $A^\prime$ the complement of $A$ with $A \cap A^\prime = \emptyset $ and
$A \cup A^\prime = \Omega $ that is $A$ and $A^\prime$ are mutually exclusive and exhaustive.
Then
\begin{equation}
\label{Bayes1}
  P(A/B) = \frac{P(B/A)P(A)}{P(B)} = \frac{P(B/A)P(A)}{P(B/A)P(A) + P(B/A^\prime ) P(A^\prime)}.
\end{equation}
Likewise we can also have the following.
\begin{equation}
\label{Bayes2}
  P(A^\prime /B) = \frac{P(B/A^\prime )P(A^\prime )}{P(B/A)P(A) + P(B/A^\prime ) P(A^\prime)}.
\end{equation}
We map $A, A^\prime , B$ to the events $ E_2 , E_2^\prime , E_1$.
\begin{eqnarray*}
A         &\mapsto& E_2         , \quad \text{$n$ is Composite event},  \\
A^\prime  &\mapsto& E_2^\prime  , \quad \text{$n$ is Prime event},  \\
B         &\mapsto& E_1         . \quad \text{NoWitness event in $t$ rounds}.
\end{eqnarray*}
$ $ \\ $ $
\begin{lem}
\label{lb10}
For complementary events $E_1 \cap E_2$ and $E_1 \cap E_2^\prime$ we have the following.
\begin{equation}
\label{lb11}
P (E_2 / E_1 ) = 1 - P (E_2^\prime / E_1 )
\end{equation}
\end{lem}
For complementary events $E_1 \cap E_2$ and $E_1 \cap E_2^\prime$  we
have
\begin{eqnarray*}
P(E_1 ) &=& P(E_1 \cap E_2 ) + P(E_1 \cap E_2^\prime)  \\
\frac{P(E_1 )}{P(E_1 )} &=& \frac{P(E_1 \cap E_2 ) + P(E_1 \cap E_2^\prime)}{P(E_1 )}  \\
          1             &=& \frac{P(E_1 \cap E_2 )}{P(E_1 )} + \frac{P(E_1 \cap E_2^\prime)}{P(E_1 )}  \\
          1             &=& P(E_2 / E_1 ) + P(E_2^\prime /E_1 ),
\end{eqnarray*}
and after solving for $ P(E_2 / E_1 )$ the lemma follows.
$ $ \\ $ $
Furthermore we have the following for the Solovay-Strassen algorithm
\begin{equation}
\label{lbSS}
P( E_1 / E_2 ) \leq \frac{1}{2^t}.
\end{equation}
By Equation~\ref{Bayes2} we have the following with the substitutions obtained through the
three mappings for $A, A^\prime $ and $B$.
\begin{equation}
\label{Bayes3}
  P(E_2^\prime / E_1 ) = 
  \frac{P(E_1 /E_2^\prime)P(E_2^\prime)}{P(E_1 / E_2)P(E_2) + P(E_1 /E_2^\prime ) P(E_2^\prime )}.
\end{equation}
By the prime number theorem, quoting the density of primes, we have the following.
\begin{equation}
\label{lb13}
  P( E_2^\prime ) \approx \frac{1}{\ln{n}}.
\end{equation}
Furthermore,
\begin{equation}
\label{lb14}
  P( E_2 ) = 1 - P( E_2^\prime ) \approx 1 - \frac{1}{\ln{n}}.
\end{equation}
Moreover,
\begin{equation}
\label{lb15}
  P( E_1 / E_2^\prime ) = \frac{P( E_1 \cap E_2^\prime}{P(E_2^\prime )} = 1,
\end{equation}
since no witness for the compositeness of $n$ is to be generated for $n$ a 
prime number.
Therefore Eq.~(\ref{Bayes3}) by way of Eq.(\ref{lb13}),  Eq.(\ref{lb14}),  Eq.(\ref{lb15}
yields the following.
\begin{eqnarray}
\label{Bayes4}
  P(E_2^\prime / E_1 ) &=& 
  \frac{P(E_1 /E_2^\prime)P(E_2^\prime)}
       {P(E_1 / E_2)P(E_2) + P(E_1 /E_2^\prime ) P(E_2^\prime )} 
  \nonumber\\
                       &=&  
  \frac{ 1 \cdot \frac{1}{\ln{n}}}
       {P(E_1 /E_2 ) \cdot ( 1 - \frac{1}{\ln{n}})+ 1 \cdot \frac{1}{\ln{n} }} 
  \nonumber\\
                       &=&  
  \frac{ 1 }
       {P(E_1 /E_2 ) \cdot ( \ln{n} - 1)+ 1  }  \text{by way of Eq.~(\ref{lbSS})}
  \nonumber\\
                       &\geq&  
  \frac{ 1 }
       {1 + \frac{\ln{n}-1}{2^t}   }  
  \nonumber\\
                       &\geq&  
       {1 - \frac{\ln{n}-1}{2^t}   }  ,
\end{eqnarray}
where in the last equation one can use $1/(1+x) \geq 1 -x $ for $0<x< 1$.
Therefore one can also write
\begin{equation}
\label{SS1a}
  P(E_2^\prime / E_1 ) \geq 1 - \frac{\ln{n}-1}{2^t}  \geq  1- \frac{\ln{n}}{2^t},
\end{equation}
and by using Lemma~\ref{lb10} one can obtain the following.
\begin{equation}
\label{SS1b}
  P(E_2 / E_1 ) \leq \frac{\ln{n}-1}{2^t} \leq \frac{\ln{n}}{2^t}.
\end{equation}
The corresponding results for the 
Rabin-Miller probabilistic primality test then become 
as follows.
\begin{equation}
\label{MR1a}
  P(E_2^\prime / E_1 ) \geq 1 - \frac{\ln{n}-1}{4^t}  \geq  1- \frac{\ln{n}}{4^t},
\end{equation}
\begin{equation}
\label{MR1b}
  P(E_2 / E_1 ) \leq \frac{\ln{n}-1}{4^t} \leq \frac{\ln{n}}{4^t}.
\end{equation}
This completes the proof of the lemma.
\end{proof}

\noindent
Let $n \in \mb{Z}_+$ be an odd composite integer with $n> 10$,
such that $n-1 = 2^k l$, where $l$ is an odd integer.
Furthermore let 
$n = p_1^{a_1}  p_2^{a_2} \ldots p_r^{a_r}$ be the prime
decomposition of $n$, and
let $d(n)=r$ indicate the number of distinct prime divisors
of $n$.

Let $b(n)$ be the largest integer value such that
$\dv{2^{b(n)}}{p-1}$, for every prime number $p$ such that
$\dv{p}{n}$.

It has already been shown that $n$ is a prime number if and only
if for every $a$ such  $\ndv{n}{a}$
either
\[
 P_1 (a,n) : a^{l} \equiv 1 \pmod n ,
\]
or 
\[
 P_2 (a,n,m) :  a^{2^m l} \equiv -1 \pmod n ,
\]
for some integer $m$ such that $0 \leq m  <   k$.

\noindent
Show that if $ a \in MRnw(n)$
then the following holds
\begin{equation}
\label{mrc1}
a^{2^{b(n)-1} \cdot l} \equiv \pm 1 \pmod n .
\end{equation}

\begin{proof}
$ $ \\ $ $
We reconsider the following from an earlier discussion.
If $a \in MRw(n)$, that is, $a$ is a $MRw(a,n)$ or $a$ is a Miller-Rabin
witness of the compositeness of $n$, then the following apply.
\begin{equation}
\label{mrpc1b}
\text{n is composite} \Leftrightarrow
\exists a: \quad \dv{n}{a} \quad \vee \quad
\begin{cases}
 a^l       \not\equiv 1 \pmod n  & \text{Case } C_1(a) = \overline{P_1 (a,n)} \\
           \text{and}         &                  \\
\forall m, 0 \leq m < k : a^{2^m l} \not\equiv -1 \pmod n 
                             & \text{Case } C_2 (a,n,m) = \overline{P_2 (a,n,m)} \\
\end{cases}
\end{equation}
On the other hand if $a \not\in MRw(n)$ that is,
$a \in MRnw(n)$ or $a$ a Miller-Rabin NON-witness of the
compositeness of $n$ the following apply then.
Integer $n$ can be a prime number or not: we classify it as a
PseudoPrime then.
\begin{equation}
\label{mrc1a}
\text{n is prime} \Leftrightarrow
\forall a : \quad \ndv{n}{a} \quad \wedge \quad
\begin{cases}
 a^l       \equiv 1 \pmod n  & \text{Case } P_1 (a,n) \\
           \text{or}         &                  \\
\exists m, 0 \leq m < k : a^{2^m l} \equiv -1 \pmod n 
                             & \text{Case } P_2 (a,n,m) \\
\end{cases}
\end{equation}
Consider  $a \in MRnw(n)$. There can be one of two cases
then: $P_1 (a,n)$ or $P_2 (a,n,m)$ for some $m$.
$ $ \\ $ $
{\bf Case 1.} If $P_1 (a,n)$ is the case then
 $a^{l} \equiv 1 \pmod n$ and we are done since then
$ a^{2^{b(n)-1} \cdot l} \equiv 1 \pmod n$.
$ $  \\ $ $ 
{\bf Case 2.} Otherwise $P_2 (a,n,m)$ is the case for
some $m$ such that $0\leq m < k$. Call that $m$, $j$ that
is for $m=j$ we have the following:
$ a^{2^{j} \cdot l} \equiv -1 \pmod n$.
For any prime number $p$ dividing $n$ 
(e.g. $p_1 , p_2 , \ldots , p_r$) we also 
have the following.
\begin{equation}
\label{mrc1c}
 a^{2^{j} \cdot l} \equiv -1 \pmod p .
\end{equation}
Consider $u= ord_p (a)$. The $u$ is the smallest positive
integer such that $a^u \equiv 1  \pmod p$.
From Eq.~(\ref{mrc1c}), and the definition of $u$
we conclude $\dv{u}{2^{j} \cdot l}$.
From Eq.~(\ref{mrc1c}) we obtain
$a^{2^{j+1} \cdot l} \equiv -1 \pmod p$, and
from the definition of $u$ we further obtain
that $\dv{u}{2^{j+1} \cdot l}$,
Furthermore $\ndv{u}{2^{j} \cdot l}$ since otherwise
it would have been $  a^{2^{j} \cdot l} \equiv 1 \pmod p$.
Therefore $u$ has $2^{j+1}$ as the largest power of two dividing
it ($u$) that is $\dv{2^{j+1}}{u}$.
Furthermore for prime $p$ we have by Fermat's little theorem
$a^{p-1} \equiv 1 \pmod p$, which implies $\dv{u}{p-1}$,
and therefore $\dv{2^{j+1}}{p-1}$.
$ $ \\ $ $
The latter implies $j+1 \leq b(n)$.
$ $ \\ $ $
We then have two cases.
$ $ \\ $ $
{\bf Case 2a: $j+1 = b(n)$.}
Then $j=b(n)-1$ and therefore by Eq.~(\ref{mrc1c})
we have
\[
a^{2^{b(n)-1} \cdot l} \equiv -1 \pmod p .
\]
$ $ \\ $ $
{\bf Case 2b: $j+1 < b(n)$.}
Then $j < b(n)-1$. By Eq.~(\ref{mrc1c}) and squaring
we obtain the following.
\[
a^{2^{j+1} \cdot l} \equiv 1  \pmod p ,
\]
which implies
\[
a^{2^{b(n)-1} \cdot l} \equiv 1 \pmod p .
\]
The latter concludes the proof of this proposition.
\end{proof}

\subsection{Proving the Rabin bound}

We present a  proof of Proposition~\ref{mrp}
that is based on \cite{RC2e} 
originating from the paper of \cite{Monier}.
Proposition~\ref{mrp} will utilize the result
of Theorem~\ref{mrbound}.
Moreover,  Theorem~\ref{mrbound} relies
on a series of propositions itself.

\begin{thm}[ MillerRabin bound]
\label{mrbound}
In a follow-up problem the following will be shown.
\[
 |C(n)| \leq \phi(n) /4 ,
\]
which would then imply
\[
 |MRw(n) | \geq (n-1) - \phi(n) /4 .
\]
\end{thm}
\begin{proof}
It follows by way of
Proposition~\ref{thaux0} shown earlier,
Proposition~\ref{thaux1},
Proposition~\ref{thaux2}, and
Proposition~\ref{thaux3} below.
\end{proof}

\begin{prp}
\label{thaux0}
Let $n \in \mb{Z}_+$ be an odd composite integer with $n> 10$,
such that $n-1 = 2^k l$, where $l$ is an odd integer.
Furthermore let 
$n = p_1^{a_1}  p_2^{a_2} \ldots p_r^{a_r}$ be the prime
decomposition of $n$, and
let $d(n)=r$ indicate the number of distinct prime divisors
of $n$.
Let $b(n)$ be the largest integer value such that
$\dv{2^{b(n)}}{p-1}$, for every prime number $p$ such that
$\dv{p}{n}$.
\noindent
$ $ \\ $ $
(a) Let
\begin{equation}
\label{mrc2}
C(n) = \left\{ a \in \mb{U}_n : a^{2^{m} \cdot l } \equiv \pm 1 \pmod n
       \right\}
\end{equation}
Show that if $ a \in MRnw(n)$
then $a \in C(n)$.
$ $ \\ $ $
\noindent
(b) Furthermore show that $ MRnw(n) \subseteq C(n)$.
\end{prp}

\begin{proof}
$ $ \\ $ $
Both (a) and (b) follow trivially from the prior discussion.
\end{proof}

\noindent
\begin{prp}[Auxiliary result 1]
\label{thaux1}
Let $n = p_1^{a_1}  p_2^{a_2} \ldots p_r^{a_r}$ be the prime
decomposition of composite $n$, 
let $d(n)=r$ indicate the number of distinct prime divisors
of $n$.
Furthermore, let $n-1 = 2^k l$, where $l$ is an odd integer.
%Let $b(n)$ be the largest integer value such that
%$\dv{2^{b(n)}}{p-1}$, for every prime number $p$ such that
%$\dv{p}{n}$.
Let $y \in \mb{Z}$ such that $\gcd(y,n)=1$ and let
$b \in \mb{Z}_+$.
$ $  \\ $ $
(a)
Then,  the following modular equation
\begin{equation}
\label{mrc3}
 x^{v} \equiv y \pmod{p_i^{a^i}} ,
\end{equation}
has a solution $x$ mod $p_i^{a_i}$ if
\[
 d_i = \dv{\gcd(v, \phi (p_i^{a_i} ))}{ind_{g_i}(y)},
\]
and if there is at least one solution, conclude then
that the number of solutions is $d_i$.
$ $ \\ $ $
(b) 
Using CRT then  the number of solutions of
\begin{equation}
\label{mrc3g}
 x^{v} \equiv y \pmod{n} ,
\end{equation}
is given by the following expression
\begin{equation}
\label{mrc3s}
 \prod_{i=1}^{r} d_i 
 =
 \prod_{i=1}^{r}  \gcd(v, \phi (p_i^{a_i} ))
 =
 \prod_{i=1}^{r}  \gcd(v, (p_i^{a_i-1} (p_i -1) )) .
\end{equation}
\end{prp}

\begin{proof}
$ $ \\ $ $
(a) Let $g_i$ be a generator of $\mb{U}_{p_i^{a_i}}$.
Let $ind_{g_i} (x)$ and $ind_{g_i} (y)$ be the indices of 
$x, y$ respectively.
\begin{eqnarray}
x^v  &\equiv&  y \pmod{n} 
   \Leftrightarrow \nonumber \\
x^v  &\equiv&  y \pmod{p_i^{a_i}} \ \forall i=1, \ldots , r 
   \Leftrightarrow \nonumber \\
g_i^{ind_{g_i}(x) v}  &\equiv&  g_i^{ind_{g_i}(y) } \pmod{p_i^{a_i}} 
   \Leftrightarrow \nonumber \\
\label{mrc3a}
{ind_{g_i}(x) v}  &\equiv&  {ind_{g_i}(y) } \pmod{\phi( p_i^{a_i} )} 
\end{eqnarray}
Viewing the last equation Eq.(\ref{mrc3a}) as a modular equation and 
invoking the results associated with 
Eq.(\ref{tlincond}) we obtain the following.
A solution  for Eq.(\ref{mrc3a}) exists if and only if
\[
 \dv{d_i}{ind_{g_i}(y) } ,
\]
where
\[
 d_i = \gcd( v, \phi( p_i^{a_i} ) ).
\]
If one solution exists, then the number of solution
mod $p_i^{a_i}$ is equal to $d_i$.
$ $ \\ $ $
(b) Follows by the CRT and
the fact that
$\phi (p_i^{a_i} ) =(p_i^{a_i-1} (p_i -1) )$.
\end{proof}

%HERE4

\noindent
\begin{prp}[Auxiliary result 2]
\label{thaux2}
Let $n = p_1^{a_1}  p_2^{a_2} \ldots p_r^{a_r}$ be the prime
decomposition of composite $n$, 
let $d(n)=r$ indicate the number of distinct prime divisors
of $n$.
Furthermore, let $n-1 = 2^k l$, where $l$ is an odd integer.
Let $b(n)$ be the largest integer value such that
$\dv{2^{b(n)}}{p-1}$, for every prime number $p$ such that
$\dv{p}{n}$.

\noindent
We reintroduce Eq.~(\ref{mrc2}).
\begin{equation}
%\label{mrc2}
C(n) = \left\{ a \in \mb{U}_n : a^{2^{m} \cdot l } \equiv \pm 1 \pmod n
       \right\}
\end{equation}

\noindent
Show the following.
\begin{equation}
\label{mrc4}
 c= |C(n)| =
      2 
  \cdot 
      \prod_{i=1}^{r} 2^{(b(n)-1)} \gcd(p-1,l) 
           =
      2 
  \cdot 
      2^{(b(n)-1)\cdot d(n)} 
  \cdot 
      \prod_{i=1}^{r} \gcd(p-1,l) .
\end{equation}
\end{prp}

\begin{proof}
$ $ \\ $ $
For $b(n)$ let us define $v= 2^{b(n)-1} \cdot l$.
Equation~(\ref{mrc2}) leads us to count the number
of $a$ of $C(n)$ that satisfy
modular equation
\begin{equation}
\label{mrc3p}
 a^v \equiv  1 \pmod n \Leftrightarrow
 a^v \equiv  1 \pmod{p_i^{a_i}}  \quad \forall i=1,2, \ldots , r
\end{equation}
or modular equation
\begin{equation}
\label{mrc3q}
 a^v \equiv -1 \pmod n \Leftrightarrow
 a^v \equiv -1 \pmod{p_i^{a_i}}  \quad \forall i=1,2, \ldots , r
\end{equation}
$ $ \\ $ $
{\bf Case 1: Number of solutions of Eq.(\ref{mrc3p}).}
We start with the former.
Eq.(\ref{mrc3p}) by Eq.(\ref{mrc3}) has a number of solutions equal to
\[
 d_i = \gcd( v, \phi( p_i^{a_i} ) ),
\]
as long as 
\[
 \dv{d_i}{ind_{g_i}(y) } =
 \dv{d_i}{ind_{g_i}(1) }
\]
which is the case 
for $ind_{g_i}(1) = \phi( p_i^{a_i} )$.
Then
\begin{eqnarray}
\forall i=1,2, \ldots r \quad :
\gcd( v, \phi( p_i^{a_i} ) )        
   &=&
\gcd( v,  p_i^{a_i -1} (p_i -1) )  \nonumber \\
   &=&
\gcd( v,  (p_i -1) )  \nonumber \\
   &=&
\gcd( 2^{b(n)-1} \cdot l ,  (p_i -1) )   \nonumber \\
\label{mrc3b}
   &=&
2^{b(n)-1}\cdot \gcd( l ,  (p_i -1) ) .
\end{eqnarray}
The term $p_i^{a_i -1}$ was removed from the first
equation above, because
$\ndv{p_i}{v}$. If it were
$\dv{p_i}{v}$ then since $\dv{v}{n-1}$ we
would have had $\dv{p_i}{n-1}$. Since
$\dv{p_i}{n}$ the last two divisibility results
would lead to $\dv{p_i}{1}$ i.e. $p_i =1$.
This violates the assumption of the unique factorization
of $n$ where $p_i$ are prime numbers and in fact greater
than two.
$ $ \\ $ $
By the CRT and combining for $i=1 , 2, \ldots , r$,
Eq.~(\ref{mrc3b}) we obtain 
that  for modular equation
\[
 a^v \equiv  1 \pmod n 
\]
the number of its solutions is as follows.
\[
\prod_{i=1}^{r} 2^{b(n)-1}\cdot \gcd( l ,  (p_i -1) ) 
=
2^{(b(n)-1)r }\cdot
\prod_{i=1}^{r} \gcd( l ,  (p_i -1) ) 
=
2^{(b(n)-1)d(n) }\cdot
\prod_{i=1}^{r} \gcd( l ,  (p_i -1) )  .
\]
$ $ \\ $ $
{\bf Case 2: Number of solutions of Eq.(\ref{mrc3q}).}
We continue with the latter
Eq.(\ref{mrc3q}) to establish the number of 
solutions of
\[
 a^v \equiv -1 \pmod n 
\]
by establishing the number of solutions of
\[
 a^v \equiv -1 \pmod{p_i^{a_i}}  \quad \forall i=1,2, \ldots , r .
\]
Note that the latter implies
\[
 a^{2v} \equiv 1 \pmod{p_i^{a_i}}  \quad \forall i=1,2, \ldots , r .
\]
Therefore we find the number of solutions of
\[
 a^{2v} \equiv 1 \pmod{p_i^{a_i}}  \quad
\wedge \quad
 a^{v} \not\equiv 1 \pmod{p_i^{a_i}}  
\quad \forall i=1,2, \ldots , r .
\]
We start with the former.
The $d_i$ is slightly different.
\[
 d_i = \gcd( 2v, \phi( p_i^{a_i} ) ),
\]
We note that we need
\[
 \dv{d_i}{ind_{g_i}(y) } =
 \dv{d_i}{ind_{g_i}(1) }
\]
which is the case as before for Case 1.
Then
\begin{eqnarray}
\forall i=1,2, \ldots r \quad :
\gcd( 2v, \phi( p_i^{a_i} ) )        
   &=&
\gcd( 2v,  p_i^{a_i -1} (p_i -1) )  \nonumber \\
   &=&
\gcd( 2v,  (p_i -1) )  \nonumber \\
   &=&
\gcd( 2^{b(n)} \cdot l ,  (p_i -1) )   \nonumber \\
\label{mrc3c}
   &=&
2^{b(n)}\cdot \gcd( l ,  (p_i -1) ) .
\end{eqnarray}
% a = a \wedge \bar{b} + a \wedge b =>
% a - a \wedge b = a \wedge \bar{b} =>
\[
 a^{2v} \equiv 1 \pmod{p_i^{a_i}}  \quad
\wedge \quad
 a^{v} \not\equiv 1 \pmod{p_i^{a_i}}  
\quad \forall i=1,2, \ldots , r .
\]
Therefore the number of solutions of
\[
 a^{2v} \equiv 1 \pmod{p_i^{a_i}}  \quad
\wedge \quad
 a^{v} \not\equiv 1 \pmod{p_i^{a_i}}  
\quad \forall i=1,2, \ldots , r 
\]
is equal to 
\[
2^{b(n)}\cdot \gcd( l ,  (p_i -1) ) 
-
2^{b(n)-1}\cdot \gcd( l ,  (p_i -1) ) 
=
2^{b(n)-1}\cdot \gcd( l ,  (p_i -1) ) ,
\]
just like case 1.
This concludes the case and the result.
\end{proof}

\noindent
\begin{prp}[Auxiliary result 3]
\label{thaux3}
Let $n = p_1^{a_1}  p_2^{a_2} \ldots p_r^{a_r}$ be the prime
decomposition of composite $n$,  and
let $d(n)=r$ indicate the number of distinct prime divisors
of $n$.
Furthermore, let $n-1 = 2^k l$, where $l$ is an odd integer.
Let $b(n)$ be the largest integer value such that
$\dv{2^{b(n)}}{p-1}$, for every prime number $p$ such that
$\dv{p}{n}$.
$ $ \\ $ $
\noindent
For
$
C(n) = \left\{ a \in \mb{U}_n : a^{2^{m} \cdot l } \equiv \pm 1 \pmod n
       \right\}
$
Then, we have the following.
\begin{equation}
\label{mrc5}
 c= |C(n)|  \leq \frac{\phi(n)}{4},
\end{equation}
for all odd composite $n>10$.
\end{prp}

\begin{proof}
$ $ \\ $ $
From Eq.(\ref{mrc4}) we have the following.
\begin{eqnarray}
\label{mrc4a}
 c= |C(n)| 
    &=&
      2 
  \cdot 
      \prod_{i=1}^{r} 2^{(b(n)-1)} \gcd(p-1,l) 
           \\
\label{mrc4b}
 \frac{\phi(n)}{|C(n)|} &=&
 \frac{1}{2} \cdot
 \prod_{i=1}^{r}
   \frac{p_i^{a^i -1} (p_i -1)}{
      2^{(b(n)-1)} 
      \gcd(p-1,l)}
\end{eqnarray}
$ $ \\ $ $
We note a few thing related to Eq.(\ref{mrc4b}).
Since $\dv{2^{b(n)-1}}{2^{b(n)}}$
and
$\dv{2^{b(n)}}{p-1}$ for each  $i$, we have that
$\frac{(p_i -1)}{2^{b(n)-1}}$ is an integer and also an
even number. Moreover $\dv{\gcd(l, p_i -1)}{p_i -1}$
and $\gcd(l, p_i -1)$ is an odd number,
since $l$ is odd.
Thus the product term of Eq.(\ref{mrc4b}) is an integer
and a multiple of two and multiplied by the outside
$1/2$ still yields an integer.
We now perform a case analysis.
$ $ \\ $ $
{\bf Case 1: $r=d(n) \geq 3$.}
Then
\[
 \frac{\phi(n)}{|C(n)|} \geq 
   \frac{1}{2} \cdot 2 \cdot 2 \cdot 2  = 4.
\]
$ $ \\ $ $
{\bf Case 2: $r=d(n) =2$, and $n$ is not squarefree.}
Then for some $p_i$ we have $a_i \geq 2$. In the numerator
of Eq.(\ref{mrc4b}) $p_i^{a_i -1}$ contributes a
                    $p_i^{a_i -1} \geq p_i \geq 3$.
\[
 \frac{\phi(n)}{|C(n)|} \geq 
   \frac{1}{2} \cdot \left( 3 \cdot 2 \right) \cdot 2  = 6.
\]
$ $ \\ $ $
{\bf Case 3: $r=d(n) =2$, but $n$ is squarefree.}
Let $n=pq$ where $p< q < n$. We have that
$\dv{2^{b(n)}}{p-1}$
and
$\dv{2^{b(n)}}{q-1}$.
We distinguish two subcases.
$ $ \\ $ $
{\bf Case 3a: $\dv{2^{b(n)+1}}{q-1}$.}
This means $q-1 = 4 \cdot 2^{b(n)-1} \cdot R$.
Then we have as follows.
\[
 \frac{\phi(n)}{|C(n)|} \geq 
   \frac{1}{2} \cdot 2 \cdot \frac{q-1}{2^{b(n)-1} \gcd(l,q-1)}
 \geq \frac{1}{2} \cdot 2 \cdot 4 = 4.
\]
$ $ \\ $ $
{\bf Case 3b: $\ndv{2^{b(n)+1}}{q-1}$ but $\dv{2^{b(n)1}}{q-1}$.}
Then we have as follows.
\begin{equation}
\label{mrc4p}
 \frac{\phi(n)}{|C(n)|} \geq 
   \frac{1}{2} \cdot 2 \cdot \frac{q-1}{2^{b(n)-1} \gcd(l,q-1)}
   \frac{1}{2} \cdot 2 \cdot \frac{2\cdot Q}{\gcd(l,q-1)}
\end{equation}
Note that $n-1 = 2^k l$, where $b(n) \leq k$.
Moreover,  $q-1= 2^{b(n)} Q$, where $Q$ can be even or odd.
Furthermore,
\[
 n-1 =pq -1 = p(q-1)+(p-1)
\Rightarrow
 n-1 \equiv p-1 \pmod{q-1}
\Rightarrow
\ndv{q-1}{n-1}.
\]
The latter implies that there is a prime divisor of $q-1$
that does not divide $n-1$. Since $b(n) \leq k$, this divisor
is not a two. So it must be an odd number and it is at least 3.
Therefore $q-1 =2 Q = 2 \cdot 3 \cdot P$.
Therefore we refine the previous bound of Eq.(\ref{mrc4p})
as follows.
\begin{equation}
\label{mrc4q}
 \frac{\phi(n)}{|C(n)|} \geq 
   \frac{1}{2} \cdot 2 \cdot \frac{2\cdot Q}{\gcd(l,q-1)}
   \geq
   \frac{1}{2} \cdot 2 \cdot \frac{2\cdot 3 \cdot P}{\gcd(l,q-1)}
   \geq 6.
\end{equation}
$ $ \\ $ $
{\bf Case 4: $d(n)=1$ and $n$ is not squarefree i.e. $n=p^{a}$,
for some $a \geq 2$ and $p\geq 3$.}
A minimal such  value is for $a=2$ and $p=3$ and it is $n=3^2 =9$.
We then have he following.
\[
 \frac{\phi(n)}{|C(n)|} \geq 
   \frac{1}{2}  \cdot \frac{p^{a-1} (p-1)}{2^{b(n)-1} \gcd(l, p-1)}
\]
Note that $p-1=2^k \cdot l$. Moreover $b(n)=k$.
Then $\gcd(l,p-1) = \gcd(l, 2^k \cdot l) =l$.
Therefore
\[
    \frac{1}{2}\cdot  \frac{p^{a-1} (p-1)}{2^{b(n)-1} \gcd(l, p-1)}
   =
   \frac{1}{2}  \cdot \frac{p^{a-1} \cdot  2^k \cdot  l}{2^{k-1} l}
   =  p^{a-1}.
\]
If $p=3$ and $a=2$ this $p^{a-1}$ is a $3$. 
We discard this possibility by requiring
$n>9$ and for composite and odd $n$ this becomes $n>10$.
In all other cases $p^{a-1}$ is at least  $5$.
\end{proof}

\chapter{Multiplicative functions}

\section{Multiplicative functions}

\begin{dfn}
A function $f$ on $\mb{Z}_+$ is multiplicative if
\[
f(ij)=f(i)f(j) \quad \forall i \in \mb{Z}_+ ,  \forall j \in \mb{Z}_+ ,
\text{where} \ \gcd(i,j)=1.
\]
\end{dfn}

\begin{lem}
If $f$, as previously defined, is multiplicative and 
not zero then $f(1)=1$.
\end{lem}
\begin{proof}
This derives for $i=j=1$ and $f(1)=f(1)f(1)$,
provided   that   $f(1) \neq 0$.
Otherwise, since $f$ is not zero then there 
exists a $k$ such that $f(k) \neq 0$. Then 
$f(k \cdot 1) =f(k)f(1)$ and dividing
by $f(k) \neq 0$ we reach the same conclusion.
Note that if $f(1)=0$ then $f(k \cdot 1 ) = f(k) f(1) = 0$,
and therefore $f(k)=0$ for all $k$.
\end{proof}

\begin{prp}
If $n$ is as follows,
\[
  n = \prod_{i=1}^{k} p_i^{a_i},
\]
then $f(n) = \prod_i f(  p_i^{a_i} )$,
where $f$ is a multiplicative function.
\end{prp}
\begin{proof}
One can use induction on $k$. If $k=1$
Then obviously, $f(n) = f( p_1^{a_1}  = \prod_{i=1}^{1} f(p_1^{a_1} )$.
Now consider that the result is true of $k$ or less and we want
to show it for $k+1$., given that $\gcd( p_l , p_m )=1$ for all
$l,m =1 , \ldots , k+1$ we have
\begin{eqnarray*}
f(n) &=& f( \prod_{i=1}^{k+1} p_i^{a_i} ) \\
     &=& f( \prod_{i=1}^{k} p_i^{a_i} \cdot p_{k+1}^{a_{k+1}}) \\
     &=& f( \prod_{i=1}^{k} p_i^{a_i}) \cdot f( p_{k+1}^{a_{k+1}})
\end{eqnarray*}
and then applying the inductive hypothesis for $k=k$ on the first
term of the right hand side.
\end{proof}

\begin{prp}
\label{prpgmult}
If $f$ is a multiplicative function, and
\[
  g(n) = \sum_{\dv{d}{n}} f(d),
\]
then $g(n)$ is also multiplicative.
\end{prp}

\begin{proof}
For $f$ multiplicative and $p,q$ such that $\gcd(p,q)=1$
we have $f(pq)=f(p)f(q)$. If $\dv{d}{pq}$ and given $\gcd(p,q)=1$
then either $\dv{d}{p}$ or $\dv{d}{q}$ or $d=d_1 d_2$ and
$\dv{d_1}{p}$ and $\dv{d_2}{q}$ and $\gcd(d_1 , d_2 )=1$.
By the multiplicativity of $f$, $f(d_1 d_2 ) = f(d_1 ) f(d_2 )$.
Then we have the following.
\begin{eqnarray*}
g(pq) &=& \sum_{\dv{d}{pq}} f(d) \\
      &=& \sum_{\dv{d_1}{p}, \dv{d_2}{q}} f(d_1 d_2) \\
      &=& \sum_{\dv{d_1}{p}, \dv{d_2}{q}} f(d_1 ) f( d_2) \\
      &=& \sum_{\dv{d_1}{p}} f(d_1 ) \sum_{\dv{d_2}{q}}  f( d_2) \\
      &=& g(p)                                g(q).
\end{eqnarray*}
\end{proof}

\begin{prp}
If $f$, $g$ are multiplicative functions, and
\[
  M(n) = \sum_{\dv{d}{n}} f(d) g(n/d),
\]
then $M(n)$ is also multiplicative.
\end{prp}

\begin{proof}
For $f,g$ multiplicative and $p,q$ such that $\gcd(p,q)=1$
we have $f(pq)=f(p)f(q)$ and also $g(pq)=g(p)g(q)$.
If $\dv{d}{pq}$ and given $\gcd(p,q)=1$
then either $\dv{d}{p}$ or $\dv{d}{q}$ or $d=d_1 d_2$ and
$\dv{d_1}{p}$ and $\dv{d_2}{q}$ and $\gcd(d_1 , d_2 )=1$.
(The latter case absorbs the two former ones that imply then
$d_2 =1$ or $d_1 =1$ respectively.)
Moreover $\gcd( p/d_1 , q/d_2 )=1$.
Then we have the following.
\begin{eqnarray*}
M(pq) &=& \sum_{\dv{d}{pq}} f(d) g(\frac{pq}{d}) \\
      &=& \sum_{\dv{d_1}{p} ,\ \dv{d_2}{q}}
            f(d_1 d_2 ) g(\frac{pq}{d_1 d_2}) \\
      &=& \sum_{\dv{d_1}{p} ,\ \dv{d_2}{q}}
            f(d_1 ) f( d_2 ) g(\frac{p}{d_1}) \ g(\frac{q}{d_2}) \\
      &=& \sum_{\dv{d_1}{p}}  f(d_1 ) g(\frac{p}{d_1})
          \sum_{\dv{d_2}{q}}  f(d_2 ) \ g(\frac{q}{d_2}) \\
      &=& M(p) M(q).
\end{eqnarray*}
\end{proof}

\section{Totient function redefined}

\begin{dfn}[Euler's totient function]
For $n \in \mb{Z}_+$, we define $\phi (n)$ to be
the number of positive integers less than $n$ that
are relatively prime to $n$.
\[
 \phi (n) = \sum_{1 \leq i < n ,\  \gcd(i,n)=1} 1.
\]
\end{dfn}

\begin{prp}
\label{prpphi}
For any $n \in \mb{Z}_+$ we have the following.
\[
 n = \sum_{\dv{d}{n}} \phi (d ).
\]
\end{prp}
\begin{proof}
Let $n_d$ be the number of integers $i$ such that
$\gcd(i,n)=d$ or equivalently $\gcd(i/d,n/d) =1$.
By definition the number of $i/d$ is $\phi (n/d)$
and so is $n_d$ i.e. $n_d = \phi (n/d)$.
If $\dv{d}{n}$ then there exists an integer in $\mb{Z}_+$,
$d_1 $, such that $n=d d_1$.
Moreover, $\dv{d_1}{n}$.
\begin{eqnarray*}
 n    &=& \sum_{\dv{d}{n}} n_d \\
      &=& \sum_{\dv{d}{n}} \phi ( \frac{n}{d} ) \\
      &=& \sum_{\dv{d}{n}, d d_1 =n} \phi ( d_1 ) \\
      &=& \sum_{\dv{d_1}{n}} \phi ( d_1 ) \\
      &=& \sum_{\dv{d}{n}} \phi ( d ) \\
\end{eqnarray*}
In the last derivation we renamed variable $d_1$ into
a new variable $d$.
\end{proof}

\begin{prp}
For $n \in \mb{Z}_+$ we have the following.
\begin{enumerate}
\item For $n$ a prime number $\phi (n)=n-1$. Moreover
if $\phi (n)=n-1$ then $n$ is a prime number.
\item  For $\gcd (p,q)=1$, Euler's $\phi$ function is
multiplicative and therefore
\[
  \phi (pq) = \phi (p) \phi (q), \quad \gcd(p,q)=1.
\]
\end{enumerate}
\end{prp}
\begin{proof}
$ $ \\ $ $
(a)
Say $n$ is a prime number. Then all $1,2, \ldots , n-1$
are relatively prime to $n$ and therefore $\phi (n) = n-1$,
as required.
$ $ \\  $ $
Assume now that there is a composite number $n$ such that
$\phi (n) = n-1$. If $n$ is a composite number there exists
a $d \neq 1, n$ such that $\dv{d}{n}$.
Then in
\[
  n = \sum_{\dv{d}{n}} \phi (d ),
\]
there are three contributions: $\phi (1) , \phi(d) , \phi (n)$
that re non-zero. For $\phi (1)=1$, and $\phi (d) \geq 1$ imply
$\phi (n) \leq n -2$ and this contradicts the
$\phi (n) = n-1$. Thus $n$ is not compositive and it must be
prime number.
$ $ \\ $ $
(b)
In Proposition~(\ref{prpgmult}) we had
\[
   g(n) = \sum_{\dv{d}{n}} f(d),
\]
and in Proposition~(\ref{prpphi}) we had
\[
  n = \sum_{\dv{d}{n}} \phi (d ).
\]
Consider $g(n)=n$. Then $g$ is a multiplicative function.
Consider $f(d)=\phi(d)$.
By way of Proposition~(\ref{prpgmult}), with the stated
substitutions, if we trace the proof there forward and backwards we
conclude that  if $g(n)$ is multiplicative then
$f(n)$ is multiplicative.
This establishes that $\phi (n)$ is a multiplicative function.
\end{proof}

\section{M\"obius function}

\begin{dfn}[M\"obius function]
\[
\mu (n)     =
\begin{cases}
 1     & \text{if } n=1,\\
 0     & \text{if n contains a square factor},  \\
(-1)^k & \text{if n is the product of k distinct primes}.
\end{cases}
\]
\end{dfn}

\begin{prp}
Let $n \in \mb{Z}_+$. Then we have the following.
\begin{enumerate}
\item $\mu (n)$ is multiplicative that is, for $p,q \in \mb{Z}_+$
with $\gcd(p,q)=1$ the following applies.
\[
 \mu (pq) = \mu (p) \mu (q).
\]
\item  We have the following equality for the M\"obius function.
\[
 I(n) = \sum_{\dv{d}{n}} \mu (d),
\]
where
\[
 I  (n)     =
\begin{cases}
 1     & \text{if } n=1,\\
 0     & \text{if } n >1.  \\
\end{cases}
\]
\end{enumerate}
\end{prp}
\begin{proof}
$ $ \\ $ $
(1)  Consider $p=q=1$.
Then $\mu (pq)= \mu(p) \mu (q)=1$.
Consider $\dv{r^2}{p}$. Since $\gcd(p,q)=1$ then
$\dv{r^2}{pq}$. Likewise for the $\dv{r^2}{q}$ case.
Then $\mu (pq)= \mu(p) \mu (q)=0$.
Therefore we may assume that both $p,q$ are square free.
Thus $p = p_1 \ldots p_k$ and $q = q_1 \ldots q_l$.
Moreover $\gcd(p,q)=1$ and thus $\gcd(p_i , q_j)=1$.
\begin{eqnarray*}
\mu (pq) &=&
\mu (  p_1 \ldots p_k q_1 \ldots q_l) \\
         &=& (-1)^{k+l} \\
         &=& (-1)^{k} (-1)^{l} \\
         &=& \mu (  p_1 \ldots p_k  )
             \mu (q_1 \ldots q_l)\\
         &=& \mu (p) \mu (q),
\end{eqnarray*}
as needed.
$ $ \\ $ $
(2) If $n=1$ then $I(n)=1$.
Moreover
\[
 I(n) = \sum_{\dv{d}{n}} \mu (d)
\]
is true as the left-hand side is $I(n)=I(1)=1$
and the right-hand side for $n=1$ has only one term
for $d=1$ and $\mu (d) = \mu (1)=1$.
Thus the expression is true.
Moreover if $n>1$
by way of Proposition~(\ref{prpgmult}),
and the multiplicativity of $\mu(n)$
consider only the case where $n$ is a prime power,
$n=p^k$, with the general case proven by induction.
We have the following.
\[
I(p^k ) = \sum_{\dv{d}{p^k}} \mu (d)
        = \mu(1) + \mu(p) + \mu (p^2 ) + \ldots + \mu(p^k )
        =   1    +  (-1)  +    0       + \ldots +  0
        =   1 +(-1) =0
\]
Noting that $p^2 , p^3 , \ldots $ are not square-free,
the result follows.
An alternative proof of the $n \neq 1$ is as follows.
As prime powers greater than 1 lead to a zero, sum the
only possibility left is for $n= p_1 \ldots p_k$.
Then we have the following.
\begin{eqnarray*}
\sum_{\dv{d}{n}} \mu (d) &=&
    \mu (p_1 ) + \mu (p_2 ) + \ldots + \mu( p_k )+
    \mu (p_1 p_2 ) + \ldots + \mu (p_1 p_2 p_3  ) +
        \ldots + \ldots + \mu (p_1 \ldots p_k ) \\
    &=& \sum_{i=0}^{k} {k \choose i} (-1)^i = (1-1)^k = 0.
\end{eqnarray*}
\end{proof}

The following result will be proved again using
M\"obius inversion formula as a corollary of the
corresponding theorem.

\begin{lem}
For any $n \in \mb{Z}_+$ we have the following.
\[
 \phi (n) = \sum_{\dv{d}{n}}  \mu (d) \frac{n}{d} .
\]
\end{lem}
\begin{proof}
We have by definition the following.
\[
 \phi (n) = \sum_{1 \leq i < n ,\  \gcd(i,n)=1} 1
\Rightarrow
   n = \sum_{\dv{d}{n}} \phi (d).
\]
We also have the following equality for the
M\"obius function.
\[
 I(n) = \sum_{\dv{d}{n}} \mu (d),
\]
\begin{eqnarray*}
 \phi (n) &=& \sum_{i=1, \ \gcd(i,n)=1}^{n}   1  \\
          &=& \sum_{i=1}^{n}   I( \gcd(i,n) )   \\
          &=& \sum_{i=1}^{n} \   \sum_{\dv{d}{\gcd(i,n)}} \mu (d ) \\
          &=& \sum_{i=1}^{n} \   \sum_{\dv{d}{i}, \ \dv{d}{n}} \mu (d )
\end{eqnarray*}
For a given and fixed $d$ such $\dv{d}{n}$
the number of values $i$ such that $\dv{d}{i}$ i.e. that
are multiple of $d$ is exactly $n/d$.
For $\dv{d}{i}$ we have $i=da$ and since $i\leq n$ we have
$da\leq n$ and therefore $a \leq n/d$.
We can then rewrite the last double sum as follows.
\begin{eqnarray*}
 \phi (n) &=& \sum_{i=1}^{n} \   \sum_{\dv{d}{i}, \ \dv{d}{n}} \mu (d ) \\
          &=& \sum_{\dv{d}{n}} \sum_{a=1}^{\frac{n}{d}} \mu (d ) \\
          &=& \sum_{\dv{d}{n}} \mu (d )  \sum_{a=1}^{\frac{n}{d}} 1 \\
          &=& \sum_{\dv{d}{n}} \mu (d )  \frac{n}{d}  \\
          &=& \frac{n}{d}  \sum_{\dv{d}{n}} \mu (d ) .
\end{eqnarray*}
\end{proof}

\begin{prp}
For $n \in \mb{Z}_+$ we have the following.
\begin{enumerate}
\item For $n$ a prime power $n=p^k$ where $k>1$ we
have
\[
 \phi (n) = \phi (p^k ) = p^k - p^{k-1} = n (1-1/p).
\]
\item  For $n$ a composite number with prime decomposition
as follows
\[
  n = \prod_{i=1}^{k} p_i^{a_i},
\]
we have the following.
\[
  \phi (n) = n \prod_{\dv{p}{n}} \left( 1 - \frac{1}{p} \right)
           = p_1^{a_1}  \left( 1 - \frac{1}{p_1} \right)
             p_2^{a_2}  \left( 1 - \frac{1}{p_2} \right)
             \ldots
             p_k^{a_k}  \left( 1 - \frac{1}{p_k} \right) .
\]
\end{enumerate}
\end{prp}
\begin{proof}
$ $ \\ $ $
(1) For $n=p^k$ the only divisors of $n$ are multiples of $p$
and there are $n/p = p^{k-1}$ of them. Therefore the number of
integers relatively prime to (i.e. non divisors of )   $n$ are
$p^k - p^{k-1}$, as needed.
$ $ \\ $ $
(2)
Use induction and the multiplicativity of $\phi (n)$.
This completes the proof.
$ $ \\ $ $
Or use induction as follows, but we will need a result
to be proven later.
For $n=1$, there is no product and thus $\phi (1)=1$.
For $n \geq 2$, given that
\[
  n = \prod_{i=1}^{k} p_i^{a_i},
\]
we need to show
\begin{eqnarray*}
  \phi (n) &=& n \prod_{\dv{p}{n}} \left( 1 - \frac{1}{p} \right)  \\
           &=& n \prod_{i=1}^{k}  \left( 1 - \frac{1}{p_i} \right) \\
           &=& n \left( 1- \sum_i \frac{1}{p_i}
                         + \sum_{i,j} \frac{1}{p_i p_j } +
                 \ldots + (-1)^k \sum \frac{1}{p_1 \ldots p_k} \right).
\end{eqnarray*}
The parenthesized expression is
\[
\sum_{\dv{d}{n}} \mu (d) \cdot \frac{1}{d}.
\]
Therefore
\begin{eqnarray*}
           &=& n \prod_{\dv{p}{n}} \left( 1 - \frac{1}{p} \right)  \\
           &=& n \left( 1- \sum_i \frac{1}{p_i}
                         + \sum_{i,j} \frac{1}{p_i p_j } +
                 \ldots + (-1)^k \sum \frac{1}{p_1 \ldots p_k} \right). \\
           &=& n \sum_{\dv{d}{n}} \mu (d) \frac{1}{d} \\
           &=&  \sum_{\dv{d}{n}} \mu (d) \frac{n}{d} \\
           &=&   \phi (n),
\end{eqnarray*}
where the last derivation remains to be shown.
\end{proof}

\section{Dirichlet product}

\begin{dfn}
A function $f$ on $\mb{Z}_+$ is arithmetic
if its range is a subset of the composite numbers.
\end{dfn}

\begin{dfn}[Dirichlet product]
For two arithmetic functions $f,g$ the Dirichlet product
is defined as follows
\[
 (f*g) (n) = \sum_{\dv{d}{n}} f(d) g(\frac{n}{d})
           = \sum_{ab=n} f(a) g(b).
\]
\end{dfn}

\begin{thm}
The Dirichlet product is commutative and associative that is
$f*g=g*f$ and $(f*g)*h = f*(g*h)$ for three arithmetic functions
$f,g,h$.
\end{thm}
\begin{proof}
Commutativity follows from the definition.
\[
 (f*g) (n) = \sum_{\dv{d}{n}} f(d) g(\frac{n}{d})
           = \sum_{ab=n} f(a) g(b).
           = \sum_{\dv{d}{n}} g(d) f(\frac{n}{d})
 (g*f) (n).
\]
Associativity follows similarly, after observing.
\[
 ((f*g)*h) (n) = ( \sum_{ab=d} f(a) g(b)) \sum{dc=n} h(c)
               = \sum_{abc=n} f(a) g(b) h(c).
\]
\end{proof}

\section{Unit function}

\begin{dfn}[Unit function]
\[
 U  (n)     = 1,  \ \forall n \in \mb{Z}_+ .
\]
\end{dfn}

\begin{cor}
For all arithmetic functions $f$ we have
\[
  f* I = I * f = f
\]
\end{cor}
\begin{proof}
\[
 (f*I) (n) = \sum_{\dv{d}{n}} f(d) I(\frac{n}{d}) = f(n).
\]
Function $I(\frac{n}{d})$ is 1 for $n=d$ and 0 otherwise.
Therefore the sum has one non-zero term for $d=n$ and then
$f(d)=f(n)$.
\end{proof}

\section{Dirichlet and M\"obius inversions}

\begin{thm}[Dirichlet inverse]
Given an arithmetic function $f$, where $f(1) \neq 0$,
there exists the inverse $f^{-1}$ of $f$ such that
\[
   f * f^{-1} = f^{-1} * f = I.
\]
Moreover the calculation of $f^{-1}$ is recursive
as follows for $n \geq 1$,
\begin{eqnarray*}
  f^{-1} (1) &=& \frac{1}{f(1)} \\
  f^{-1} (n) &=& - \frac{1}{f(1)} \sum_{\dv{d}{n}, d \neq n}
         f( \frac{n}{d}) f^{-1} (d) , \quad n >1 .
\end{eqnarray*}
\end{thm}
\begin{proof}
Use induction.
For $n=1$ we have
\[
   (f * f^{-1})(1) =  I(1),
\]
or equivalently $1 = f(1) f^{-1} (1)$ and the result follows.
for $f(1) \neq 0$. For $n>1$ let us assume that $f^{-1} (i)$
has been calculated for $i<n$.
Then
\[
   (f * f^{-1})(n) =  I(n)
\]
and thus by definition of the Dirichlet product
%I(n) is 1 n=1 and 0 otherwise ; we have n>1
\begin{eqnarray*}
  0  &=& \sum_{\dv{d}{n}} f( \frac{n}{d}) f^{-1} (d)  \\
     &=& f(1) f^{-1} (n) +
        \sum_{\dv{d}{n}, d \neq n} f( \frac{n}{d}) f^{-1} (d) \\
  f^{-1} (n) &=& - \frac{1}{f(1)} \sum_{\dv{d}{n}, d \neq n}
         f( \frac{n}{d}) f^{-1} (d) .
\end{eqnarray*}
\end{proof}

\begin{thm}[M\"obius inversion]
Let $f$ be any arithmetic function and if
\label{mobinf}
\[
 g(n) = \sum_{\dv{d}{n}}   f (d)
\]
then
\[
 f(n) = \sum_{\dv{d}{n}} \mu(\frac{n}{d})   g (d)
      = \sum_{\dv{d}{n}} \mu (d) g(\frac{n}{d}) .
\]
\end{thm}
\begin{proof}
$ $ \\ $ $
{\bf 1. By use of Dirichlet products.}
An easy proof follows by using Dirichlet products.
Function $g(n)$ is in fact
\[
  g(n) =  (f * U) (n)
\]
Then we have the following
\[
 g * \mu = (f*U) * \mu ) = f * ( U  * \mu ) = f * (\mu * U)
 = f * I = f,
\]
with the latter being equivalent to
\[
\sum_{\dv{d}{n}} g(d) \mu(\frac{n}{d}) = f(n),
\]
after noting that for the M\"obius function
\[
 I(n) = (\mu * U) (n),
\]
and the $*$ in this proof indicates a Dirichlet product.
%%%%
%%%%%
$ $ \\ $ $
{\bf 2. By direct methods: sums.}
For
\[
 g(n) = \sum_{\dv{d}{n}}   f (d),
\]
we note
the following that will be used in the sums to follow.
\[
n = d \frac{n}{d} = d n_1 = d_1 \frac{d}{d_1} n_1 = d_1 d_2 n_1,
\]
where
\[
\frac{n}{d}= n_1 \Rightarrow n = d n_1 ,
\]
and
\[
\frac{d}{d_1} = d_2 \Rightarrow d = d_1 d_2 .
\]
Furthermore, note that
\[
 I(n) = \sum_{\dv{d}{n}} \mu (d).
\]
\begin{eqnarray*}
\sum_{\dv{d}{n}} \mu (d) g(\frac{n}{d}) &=&
\sum_{\dv{d}{n}} \mu (\frac{n}{d}) g (d) \\
 &=&
\sum_{\dv{d}{n}} \mu (\frac{n}{d}) \sum_{\dv{d_1}{d}}   f (d_1 ) \\
 &=&
\sum_{n_1 d = n} \mu (n_1 ) \sum_{d_1 d_2 = n}   f (d_1 ) \\
 &=&
\sum_{n_1 d_1 d_2 = n} \mu (n_1 )  f (d_1 ) \\
 &=&
\sum_{d_1 n_1 d_2 = n}   f (d_1 )
\sum_{n_1 d_2 = \frac{n}{d_1}} \mu (n_1 )   \\
 &=&
\sum_{\dv{d_1}{n}}   f (d_1 )
\sum_{\dv{n_1}{\frac{n}{d_1}}} \mu (n_1 )   \\
                                        &=&
\sum_{\dv{d_1}{n}    }   f (d_1 ) I(\frac{n}{d_1})   \\
                                        &=&
\sum_{d_1 = n        }   f (d_1 )  = f(n).
\end{eqnarray*}
%%  For
%%  \[
%%   g(n) = \sum_{\dv{d}{n}}   f (d),
%%  \]
%%  we note
%%  the following that will be used in the sums to follow.
%%  \[
%%  \dv{d_1}{n} \Rightarrow  n = d_1 n_1 ,  n_1 = \frac{n}{d_1}
%%  \]
%%  \[
%%  \dv{d_1}{d} \Rightarrow  d = d_1 d_2 ,  d_2 = \frac{d}{d_1},
%%  \]
%%  \[
%%  \dv{d  }{n} \Rightarrow  n = d D ,       D = \frac{n}{d}=
%%                                           \frac{\frac{n}{d_1}}{\frac{d}{d_1}} =
%%                                           \frac{n_1}{d_2}.
%%  \]
%%  \begin{eqnarray*}
%%  \sum_{\dv{d}{n}} \mu (d) g(\frac{n}{d}) &=&
%%  \sum_{\dv{d}{n}} \mu (\frac{n}{d}) g (d) \\
%%   &=&
%%  \sum_{\dv{d}{n}} \mu (\frac{n}{d}) \sum_{\dv{d_1}{d}}   f (d_1 ) \\
%%                                          &=&
%%  \sum_{\dv{d_1}{n}}   f (d_1 )
%%  \sum_{\dv{d_1}{d},\ \dv{d}{n} }
%%  \mu (\frac{n}{d}) \\
%%                                          &=&
%%  \sum_{\dv{d_1}{n}}   f (d_1 )
%%  \sum_{\dv{d_2}{n_1}} \mu (\frac{n_1}{d_2}) \\
%%                                          &=&
%%  \sum_{\dv{d_1}{n}}   f (d_1 )
%%  \sum_{\dv{d_2}{n_1}} \mu (d_2 ) \\
%%                                          &=&
%%  \sum_{\dv{d_1}{n}}   f (d_1 )  I(n_1 ) \\
%%                                          &=&
%%  \sum_{d_1 = n}   f (d_1 )          \\
%%                                          &=&
%%                     f(n).
%%  \end{eqnarray*}
\end{proof}
A corollary of the M\"obius inversion formula establishes
a result for $\phi (n)$ as follows.
\begin{cor}
For every $n \in \mb{Z}_+$ we have the following.
\[
  \phi (n) = \sum_{\dv{d}{n}} \mu (d) \frac{n}{d}.
\]
\end{cor}
\begin{proof}
From the M\"obius inversion formula and $f(n),g(n)$
arithemtic functions as follows, we have the following.
If
\[
 g(n) = \sum_{\dv{d}{n}}   f (d)
\]
then
\[
 f(n) = \sum_{\dv{d}{n}} \mu(\frac{n}{d})   g (d)
      = \sum_{\dv{d}{n}} \mu (d) g(\frac{n}{d}) .
\]
Let $f(n) = \phi (n)$ and let $g(n)=n$.
From Proposition~(\ref{prpphi}) we have
\[
 n = \sum_{\dv{d}{n}} \phi (d )  \Leftrightarrow
 g(n) = \sum_{\dv{d}{n}} f (d ) .
\]
The precondition of  Proposition~(\ref{mobinf})
is true. Therefore the conclusion is as follows.
\begin{eqnarray*}
 f(n) &=& \sum_{\dv{d}{n}} \mu (d) g(\frac{n}{d})  \Leftrightarrow\\
 \phi(n) &=& \sum_{\dv{d}{n}} \mu (d) g(\frac{n}{d})  \Leftrightarrow\\
 \phi(n) &=& \sum_{\dv{d}{n}} \mu (d)   \frac{n}{d} .
\end{eqnarray*}
\end{proof}

\begin{thm}[M\"obius inversion converse]
Let $f$ be any arithmetic function and if
\label{mobinfcon}
\[
 f(n) = \sum_{\dv{d}{n}} \mu(\frac{n}{d})   g (d)
      = \sum_{\dv{d}{n}} \mu (d) g(\frac{n}{d}) ,
\]
then
\[
 g(n) = \sum_{\dv{d}{n}}   f (d).
\]
\end{thm}
\begin{proof}
$ $ \\ $ $
Function $f(n)$ is in fact
\[
  f(n) =  (g * \mu ) (n) ,
\]
and we want to show the following.
\[
  g = f * U \Leftrightarrow
  g(n) =  (f * U) (n) .
\]
Then we have the following
\[
 f *  U  = (g * \mu ) * U =  g * (\mu * U) = g * I = g,
\]
with the latter being equivalent to
\[
g(n) = \sum_{\dv{d}{n}} f(d) U (\frac{n}{d}) =  \sum_{\dv{d}{n}} f(d).
\]
\end{proof}

Consider the two functions
\[
 a(n) = \sum_{\dv{d}{n}} 1,
\]
and
\[
 b(n) = \sum_{\dv{d}{n}} d.
\]
Both can be inverted and
\[
 1 = \sum_{\dv{d}{n}} \mu ( \frac{n}{d} ) a(d),
\]
and
\[
n = \sum_{\dv{d}{n}} \mu ( \frac{n}{d} ) b(d).
\]

% 6. Appendix
%%%%%%%%%%%%%%%%%%%%%%%%%%%%%%%%%%%%%%%%%%%%%%%%%%%%%%%%%%%%%%%%%%%%
%\appendix

% 7. Backmatter
%%%%%%%%%%%%%%%%%%%%%%%%%%%%%%%%%%%%%%%%%%%%%%%%%%%%%%%%%%%%%%%%%%%%
\backmatter

% 8. Bibliography
%%%%%%%%%%%%%%%%%%%%%%%%%%%%%%%%%%%%%%%%%%%%%%%%%%%%%%%%%%%%%%%%%%%%
%    Bibliography styles amsplain or harvard are also acceptable.
%\bibliographystyle{amsalpha}
%\bibliography{}

\begin{thebibliography}{A}

\bibitem{AS}
     M.~Abramowitz and I.~A.~Stegun.
     Handbook of mathematical functions
with formulas, graphs, and mathematical tables.
New York: Dover Publications.  Ninth printing.


\bibitem{Bach}
E.~Bach.
Explicit bounds for primality testing and related problems.
Mathematics of Computation, Vol 55, No. 191 (Jul., 1999), pp. 355-380.



\bibitem{ConradSS}
K.~Conrad.
The Solovay-Strassen test.\\
{https://kconrad.math.uconn.edu/blurbs/ugradnumthy/solovaystrassen.pdf
}
[Accessed: 2026/05/04]


\bibitem{ConradMR}
K.~Conrad.
The Miller-Rabin     test.\\
{https://kconrad.math.uconn.edu/blurbs/ugradnumthy/millerrabin.pdf}
[Accessed: 2026/05/05]


\bibitem{RC2e}
R.~Crandall and C.~Pomerance.
Prime Numbers: A computational perspective.
Second edition. Springer, 2005.

%\bibitem{KL51}
%     S.~Kullback and R.~A.~Leibler.
%     On information and sufficiency.
%The Annals of Mathematical Statistics. 22 (1):79-86, 1952.
%%: 493507. doi:10.1214/aoms/1177729330.
%%ISSN 0003-4851. JSTOR 2236576.



\bibitem{M76}
G.~L.~Miller.
Riemann's hypthesis and tests for primality.
Journal of computer and system sciences, 13, 300-317 (1976),
Academic Press.


\bibitem{Monier}
L.~Monier.
Evaluation and comparison of two efficient probabilistic
primality testing algorithms.
Theoretical Computer Science, 12(1980), 97-108, North-Holland
Publishing Company.

\bibitem{R80}
M.~O.~Rabin.
Probabilistic algorithm for testing primality.
Journal of Number Theory, 12, 128-138 (1980), Academic Press.


\bibitem{SS77}
R.~Solovay and V.~Strassen.
A fast Monte-Carlo test for primality.
SIAM Journal of computing, 6 (1977), 84-85.


\bibitem{Wiki26}
Wikipedia.
{https://en.wikipedia.org/wiki/Floor\_and\_ceiling\_functions}
[Accessed: 2026/05/28]










\end{thebibliography}
%    See note above about multiple indexes.

% ALEXG : \include{biblio}
\bibliographystyle{amsalpha}

%  9. Index
%%%%%%%%%%%%%%%%%%%%%%%%%%%%%%%%%%%%%%%%%%%%%%%%%%%%%%%%%%%%%%%%%%%%

%\printindex
%\include{index}
%\printindex
%\addcontentsline{toc}{chapter}{Index}

\end{document}